\documentclass[11pt,leqno]{article}

\usepackage{mathrsfs,amssymb,amsmath,amsthm,amsfonts, color, mathscinet}
\usepackage[dvips]{graphicx}

\makeatletter
	
	\@addtoreset{equation}{section}
\makeatother

\begin{document}

\renewcommand{\theenumi}{\rm (\roman{enumi})}
\renewcommand{\labelenumi}{\rm \theenumi}

\newtheorem{thm}{Theorem}[section]
\newtheorem{defi}[thm]{Definition}
\newtheorem{lem}[thm]{Lemma}
\newtheorem{prop}[thm]{Proposition}
\newtheorem{cor}[thm]{Corollary}
\newtheorem{exam}[thm]{Example}
\newtheorem{conj}[thm]{Conjecture}
\newtheorem{rem}[thm]{Remark}
\allowdisplaybreaks

\title{Construction of a non-Gaussian and rotation-invariant $\Phi ^4$-measure and associated flow on ${\mathbb R}^3$ through stochastic quantization}

\author{Sergio Albeverio$^\dagger$ and Seiichiro Kusuoka$^*$
\vspace{5mm}\\
\small $^\dagger$Institute for Applied Mathematics and HCM, University of Bonn\\
\small Endenicher Allee 60, 53115 Bonn, Germany \\
\small e-mail address: albeverio@iam.uni-bonn.de\vspace{5mm}\\
\small $^*$Department of Mathematics, Graduate School of Science, Kyoto University,\\
\small Kitashirakawa-Oiwakecho, Sakyo-ku, Kyoto 606-8502, Japan\\
\small e-mail address: {kusuoka@math.kyoto-u.ac.jp}}
\maketitle

\begin{abstract}
A new construction of non-Gaussian, rotation-invariant and reflection positive probability measures $\mu$ associated with the $\varphi ^4_3$-model of quantum field theory is presented. Our construction uses a combination of semigroup methods, and methods of stochastic partial differential equations (SPDEs) for finding solutions and stationary measures of the natural stochastic quantization associated with the $\varphi ^4_3$-model.
Our starting point is a suitable approximation $\mu _{M,N}$ of the measure $\mu$ we intend to construct.
$\mu _{M,N}$  is parametrized by an $M$-dependent space cut-off function $\rho _M: {\mathbb R}^3\rightarrow {\mathbb R}$ and an $N$-dependent momentum cut-off function $\psi _N: \widehat{\mathbb R}^3 \cong {\mathbb R}^3 \rightarrow {\mathbb R}$, that act on the interaction term (nonlinear term and counterterms).
The corresponding family of stochastic quantization equations yields solutions $(X_t^{M,N}, t\geq 0)$ that have $\mu _{M,N}$ as an invariant probability measure.
By a combination of probabilistic and functional analytic methods for singular stochastic differential equations on negative-indices weighted Besov spaces (with rotation invariant weights) we prove the tightness of the family of continuous processes $(X_t^{M,N},t \geq 0)_{M,N}$.
Limit points in the sense of convergence in law exist, when both $M$ and $N$ diverge to $+\infty$.
The limit processes $(X_t; t\geq 0)$ are continuous on the intersection of suitable Besov spaces and any limit point $\mu$ of the $\mu _{M,N}$ is a stationary measure of $X$.
$\mu$ is shown to be a rotation-invariant and non-Gaussian probability measure and we provide results on its support.
It is also proven that $\mu$ satisfies a further important property belonging to the family of axioms for Euclidean quantum fields, it is namely reflection positive.
\end{abstract}

{\bf AMS Classification Numbers:}
81S20, 81T08, 60H15, 35Q40, 35R60, 35K58

 \vskip0.2cm


\section{Introduction}\label{sec:intro}

\subsection{Background}

This paper is concerned with the construction of solutions of a certain stochastic partial differential equation (SPDE) of parabolic type, with singular coefficients, known as stochastic quantization equation (SQE) for the $\varphi ^4_3$-model of relativistic quantum fields on the Euclidean $3$-dimensional space-time ${\mathbb R}^3$.
Particular attention is given to the construction of an invariant probability measure for the process associated with the SQE, that is stationary (i.e. invariant) with respect to the action induced on the state space by the Euclidean group (including reflections)  of rigid transformations of ${\mathbb R}^3$.
The property of Euclidean invariance is particularly important for assuring the relativistic invariance (i.e. invariance under the full Poincar\'e group) of the associated relativistic quantum fields.
In this sense the construction of solutions of the SQE and their invariant measures provides an alternative construction of the $\phi ^4_3$-Euclidean and relativistic fields by other methods than those used in the constructive quantum field theory.

In order to better understand the motivations for our approach, let us shortly recall the origins of the problems that we are dealing with (for more details see also the references given below, and the introduction in our previous paper \cite{AK} and in the paper by Gubinelli and Hofmanov\'a \cite{GuHo2}).
The most simple classical relativistic equation describing scalar waves is the linear wave equation 
\begin{equation}\label{eq:introwave}
\Box \varphi _{\rm cl} (t,\vec{x}) = - m_0^2 \varphi _{\rm cl} (t,\vec{x})
\end{equation}
where $\varphi _{\rm cl}$ is a real valued function from the Minkowski space-time $M^d= {\mathbb R}\times {\mathbb R}^\sigma$, $\sigma \in {\mathbb N}\cup \{ 0\}$, $d:= 1+\sigma$, $\Box = \frac{\partial ^2}{\partial t^2} - \triangle _{\vec{x}}$ being the d'Alembert operator, $(t,\vec{x})\in M^d$, with $t$ standing for time and $\vec{x}$ for space. $m_0>0$ is a mass parameter.
The modification of the equation \eqref{eq:introwave} by some local nonlinear term $-\lambda V'(\varphi _{\rm cl} (t,\vec{x}))$, where $V\in C^1({\mathbb R};{\mathbb R})$ and $\lambda >0$ is a parameter, is the nonlinear Klein-Gordon equation of $M^d$
\begin{equation}\label{eq:introKG}
\Box \varphi _{\rm cl} = - m_0^2 \varphi _{\rm cl} -\lambda V' (\varphi _{\rm cl}) .
\end{equation}
To discuss the quantization of \eqref{eq:introKG} it is useful to consider first the case that $\lambda =0$ and $\sigma =0$.
In this case the quantization of \eqref{eq:introKG} reduces to the quantization of the solution of the equation
\[
\frac{d^2}{dt^2} \varphi _{\rm cl} (t) = - m_0^2 \varphi _{\rm cl} (t) .
\]
This is just Newton's equation describing the evolution in time $t$ of a classical harmonic oscillator.
The quantized version $\varphi (t)$ of $\varphi _{\rm cl}(t)$ can be realized by identifying $\varphi$ at $t=0$ with the multiplication operator by the coordinate in the complex Hilbert space ${\mathcal H}= L^2({\mathbb R},dx)$ and letting evolve $\varphi (0)$ by the unitary group generated by the self-adjoint Hamilton operator in ${\mathcal H}$ given on a dense domain ${\mathcal S}({\mathbb R})$, e.g., by
\[
H= -\frac{1}{2} \frac{d^2}{dx^2} + \frac{m_0^2}{2} x^2 -\frac{m_0}{2}.
\]
One has on a dense domain $\varphi (t) = e^{-itH}\varphi (0) e^{itH}$ ($t\in {\mathbb R}$), where $e^{itH}$ is the unitary group on ${\mathcal H}$ generated by $H$.
Of particular interest for this case (but also for the wanted quantized version of \eqref{eq:introKG}) is the study of the mean value of products of the operators $\varphi$ taken at different times, the average being understood in the sense of scalar products taken in the Hilbert space.
In fact it turns our that it is enough to look at the quantities
\begin{equation}\label{eq:intro3}
(\Omega , \varphi (t_1)\cdots \varphi (t_n) \Omega )
\end{equation}
where $(\cdot , \cdot )$ is the scalar product in ${\mathcal H}$ and $\Omega$ is the (normalized) eigenfunction to the lowest eigenvalue, zero, of $H$ ($\Omega$ is called ``ground state''  or ``vacuum'' state).
A simple calculation shows that $\Omega$ is the function $x \mapsto \left( \frac{m_0}{\pi}\right) ^{1/4} e^{-\frac{m_0}{2}x^2}$ in ${\mathcal H}$.
\eqref{eq:intro3} is a special case (for $\sigma =0$ and $\lambda =0$) of Wightman's functions that were found out to be of particular interest for the study of the quantization of \eqref{eq:introKG} (and more general nonlinear Klein-Gordon equations in the quantum field theory, see, e.g., \cite{Jo1}, \cite[Chapter 3-4]{StWi}, \cite{BLT} and \cite[Chapter 2]{Stro}).

For the construction of a corresponding structure for the case that $\sigma \geq 1$, the first problem we meet is the choice of Hilbert space.
For this it is useful to recall (from non relativistic quantum mechanics) that for the harmonic oscillator the Hilbert space is ${\mathcal H}=L^2({\mathbb R},dx)$ and the quantized version $\varphi (0)$ of $\varphi _{\rm cl}(0)$ acts as multiplication by the coordinate function $x$.
If we try to adapt this directly for quantum fields $\varphi (t,\vec{x})$, $x\in {\mathbb R}^\sigma$, the operator in $\varphi (0,\vec{x})$ will be multiplication by a function of $\vec{x}$ that will belong to some infinite-dimensional space $\Sigma$, but then the measure structure (in analogy with the above $dx$) should be given by a ``volume measure'' on $\Sigma$.
Such an approach meets difficulties, in fact it does not work if $\Sigma$ is a Hilbert space (see, e.g. \cite{AlMa} and references therein for the absence of regular $\sigma$-additive measures on infinite-dimensional Hilbert spaces).
To find out of this ``impasse'', it is useful to realize that $L^2({\mathbb R},dx)$ for the harmonic oscillator can be replaced by a suitable unitary equivalent space.
E.g. the map sending $f\in L^2({\mathbb R},dx)$ into $\frac{f}{\sqrt{\rho}} \in L^2({\mathbb R}, \mu _0)$, where $\mu _0$ is the probability measure on the Borel subsets of ${\mathbb R}$ given by $\mu _0 (dx):= \rho (x) dx$, with $\rho (x):= \left( \frac{m_0}{\pi }\right) ^{1/2} e^{-m_0 x^2}$ ($x\in {\mathbb R}$), does the task.
$\mu _0$ is a centered Gaussian measure (with variance $(2m_0)^{-1}$), and $f \mapsto \frac{f}{\sqrt{\rho}}$ is a unitary map between $L^2({\mathbb R},dx)$ and $L^2({\mathbb R},\mu _0)$.

Note that by this map the Hamiltonian $H$ is unitary equivalent to $H_{\mu _0} = -\frac{1}{2}\frac{d^2}{dx^2} + m_0 x \frac{d}{dx}$ acting on a dense domain in $L^2({\mathbb R}, \mu _0)$.
More precisely, this is the unique self-adjoint positive operator properly associated with the classical Dirichlet form ${\mathcal E}_{\mu _0}(u,v) = \frac{1}{2}\int _{\mathbb R} \frac{du}{dx} \frac{dv}{dx} \mu _0 (dx)$ given by $\mu _0$ (with $u,v$ in a dense domain in $L^2({\mathbb R}, \mu _0)$, see for the relations between Hamiltonians and classical Dirichlet forms, e.g. \cite{AHKS}, \cite[Chapter 1-2]{BeKo1}, \cite[Chapter 6]{BeKo2}).
Note that the function identically equal to $1$ in $L^2({\mathbb R}, \mu _0)$ is an eigenfunction to the eigenvalue zero of $H_{\mu _0}$.
The stationary stochastic process properly associated to ${\mathcal E}_{\mu _0}$ (see \cite{FOT}) is an Ornstein-Uhlenbeck (velocity) process, with a covariance operator given by $\frac{e^{-m_0|t-s|}}{2m_0}$ ($s,t\in {\mathbb R}$), and this is the kernel of the operator $(-\frac{d^2}{dt^2}+m_0^2)^{-1/2}$ (see e.g. \cite{AKKR}).
These observations are appropriate for an extension to the case $\sigma \geq 1$, that we now discuss.
In the latter case $\mu _0$ is replaced by the centered Gaussian measure $\mu _{0, \sigma}$ on ${\mathcal S}'({\mathbb R}^\sigma )$ given uniquely on the basis of Minlos's theorem by its Fourier transform 
\begin{align*}
\hat{\mu} _{0,\sigma} (f) &= \int _{{\mathcal S}'({\mathbb R}^\sigma)} e^{i\langle f, \omega \rangle} \mu _0(d\omega ) \\
&= \exp \left( -\frac{1}{2} \langle f, (-\triangle _\sigma + m_0^2)^{-1/2}f \rangle _{L^2({\mathbb R}^\sigma )} \right)
\end{align*}
with $\triangle _{\sigma}$ the Laplacian in $L^2({\mathbb R}^\sigma , dx)$ and $f\in {\mathcal S}({\mathbb R}^\sigma )$.
$\mu _{0,\sigma}$ is the ground-state measure for the quantum fields $\varphi (0,\vec{x})$ at time zero for a $d$ $(= \sigma +1)$-dimensional space time, $(t,\vec{x})\in {\mathbb R}\times {\mathbb R}^\sigma$.
The Hilbert space in which they act would be $L^2({\mathcal S}'({\mathbb R}^\sigma ),\mu _{0,\sigma })$ (in the case that $\sigma =0$, ${\mathcal S}'({\mathbb R}^\sigma)$ is replaced by ${\mathbb R}$ and $\mu _{0,\sigma}$ by $\mu _0$ as above).
Proceeding on the same line we can define a corresponding quantum Hamiltonian $H_{\mu _{0,\sigma}}$ given by $\mu _{0,\sigma}$ as the unique self-adjoint operator properly associated with the Dirichlet form ${\mathcal E}_{\mu _0}(u,v)= \frac{1}{2}\int \nabla u \cdot \nabla v d\mu _{0,\sigma}$ given by $\mu _{0,\sigma}$ (on its natural domain).
The existence and the well-definedness of the Hamiltonian are proven in \cite{AlHK1} and \cite{AlRo1}.
$H_{\mu _0,\sigma}$ corresponds indeed to the classical Hamiltonian given (modulo an additive constant) by
\[
\frac{1}{2} \int \left| \left. \frac{\partial}{\partial t} \varphi _{\rm cl} (t,\vec{x} ) \right| _{t=0}\right|^2 d\vec{x} + \frac{1}{2} \int \left| \left. \nabla _{\vec{x}} \varphi _{\rm cl} (t,\vec{x} ) \right| _{t=0}\right|^2 d\vec{x} + \frac{m_0^2}{2} \int |\varphi (0,\vec{x})|^2 d\vec{x} .
\]
(Incidentally $L^2({\mathcal S}'({\mathbb R}, \mu _{0,\sigma}))$ is isomorphic in a natural way, via the Friedrichs-Segal isomorphism, to the Fock space for the free time-zero quantum field $\varphi (0,\vec{x})$ (see e.g. \cite[Chapter 1]{BSZ} and \cite[Chapter I]{Si})).

Note that the variance of $\varphi (0,\vec{x})$ is heuristically $(-\triangle _{\sigma} + m_0^2)^{-1/2}$ valuated at $0$, which diverges, hence the time-zero quantum field ``operator $\varphi (0,\vec{x})$'' should be really understood as a ${\mathcal S}'({\mathbb R}^\sigma )$-valued Gaussian random variable with distribution $\mu _{0,\sigma}$:
\[
\varphi (f) := \langle f, \varphi \rangle, \quad f\in {\mathcal S}({\mathbb R}^\sigma),\ \varphi \in {\mathcal S}'({\mathbb R}^\sigma) ,
\]
so that the first constructive approach of models given by \eqref{eq:introKG} was developed in this frame work.
The operator
\[
\varphi (t,f) := e^{-itH_{\mu _{0,\sigma}}} \varphi (f) e^{it H_{\mu _{0,\sigma}}}
\]
is then the time $t$-field operator acting in the Hilbert space $L^2(\mu _{0,\sigma})$, with spatial test-function $f$.
To perturb $H_{\mu _{0,\sigma}}$ by an interaction term a suggestion comes again from the classical Hamiltonian corresponding to \eqref{eq:introKG} that, in the case $\lambda \neq 0$, contains a term $\lambda \int _{{\mathbb R}^\sigma} V(\varphi _{\rm cl}(0,\vec{x})) d\vec{x}$.
Due to the fact that the quantum time-zero field $\varphi (0,\vec{x})$ has to be understood in the distributional sense, an expression like ``$\lambda \int _{{\mathbb R}^\sigma} V(\varphi (0,\vec{x})) d\vec{x}$'', has only a heuristic meaning.
For $\sigma =1$, the standard way out is to ``regularize and renormalize'', substituting the above heuristic expression by ``$\lambda \int _{{\mathbb R}^\sigma} :V(\varphi _\varepsilon (0,\vec{x})): g(\vec{x}) d\vec{x}$'', where $:\cdot :$ is the Wick ordering (with respect to $\mu _{0,1}$ (see e.g. \cite[Chapter I]{Si})), $\varepsilon >0$ means convolution with a smooth function of compact support in momentum space  tending to $1$ as $\varepsilon \downarrow 0$ (``ultraviolet cut-off'') and $g$ is a smooth function of compact support in space (``infrared cut-off''), which one then eventually let converge to $1$.
In this way one gets a perturbed Hamiltonian $H_{\varepsilon ,g}$ as a well-defined self-adjoint operator on $L^2 (\mu _{0,\sigma})$.
(In the case $\sigma =2$ modifications by ``counterterms'' (see below) are needed.)
This is indeed the approach used in the first period of constructive quantum field theory (see e.g. \cite[Chapter 8]{BSZ}, \cite[Chapter 9]{GlJa} \cite{GlJa1}, \cite{GlJa1.5}, \cite{GlJa2}, \cite{GlJa3}, \cite[Chapter IV]{He}, \cite{HK} and \cite[Chapter VII-VIII]{Si}; for the case that $\sigma =0$ see also \cite{CouRen75} and \cite{Stea81}).
For a well-defined limit for $\varepsilon \downarrow 0$, $g\rightarrow 1$ for $\sigma =1$ and for $V$ of lower bounded polynomial type, a probabilistic lower bound for $H_{\varepsilon ,g}$ found by E. Nelson played a crucial role (see \cite{Ne}; see also \cite{NelsonWienerIII} for Nelson's use of probabilistic ideas in quantum theory).
This work was extended to other types of interaction terms $V$, namely an exponential or trigonometric term, for $\sigma =1$ (see \cite{HK}).
For the case where $\sigma =2$ these Hamiltonian methods were extended to the case of $V$ of a fourth power (the $\varphi ^4_3$-model) in \cite{Gl68}.
But, major further developments, both for $\sigma =1$ and $\sigma =2$ came after ``Euclidean methods'' were introduced in the area of constructive quantum field theory.

Physicists like F. Dyson and G. C. Wick (in 1949 resp. 1954-56, see references in \cite{Miller15}) already used Euclidean methods instrumentally, for going from the relativistic ``Feynman propagator'' associated with $(\Box -m_0^2)^{-1}$, appearing in ``perturbative quantum field models'' to the easier to handle Euclidean propagator $(-\triangle _d +m_0^2)^{-1}$, the passage being effectuating by formal analytic continuation from the relativistic time $t\in {\mathbb R}$ to a Euclidean time $\tilde t=-it$, by which $\Box = \frac{\partial ^2}{\partial t^2}-\triangle _{\vec x}$ is mapped into $-\triangle _d = -\frac{\partial ^2}{\partial \tilde t^2}-\triangle _{\vec{x}}$ ($\vec{x}\in {\mathbb R}^\sigma$, $\triangle _d$ being the Euclidean invariant Laplace operator on ${\mathbb R}^d$, $d=\sigma +1$).
Starting from 1956 an axiomatic framework for relativistic quantum field theory (extended from physicist's experience with the theory) was presented (see e.g. \cite{Wi}, \cite[Chapter 3]{StWi}, \cite[Chapter III]{Jo1}, \cite{BLT}, \cite[Chapter 5]{Stro} and reference therein) defined in terms of ``correlation functions'', called Wightman functions. These functions describe correlations expressed by mean values in a certain state, called ``vacuum state'', of $n$-fold products of quantum field operators at space-time $x_i$ ($i=1,2,\dots ,n$) in $d$-dimensional Minkowski space-time ${\mathcal M}^d$ (for any $n\in {\mathbb N}$).
These are required to satisfy minimal properties, next to relativistic invariance, a spectrum condition for the generator of time evolution, a cluster property and a positivity condition (assuring a complex Hilbert space structure for the physical states).
These properties were translated in physical terms by analytic continuation of the (time-ordered) Wightman functions in the time variables $t_k$ to Euclidean time variables $\tilde t_k$ resulting in corresponding correlation functions, called Schwinger functions, of the Euclidean framework in work by J. Schwinger (1951 and 1958) and T. Nakano (1959), see \cite{Miller15} (risp. \cite{Schweber} for the relativistic case) for historical references.

A more mathematically oriented point of view on the Euclidean correlation functions was taken by K. Symanzik, who proposed Euclidean methods to construct relativistic models \cite{Sy} and \cite{Sy2}.
This program was taken up by E. Nelson in 1971, when he was lecturing in Princeton on mathematical themes in relation to \cite{Sy}.
Nelson's work was followed by a series of publications (see \cite{NeProc}, \cite{Ne73c}, \cite{Ne73a} and \cite{Ne73b}).
The ground breaking observation was to look at the time-zero field measures, in particular $\mu _{0,\sigma}$ on ${\mathcal S}'({\mathbb R}^\sigma )$, as time-zero marginal measures of Euclidean invariant measures $\mu _{0, {\rm E}}$ on ${\mathcal S}'({\mathbb R}^d)$ with $d=\sigma +1$ ($\sigma \geq 0$), $\mu _{0,{\rm E}}$ being by definition the probability measure on ${\mathcal S}'({\mathbb R}^d)$ given by its Fourier transform
\begin{align*}
\widehat{\mu _{0,{\rm E}}} (f) &= \int _{{\mathcal S}'({\mathbb R}^d)} e^{i\langle f,\omega \rangle } \mu _{0,{\rm E}} (d\omega )\\
&= \exp\left( -\frac{1}{2} (f, (-\triangle +m_0^2)^{-1}f) _{L^2({\mathbb R}^d)}\right)
\end{align*}
for any $f\in {\mathcal S}({\mathbb R}^d)$, $\triangle := \triangle _d$ being the Laplacian on ${\mathbb R}^d$.
The corresponding coordinate process $\varphi (h):= \langle h, \varphi \rangle$, $h\in {\mathcal S}({\mathbb R}^d)$, $\varphi \in {\mathcal S}'({\mathbb R}^d)$ is a Gaussian random variable on ${\mathcal S}'({\mathbb R}^d)$ with mean zero and variance
\begin{equation}\label{eq:Introcorr}
\int \langle h_1, \varphi \rangle \langle h_2, \varphi \rangle \mu _{0,{\rm E}}(d\varphi) = \left\langle h_1, (-\triangle +m_0^2)^{-1} h_2 \right\rangle _{L^2({\mathbb R}^d)}
\end{equation}
i.e. $\varphi$ is a Gaussian random process (indexed by ${\mathcal S}({\mathbb R}^d)$).
The right-hand side can be looked upon as a scalar product in the Sobolev space $H^{-1}({\mathbb R}^d)$ and one can naturally look upon $\varphi$ as a random process supported by $H^{-1}({\mathbb R}^d)$.
The full Euclidean group $E({\mathbb R}^d)$ on ${\mathbb R}^d$ (generated by translations, rotations and reflections) leaves $\mu _{0,{\rm E}}$ invariant, in the sense that, for any $\beta \in E({\mathbb R}^d)$, one has $\mu (T_\beta (A)) = \mu (A)$ for any Borel subset in $H^{-1}({\mathbb R}^d)$, $T_\beta$ being the pointwise map on $H^{-1}({\mathbb R}^d)$ into itself given by
\[
\varphi (h) \circ T_\beta = \varphi (h\circ \beta ), \quad h\in H^{-1}({\mathbb R}^d).
\]
The fact mentioned above that the measure $\mu _{0,\sigma}$ can be looked upon as time-zero marginal is due to the reflection symmetry and strict Markov property of $\mu _{0,{\rm E}}$ with respect to (the $\sigma$-algebra associated with) the hyperplane $\Pi _0 := \{ (x^1, \vec{x}); x^1=0, \vec{x}\in {\mathbb R^\sigma}\}$ in ${\mathbb R}^d$, $d=\sigma +1$.
To see this, $\mu _{0,{\rm E}}$ being Gaussian, it is enough to examine its covariance, given by \eqref{eq:Introcorr}.
This has a special structure, best seen by observing that (because of the support properties of $\mu _{0,{\rm E}}$) the right-hand side of \eqref{eq:Introcorr} is also finite if $h_i\in {\mathcal S}({\mathbb R}^d)$ ($i=1,2$) are replaced by the generalized functions of the form $\delta _{t_i}(x_i^1) \chi_i (\vec{x})$ ($x_i^1 \in {\mathbb R}$, $\vec{x}_i \in {\mathbb R}^\sigma$, $t_i \in {\mathbb R}$ and $\chi _i \in {\mathcal S}({\mathbb R}^\sigma )$, $i=1,2$), and one has, by Fourier transform and an application of the ``method of residua'' for performing an integration over the first coordinate in $\widehat{\mathbb R}^d$ (the dual of ${\mathbb R}^d$, $\widehat{\cdot}$ used for Fourier transform):
\[
\left\langle h_1, (-\triangle +m_0^2)^{-1} h_2 \right\rangle _{L^2({\mathbb R}^d)} = \int _{\widehat{\mathbb R}^\sigma} \overline{\widehat{\chi _1}} (\vec{p}\, ) \frac{e^{-\sqrt{\vec{p}^{\, 2}+m_0^2} |t_2-t_1|}}{\sqrt{\vec{p}^{\, 2}+m_0^2}} {\widehat{\chi _2}}(\vec{p}\, ) d\vec{p} .
\]
Note that this is positive for $\chi _1=\chi _2$, hence one has on the right-hand side a scalar product in $L^2(\widehat{\mathbb R}^\sigma)$.
From this it is not difficult to deduce that the sharp time random field $\varphi (t,\vec{x})$ looked upon as a stochastic process in $t\in {\mathbb R}$ with state space ${\mathcal S}'({\mathbb R}^\sigma )$ is a symmetric Markov process with continuous paths and invariant measure $\mu _{0,\sigma}$.

Note that $\mu _{0,\sigma}$ has covariance $(-\triangle _\sigma + m_0^2)^{-1/2}$.
Its transition semigroup can be identified with the one generated by the unique self-adjoint operator $H_{\mu _{0,\sigma}}$ in $L^2({\mu _{0,\sigma}})$ associated with the classical Dirichlet form given by $\mu _{0,\sigma}$ (see e.g. \cite{AlbStFlour} and \cite{AlRo1}).
The unitary group generated by $H_{0,\sigma}$ has then the interpretation of time-translation group for the free relativistic quantum field over Minkowski space-time ${\mathcal M}^d$ (corresponding to the Euclidean space ${\mathbb R}^d$, $d=1+\sigma$).
This relies on the explicit form of $(-\triangle +m_0^2)^{-1}$ i.e. the inverse of a local second-order Euclidean invariant operator (see \cite{Ne73a}, \cite[Chapter III]{Si}, \cite{GRS}, \cite{AlHK1}, \cite{Ro1}).
To introduce interactions, the first idea inspired by classical field theory (or the case $\sigma =0$ we discussed above) is to consider a Feynman-Kac or Gibbs type measure
\begin{equation}\label{eq:PhiVmeas}
Z_{\varepsilon ,\Lambda}^{-1} e^{-\int _{{\mathbb R}^d} V(\varphi _\varepsilon (x)) g(x) dx} \mu _{0,{\rm E}}(d\varphi )
\end{equation}
where $\varepsilon >0$ is a regularization (``ultraviolet cut-off'') and $g$ a smooth function with compact support in ${\mathbb R}^d$ (``infrared cut-off''), $Z_{\varepsilon ,\Lambda}$ is a normalization constant.
The passage to the limit $\varepsilon \downarrow 0$ and $g\rightarrow 1$ (besides requiring renormalization) does on one hand at least heuristically guarantee locality of the interaction ($\varepsilon \downarrow 0$) and Euclidean invariance ($g\rightarrow 1$), but it does not guarantee that the strict Markov property of the new measure (in case it can be obtained by such a limit) holds, and that therefore a Hamiltonian operator exists.
In fact the proof of the strict Markov property for the exponential and polynomial interaction models , even for $\sigma =1$, was achieved only some years later in \cite{AlHo}, \cite{AHZ}, \cite{Ze84}, \cite{Gie83} and references therein and also e.g. \cite{Ro88} (for $\sigma =2$ it is still open). 
At this point a basic observation by Osterwalder and Schrader came in.

To understand this it is useful to have in mind that the goal is really to show that the limit one would like to obtain after removal of the cut-offs (i.e. for $\varepsilon \downarrow 0$ and $g\rightarrow 1$) should be such as to obtain a unitary representation of the Euclidean group, e.g. of the group of translation in the time variable of the random field $\varphi _{\varepsilon , g}(t,\vec{x})$ having distribution after removal of the cut-offs in some Hilbert space, the natural candidate being either $L^2({\mathcal S}'({\mathbb R}^d),\mu )$, $\mu$ being the limit of $\mu _{\varepsilon ,g}$ for $\varepsilon \downarrow 0$ and $g\rightarrow 1$, or $L^2({\mathcal S}'({\mathbb R}^d),\mu \upharpoonright \Sigma _0)$, $\mu \upharpoonright \Sigma _0$ being the restriction of the probability measure $\mu$ to the sub-$\sigma$-algebra in ${\mathcal S}'({\mathbb R}^d)$ generated by the time zero field.
$\mu$ should have moments, Schwinger functions, that are then Euclidean invariant.
In addition the inverse operation of analytic continuation (corresponding to the inverse of the one we described before, passing from Wightman to Schwinger functions) should lead to relativistic invariant functions (Wightman functions).
Nelson \cite{Ne}--\cite{NeProc} (see also \cite{Fr} and \cite[Chapter II]{Si}) realized that such an analytic continuation is possible if the limit $\mu$ satisfies a set of ``Euclidean axioms'', including the strict Markov property, and then the Wightman functions so obtained would satisfy the Wightman axioms.
Osterwalder and Schrader's basic observation was that a weaker axioms system for a set of Schwinger functions also leads to the satisfaction of the Wightman axioms.
There are various versions of these axioms, see \cite{OS1}, \cite{OS2}, \cite{O}, \cite{O2}, \cite{Fr}, \cite[Chapter II and IV]{Si}, \cite[Chapter 6 and 12]{GlJa} (see also \cite{Heg} and \cite{Kl} for a special case, intermediate between Nelson's axioms and Osterwalder-Schrader axioms).
For successive discussions of the axioms including their equivalence see \cite{Glaser} and \cite{Zinoviev} and references therein.
An essential point is that they underlined a ``reflection positivity'' axiom as a weaker requirement than the strict Markov property that can be more easily verified in models, being formulated through inequalities that are stable under passage to the limit.

Let us formulate reflection positivity in terms of the Fourier transform of a measure $\mu$ on ${\mathcal S}'({\mathbb R}^d)$.
$\mu$ is said to be reflection positive if the following two requirements are satisfied:
\begin{enumerate}
\item \label{reflectionpositive1}
$\mu$ is reflection symmetric, i.e. for any Borel subset $A$ in ${\mathcal S}'({\mathbb R}^d)$, we have $\mu (A) = \mu (T_\theta (A))$, where $T_\theta (A)$ is the subset of ${\mathcal S}'({\mathbb R}^d)$ obtained from $A$ by the transformation $T_\theta$ induced in ${\mathcal S}'({\mathbb R}^d)$ by the underlying reflection $(x^1,\vec{x}) \mapsto (-x^1, \vec{x})$ in ${\mathbb R}^d$.
\item \label{reflectionpositive2}
it satisfies the property that 
\[
\int _{{\mathcal S}'({\mathbb R}^d)} e^{i\langle \varphi, ^\theta \! f_j - f_k\rangle} \mu (d\varphi )
\]
is a positive definite matrix, for any $N\in {\mathbb N}$ and $\{ f_j \in {\mathcal S} ({\mathbb R}^d); j=1,2,\dots , N\}$ with supports in ${\mathbb R}_+^d = \{ (x^1,\vec{x})\in {\mathbb R}^d; x^1>0\}$.
\end{enumerate}
By \ref{reflectionpositive2} one has a scalar product on the subspace $\Sigma _+$ in $L^2({\mathcal S}'({\mathbb R}^d), \mu)$ spanned by the set of functions $\{ {\mathcal S}'({\mathbb R}^d) \ni \varphi \mapsto e^{i\langle \varphi , h_k \rangle}; h_k \in {\mathcal S}({\mathbb R}^d), {\rm supp} h_k \subset {\mathbb R}_+^d \}$. 
An important result in \cite{OS1}, \cite{OS2} and \cite{O} (see also \cite{Fr}) is that reflection positivity yields already a contraction semigroup $\{ P_t; t\geq 0 \}$ in a Hilbert space ${\mathcal H}$ associated with $\mu$, whose generator will also generate a unitary time translation group in a relativistic framework which we are going to describe shortly.

The passage from Euclidean (correlation) functions having the reflection positivity (RP) property to corresponding relativistic (correlation) functions in essence goes as follows.
Let us assume, as before, that the Euclidean functions are moments of a given probability measure on ${\mathcal S}'({\mathbb R}^d)$, assumed to be reflection positive and Euclidean invariant.
Let $x=(x^1,\vec{x})$ ($x^1\in {\mathbb R}$ and $\vec{x}\in {\mathbb R}^\sigma$), and let $U(t)$ ($t\in {\mathbb R}$) be the unitary group representing translations in direction of $x^1=t$ (thought as ``time'') in $L^2(\mu ):= L^2({\mathcal S}'({\mathbb R}^d),\mu )$.
By the RP condition,  for $F\in {\Sigma}_+$ we have that $\int \ \!\! ^\theta \! F F d\mu \geq 0$.
Let ${\mathcal N}$ be the corresponding null space (i.e. ${\mathcal N}:= \{ F\in \Sigma _+ | \ \int \ \!\! ^\theta \! F F d\mu = 0\}$).
Let ${\mathcal H}$ be the completion of $\Sigma _+| {\mathcal N}$ in the positive scalar product given by $\int \ \!\! ^\theta \! F F d\mu$.
${\mathcal H}$ is the candidate for the Hilbert space for a corresponding relativistic setting.
The canonical projection $j$ satisfies $\| j (F) \| _{\mathcal H} \leq \| F\| _{L^2(\mu )}$, as easily seen.
Let us consider on $j {\mathcal H}$ the $1$-parameter family $P_t$ defined for $t\geq 0$ by $P_t j(F) := j(U(t)F)$ ($F\in \Sigma +$).
Since $U(t) {\mathcal N} \subset {\mathcal N}$, $P_t$ is well-defined and extends uniquely to ${\mathcal H}$, it also easy to see that $P_t$ is a strongly continuous semigroup on ${\mathcal H}$ and moreover, as shown in \cite{OS1}, \cite{OS2} and \cite[Lemma 1.8]{Kl}, $P_t$ is a contraction on ${\mathcal H}$.
The generator $H$ of $P_t$ will be identified with the relativistic Hamiltonian.
Let $\Omega = j ({\mathbf 1})$, where ${\mathbf 1}$ is the identity function in $L^2(\mu )$.
Then $\| \Omega \| _{\mathcal H}=1$ and $P_t \Omega = \Omega$ for all $t\geq 0$.
$\Omega$ is going to be interpreted as relativistic vacuum.
In fact from the above it is not difficult to see that $f_i \in {\Sigma}_0$ ($i=1,2,\dots ,n$) one has, for $t_1\leq t_2 \leq \dots \leq t_n$:
\[
\left( \Omega , \tilde{f}_1 P_{t_2 -t_1} \tilde{f}_2 \cdots P_{t_n -t_{n-1}} \tilde{f}_n \Omega \right) _{\mathcal H} = \int f_{t_1} f_{t_2} \cdots f_{t_n} d\mu ,
\]
where $\tilde{f}_i j(F) = j(f_i F)$ ($F\in {\Sigma}_+$) and $f_{t_i} := U(t_i) f_i$ ($f_i \in \Sigma _0$).
Under the assumption that the moments
\[
S_n(f_1, f_2 , \dots , f_n) := \int \langle \varphi ,f_1\rangle \langle \varphi ,f_2\rangle \cdots \langle \varphi ,f_n\rangle d\mu
\]
of $\mu$ exists even when $f_i \in {\mathcal S}({\mathbb R}^d)$ are replaced by $f_i = \delta _{t_i}\times k_i$ ($i=1,2,\dots , n$, $t_n > t_{n-1}> \cdots > t_1$, $k_i \in {\mathcal S}({\mathbb R}^\sigma)$) one gets under some conditional regularity assumption (\cite[Axiom C]{Fr}) that the Schwinger functions $S_n((t_1,\vec{x}_1), (t_2,\vec{x}_2), \dots ,(t_n,\vec{x}_n))$ coincide with the Wightman functions $W_n((it_1,\vec{x}_1), (it_2,\vec{x}_2), \dots ,(it_n,\vec{x}_n))$ at the Euclidean points $t_j \in {\mathbb R}$, $\vec{x}_j \in {\mathbb R}^\sigma$ with $(t_j,\vec{x}_j) = (t_l,\vec{x}_l)$ for $j\neq l$, the Wightman functions satisfying all Wightman axioms (with the possible exception of uniqueness of the vacuum) (see \cite{Fr}, \cite{OS1}, \cite{OS2}, \cite[Theorem II.12 and II.13]{Si}, \cite{Sum2} and \cite[Theorem 12.1.1]{GlJa}).

Let us now see how this program is concretely implemented in the models we were discussing.
We start by looking again at the regularized (through ultraviolet and space cut-off) probability measure $\mu _{\varepsilon}$ given by \eqref{eq:PhiVmeas}.
The problem is now to find interpretations of $\varepsilon$ and $g$ such that the limit $\varepsilon \downarrow 0$ and $g\rightarrow 1$ of the Schwinger functions as moments of $\mu _{\varepsilon ,g}$ exists and satisfies the axioms for Euclidean Schwinger functions.
To this ``minimal program'' one should add other properties making the set of Wightman functions interesting also from the point of view of the physical interpretation (e.g. the limit measure should not be Gaussian or infinitely divisible, and a nontrivial scattering operator should exist, see e.g. \cite[Section VI.7]{Jo1}).
For scattering theory and its background see \cite[Section 2.6]{Stro}.
This program has succeeded in the cases $d=2$ and $d=3$, for some choices of cut-off and interaction terms $V$.
The most complete picture is of course for the simpler case $d=2$ and we shall first mention some results for this case, limiting ourselves to mainly give some indicative references.
Here $V$ is taken to be a lower bounded polynomial or a superposition of exponential or trigonometric terms, and $V$ is replaced by a corresponding Wick ordered expression.
In all these cases existence, and in some cases uniqueness, of Euclidean measures satisfying all Euclidean axioms and additional properties, have been shown, see e.g. \cite[Chapter 9-12]{GlJa}, \cite[Chapter VIII and IX]{Si} and \cite{Fr} for the polynomial case (called $P(\varphi )_2$-model).
The dependence of the choice of cut-offs has been discussed e.g. in \cite{GRS}, \cite{FrSi77} and \cite{AlLi04}; for uniqueness results on the limit measure see \cite{Fr}, \cite{Guerra72}, \cite{GRS72}, \cite{GRS}, \cite{AHZ}, \cite{AlBoRo}, \cite{AKR} and \cite{AlRo1}.
Also results on the structure of Gibbs states have been obtained \cite{DoMi}, \cite{AlRo1} and \cite{GRS}, together with other results, fully exploiting and adaptation of classical statistical mechanical methods, see also e.g. \cite{NeProc}, \cite{GRS} and \cite{LeNew}.
For results on the $C^\infty$ character in $\lambda$ for $\lambda \in (0,\delta ]$ ($\delta >0$) of the Schwinger functions associated with the limit measure $\mu$, see \cite{Dimock74}, \cite{Dimock76} and \cite{Sum1} (the derivatives being those indicated by perturbation theory).
Asymptotic Borel summable expansion of Schwinger functions in $\lambda$ for $V(y)=y^4$ has been established in \cite{EckMaSe}, as well as for the mass gap see \cite{EcEp}, scattering was discussed in \cite{SpZi} and \cite{DiEc}.
For Borel summability in the case of the $S$-matrix see \cite{OsSen}, \cite{EcEpFr}.
For results concerning semiclassical expansion of the Schwinger functions see \cite{Eck77}, and also e.g. \cite{Donald81}.
Other results are related to the probabilistic structure of $\mu$ looking at its Markov character, associated Dirichlet forms and associated diffusion processes, see \cite{NeProc}, \cite{Ne74}, \cite{GueRug}, \cite{AlHK1}, \cite{Ro1}, \cite{AlGieRu}, \cite{ARY} and \cite{AlbStFlour}.
Other results concern the exponential (H{\o}egh-Krohn or Liouville model) see \cite{HK}, \cite{AlHo}, \cite[Section V.6]{Si}, \cite{FrSei}, \cite{FrPa}, \cite{AHZ} and \cite{AGH}.
For trigonometric models \cite{AHK73}, \cite{FrSei} and \cite{BeGaNi} (see also for hyperbolic case, e.g. \cite{AHaRu2} and \cite{AHaRu}).
A structural result is concerning the measures $\mu$ for all such models for $d=2$ is that $\mu$ can be identified as a Hida distribution (see \cite{AHPRS1}, \cite{AHPRS2}, \cite{ARY} and \cite{HKPS}; see also \cite{AHFL} for a nonstandard analysis approach).

The case of special interest for the present paper is for $d=3$, $V(y)=y^4$ (the only known nontrivial scalar model for which all axioms for $d=3$ have been established).
In this case besides changing $V$ by Wick ordering further modifications by counterterms have been used (see \eqref{eq:intromuge} below in the context of the Euclidean method).

The first work on the construction of the $\varphi ^4_3$-model is due to James Glimm \cite{Gl68}, who proved by Hamiltonian methods in the relativistic case with a space cut-off (i.e. in a bounded region for the space variable in ${\mathbb R}^2$) the existence of a Hilbert space and a densely defined symmetric Hamiltonian, the Hilbert space being different from the original Hilbert-Fock space for the interaction free case.
Other Hamiltonian methods were used in the study of canonical commutation relations for the $\varphi ^4_3$-model in a bounded spatial region in \cite{EcOs71}, where independence (in the sense of unitary equivalence) of the representations with respect to the space cut-off has been proven.
For other work on the $\varphi ^4_3$-model by Hamiltonian methods see also \cite[Chapter IV]{He}.
The non-locally Fock property of the (time zero) canonical commutation relation is proven in \cite{Fab} (see also \cite[Chapter I]{He}).
This was shown later in \cite{AlLi} to imply the singularity of the restriction of the Euclidean $\varphi ^4_3$-measure to the one for the time-zero fields in a relative to the Euclidean free field bounded region measure (see also below for a recent paper by Barashkov and Gubinelli \cite{BaGu2} for the singularity of the non-restricted Euclidean $\varphi ^4_3$-measure with respect to the free field measure in a bounded Euclidean region).
Ultraviolet stability and the positivity of the Hamiltonian (after removal of the ultraviolet cut-off, in a bounded Euclidean space-time volume) was proven in \cite{GlJa73} (see also \cite{GlJa72}).
For other early works on renormalization for the $\varphi ^4_3$-models see \cite{Fed} (and \cite[Chapter 5]{Bat}).
The usefulness of the introduction of Euclidean methods had already been exemplified in the study of the $P(\varphi )_2$-model in the whole ${\mathbb R}^2$-space, see \cite{Guerra72}, \cite{GRS} and \cite{GlJaSp74}.
To discuss Euclidean methods in the case of the $\varphi ^4_3$-model it is useful to look back i.e. the \eqref{eq:PhiVmeas} with $V(\varphi (x))$ with $V(y)=y^4$ ($y\in {\mathbb R}$) replaced by $V_\varepsilon ((\varphi _\varepsilon (x)))$ with $\varphi _\varepsilon$ and ultraviolet regularized $\varphi$, e.g. $\varphi _\varepsilon = \chi _\varepsilon * \varphi$, with $\chi _\varepsilon$ ($\varepsilon >0$) being an $\varepsilon$-approximation of the $\delta _0$-measure on ${\mathbb R}^3$ i.e. $\chi _\varepsilon = \varepsilon ^{-3} h (\varepsilon ^{-1} \cdot )$ ($h\geq 0$, $\int _{{\mathbb R}^3} h(y) dy =1$, $h\in C^\infty _0({\mathbb R}^3 )$) with
\[
\lambda V_\varepsilon (\varphi _\varepsilon (x)) = \lambda \varphi _\varepsilon (x) ^4 + a _\varepsilon (\lambda ) \varphi _\varepsilon (x)^2 
\]
with $\lambda \geq 0$, $a_\varepsilon (\lambda ) = - \alpha \frac{\lambda }{2}\varepsilon ^{-1} - \beta \lambda ^2 \log \varepsilon + \sigma$ (see \cite[Section 23.1]{GlJa}).
Let then 
\begin{equation}\label{eq:intromuge}
\mu _{g,\varepsilon} (d\varphi ) = Z_{g,\varepsilon}^{-1} e^{-\lambda \int _{{\mathbb R}^3} V_\varepsilon (\varphi _\varepsilon (x)) g(x) dx} \mu _0 (d\varphi )
\end{equation}
For ${\mathbb R}^3$ replaced by a torus ${\mathbb T}^3$ in ${\mathbb R}^3$, the existence of a limit $\mu _g$ of $\mu _{g,\varepsilon}$ for $\varepsilon \downarrow 0$ was shown in \cite{Fe}.
The existence of a limit $\mu$ on ${\mathbb R}^3$ for $g\rightarrow 1$ and $\varepsilon \downarrow 0$ for $\sigma$ large (equivalently, $\lambda$ small enough: ``small or weak coupling case'') and the proof that all Wightman axioms are satisfied were given in  \cite{FeOs} and \cite{MaSe76} both using a ``cluster expansion'' method (inspired by classical statistical mechanics).
Let us observe that it is on the basis of these particular results in the weak coupling case that in reality one can legitimately speak of ``the $\varphi ^4_3$-model'' or ``the $\varphi ^4_3$-Euclidean measure''.
Other constructions reach quicker the goal of showing that there is a limit point for the family $(\mu _{\varepsilon , g})_{\varepsilon ,g}$, but they have difficulties either with full Euclidean invariance (see \cite{BrFrSo}; and also \cite{Pa1} and \cite{Pa4}) or reflection positivity (see \cite{AlYo1} and \cite{AlYo2}).

See also for other methods of construction of the $\varphi ^4_3$-measure the approach using wavelet renormalization group \cite{Bat}.
For other partly alternative constructions of the $\varphi ^4_3$-Euclidean measure and corresponding Schwinger function, see also e.g. \cite{BCGNOPS78} (who use a probabilistic approach); \cite{Bal} and \cite{Dimock14} (for a ``block average approach''); \cite{GawKup86} and \cite{BDH} (for methods by renormalization groups) (see also the very nice introduction of \cite{GuHo2}).

For the weak coupling case the simplest and most elegant method to obtain limit points from a lattice versions of the $\varphi ^4_3$-model is the work \cite{BrFrSo} (see also \cite[Chapter 15]{FeFrSo}) where the important skeleton inequalities are proven.
So one gets limit measures that are translation invariant and reflection positive (but not necessarily rotation invariant), at least for $M$ sufficiently large (depending on $\lambda$) (weak coupling limit) (see \cite[Theorem 1.1]{BrFrSo}).
For the full Euclidean invariance on this weak coupling limit an identification with the construction by cluster expansion methods is needed to be applied, which is obtained in the cases of some approximations (see e.g. \cite{FeOs}, \cite{MaSe77} for the construction by cluster expansion methods).
Also for this weak coupling case the existence of a mass gap was shown and the first study of the particle structure was provided \cite{Burnap76}, \cite{SeiSimon} and \cite{BrFrSo}.
The particle structure mathematically means the description of the point spectrum and the continuous spectrum, and physically it explains stable states for particles.  
The asymptotic character of the perturbation expansion in powers of $\lambda$ for the Schwinger functions associated with the Euclidean $\varphi ^4_3$-measure was shown on the basis of the important skeleton inequalities in \cite{BoFe}, and Borel summability for this expansion was established by a phase cell expansion in \cite{MaSe77}.
The fact that there is an isolated one particle state in the spectrum of the Hamiltonian is proven in \cite{Burnap77} and this permitted to apply to this model Haag-Ruelle methods for scattering theory (for the general Haag-Ruelle scattering theory see \cite{Jo1}).
Here, the existence of one particle state mathematically means that the infimum of the spectrum strictly greater than $0$ is an isolated point in the spectrum.
The asymptotic character of the scattering matrix elements was later shown in \cite{Const77}.
That the time ordered Wightman functions and the $S$-matrix elements are $C^\infty$ in $\lambda$ for $\lambda \in [0,\delta ]$ ($\delta >0$) was proven in \cite{EcEp}.
In \cite{FeRa} relativistic equations of motion for the model were discussed.

Results for large values of $\lambda$ have also been established in particular the existence of at least two vacuum vectors, see \cite{FrSiSp}.
Concerning the structure of the space of Gibbs measures very little seems to be known.
Known facts are e.g. the failure of Nelson's symmetry and the vacuum overlap (see \cite{SeiSimon}), and the mutual singularity for different $\lambda$s (see \cite{FeOs}).

%
There is a study of solutions of the particular singular SPDEs (stochastic partial differential equations) appearing in quantum field theory under the name of stochastic quantization equations (SQE).
The origin of the study of SQEs goes back to a very original approach by the physicists Giorgio Parisi and Yongshi Wu \cite{PaWu}, who associated to a candidate classical evolution equation like the one, that we have in \eqref{eq:introKG}, a stochastic differential equation in such a way that the invariant measure for such an equation is precisely a relevant measure of interest connected with the original equation.
The advantage of this procedure is that the solutions of the associated stochastic differential equation can in principle be used to make a Monte-Carlo type simulation for computing quantities associated to the measure of interest.
In recent years, for the study of such singular SPDEs, new methods have been initiated particularly by Martin Hairer (with his theory of regularity structure) and Massimiliano Gubinelli (with his approach to SSPDEs through a theory of paracontrolled distributions).
These methods remarkably work for SQEs, and now the study of SQEs is developing rapidly.
In the case of \eqref{eq:introKG}, keeping in mind that the solution of it, as we discussed above, passed through the construction of the Euclidean measure $\mu$ indicated below by \eqref{eq:PhiVmeas2}, the stochastic quantization equation takes the form of the singular SPDE
\begin{equation}\label{eq:introSQE}
dX_\tau = \left[ (\triangle - m_0^2) X_\tau - \lambda V'(X_\tau) \right] d\tau + dW_\tau , \quad \tau \geq 0,
\end{equation}
where $\tau$ is the additional ``computer time'', $X_\tau$ is a process in $\tau$ depending on the variable $x\in {\mathbb R}^d$, $\triangle$ being the Laplacian in ${\mathbb R}^d$.
$dW_\tau$ is a Gaussian white noise in $\tau$ and in $x$.
In fact, heuristically the form of \eqref{eq:introSQE} suggests an invariant measure precisely of the form of the Euclidean measure $\mu$ given as
\begin{equation}\label{eq:PhiVmeas2}
\mu (d\varphi) = ``Z^{-1} e^{-\lambda \int _{{\mathbb R}^d} V(\varphi (x)) dx} \mu _0 (d\varphi ) ",
\end{equation}
$\mu _0$ being Nelson's free field measure (indeed an invariant measure of \eqref{eq:introSQE} for the case $\lambda =0$, having mean $0$ and covariance suggested by the linear drift term in \eqref{eq:introSQE}; the case $\lambda >0$ can be mentally related to a Feynman-Kac type term perturbing $\mu _0$, associated with $\lambda V$).
The program has been implemented for $d=1$, $2$ and $3$.
In the case $d=1$ let us mention the work by T. Funaki on a bounded interval of ${\mathbb R}$ \cite{Fu}, R. Marcus on ${\mathbb R}$ \cite{Marcus}, and especially Koichiro Iwata \cite{Iwa} (much in the spirit of later work for the case $d=2$).
The latter studies $C({\mathbb R}, {\mathbb R}^n)$-valued processes with multiplicative noise and a rather general nonlinear drift term; in the case where the noise is additive and the drift is given by a gradient field $-\lambda V'$ one has \eqref{eq:introSQE} (with $V$ the sum of a convex $C^1$-function and a $C^2$-function with compact support, e.g. a ``double-well potential''), see also \cite{Iwa2}.
Questions of uniqueness of Dirichlet operators associated with the $P(\varphi )_1$-measure are discussed in \cite{KaRo} and the strong uniqueness is proven (on the full space ${\mathbb R}$).
For previous work on such uniqueness questions see e.g. \cite{AR}, \cite{Eberle}, \cite{AKaR}, \cite{Shi95}, \cite{Takeda} (who proved for the first time Markov uniqueness for models in $d=1$).
For uniqueness results on Gibbs measures that are associated with \eqref{eq:introSQE} for $d=1$, see also \cite{BetzLor}.
See also the survey \cite{AMR} and references therein.
For the case $d=2$ the first results were obtained by replacing \eqref{eq:introSQE} by a more regular noise and modifying accordingly the drift term in order to maintain the same invariant Euclidean measure (we shall henceforth call the equation \eqref{eq:introSQE} with such modification as ``$\mbox{SQE}_\varepsilon$ equation'', $\varepsilon$ being a parameter in the modified noise such that for $\varepsilon \downarrow 0$ this noise becomes Gaussian white noise on ${\mathbb R}^2$).
The first mathematical results on $\mbox{(SQE)}_\varepsilon$ were obtained by G. Jona-Lasinio and S. Mitter \cite{JoMi}.
Weak probabilistic solutions are then obtained for $\mbox{SQE}_\varepsilon$ for suitable initial conditions, both on the $2$-dimensional torus ${\mathbb T}^2$ and on ${\mathbb R}^2$.
Much further work was then done by various authors, see e.g. references in \cite{AlbStFlour}, and Markov uniqueness for SQE or $\mbox{SQE}_\varepsilon$ of the generator was shown in \cite{Takeda}.
Let us mention in particular the work \cite{DPTu} for the $\mbox{SQE}_\varepsilon$ on ${\mathbb T}^2$, where the generator of the Markov semigroup associated with $\mbox{SQE}_\varepsilon$ is shown to be essentially self-adjoint on a natural domain with respect to the Euclidean invariant $P(\varphi )_2$-measure $\mu$ on ${\mathbb T}^2$.
For the case $d=2$ and $\varepsilon =0$, the SQE on ${\mathbb T}^2$ has been solved strongly by probabilistic methods in \cite{DPDe}.
This paper introduced the by now commonly called ``Da Prato-Debussche method'' of splitting the solution $X_\tau$ as $X_\tau - Z_\tau + Z_\tau$, $(Z_\tau , \tau \geq 0)$ being the Ornstein-Uhlenbeck process associated with the linear part of \eqref{eq:introSQE} (modified by counterterms, in correspondence to those in $\mu _{\varepsilon ,g}$) and exploiting then the known property of Wick powers of $Z_\tau$.
These authors also obtained the ergodicity of the solution process for $\mu$-almost every initial data.
In ${\mathbb R}^2$ the SQE has been discussed first by the method of Dirichlet forms in \cite{BChM}, \cite{Mi}, \cite{AlRo1} and \cite{AlRo2} for quasi-every initial conditions.
In \cite{MW2} well-posedness of SQE for $d=2$ in negative Besov spaces was shown both on ${\mathbb T}^2$ and ${\mathbb R}^2$.
The complex of questions on uniqueness of Markov processes associated with the $\varphi ^4_2$- and $P(\varphi )_2$-models and the one on uniqueness of the invariant measure have been further discussed in a series of papers.
In \cite{RZZ2} ergodicity of the process associated with the $\varphi _2^4$-model on ${\mathbb T}^2$ is established as the extremality of the invariant measure (in the set of all $L$-symmetrizing measures, $L$ being the corresponding generator).
In \cite{RZZ} restricted Markov uniqueness for the $P(\varphi )_2$-model both on ${\mathbb T}^2$ and ${\mathbb R}^2$, as well as ergodicity have been proven.
For these results the authors use, in addition to methods of the theory of Dirichlet forms, results in \cite{MW2} on the construction of strong solutions of the SQE for the $P(\varphi )_2$-model in certain Besov spaces with negative indices.
In \cite{HaMatt} the strong Feller property of the semigroup associated to the SQE for the $P(\varphi )_2$-model on ${\mathbb T}^2$ was proven (as related to their proof of the corresponding property for the $\varphi ^4_3$-model on ${\mathbb T}^3$ (and the KPZ equation on ${\mathbb R}\times {\mathbb T}$)).
Related results have been established in \cite{TsWe} where the authors showed in particular that the semigroup associated to the $\varphi ^4_2$-model (on ${\mathbb T}^2$) maps bounded Borel functions into $\alpha$-H\"older continuous functions for some $\alpha \in (0,1)$.
They also showed exponential speed of approximation to the invariant measure.

Let us also mention that a derivation of the SQE for the $\varphi ^4_2$-model on ${\mathbb R}^2$ from a Kac-Ising model has been achieved in \cite{MW17}.
This goes back to work initiated in \cite{FrRu} for $d=1$ and conjectured results in \cite{GLP} for $d=2,3$ (see also \cite{HaIb}).
The SQEs for the exponential/trigonometric models on ${\mathbb R}^2$ have been discussed by Dirichlet forms in \cite{AKaMR} and by semigroup methods, similarly as in \cite{AK} (the latter for the $\varphi ^4_3$-model on ${\mathbb T}^3$).
Strong solutions for the trigonometric model on ${\mathbb R}^2$ have been discussed in \cite{AHaRu} and \cite{AHaRu2} (using Colombeau distributions method, showing in particular the necessity of renormalization) and in \cite{HaSh} by the theory of regularity structures.
The SQE for the exponential model on ${\mathbb R}^2$ has been discussed in \cite{HKK1} and \cite{HKK2}.
Let us also mention that a lot of work has been done in the mass-zero exponential model (called Liouville model), see references in \cite{AKaMR}, relating to methods of \cite{Kusuoka} and  \cite{Kahane}.

In recent years also other methods to perform stochastic quantization for all known scalar field models for $d=1$ and $d=2$ have been developed.
One is the method of dimensional reduction from corresponding supersymmetric models.
In the case of polynomial models this is discussed in \cite{KP83, KLP84} and in \cite{AlDVGu1, AlDVGu2}.
The latter also permits to cover polynomial and exponential models on ${\mathbb T}^d$ ($d\leq 2$) from supersymmetric models in dimension $d+2$, and also an elliptic equation analogue on ${\mathbb T}^d$ ($d\leq 2$).
For the exponential model optimal range of parameters has been achieved in \cite{AlDVGu2}.

Yet another method of stochastic quantization is by the use of nonlocal Dirichlet forms, this has been applied to polynomial models for $d\leq 2$ and to the $\varphi ^4_3$-model (see \cite{AlYo}).
For further references concerning the SQE and $\mbox{SQE}_\varepsilon$ for $d=2$, see \cite{AMR} and \cite{MiRo}.
In a very recent paper \cite{DPDe2} Da Prato and Debussche have proven estimates on the gradient of the transition semigroup $P_t$ for the SQE to the $\varphi ^4_2$-model on ${\mathbb T}^2$ and shown that cylinder functions form a core for the generator of this semigroup, in fact the $P_t$ maps bounded Borel functions into Lipschitz continuous functions.

We shall now pass to describe the situation with the study of the SQE for the $\varphi ^4_3$-model.
The problem of giving a rigorous meaning to solutions of this singular SPDE for this model remained open for quite a long time.
In \cite{ALZ} it was pointed out that the method of Dirichlet forms as minimally defined as not-yet closed quadratic forms on cylindrical functions would have difficulties in its being carried thorough (because the divergences would make it difficult to show closability in spaces of interest).
A breakthrough was realized in work by Martin Hairer on regularity structures \cite{Ha} and in work based on Gubinelli's extension of the method of paracontrolled distributions for singular SPDEs (see \cite{GIP}, \cite{CaCh}, \cite{FuG} and \cite{GuHo}).
Hairer's method is an innovative use of PDE methods on $C^\alpha$-spaces with negative index.
Gubinelli's methods constitute an extensions of T. Lyon's rough paths methods to the case of multidimensional time (see \cite{Gu1} and \cite{Gu2}).
Let us note in passing that since Hairer's work the SPDE for the $\varphi ^4_3$-model is also called,  besides SQE, the equation for the ``dynamical $\varphi ^4_3$-model''.
The methods of Hairer and Gubinelli apply to many other SPDEs besides the one for the $\varphi ^4_3$-model, we shall however concentrate ourselves in mentioning papers directly related to the $\varphi ^4_3$-model (for other applications see, e.g. \cite{BaHo} and \cite{FHSX}).
In the original work of Hairer solutions for the SQE to the $\varphi ^4_3$-model on ${\mathbb T}^3$ were found on any space $C^\alpha$ with $\alpha \in (-2/3, -1/2)$ and with initial conditions in the same space.
Various approximations results for the solutions have been found, from other interaction terms \cite{HaX} or from a lattice approximation \cite{HaMa}.
Other proofs of local in time well-posedness of the SQE on ${\mathbb T}^3$ have been obtained, e.g. by renormalization group methods \cite{Kup}.
Existence and uniqueness local in time of solutions of this model on ${\mathbb T}^3$ have also been obtained in \cite{CaCh} by the method of paracontrolled distributions.
The extension to global solutions on ${\mathbb T}^3$ was suggested in \cite{Ha} and \cite{HaMa} in view of the Bourgain method, and carried through in \cite{MW3}.
We shall discuss this reference below, together with other papers that were done roughly at the same time when our paper \cite{AK} was announced in the arXiv (2017).
Hairer and Mattingly proved in \cite{HaMatt} the strong Feller property of the SQE for ${\mathbb T}^3$, for initial data of suitable regularity.
In our previous paper \cite{AK} we constructed global solutions of SQE on ${\mathbb T}^3$ using semigroup methods, providing also a proof of invariance of the equilibrium measure (in this sense a new construction of a probability measure associated with the $\varphi ^4_3$-model on ${\mathbb T}^3$, the measure having the same form as the $\varphi ^4_3$-model Euclidean measure when reduced to the torus ${\mathbb T}^3$, in fact coinciding with it when a finite-dimensional approximation is considered).
In \cite{AK} the solution to the SQE on ${\mathbb T}^3$ was shown to exist in suitable Besov spaces of negative index.
The limit process has continuous paths and has a translation invariant $\varphi ^4_3$-measure as an invariant measure.
A local in time uniqueness result is expected to hold on the basis of results in \cite{Ha} and \cite{CaCh}.

In \cite{MW3} Mourrat and Weber proved existence and uniqueness of global solutions in time of the SQEs on the $3$-dimensional torus ${\mathbb T}^3 =({\mathbb R}/2\pi {\mathbb Z})^3$.
They also proved the stronger property of ``coming down from infinity'' of the solutions, in the sense that after a finite time the solutions are contained in a compact set in the state space uniformly with respect to the initial condition.
In \cite{GuHo} Gubinelli and Hofmanov\'a proved the existence and uniqueness of solutions global in the Euclidean space ${\mathbb R}^3$ and in computer time $\tau$ of the SQEs, also with an associated ``coming down from infinity'' property.
Their results use the study they perform of the elliptic version of the SQEs on ${\mathbb R}^d$ with $d=4,5$ (that have the same kind of singularities as the stochastic quantization equations for $d=2,3$).
The usefulness of the study of such elliptic equations is inspired by the heuristic dimensional reduction suggested by work of Parisi and Sourlas \cite{PaSo}, partially validated by \cite{KP83, KLP84} in a regularized version of the models.
We already recalled above in our discussion of the SQE for the $\varphi ^4_2$-model that dimensional reduction has been worked out mathematically without regularization for the elliptic case $d=4$ (corresponding to $d=2$ in the stochastic quantization equation case) in \cite{AlDVGu1} (see also \cite{AlDVGu2} and \cite{DVGu} for this reduction technique, based on the use of supersymmetry).
For handling the global problem in ${\mathbb R}^3$ Gubinelli and Hofmanov\'a \cite{GuHo} introduced  Besov spaces with suitable weights.
In order to obtain the necessary estimates to gain control on the behavior at spatial infinity of solutions of the equations, they introduced a new localization theory by which they split distributions belonging to weighted spaces into an irregular component behaving well at spatial infinity and a smooth component growing in space.
In this way the singular SPDEs under study split into two equations, one containing the irregular component in a linear (or almost linear) way, easy to handle in weighted spaces, and another equation containing all regular terms and all the nonlinear terms that can be handled by PDE arguments adapted to the weighted setting.
In the (parabolic) SQE case the uniqueness of solutions is proven essentially by $L^2$-energy estimates, the property of ``coming down from infinity'' (proven first in \cite{MW3}) is established by choosing suitable time dependent weights and fully exploiting energy-time estimates  (in $L^2$-spaces).
The paper also compares the methods and results obtained with the ones presented in \cite{MW3} (concerning the model on the torus ${\mathbb T}^3$).
The global uniqueness problem for the elliptic case remained open.

A subsequent paper \cite{GuHo2} by Gubinelli and Hofmanov\'a presents an alternative construction of a translation invariant $\varphi ^4$-measure (as invariant measure of the SQE) on ${\mathbb R}^3$ by using methods of the theory of singular SPDEs.
Since the results of \cite{GuHo2} are closely related to those in our present paper, we shall describe them in some details in order also to compare them with our methods and results.
Gubinelli and Hofmanov\'a construct an invariant measure $\nu _\lambda$ of the SQE for all values of the coupling constant $\lambda$ as limit points of the tight family $(\nu _{M,\varepsilon})$ (for any $\lambda >0$ and $m_0 >0$) and prove that it satisfies the properties of invariance under translations in the underlying Euclidean space ${\mathbb R}^3$, the important axiom of reflection positivity as well as the regularity axiom of OS (in the form of \cite{OS2}).
These properties are sufficient for the analytic continuation of the moment functions to $\nu _\lambda$ to obtain the corresponding functions in Minkowski space $M^3$ satisfying the axioms of invariance under the translation group in $M^3$, positivity (i.e. Hilbert space structure) and regularity (in the sense of corresponding Wightman functions axioms).
They do not however prove the rotation invariance property for $\nu _\lambda$ (and hence they cannot deduce Lorentz invariance of the Wightman functions).
In addition Gubinelli and Hofmanov\'a prove that their measure $\nu _\lambda$ is non-Gaussian.
They also deduce an integration by parts formula and from this they show that the Dyson-Schwinger equations (derived in \cite{FeRa}, on the basis on previous work on the Euclidean approach to constructive quantum field theory) for the Schwinger functions (moment functions of $\nu _\lambda$) are satisfied.
This construction is very important, since it is streamlined and avoids techniques like cluster expansion, correlation or skeleton inequalities of the constructive quantum field approach.
The main advantage (as also of our construction in our previous paper \cite{AK} on the torus model) is to have reflection positivity assured by starting with measures that are finite-dimensional approximations of the one associated with the SQE of the $\phi ^4_3$-model on a bounded lattice (that is both translation invariant and reflection positive).

The choice of approximation chosen in \cite{GuHo2} is namely of the form 
\begin{align*}
&\nu _{\lambda ,M,\varepsilon} =\nu _{M,\varepsilon} \\
&= Z_{M,\varepsilon}^{-1} \exp \left\{ -2\varepsilon ^d \sum _{\Lambda _{M,\varepsilon}} \left( \frac{\lambda}{4} |\varphi |^4 + \frac{-3\lambda a_{M,\varepsilon} + 3 b^2 b_{M,\varepsilon} +m_0^2}{2}|\varphi | ^2 + \frac{1}{2} |\nabla _{\varepsilon} \varphi |^2 \right) \right\} \\
&\quad \times \prod _{x\in {\Lambda }_{M,\varepsilon}} d \varphi (x) ,
\end{align*}
where $\nabla _{\varepsilon}$ denotes the discrete gradient on the periodic lattice $\Lambda _{M,\varepsilon} = (\varepsilon {\mathbb Z}/M{\mathbb Z})^3$, $M$ is an infrared cut-off and $\varepsilon$ is an ultraviolet cut-off.
$Z_{M,\varepsilon}$ is a normalization constant.
$\nu _{M,\varepsilon}$ is both translation and reflection-positive (the latter can be seen by methods discussed e.g. in \cite[Theorem 7.10.3]{GlJa}and \cite[Theorem VIII.15.]{Si}).
The removal $\varepsilon \downarrow 0$ and $M\rightarrow \infty$ is obtained by choosing appropriately the counterterms $a_{M,\varepsilon}$ and $b_{M,\varepsilon}$.
The essential point is the proof of the tightness of the family $(\nu _{M,\varepsilon})_{M,\varepsilon}$ as embedded in ${\mathcal S}'({\mathbb R}^3)$.
This is achieved by looking at the analogue on the periodic lattice ${\Lambda}_{M,\varepsilon}$ of the SQE for the $\phi ^4_3$-model.
This analogue constitutes a finite-dimensional system of the SDE which is globally well-posed and has $\nu _{M,\varepsilon}$ as its unique invariant measure.
In this sense this corresponds to our method of studying the SQE on ${\mathbb T}^3$ in our previous paper \cite{AK}.
But in the case of \cite{GuHo2} an extension of a renormalized energy method (similar to the one in \cite{AK} for the case of ${\mathbb T}^3$) is used to cope with the case of the weighted Sobolev spaces needed to handle the global problem in ${\mathbb R}^3$.
This extension of the renormalized energy method is based on the previous other work by Gubinelli and Hofmanov\'a (we already mentioned above \cite{GuHo} ), in turn using work in \cite{MaPe} on Besov spaces on the lattices.
As compared to \cite{GuHo} in the paper \cite{GuHo2} there is a more detailed control on the solutions of the SQE, both stationary and non-stationary.
This leads in particular to establish that any accumulation point $\nu _\lambda$ of the tight family $(\nu _{M,\varepsilon})_{M,\varepsilon}$ (for any $\lambda >0$ and $m_0^2 >0$) is translation invariant, reflection-positive and non-Gaussian.
In addition, for any small $\kappa >0$, Gubinelli and Hofmanov\'a prove that there exists $\sigma >0$, $\beta >0$ and $v=O(\kappa) >0$ such that the stretched exponential integrability
\[
\int _{{\mathcal S}'({\mathbb R}^3)} \exp \left\{ \beta \| (1+|\cdot |^2) ^{-\sigma} \varphi \| _{H^{-1/2-\kappa}} ^{1-v}\right\} \nu (d\varphi ) < \infty 
\]
holds.
The authors of \cite{GuHo2} show that this is enough to establish the regularity axiom of \cite{OS2} for the Schwinger functions in the form
\[
|S_n(f_1,f_2,\dots ,f_n)| \leq k^n (n!)^\beta \prod _{i=1}^n \| f_i \| _{H^{1/2+2\kappa} (\rho ^{-2})} .
\]
Moreover, they show that $\nu _\lambda$ satisfies an integration by parts formula leading to the hierarchy of Dyson-Schwinger equations for the Schwinger functions already mentioned above.
Let us also note that by this construction the authors provide an interpretation of the cubic term in the integration by parts as being given by a random distributions (the problem of giving a meaning to that term was mentioned before in \cite{ALZ}).
All results hold for all $\lambda >0$ and $m_0^2 >0$.
Moreover an extension of the construction for a fractional variant of the model is provided, with $-\triangle$ replaced by $(-\triangle )^\gamma$ ($\gamma \in (21/22, 1)$), and their argument is naturally extendable to multicomponent models with $O(N)$-symmetry.

Let us also mention a few other papers that also appeared after our submission of \cite{AK}.
In \cite{MoWe} Mourrat and Weber established results similar to those in Section 4 of \cite{GuHo2}.
The authors proved a priori bounds for the solutions of the $\phi ^4_3$-SQE.
Control on a compact space-time set was obtained independently of boundary conditions.
They pointed out that their bounds could be used in a compactness argument to construct solutions on the full space ${\mathbb R}^3$ and invariant measures.

In \cite{BaGu1} a variational approach for the construction of the finite volume $\phi ^4_3$-measure $\nu$ has been presented and shown to hold.
The method used in that paper is based on methods of stochastic control: in fact that the Laplace transform of an approximation by regularization of the $\varphi ^4_3$-measure $\nu _T$ (with renormalization terms) on the $3$-torus is interpreted as the value function of a stochastic optimal control problem along the flow of the scale regularization parameter $T$.
Under removal of the regularization parameter (i.e. $T\rightarrow \infty$) it is shown by $\Gamma$-convergence (of variational functionals) that a limit probability measure $\nu$ exists whose Laplace transform is also characterized by a variational principle.
This is an important result, providing an explicit formula also for $\nu$ as $\varphi ^4_3$-measure on the $3$-torus.
Let us remark that this approach was also used in successive work in \cite{OhRoSoWa} for a study of the SQE for the hyperbolic Sine-Gordon model on ${\mathbb T}^2$ (for this model, in the Euclidean version, see also the recent work \cite{AKaMR} and references therein).

In \cite{BaGu2} Barashkov and Gubinelli show that the $\phi ^4_3$-measure on ${\mathbb T}^3$ is absolutely continuous with respect to a new measure constructed by a random shift of Nelson's Gaussian free field.
To show this the authors prove a Girsanov type theorem with respect to the filtration generated by a scale parameter.
As a byproduct they prove the long standing conjecture (see e.g. \cite{Fe} and \cite{ALZ}) that the $\phi ^4_3$-measure (on ${\mathbb T}^3$) is singular with respect to Nelson's Gaussian free field measure.
For achieving this a Brownian motion $(W_t)_{t\geq 0}$ with values in ${\mathcal S}'({\mathbb T}^3)$ is considered, as a regularization of Nelson's Gaussian free field measure $\mu$ on ${\mathbb T}^3$ ``at scale $t$''.
The authors also show that the $\phi ^4_3$-measure $\nu$ on ${\mathcal S}'({\mathbb T}^3)$ is the weak limit for $T\rightarrow \infty$ of the path space measure ${\mathbb P}^T$ of $(W_t)_{t\geq 0}$, with density $Z_T ^{-1} \exp (-V_T(W_T))$ with respect to $\mu$ and $V_T(\varphi ):= \lambda \int (\varphi (x)^4 -a_T \varphi (x) + b_T) dx$, where $(a_T,b_T)$ is a family of appropriate diverging (for $T\rightarrow +\infty$) counterterms.
An adapted process $v$ with values in $L^2({\mathbb T}^3)$ is constructed and a Girsanov transformed measure ${\mathbb Q}^v$ such that $W_t$ is a solution of the path dependent SDE with a polynomial drift term of the form $V_t({\tilde W}_s, s\in [0,t])$, with ${\tilde W}_t$ a Gaussian martingale under ${\mathbb Q}^v$.
${\mathbb Q}^v$ can then be looked upon as a natural reference measure for the $\phi ^4_3$-measure on ${\mathbb T}^3$, instead of the free field measure $\mu$, that is shown to be singular both for ${\mathbb Q}^v$ and for $\nu$.
This construction is expected to have interesting applications also to many other areas.

Let us now describe the structure of the present paper.
In Section \ref{sec:main} we present the setting of our work and the main theorems (Theorems \ref{thm:tight2} and \ref{thm:Phi43measure}) followed by a comparison of our work with the results in \cite{GuHo2} and \cite{AK}.
The setting of the present work is exposed at the beginning of Section \ref{sec:main} starting by explaining the notation for our finite-dimensional version $\mu _{M,N}$ (see \eqref{eq:UMN} and \eqref{eq:muMN} below) of the $\phi ^4_3$-measure $\mu$ that we intend to construct.
This contains, in the terminology discussed before in this introduction, both an ``infrared cut-off'' expressed by the smooth non-negative function $\rho _M(x)$ that cut-off the integral in the space variable $x\in {\mathbb R}^3$ in \eqref{eq:UMN} at values $|x| \geq 2M$ ($M\in {\mathbb N}$) and an ``ultraviolet cut-off'' expressed by the presence in \eqref{eq:UMN} of $P_{M,N}$ given by $\rho _M(\cdot ) P_N$, with $P_N$ such that the Fourier transform of its action on a Borel function $f\in L^2({\mathbb R}^3;{\mathbb C})$ amounts to the multiplication by $\psi _N(\xi )$, with $\psi _N$ an even, positive, smooth function on ${\mathbb R}^3$ that vanishes for values in the Fourier transform variable $\xi$ such that $|\xi |\geq 2^{N+1}$.
The measure $\mu _{M,N}$ is thus of Gibbs-type, with a doubly regularized Gibbsian function $U_{M,N}$, and reference measure $\mu _0$, namely Nelson's free field measure with a positive mass $m_0$, i.e. the centered Gaussian measure on ${\mathcal S}'({\mathbb R}^3)$ with the covariance operator $[2(-\triangle +m_0^2)]^{-1}$.
$U_{M,N}$ contains the coupling constant $\lambda \in (0,\lambda _0]$ with $\lambda _0>0$ given, and renormalization terms $C_1^{(N)}$ and $C_2^{(M,N)}$ chosen to be divergent as $M,N\rightarrow \infty$ in such a way as to compensate the divergence of the fourth power term, making possible that $\mu _{M,N}$ converges to the candidate $\mu$ for the $\phi ^4_3$-measure.
The difference between our approximation of $\mu$ as compared to the one used by \cite{GuHo2} is that in our case $C_2^{(M,N)}$ depends on $x \in {\mathbb R}^3$, whereas in the translation invariant approximation used in \cite{GuHo2} the analogue of $C_2^{(M,N)}$ is independent of $x$.
As we will see our choice has the advantage of assuring rotation invariance for the limit point $\mu$.

Our main results are Theorems \ref{thm:tight2} and \ref{thm:Phi43measure}.
Theorem \ref{thm:tight2} asserts that for any $\lambda >0$ and $m_0^2>9/2$ (this can be relaxed to $m_0^2 >0$, see Remark \ref{rem:Zhu}), there is a subsequence $M(k)$ and $N(k)$ for $k\in {\mathbb N}$ of $M,N\rightarrow \infty$ such that $\mu _{M(k), N(k)}$ converges weakly to a limit probability measure $\mu$.
Moreover, $\mu$ is a stationary measure of a continuous process $X$ on a certain intersection of weighted Besov spaces with negative regularity index $-(1/2)-\varepsilon$ ($\varepsilon >0$) and with weight $\nu (x) = (1+|x|^2)^{-\sigma /2}$ ($\sigma \in (3,\infty )$).
$X$ is obtained as limit of the processes $X_t^{M,N}$ given by \eqref{eq:SDEN3} below, with initial law given by $\mu _{M,N}$, as $M,N\rightarrow \infty$.
Theorem \ref{thm:Phi43measure} gives; if the cut-off functions $\rho$ and $\psi$ are chosen to be rotation invariant, then $\mu$ is invariant with respect to rotations and reflections in the underling space ${\mathbb R}^3$.
It is also shown that the support of $\mu$ is contained in the weighted Besov space $B_p^{-(1/2)-\varepsilon} (\nu ^{p/2})$ for all $p\in [2,\infty )$.
Moreover, the square moment estimate \ref{thm:Phi43measure4} in Theorem \ref{thm:Phi43measure} holds.
In addition we show that our limit measure $\mu$ is not Gaussian on ${\mathcal S}'({\mathbb R}^3)$ and is reflection positive.

Our method at present does not lead to growth bounds for the moment functions of $\mu$, as were established in \cite[Theorem 1.1]{GuHo2} for their construction of $\nu _\lambda$.
The reason is that; in the present paper approximation operators $P_N$ and $\rho _M$ are mixed in the stochastic quantization equation (see \eqref{eq:SDEN3} below) of approximation measures, while such approximation operators do not appear in the case of approximations by discretization and tori (see (3.1) in \cite{GuHo2}).
We remark that this difference comes from the approximations.
As in \cite{GuHo2} the cluster property of the limit measure $\mu$ is not discussed as well as the question of whether the corresponding relativistic $\phi ^4_3$ would yield (at least for small values of $\lambda$) nontrivial scattering.
As we recalled above, these additional properties have been proven for the Euclidean measures constructed via the cluster expansion (see \cite{FeOs} and \cite{MaSe76}). 
We expect that our limit measure $\mu$ at least for small values of $\lambda$ will also be translation invariant.
The question about identification of our limit measure with the one constructed in \cite{GuHo2} or, more generally, a uniqueness result for reflection positive, rotation and translation invariant probability measures that are invariant measures for the SQE \eqref{eq:introSQE} and would coincide with the Euclidean measure of \cite{FeOs} and \cite{MaSe76} (at least in the ``weak coupling limit'') have to be left for future investigations.

After the preprint version of the present paper appeared, this topic continues to be studied intensively and a lot of new papers have appeared.
For examples, \cite{AlBoDVGu} about the stochastic quantization of Euclidean fermionic quantum field theories, \cite{ChChHaSh} the stochastic quantization of the Yang-Mills-Higgs model, \cite{BaDV} the elliptic stochastic quantization for the sinh-Gordon quantum field theory and \cite{ShZZ} the perturbation of the $\Phi^4_2$-model.
In particular, the type of approximation in the present paper is applied in \cite{ShZZ} for the $\Phi^4_2$-model to be considered on the whole space.

In the next section we shall provide the precise setting and the statement of the main results in the present paper.

\subsection{Setting and Main theorems}\label{sec:main}

In this section we give the precise setting and the statement of the main results in the present paper.

Denote the spaces of Schwartz-class functions and of tempered distributions by ${\mathcal S}({\mathbb R}^3)$ and ${\mathcal S}'({\mathbb R}^3)$, respectively, and denote this pairing by $\langle \cdot , \cdot \rangle$.
We remark that the pairing $\langle \cdot , \cdot \rangle$ of ${\mathcal S}({\mathbb R}^3)$ and ${\mathcal S}'({\mathbb R}^3)$ is given by the extension of the inner product on $L^2(dx):= L^2({\mathbb R}^3, dx)$.
Let $\psi$ be a non-negative $C^\infty $-function on ${\mathbb R}^3$ such that $\psi (\xi ) =1$ for $|\xi| \leq 1$, $\psi (\xi )=0$ for $|\xi | \geq 2$ and $\psi (\xi ) = \psi (-\xi )$ for $\xi \in {\mathbb R}^3$, and let $\rho$ be a non-negative $C^\infty $-function on ${\mathbb R}^3$ such that $\rho (x) =1$ for $|x| \leq 1$ and $\rho (x)=0$ for $|x| \geq 2$ .

Let ${\mathcal F}$ and ${\mathcal F}^{-1}$ be the Fourier transform and the inverse Fourier transform operators on ${\mathbb R}^3$, respectively (see Section \ref{sec:Besov} for details).
For  $M, N\in {\mathbb N}$ define $P_N$, $P_{M,N}$ and $P_{M,N}^*$ as the mappings from ${\mathcal S}'({\mathbb R}^3)$ to ${\mathcal S}'({\mathbb R}^3)$ given by
\begin{align*}
P_N f (x) &:= {\mathcal F}^{-1} \left[ \psi _N ({\mathcal F} f) \right] (x),\\
P_{M,N} f (x) &:= \rho _M (x) P_N f (x) ,\\
P_{M,N}^* f(x) &:= P_N (\rho _M f) (x), 
\end{align*}
for $f\in {\mathcal S}'({\mathbb R}^3)$, where
\begin{align*}
\psi _N (\xi ) &:= \psi (2^{-N} \xi)  , \quad \xi \in {\mathbb R}^3 \\
\rho _M (x)&:= \rho (M^{-1} x), \quad x\in {\mathbb R}^3 .
\end{align*}
Note that $P_{M,N}^*$ is the dual operator of $P_{M,N}$ with respect to $L^2(dx)$.

Let $\lambda _0\in (0,\infty )$ and let $\lambda \in (0,\lambda _0]$ be fixed.
Let $m_0\in (0,\infty )$ be a fixed constant interpreted as the mass of the underling Euclidean free field.
Define  $C_1^{(N)}$ and $C_2^{(M,N)}$ by
\begin{align*}
C_1^{(N)} &:= \left\langle [2(m_0^2-\triangle )]^{-1} P_N^2 \delta _0, \delta _0 \right\rangle \\
C_2^{(M,N)} (x) &:= 2\sum _{i,j=-1}^\infty {\mathbb I}_{[-1,1]}(i-j) \\
&\qquad \times \int _0^\infty e^{t(\triangle -m_0^2)} P_N^2 \left[ \left( \Delta _i P_N \left[ \rho _M^2 V_N(t) \left( \rho _M^2 \Delta _j P_N \delta _x \right) \right] \right) ^2 \right] (x) dt
\end{align*}
where $\delta _x$ is the Dirac delta function with support at $x$, $\triangle$ is the Laplacian on ${\mathbb R}^3$, $\{ \Delta _j; j\in {\mathbb N}\cup \{ -1,0\} \}$ is the (Littlewood-Paley) nonhomogeneous dyadic blocks (see Section \ref{sec:Besov} for details), and
\[
V_N (t) := P_N ^2 [2(m_0^2-\triangle )]^{-1} e^{t(\triangle -m_0^2)}.
\]
Define a function $U_{M,N}$ on ${\cal S}'({\mathbb R}^3)$ by
\begin{equation}\label{eq:UMN}\begin{array}{l}
\displaystyle U_{M,N}(\phi ) \\
\displaystyle := \int _{{\mathbb R}^3} \left\{ \frac{\lambda}{4} (P_{M,N} \phi ) (x)^4 - \frac{3\lambda }{2}\left( C_1^{(N)} -3\lambda C_2^{(M,N)}(x)\right) \rho _M (x)^2 (P_{M,N} \phi )(x)^2 \right\} dx ,
\end{array}\end{equation}
for $\phi \in {\mathcal S}'({\mathbb R}^3)$, and consider the probability measure $\mu _{M,N}$ on ${\cal S}'({\mathbb R}^3)$ given by
\begin{equation}\label{eq:muMN}
\mu _{M,N} (d\phi ) = Z_{M,N} ^{-1} \exp \left( -U_{M,N}(\phi ) \right) \mu _0 (d\phi) ,
\end{equation}
where $Z_{M,N}$ is the normalizing constant (so that the total mass of $\mu _{M,N}$ is $1$), and $\mu _0$ is the free field measure with a positive mass $m_0$, i.e. the centered Gaussian measure on ${\cal S}'({\mathbb R}^3)$ with the covariance operator $[2(-\triangle + m_0^2)]^{-1}$.
We remark that $\{ \mu _{M,N}\}$ is an approximation sequence for the $\Phi ^4_3$-measure, which will be constructed below as an invariant probability measure of the flow associated with the stochastic quantization equation, and $C_1^{(N)}$ and $C_2^{(M,N)}(x)$ are provided for renormalization.
We also remark that $C_1^{(N)}$ is independent of $M$ and $x\in {\mathbb R}^3$, while $C_2^{(M,N)}(x)$ depends on $M$ and $x\in {\mathbb R}^3$.
The dependence of the second renormalization term $C_2^{(M,N)}(x)$ on $M$ and $x$ comes from our approximation in space by $\rho _M$.
Let us remark that for a torus approximation of ${\mathbb R}^3$ as in \cite{GuHo2} the analogue of $C_2^{(M,N)}$ can be taken to be independent of $x$, because of translation invariance.
This is however not the case with the choice of our approximation by $\rho _M$.

Consider the stochastic partial differential equation (SPDE)
\begin{equation}\label{eq:SDEN3}\left\{ \begin{array}{rl}
\displaystyle \partial _t X_t^{M,N}(x)
&\displaystyle \!\!\! = \dot{W}_t(x) - (-\triangle +m_0^2) X_t^{M,N}(x) \\[3mm]
&\displaystyle \quad \!\!\! - \lambda P_{M,N}^* \left\{ (P_{M,N} X_t^{M,N})^3 (x) \right. \\
&\displaystyle \quad \!\!\! \hspace{1.5cm} \left. -3 \left( C_1^{(N)} -3\lambda C_2^{(M,N)}(x)\right) \rho _M (x)^2 P_{M,N} X_t^{M,N}(x) \right\} ,\\
\displaystyle X_0^{M,N}(x)
&\displaystyle \!\!\! = \xi _{M,N} (x), \quad t\in [0,\infty ),\ x\in {\mathbb R}^3
\end{array}\right. \end{equation}
where the initial condition $\xi _{M,N}$ is a ${\cal S}'({\mathbb R}^3)$-valued random variable with law $\mu _{M,N}$ and $\dot{W}_t (x)$ is a (centered) Gaussian white noise (in the variable $t$ and in the space variable $x\in {\mathbb R}^3$) independent of $\xi _{M,N}$.
The equation \eqref{eq:SDEN3} comes from the stochastic quantization equation corresponding to the measure $\mu _{M,N}$ with initial condition $\xi _{M,N}$.
From now on we shall assume that
\begin{equation}\label{eq:m0}
m_0^2 > \frac{9}{2} .
\end{equation}
This assumption can be relaxed to $m_0^2>0$ by modifying $\nu$ defined below. 
For detail, see Remark \ref{rem:Zhu} below.
Here and below, we shall also use weighted Besov spaces.
Let $\nu$ be a smooth function on ${\mathbb R}^3$ such that $\nu$ is positive everywhere, and denote the $L^p$-space and Besov spaces with weight function  $\nu$ by $L^p(\nu )$ and $B_{p,r}^s(\nu )$, respectively (see Section \ref{sec:Besov} for the details of the definition of weighted Besov spaces).

Let us take $\nu (x) := (1+|x|^2)^{-\sigma /2}$ for $x\in {\mathbb R}^3$ and some $\sigma \in (3,\infty )$.
Then, $\nu \in L^1(dx)$ and the unique solution $X_t^{M,N}$ to \eqref{eq:SDEN3} is a $B_2^{-\alpha}(\nu )$-valued stationary Markov process for $\alpha \in (1/2,\infty)$ (see Theorem \ref{thm:invN}).

In the setting above, in particular the above choice of $\nu$, we have the following main results.

\begin{thm}\label{thm:tight2}
Let $\{ M_N; N\in {\mathbb N}\}$ be a given ${\mathbb N}$-valued sequence such that $\lim _{N\rightarrow \infty} M_N =\infty$ and
\[
\lim _{N\rightarrow \infty} M_N 2^{-\delta N} =0
\]
for all $\delta \in (0,1]$.
Then, for $\varepsilon \in (0, 1/16]$, there exists a sequence $\{ X^{M_{N(k)},N(k)}; k\in {\mathbb N}\}$ such that $\lim _{k\rightarrow \infty} N(k) =\infty$ and $\{ X^{M_{N(k)},N(k)}; k\in {\mathbb N}\}$ converges in law on $C([0,\infty ); {B_{4/3}^{-1/2-\varepsilon }(\nu )}\cap B_{12/5}^{-1/2-\varepsilon} (\nu ^{9/5}))$.
Moreover, if $X$ is the limit in law of $\{ X^{M_{N(k)},N(k)}; k\in {\mathbb N}\}$ on $C([0,\infty ); B_{4/3}^{-1/2-\varepsilon }(\nu ) \cap B_{12/5}^{-1/2-\varepsilon} (\nu ^{9/5}))$, then $X$ is a continuous stationary process on $B_{4/3}^{-1/2-\varepsilon }(\nu ) \cap B_{12/5}^{-1/2-\varepsilon} (\nu ^{9/5})$ and the limit probability measure $\mu$ in the weak convergence sense of the associated subsequence $\{ \mu _{M_{N(k)}, N(k)}\}$ is a stationary measure of $X$.
\end{thm}

The assumption on $M_N$ is provided for $M_N$ to diverge sufficiently slowly compared with $N$.
Since $\{ \mu _{M,N}\}$ is a proper two-parameter approximation sequence of the $\Phi ^4_3$-measure on ${\mathbb R}^3$, the measure $\mu$ constructed in Theorem \ref{thm:tight2} is our construction of the $\Phi ^4_3$-measure (that we shall call henceforth for simplicity the $\Phi ^4_3$-measure).
The stochastic process $X$ constructed in Theorem \ref{thm:tight2} is a flow associated to $\mu$.
Moreover, we are able to show that the $\Phi ^4_3$-measure $\mu$ constructed in Theorem \ref{thm:tight2} has the following properties.

\begin{thm}\label{thm:Phi43measure}
Let $\mu$ be the measure obtained in Theorem \ref{thm:tight2}.
Then, $\mu$ has the following properties:
\begin{enumerate}

\item \label{thm:Phi43measure2}
If the functions $\psi$ and $\rho$, which are used to define the approximation of the $\Phi ^4_3$-measure, are chosen to be radially symmetric, then $\mu$ is rotation and reflection invariant, i.e. 
\[
\int f(\phi ) \mu (d\phi ) = \int f( \theta \phi ) \mu (d\phi ) , \quad f\in C_b ({\mathcal S}'({\mathbb R}^3)), \ \theta \in {\rm O}(3)
\]
where ${\rm O}(3)$ is the orthogonal group on ${\mathbb R}^3$ and $(\theta \phi )(x) := \phi (\theta x)$ for $x\in {\mathbb R}^3$.

\item \label{thm:Phi43measure5}
For $f\in {\mathcal S}({\mathbb R}^3)$, define $f^\dagger \in {\mathcal S}({\mathbb R}^3)$ by $f^\dagger (x_1,x_2,x_3):= f(-x_1,x_2,x_3)$ for $(x_1,x_2,x_3) \in {\mathbb R}^3$.
If $\psi^\dagger = \psi $ and $\rho ^\dagger  =\rho$, then $\mu$ has the reflection positivity property i.e. 
\[
\int \overline{F^\dagger (\phi )} F (\phi ) \mu (d\phi ) \geq 0
\]
( $\overline{\, \cdot \,}$ means complex conjugation) for any $F\in C_b({\mathcal S}'({\mathbb R}^3); {\mathbb C})$ given by
\[
F(\phi ) = \sum _{j=1}^n a_j e^{\sqrt{-1}\langle f_j , \phi \rangle}
\]
for some $a_1,a_2, \dots , a_n \in {\mathbb C}$, $f_1, f_2 ,\dots , f_n\in {\mathcal S}({\mathbb R}^3_+)$ and $n\in {\mathbb N}$, where ${\mathcal S}({\mathbb R}^3_+)$ is the set of functions $f \in {\mathcal S}({\mathbb R}^3)$ with support in $\{ (x_1,x_2,x_3); x_1> 0\}$, and $F^\dagger$ is defined by
\[
F^\dagger (\phi ) = \sum _{j=1}^n a_j e^{\sqrt{-1}\langle f_j^\dagger , \phi \rangle} .
\]

\item \label{thm:Phi43measure4}
For each $p\in [2,\infty )$
\[
\int \| \phi \| _{B_p ^{-1/2-\varepsilon }(\nu ^{p/2})}^2 \mu (d\phi) <\infty .
\]
In particular, the support of $\mu$ is included in $B_p ^{-1/2-\varepsilon }(\nu ^{p/2})$ for all $p\in [2,\infty )$.

\item \label{thm:Phi43measure3}
$\mu$ is not a Gaussian measure on ${\mathcal S}'({\mathbb R}^3)$.
Here, the definition that a measure $\mu$ on ${\mathcal S}'({\mathbb R}^3)$ is a Gaussian measure is that for any $n\in {\mathbb N}$ and $f_1, f_2 ,\dots , f_n\in {\mathcal S}({\mathbb R}^3)$, $(\langle f_1,\phi \rangle , \langle f_2,\phi \rangle , \dots , \langle f_n,\phi \rangle )$ has an $n$-dimensional Gaussian distribution under $\mu$.

\end{enumerate}
\end{thm}

Theorem \ref{thm:Phi43measure} implies that if $\psi$ and $\rho$ are radially symmetric, our $\Phi ^4_3$-measure constructed in Theorem \ref{thm:tight2} is nontrivial (i.e. non-Gaussian), rotation invariant and reflection positive.

\begin{rem}\label{rem:Zhu}
The assumption \eqref{eq:m0} is relaxed to $m_0^2>0$ by replacing $\nu (x)= (1+|x|^2)^{-\sigma /2}$ with $\nu _{a}(x)= (1+a |x|^2)^{-\sigma /2}$ for some $a\in (0, 2m_0^2/9)$.
We introduced \eqref{eq:m0} to choose $\sigma$ so that \eqref{eq:assm0} holds, and \eqref{eq:assm0} is applied in the proof of Proposition \ref{prop:eng}.
However, by the replacement $\nu$ by $\nu _a$ for such $a$, we are able to choose $\sigma$ so that 
\[
9< \sigma ^2 < \frac{2m_0^2}{a^2}
\]
instead of \eqref{eq:assm0}.
This enable us to obtain \eqref{eq:propeng02} with some positive constant $c$, because
\[
|\nabla \log \nu _a(x)| = \frac{\sigma a |x|}{1+a|x|^2} \leq  \frac{\sigma \sqrt{a}}{2}.
\]
We got this technique from Professors Rongchan Zhu and Xiangchan Zhu.
\end{rem}

As compared with the results in \cite{GuHo2} we have the same statement of non-Gaussianity of the limit measure $\mu$ and its reflection positivity, but in addition we have also the rotation (and reflection) invariance of the limit measure, however we do not discuss here translation invariance (a property assured instead in \cite{GuHo2} for their limit measure).
The main difference between our method of approximation and the discrete approximation by Gubinelli and Hofmanov\'a that uses a doubly scaled torus, consists in the fact that we avoid the use of a torus but have instead an approximation by a localization of the interaction (by having a space cut-off).
We would have to regain the translation invariance in the limit, but we keep the rotation invariance all the way along, as opposite to Gubinelli and Hofmanov\'a.
Also the interplay of cut-off and weights are reflected by technical differences in the construction of the relevant Besov spaces in our construction and in the one of Gubinelli and Hofmanov\'a \cite{GuHo2}.
Moreover, also the use of the estimates on the approximations of the Ornstein-Uhlenbeck processes and their powers that we need are similar but different from the torus estimates of \cite{GuHo2}.
Here we also use a technique similar to the one we developed in \cite[Theorem 4.1]{AK} to transform \eqref{eq:SDEN3} to another partial differential equation with random coefficients which can be take limit as $M,N\rightarrow \infty$, by introducing a pair of random variables $(\xi _{M,N}, \zeta)$ with law invariant for the system $(Y_t^{M,N}, Z_t)$, where $Z_t$ is the Ornstein-Uhlenbeck process discussed in Section \ref{sec:OU}.
We consider the SPDE \eqref{eq:SDEN3} for a process $X_t^{M,N}$ with initial condition $\xi _{M,N}$ and a Gaussian white noise independent of $(\xi _{M,N},\zeta )$ (the same as the one used for the definition of $Z_t$).
We then consider $X_t^{M,N} -Z_t$ and write an SPDE for it \eqref{eq:PDEX1} below, containing the terms ${\mathcal Z}_t^{(k,N)}$ discussed in Section \ref{sec:OU}.
And we iterate the procedure, so to get estimates suitable for applying the theory of paraproducts (as first applied to singular SPDEs in work by Gubinelli and coworkers).
We get part of the necessary estimates following our previous work \cite{AK}, but we meet the problem that our approximation, by the use of $P_N$, is not symmetric on $L^2(\nu )$.
We overcome this problem by controlling the relevant commutators between the weight functions for the Besov spaces and the approximations we use.

The organization of the other sections of the present paper is as follows.
Section \ref{sec:Besov} is provided in order to prepare the function spaces and the inequalities which we apply in the proofs of the main theorems.
In Section \ref{subsec:Besov}, we survey weighted Besov spaces and paraproducts.
In Section \ref{subsec:approximation}, we present some properties of our approximation operators and commutator estimates.
These properties imply that our approximation is sufficiently good for applying methods for singular stochastic partial differential equations.
We remark that similar properties are also discussed in \cite{GuHo2} within a different setting.

In Section \ref{sec:OU} we provide the convergence and estimates for some polynomials of the infinite-dimensional Ornstein-Uhlenbeck processes, which appear in the transformation in Section \ref{sec:trans} of the stochastic quantization equations with respect to the approximation measures of the $\Phi^4_3$-measure.
In the present paper, we are concerned with approximations not on a torus, but on the whole space ${\mathbb R}^3$.
So we discuss the estimates of the polynomials of the Ornstein-Uhlenbeck processes on weighted Besov spaces by using the heat semigroup and corresponding Green kernel in the underlying space ${\mathbb R}^3$, while a similar argument was done in \cite{AK} by the Fourier expansion on the torus ${\mathbb T}^3$.

In Section \ref{sec:trans}, similarly as in our previous paper \cite{AK}, we consider stochastic quantization equations with respect to the approximation measures of the $\Phi^4_3$-measure.
The handling of these equations is based on methods related to the theory of the regularity structure and the paracontrolled calculus, and techniques to handle the singular stochastic partial differential equations with locally Lipschitz coefficients.

In Section \ref{sec:asymp} we prepare some estimates for the proofs in Section \ref{sec:tight} on the bounds of the solutions to the approximation equations for the limit-behaviour in the parameter for the smoothing.
The estimates obtained in this section imply that under the localization of interaction, the solutions to the approximation equations are asymptotically the same in the parameter for the smoothing as for the Ornstein-Uhlenbeck process. 
The estimates in this section did not appear in \cite{AK}; we deduce and apply them to control the commutators between the weight functions of weighted Besov spaces and the approximation operators in the next section.

In Section \ref{sec:tight} we give a uniform estimate for some functionals of the solutions to the approximation equations and finally prove Theorem \ref{thm:tight2}.
Many parts of the proof are similar to those in \cite{AK}.
However, as mentioned above, unlike \cite{AK} the commutators between the weight functions of weighted Besov spaces and the approximation operator appear.
In Section \ref{sec:tight} we omit the proofs of the estimates which are same as those in \cite{AK}, and only concentrate ourselves on the estimates that are different from those in \cite{AK}, because of the differences in the frameworks.

In Section \ref{sec:propertiesPhi43} we prove the properties in Theorem \ref{thm:Phi43measure} of our $\Phi ^4_3$-measure.
As proved in Section \ref{subsec:invariance}, we have the rotation and reflection invariance as an advantage of our approximation of the $\Phi ^4_3$-measure.
In Section \ref{subsec:RP}, \ref{subsec:support} and \ref{subsec:nontrivial} we prove the reflection positivity, a result on the support of our $\Phi ^4_3$-measure and prove the non-Gaussian character of the constructed measure, respectively.

\vspace{5mm}\noindent
{\bf Acknowledgment.}
The authors are very grateful to Professor Massimiliano Gubinelli and Francesco De Vecchi for stimulating discussions on topics related to this work, to Professors Rongchan Zhu and Xiangchan Zhu for giving us the technique in Remark \ref{rem:Zhu}, and also to Hirotatsu Nagoji for finding mistakes of early versions.
The second author would like to thank the Isaac Newton Institute for Mathematical Science for an invitation and hospitality during the programme ``Scaling limit, rough paths, quantum field theory''.
He is also grateful to the Institute of Applied Mathematics, and the Hausdorff Center of Mathematics of the University of Bonn for hospitality
The second author is also grateful to Professors David Brydges, Martina Hofmanov\'a, Masato Hoshino, Yuzuru Inahama, Hiroshi Kawabi, Yoshio Tsutsumi and Baris Ugurcan for helpful discussions and valuable comments.
The first named author gratefully acknowledges financial support by CIB (where our collaboration on \cite{AK} and the present paper was initiated), and the second named author gratefully acknowledges support by the JSPS grants KAKENHI Grant numbers 25800054, 17K14204 and 21H00988.
The authors also thank to the anonymous referee for one's careful reading and useful comments. 

\section{Weighted Besov spaces, paraproducts and estimates of functions}\label{sec:Besov}

\subsection{Weighted Besov spaces and paraproducts}\label{subsec:Besov}

Let $dx$ be the 3-dimensional Lebesgue measure on ${\mathbb R}^3$.
Let $L^p(\nu )$ be the $p$th-order integrable function space with respect to the measure $\nu (x) dx$, where $\nu$ is a positive Borel function on ${\mathbb R}^3$ and $p\in [1,\infty ]$.
We remark that in the present paper the $L^\infty (\nu)$-norm is defined by
\[
\| f\| _{L^\infty (\nu )} := \mathop{\rm ess.sup} _{x\in {\mathbb R}^3} |f(x)| \quad \mbox{with respect to}\ dx .
\]
Let $\chi$ and $\varphi$ be functions in $C^\infty ([0,\infty );[0,1])$ such that the supports of $\chi$ and $\varphi$ are included by $[0,4/3)$ and $[3/4, 8/3]$ respectively, and that
\[
\chi (r ) + \sum _{j=0}^\infty \varphi (2^{-j}r ) =1, \quad r\in [0,\infty ).
\]
We remark that $\chi$ and $\varphi$ satisfy that
\begin{align*}
&\varphi (2^{-j}r )  \varphi (2^{-k}r ) =0, \quad r \in [0,\infty ),\ j,k \in {\mathbb N}\cup \{ 0\} \ \mbox{such that}\ |j-k| \geq 2,\\
&\chi (r ) \varphi (2^{-j}r ) =0, \quad r \in [0,\infty),\ j \in {\mathbb N}.
\end{align*}
For the existence of $\chi$ and $\varphi$, see Proposition 2.10 in \cite{BCD}.
Throughout this paper, we fix $\chi$ and $\varphi$, and we do not mention explicitly this dependence.

Let ${\mathcal S}({\mathbb R}^3)$ and ${\mathcal S}'({\mathbb R}^3)$ be the Schwartz (test function) space and the space of tempered distributions on ${\mathbb R}^3$, respectively.
We define the (Littlewood-Paley) nonhomogeneous dyadic blocks $\{ \Delta _j; j\in {\mathbb N} \cup \{ -1, 0\}\}$ by
\[
\begin{array}{lll}
\Delta _{-1} f (x)&= \left[ {\mathcal F}^{-1} \left( \chi (|\cdot |) {\mathcal F} f \right) \right] (x), & x\in {\mathbb R}^3 \\
\Delta _j f (x)&= \left[ {\mathcal F}^{-1} \left( \varphi (2^{-j} |\cdot |){\mathcal F} f \right) \right] (x), & x\in {\mathbb R}^3, \ j \in {\mathbb N}\cup \{ 0\} ,
\end{array}
\]
where ${\mathcal F}$ and ${\mathcal F}^{-1}$ are the Fourier transform and inverse Fourier transform operators respectively, i,e. $\mathcal F$ is the automorphism of ${\mathcal S}'({\mathbb R}^3)$ given by the extension of the map
\[
g \mapsto \widehat g (\xi )= \frac{1}{(2\pi)^{3/2}}\int _{{\mathbb R}^3} g(x) e^{-\sqrt{-1} x\cdot \xi} dx , \quad g\in {\mathcal S}({\mathbb R}^3),
\]
where $x\cdot \xi := \sum _{j=1}^3 x_j \xi _j$ for $x=(x_1,x_2,x_3), \xi =(\xi _1, \xi _2, \xi _3) \in {\mathbb R}^3$, and ${\mathcal F}^{-1}$ is the inverse operator of ${\mathcal F}$, respectively (see Section 1.2 in \cite{BCD}).
As a family of pseudo-differential operators, $\{ \Delta _j; j\in {\mathbb N} \cup \{ -1, 0\}\}$ is given by
\[
\Delta _{-1} f = \chi \left( \sqrt{-\triangle} \right) f, \quad \Delta _j f = \varphi \left( 2^{-j} \sqrt{-\triangle} \right) f \quad j \in {\mathbb N}\cup \{ 0\}
\]
where $\triangle$ is the Laplace operator on ${\mathbb R}^3$.

We define the weighted Besov norm $\| \cdot \| _{B_{p,r}^s(\nu )}$ and the weighted Besov space $B_{p,r}^s (\nu)$ on ${\mathbb R}^3$ with $s \in {\mathbb R}$ and $p,r \in [1,\infty]$ by
\begin{align*}
\| f \| _{B_{p,r}^s (\nu)} &:= \left\{ \begin{array}{ll}
\displaystyle \left( \sum _{j=-1}^\infty 2^{jsr} \| \Delta _j f \| _{L^p(\nu)}^r \right) ^{1/r} , & r\in [1,\infty) ,\\
\displaystyle \sup _{j \in {\mathbb N}\cup \{ -1,0\}} 2^{js} \| \Delta _j f \| _{L^p(\nu )} , & r= \infty ,
\end{array} \right. \\
B_{p,r}^s(\nu )& := \{ f\in {\mathcal S}'({\mathbb R}^3); \| f \| _{B_{p,r}^s(\nu )} <\infty \} .
\end{align*}
It is easy to see that if the total mass of the measure $\nu (x) dx$ is finite, we have $B_{p_1,r_1}^{s_1}(\nu ) \subset B_{p_2,r_2}^{s_2} (\nu )$ for $s_1,s_2 \in {\mathbb R}$ and $p_1,p_2, r_1,r_2\in [1,\infty ]$ such that $s_1\geq s_2$, $p_1\geq p_2$ and $r_1\leq r_2$.
Similarly to the Besov spaces without weights, one has that $B_{p,\infty}^s (\nu ) \subset B_{p,1}^{s'} (\nu )$ for $p\in [1,\infty ]$ and $s,s' \in {\mathbb R}$ such that $s'<s$
(see Corollary 2.96 in \cite{BCD}).

\begin{rem}
Our definition of weighted Besov spaces is different from those in \cite{GuHo}, but the same as in \cite{MW2}.
\end{rem}

Below we often assume that $\nu (x)= (1+|x|^2)^{-\sigma /2}$ for some $\sigma \in [0,\infty )$. 
It is known that the weights of this type are good for function spaces (see e.g. Section 3 in \cite{HaLa}).
We give a relation between Besov spaces and Sobolev spaces as a proposition.

\begin{prop}\label{prop:BS}
Assume that $\nu (x)= (1+|x|^2)^{-\sigma /2}$ for some $\sigma \in [0,\infty )$.
Then, for $m\in {\mathbb Z}$ and $p\in [1,\infty ]$, we have 
\[
C^{-1} \| f\|_{B_{p,\infty}^m (\nu )} \leq \| f\|_{W^{m,p} (\nu )} \leq C \| f\|_{B_{p,1}^m (\nu )} 
\]
where $C$ is a (positive) constant depending on $\sigma$, $m$ and $p$, $\| f\|_{W^{m,p} (\nu )}$ is the weighted Sobolev norm defined by
\begin{align*}
\| f\|_{W^{m,p} (\nu )} &:= \sum _{k=0}^m \| \nabla ^k f\| _{L^p(\nu)} , \quad m\in {\mathbb N}\cup \{ 0\} \\
\| f\|_{W^{-m,p} (\nu )} &:= \sup _{g\in {\mathcal S}({\mathbb R}^3); g\neq 0} \frac{|_{{\mathcal S}'} \langle f,g\rangle _{\mathcal S}|}{\| g\|_{W^{m,p^*} (\nu )}} , \quad m\in {\mathbb N},
\end{align*}
and $p^*$ is the H\"older conjugate of $p$.
\end{prop}

\begin{proof}
It is sufficient to prove the result for the case that $m=0$, because of the effect of derivatives (see Proposition 3 in \cite{MW2}) for $m\in {\mathbb N}$ and the duality of Besov spaces (see Proposition \ref{prop:Besov}\ref{prop:Besov2}).
The triangle inequality implies
\[
\| f\|_{L^p (\nu )} \leq \sum _{j=-1}^\infty \| \Delta _j f\|_{L^p (\nu )} = \| f\|_{B_{p,1}^0 (\nu )}.
\]
Hence, we have the second inequality of the assertion.
Let $\varphi _j := \varphi (2^{-j} \cdot)$ for $j\in {\mathbb N}\cup \{ 0\}$, $\varphi _{-1} := \chi $ and $\varphi _{-2}=0$.
The fact that 
\begin{equation}\label{eq:nu}
\nu (x) \leq 2^{\sigma /2} (1+|x-y|^2)^{\sigma /2} \nu (y) \quad x,y\in {\mathbb R}^3
\end{equation}
implies
\begin{align*}
&\| f\|_{B_{p,\infty}^0 (\nu )} \\
&= \sup _{j\in {\mathbb N}\cup \{ -1,0\}} \left\| \Delta _j f \right\| _{L^p(\nu)} \\
&= \sup _{j\in {\mathbb N}\cup \{ -1,0\}} \left( \int _{{\mathbb R}^3} \left( \int _{{\mathbb R}^3} \left( {\mathcal F}^{-1} \varphi _j \right) (x-y) f(y)dy \right) ^p \nu (x) dx\right) ^{1/p}\\
&\leq 2^{\sigma /(2p)} \sup _{j\in {\mathbb N}\cup \{ -1,0\}} \left( \int _{{\mathbb R}^3} \left( \int _{{\mathbb R}^3} \left( 1+|x-y|^2\right) ^{\sigma /2p} \left( {\mathcal F}^{-1} \varphi _j \right) (x-y) f(y) \nu (y)^{1/p} dy \right) ^p dx\right) ^{1/p} .
\end{align*}
Hence, by Young's inequality we have
\[
\| f\|_{B_{p,\infty}^0 (\nu )} \leq  2^{\sigma /(2p)} \sup _{j\in {\mathbb N}\cup \{ -1,0\}} \left( \int _{{\mathbb R}^3} (1+|x|^2)^{\sigma /2p} |({\mathcal F}^{-1} \varphi _j) (x)| dx \right)  \| f\| _{L^p (\nu )} .
\]
In view of this inequality, it is sufficient to show that
\[
\sup _{j\in {\mathbb N}\cup \{ -1,0\}} \int _{{\mathbb R}^3} (1+|x|^2)^{\sigma /2p} |({\mathcal F}^{-1} \varphi _j) (x)| dx < \infty .
\]
However, this claim is a special case ($\alpha = \beta$ and $s=1$) of \eqref{eq:propBesov2} in the proof of Proposition \ref{prop:Besov} below.
So, we omit it.
\end{proof}

Let us now make, for simplicity, the convention that all constants $C$ appearing in estimates (sometimes with indices indicating their dependence on parameters) are positive.

\begin{cor}\label{cor:BS}
Assume that $\nu (x)= (1+|x|^2)^{-\sigma /2}$ for some $\sigma \in [0,\infty )$, and define the Sobolev space $W^{s,p}(\nu )$ for $s\in {\mathbb R}\setminus {\mathbb Z}$ by the interpolation spaces of $W^{m,p}(\nu )$ with $m\in {\mathbb Z}$.
Then, for $s\in {\mathbb R}$ and $p\in [1,\infty ]$, we have 
\[
C^{-1} \| f\|_{B_{p,\infty}^s (\nu )} \leq \| f\|_{W^{s,p} (\nu )} \leq C \| f\|_{B_{p,1}^s (\nu )} 
\]
where $C$ is a constant depending on $\sigma$, $s$ and $p$.
\end{cor}

\begin{proof}
Since $\{ B_{p,r}^s; s\in {\mathbb R}\}$ is the series of interpolation spaces, the assertion follows from the interpolation inequalities and Proposition \ref{prop:BS}.
\end{proof}

\begin{prop}\label{prop:Besov}
\begin{enumerate}
\item \label{prop:Besov1} For $s_1, s_2 \in {\mathbb R}$, $\theta \in (0,1)$ and $p, r\in [1,\infty ]$
\[
\| f\| _{B_{p,r}^{\theta s_1 + (1-\theta )s_2}(\nu )} \leq \| f\| _{B_{p,r}^{s_1} (\nu )}^{\theta } \| f\| _{B_{p,r}^{s_2} (\nu )}^{1-\theta} , \quad f\in B_{p,r}^{s_1} (\nu ) \cap B_{p,r}^{s_2} (\nu ) .
\]

\item \label{prop:Besov2}
 For $s\in {\mathbb R}$ and $p_1, p_2 ,r_1,r_2 \in [1,\infty ]$ such that $1= 1/p_1+1/p_2$ and $1= 1/r_1 + 1/r_2$, there exists a constant $C$ depending on $s$, $p_1$, $p_2$, $r_1$ and $r_2$ satisfying
\[
\left| \int _{{\mathbb R}^3} f g \nu dx \right| \leq C \| f\| _{B_{p_1,r_2}^{-s}(\nu )} \| g\| _{B_{p_2,r_2}^{s}(\nu )} , \quad f\in B_{p_1,r_2}^{-s} (\nu ),\ g \in B_{p_2,r_2}^{s} (\nu ).
\]

\item \label{prop:Besov3}
Assume that $\nu (x)= (1+|x|^2)^{-\sigma /2}$ for some $\sigma \in [0,\infty )$.
For $\alpha \in {\mathbb R}$, $\beta \in [0,\infty )$, and $p,r\in [1,\infty ]$, there exists a constant $C$ depending on $\alpha$, $\beta$, $p$ and $r$ such that
\[
\| e^{t\triangle} \Delta _j f\| _{B_{p,r}^{\alpha}(\nu )} \leq C (2^{2j} t) ^{-\beta} \| \Delta _j f\| _{B_{p,r}^{\alpha}(\nu )}
\]
for $j\in {\mathbb N}\cup \{ 0\}$, $f\in B_{p,r}^{\alpha}(\nu )$, where $\{ e^{t\triangle }; t\geq 0\}$ is the heat semigroup generated by $\triangle $ on ${\mathbb R}^3$.
In particular,
\[
\| e^{t\triangle} f\| _{B_{p,r}^{\alpha}(\nu )} \leq C (1+t^{-\beta}) \| f\| _{B_{p,r}^{\alpha -2 \beta}(\nu )} ,\quad f\in B_{p,r}^{\alpha - 2\beta}(\nu ) .
\]

\item \label{prop:Besov4}
Assume that $\nu (x)= (1+|x|^2)^{-\sigma /2}$ for some $\sigma \in [0,\infty )$.
For $\alpha , \beta \in {\mathbb R}$, $p,q \in [1,\infty )$ and $r \in [1,\infty ]$ such that
\[
\alpha < \beta , \quad p>q, \quad \beta -\alpha = 3\left( \frac{1}{q} - \frac{1}{p} \right) ,
\]
there exists a constant $C$ depending on $\sigma$, $\alpha$, $\beta$, $p$ and $q$ such that
\[
\| f\| _{B_{p,r}^{\alpha}(\nu )} \leq C \| f\| _{B_{q,r}^{\beta}(\nu ^{q/p})} ,\quad f\in B_{p,r}^{\beta}(\nu ^{q/p}) .
\]
\end{enumerate}
\end{prop}

\begin{proof}
The proofs for \ref{prop:Besov1}, \ref{prop:Besov2} and \ref{prop:Besov3} are similar to the case without weights.
See Theorem 2.80 and Proposition 2.76 in \cite{BCD} for \ref{prop:Besov1} and \ref{prop:Besov2}.

For the first assertion of \ref{prop:Besov3} it is sufficient to prove the case where $r=1$.
The proof of \ref{prop:Besov3} is similar to Lemma 3 and Proposition 5 in \cite{MW2}.
First we show that there exists $\varepsilon \in (0,1)$ such that
\begin{equation}\label{eq:propBesov3-00}
\| e^{t\triangle} \Delta _j f\| _{L^p (\nu )} \leq C e^{-\varepsilon 2^{2j}t} \| \Delta _j f\| _{L^p (\nu )}
\end{equation}
for $j\in {\mathbb N} \cup \{ 0\}$.
Set $\varphi _j$ for $j\in {\mathbb N}\cup \{ -1, 0\}$ as in the proof of Proposition \ref{prop:BS}.
The fact that $\Delta _i \Delta _j =0$ for $|i-j|\geq 2$ and \eqref{eq:nu} imply
\begin{align*}
&\| e^{t\triangle} \Delta _j f\| _{L^p (\nu )} \\
&\leq C \sum _{i=j-1}^{j+1} \| \Delta _{i} \Delta _j e^{t\triangle} f\| _{L^p(\nu )} \\
&\leq C \sum _{i=j-1}^{j+1} \left( \int _{{\mathbb R}^3} \left| \int _{{\mathbb R}^3} ({\mathcal F}^{-1} \varphi _i) (x-y) e^{-|x-y|^2/(4t)} (\Delta _j f) (y) dy \right| ^p \nu (x) dx \right) ^{1/p} \\
&\leq C 2^{\sigma /2p} \sum _{i=j-1}^{j+1} \left( \int _{{\mathbb R}^3} \left| \int _{{\mathbb R}^3} (1+|x-y|^2)^{\sigma /2p} \left[ {\mathcal F}^{-1} \left( e^{-2t|\cdot |^2} \varphi _i \right) \right] (x-y) \right. \right. \\
&\quad \hspace{7cm} \left. \left. \phantom{\int} \times (\Delta _j f) (y) \nu (y)^{1/p} dy \right| ^p dx \right) ^{1/p} .
\end{align*}
Hence, by Young's inequality we have
\begin{align*}
&\| e^{t\triangle} \Delta _j f\| _{L^p (\nu )} \\
&\leq C \| \Delta _j f\| _{L^p(\nu)} \sum _{i=j-1}^{j+1} \int _{{\mathbb R}^3} \left| (1+|y|^2)^{\sigma /2p} \left[ {\mathcal F}^{-1} \left( e^{-2t|\cdot |^2} \varphi _i \right) \right] (y) \right| dy .
\end{align*}
In view of this inequality, it is sufficient to show
\[
\max_{i=j,j\pm 1} \int _{{\mathbb R}^3} \left| (1+|y|^2)^{\sigma /2p} \left[ {\mathcal F}^{-1} \left( e^{-2t|\cdot |^2} \varphi _i \right) \right] (y) \right| dy \leq C e^{-\varepsilon 2^{2j}t} .
\]
for some $\varepsilon \in (0,1)$.
The proof of this estimate is similar to that of Lemma 2.4 in \cite{BCD}. So, we omit it.

Next we prove
\begin{equation}\label{eq:propBesov3-01}
\| \Delta _i f\| _{L^p (\nu )} \leq C \| f\| _{B_{p,r}^{\alpha -2\beta} (\nu )}, \quad f\in B_{p,r}^{\alpha - 2\beta}(\nu ),\ i \in \{-1, 0\}.
\end{equation}
For $m\in {\mathbb N}$,
\begin{align*}
&\| \Delta _i f\| _{L^p (\nu )}^p \\
&= \int _{{\mathbb R}^3} \left| \int _{{\mathbb R}^3} \left[ (1-\triangle )^m {\mathcal F}^{-1} \varphi _i \right] (x-y) \cdot \left[ (1-\triangle )^{-m} f \right] (y)  dy \right| ^p \nu (x) dx \\
&\leq 2^{\sigma /2} \int _{{\mathbb R}^3} \left| \int _{{\mathbb R}^3} (1+|x-y|^2)^{\sigma /2p} \left[ {\mathcal F}^{-1} \left( (1+|\cdot |^2 )^m \varphi _i \right) \right] (x-y) \right. \\
&\quad \hspace{5cm} \left. \phantom{\int} \times \left[ (1-\triangle )^{-m} f \right] (y) \nu (y)^{1/p} dy \right| ^p dx .
\end{align*}
Hence, by Young's inequality we have
\begin{align*}
&\| \Delta _i f\| _{L^p (\nu )}^p \\
&\leq C \left\| (1-\triangle )^{-m} f \right\| _{L^p (\nu)}^p \left( \int _{{\mathbb R}^3} (1+|y|^2)^{\sigma /2p} \left| \left[ {\mathcal F}^{-1} \left( (1+|\cdot |^2 )^m \varphi _i \right) \right] (y) \right| dy \right) ^p .
\end{align*}
Since $i=-1,0$ and $(1+|\cdot |^2 )^m \varphi _i \in {\mathcal S}({\mathbb R}^3)$, we obtain
\[
\| \Delta _i f\| _{L^p (\nu )} \leq C \left\| (1-\triangle )^{-m} f \right\| _{L^p (\nu)} .
\]
In view of Proposition \ref{prop:BS}, by taking $2m> |\alpha -2\beta |$, we have \eqref{eq:propBesov3-01}.

Now we prove \ref{prop:Besov3} by applying \eqref{eq:propBesov3-00}.
From \eqref{eq:propBesov3-00} and the fact that $\Delta _i \Delta _j =0$ for $|i-j|\geq 2$, we have that for $j\in {\mathbb N}\cup \{ 0\}$ and $i\in {\mathbb N}\cup \{ -1,0\}$
\begin{align*}
&2^{\alpha i} \| e^{t\triangle} \Delta _i \Delta _j f\| _{L^p (\nu )} \\
&\leq C 2^{\alpha i} e^{-\varepsilon 2^{2j}t} \| \Delta _i \Delta _j f\| _{L^p (\nu )} \\
&\leq C 2^{\alpha i} (2^{2j}t)^{\beta} e^{-\varepsilon 2^{2j}t} (2^{2j}t)^{-\beta} \| \Delta _i \Delta _j f\| _{L^p (\nu )} \\
&\leq C 2^{\alpha i} (2^{2j}t)^{-\beta} \| \Delta _i \Delta _j f\| _{L^p (\nu )} .
\end{align*}
The first assertion follows from this inequality and \eqref{eq:propBesov3-01}.
Since \eqref{eq:propBesov3-00} also implies that for $j\in {\mathbb N} \cup \{ 0\}$
\begin{align*}
2^{\alpha j}\| e^{t\triangle} \Delta _j f\| _{L^p (\nu )} &\leq C t^{-\beta} 2^{(\alpha -2\beta ) j } (2^{2j}t)^{\beta} e^{-\varepsilon 2^{2j}t} \| \Delta _j f\| _{L^p (\nu )} \\
&\leq C t^{-\beta} 2^{(\alpha -2\beta ) j } \| \Delta _j f\| _{L^p (\nu )},
\end{align*}
by \eqref{eq:propBesov3-01} we get the second assertion.

Next we prove \ref{prop:Besov4}.
We remark that the proof is similar to that of Proposition 3.7 in \cite{MW2}.
The fact that $\Delta _i \Delta _j =0$ for $|i-j|\geq 2$, together with \eqref{eq:nu} implies
\begin{align*}
&\| \Delta _j f\|_{L^p(\nu )} \\
&= \left( \int _{{\mathbb R}^3} \left| \int _{{\mathbb R}^3} ({\mathcal F}^{-1} \varphi _j) (x-y) [(\Delta _{j-1}+ \Delta _j + \Delta _{j+1}) f] (y) dy \right| ^p \nu (x) dx \right) ^{1/p} \\
&\leq 2^{\sigma /2p} \left( \int _{{\mathbb R}^3} \left| \int _{{\mathbb R}^3} (1+|x-y|^2)^{\sigma /2p} ({\mathcal F}^{-1} \varphi _j) (x-y) \right. \right. \\
&\quad \hspace{5cm} \left. \left. \phantom{\int} \times [(\Delta _{j-1}+ \Delta _j + \Delta _{j+1}) f] (y) \nu (y)^{1/p} dy \right| ^p dx \right) ^{1/p} .
\end{align*}
Hence, by letting $s\in [1,\infty]$ such that $1/s + 1/q = 1/p +1$, from the Young's inequality we have
\begin{align*}
&\| \Delta _j f\|_{L^p(\nu )}\\
&\leq 2^{\sigma /2p} \left( \int _{{\mathbb R}^3} (1+|x|^2)^{\sigma s/2p} |({\mathcal F}^{-1} \varphi _j) (x)|^s dx \right) ^{1/s} \left\| \nu ^{1/p } (\Delta _{j-1}+ \Delta _j + \Delta _{j+1}) f \right\| _{L^q(1)}.
\end{align*}
Thus, we have
\[
\| f\|_{B_{p,r}^\alpha (\nu )} \leq C \left( \sup _{j\in {\mathbb N}\cup \{ -1,0\}} 2^{js(\alpha -\beta )}\int _{{\mathbb R}^3} (1+|x|^2)^{\sigma s/2p} |({\mathcal F}^{-1} \varphi _j) (x)|^s dx \right) ^{1/s} \left\| f \right\| _{B_{q,r}^\beta (\nu ^{q/p})}.
\]
In view of this inequality, it is sufficient to show
\begin{equation}\label{eq:propBesov2}
\sup _{j\in {\mathbb N}\cup \{ -1,0\}} 2^{js(\alpha -\beta )}\int _{{\mathbb R}^3} (1+|x|^2)^{\sigma s/2p} |({\mathcal F}^{-1} \varphi _j) (x)|^s dx <\infty .
\end{equation}
Since
\begin{equation}\label{eq:scalephi}
({\mathcal F}^{-1} \varphi _j) (x) = 2^{3j} ({\mathcal F}^{-1} \varphi ) (2^j x), \quad x\in {\mathbb R}^3
\end{equation}
and ${\mathcal F}^{-1} \varphi \in {\mathcal S}({\mathbb R}^3)$, we have
\begin{align*}
&\sup _{j \in {\mathbb N}\cup \{ 0\}} 2^{js(\alpha -\beta )} \int _{{\mathbb R}^3} (1+|x|^2)^{\sigma s /2p} | ({\mathcal F}^{-1} \varphi _j) (x) |^s dx \\
&= \sup _{j \in {\mathbb N}\cup \{ 0\}} 2^{js(\alpha -\beta )} 2^{3j(s-1)} \int _{{\mathbb R}^3} (1+2^{-2j}|x|^2)^{\sigma s /2p} | ({\mathcal F}^{-1} \varphi ) (x)| dx .
\end{align*}
The assumption in \ref{prop:Besov4} implies that $s(\alpha -\beta ) + 3(s-1) =0$ and the finiteness for the case that $j=-1$ is trivial.
Therefore, we obtain (\ref{eq:propBesov2}).
\end{proof}

Next we prepare the notations and estimates of paraproducts by following Chapter 2 in \cite{BCD}.
Let
\[
S_j f := \sum _{k =-1}^{j-1} \Delta _{k} f, \quad j\in {\mathbb N}\cup \{ 0\} .
\]
For simplicity of notation, let $\Delta _{-2}f :=0$ and $S_{-1}f :=0$.
We define
\begin{align*}
f \mbox{\textcircled{\scriptsize$<$}} g &:= \sum _{j=0}^\infty (S_j f) \Delta _{j+1} g \\
f \mbox{\textcircled{\scriptsize$=$}} g &:= \sum _{j=-1}^\infty \Delta _{j} f  \left( \Delta _{j-1} g + \Delta _{j} g + \Delta _{j+1} g\right) \\
f \mbox{\textcircled{\scriptsize$>$}} g &:= g \mbox{\textcircled{\scriptsize$<$}} f
\end{align*}
By the definitions of $\{ \Delta _j\}$, $\{ S_j\}$, $\mbox{\textcircled{\scriptsize$<$}}$, $\mbox{\textcircled{\scriptsize$=$}}$, and $\mbox{\textcircled{\scriptsize$>$}}$, we have a decomposition
\[
fg = f \mbox{\textcircled{\scriptsize$<$}} g + f \mbox{\textcircled{\scriptsize$=$}} g + f \mbox{\textcircled{\scriptsize$>$}} g,
\]
which is called the Bony decomposition.
Let $f \mbox{\textcircled{\scriptsize$\leqslant$}} g := f \mbox{\textcircled{\scriptsize$<$}} g + f \mbox{\textcircled{\scriptsize$=$}} g$ and $f \mbox{\textcircled{\scriptsize$\geqslant$}} g := f \mbox{\textcircled{\scriptsize$>$}} g + f \mbox{\textcircled{\scriptsize$=$}} g$.

\begin{prop}\label{prop:paraproduct}
\begin{enumerate}

\item \label{prop:paraproduct2} For $s\in {\mathbb R}$ and $p, p_1, p_2 ,r \in [1,\infty ]$ such that $1/p = 1/p_1 + 1/p_2$,
\[
\| f \mbox{\textcircled{\scriptsize$<$}} g \| _{B_{p,r}^s (\nu )} \leq C\| f\| _{L^{p_1}(\nu )} \| g\| _{B_{p_2,r}^s(\nu )}, \quad f \in L^{p_1}(\nu ) ,\ g \in B_{p_2,r}^s (\nu ) ,
\]
where $C$ is a constant depending on $s$, $p_1$, $p_2$ and $r$.

\item \label{prop:paraproduct3} For $s\in {\mathbb R}$, $t\in (-\infty ,0)$, and $p, p_1, p_2 ,r_1,r_2 \in [1,\infty ]$ such that
\[
\frac 1p = \frac{1}{p_1} + \frac{1}{p_2} \quad \mbox{and}\quad \frac 1r = \min \left\{ 1, \frac{1}{r_1} + \frac{1}{r_2} \right\},
\]
\[
\| f \mbox{\textcircled{\scriptsize$<$}} g \| _{B_{p,r}^{s+t}(\nu )} \leq C\| f\| _{B_{p_1,r_1}^t (\nu )} \| g\| _{B_{p_2,r_2}^s (\nu )}, \quad f \in B_{p_1 ,r_1}^t (\nu ) ,\ g \in B_{p_2,r_2}^s (\nu ) ,
\]
where $C$ is a constant depending on $s$, $p_1$, $p_2$, $r_1$ and $r_2$.

\item \label{prop:paraproduct4} For $s_1 ,s_2\in {\mathbb R}$ such that $s_1+s_2>0$, and $p, p_1, p_2, r, r_1,r_2 \in [1,\infty ]$ such that
\[
\frac 1p = \frac{1}{p_1} + \frac{1}{p_2} \quad \mbox{and}\quad \frac 1r = \frac{1}{r_1} + \frac{1}{r_2},
\]
it holds that
\[
\| f \mbox{\textcircled{\scriptsize$=$}} g \| _{B_{p,r}^{s_1+s_2}(\nu )} \leq C\| f\| _{B_{p_1,r_1}^{s_1}(\nu )} \| g\| _{B_{p_2,r_2}^{s_2} (\nu )}, \quad f \in B_{p_1,r_1}^{s_1}(\nu ) ,\ g \in B_{p_2,r_2}^{s_2} (\nu ) ,
\]
where $C$ is a constant depending on $s_1$, $s_2$, $p_1$, $p_2$, $r_1$ and $r_2$.

\item \label{prop:paraproduct4+}
For $s\in (-\infty ,0)$ and $t\in (0,\infty )$ such that $s +t >0$, and $p, p_1, p_2, r, r_1,r_2 \in [1,\infty ]$ such that
\[
\frac 1p = \frac{1}{p_1} + \frac{1}{p_2} \quad \mbox{and}\quad \frac 1r = \frac{1}{r_1} + \frac{1}{r_2},
\]
it holds that
\begin{align*}
\| f g \| _{B_{p,r}^{s}(\nu )} &\leq C\| f\| _{B_{p_1,r_1}^{s}(\nu )} \| g\| _{B_{p_2,r_2}^{t} (\nu )}, \quad f \in B_{p_1,r_1}^{s}(\nu ) ,\ g \in B_{p_2,r_2}^{t} (\nu ) , \\
\| f g \| _{B_{p,r}^{t}(\nu )} &\leq C\| f\| _{B_{p_1,r_1}^{t}(\nu )} \| g\| _{B_{p_2,r_2}^{t} (\nu )}, \quad f \in B_{p_1,r_1}^{t}(\nu ) ,\ g \in B_{p_2,r_2}^{t} (\nu ) ,
\end{align*}
where $C$ is a constant depending on $s$, $t$, $p_1$, $p_2$, $r_1$ and $r_2$.

\item \label{prop:paraproduct5}
For $s\in (0,\infty )$, $\varepsilon \in (0,1)$ and $p, p_1, p_2 ,r \in [1,\infty ]$ such that $1/p = 1/p_1 + 1/p_2$,
\[
\| f^2\| _{B_{p,r}^s (\nu )} \leq C \| f\| _{L^{p_1}(\nu )} \| f\| _{B_{p_2,r}^{s+\varepsilon }(\nu )} , \quad f\in L^{p_1} (\nu ) \cap B_{p_2,r}^{s+\varepsilon} (\nu )
\]
where $C$ is a constant depending on $s$, $\varepsilon$, $p_1$, $p_2$ and $r$.

\item \label{prop:paraproduct5.5}
For $s\in (0,\infty )$, $\varepsilon \in (0,1)$ and $p,r \in [1,\infty ]$,
\[
\| f^3\| _{B_{p,r}^s (\nu )} \leq C \| f\| _{L^{4p}(\nu )}^2 \| f\| _{B_{2p,r}^{s+\varepsilon }(\nu )} , \quad f\in L^{4p} (\nu ) \cap B_{2p,r}^{s+\varepsilon} (\nu )
\]
where $C$ is a positive constant depending on $s$, $\varepsilon$, $p$ and $r$.
\end{enumerate}
\end{prop}

\begin{proof}
For \ref{prop:paraproduct2} and \ref{prop:paraproduct3}, see Proposition A.7 in \cite{MW3}.
The proofs of  \ref{prop:paraproduct4} are similar to the case of the Besov spaces without weights
(see Theorem 2.85 in \cite{BCD}).
From \ref{prop:paraproduct2}, \ref{prop:paraproduct3} and \ref{prop:paraproduct4}, \ref{prop:paraproduct4+} follows .
The proofs of \ref{prop:paraproduct5} and \ref{prop:paraproduct5.5} are also similar to those of Proposition 2.1 (v) and (vi) in \cite{AK}, respectively.
\end{proof}

\begin{prop}\label{prop:cptembedding}
Assume that $\nu (x)= (1+|x|^2)^{-\sigma /2}$ for some $\sigma \in (3,\infty )$.
Then, for $\alpha , \beta \in {\mathbb R}$ such that $\alpha <\beta$, and $p,q,r\in [1,\infty ]$ such that $p<q$, the embedding $B_{q,r}^{\beta}(\nu ) \hookrightarrow B_{p,r}^{\alpha}(\nu )$ is compact.
\end{prop}

\begin{proof}
For a bounded sequence $\{ f_n\} \subset B_{q,r}^{\beta}(\nu )$, by considering $\{ (1-\triangle )^{-(\alpha +\beta )/4}f_n\}$ instead of $\{ f_n\}$, we are able to assume $\alpha < 0< \beta$.
Hence, in view of the fact that $B_{p,\infty}^s (\nu ) \subset B_{p,1}^{s'} (\nu )$ for $p\in [1,\infty ]$ and $s,s' \in {\mathbb R}$ such that $s'<s$, and Proposition \ref{prop:BS}, it is sufficient to show
the compactness of the embedding $L^q (\nu ) \hookrightarrow B_{p,1}^{-\alpha}(\nu )$ for $\alpha \in (0,1)$ and $p,q\in [1,\infty ]$ such that $p<q$.

Let $\varphi _j := \varphi (2^{-j} \cdot)$ for $j\in {\mathbb N}\cup \{ 0\}$ and $\varphi _{-1} := \chi$, where $\varphi$ and $\chi$ are functions prepared for the definition for Besov spaces (see the beginning of Section \ref{subsec:Besov}).
Let $\{ f_n\}$ be a bounded sequence in $L^p(\nu )$.
Since $\varphi _i$ has a compact support for each $i$, then for each $i,j\in {\mathbb N}\cup \{ -1,0\}$ and $m\in {\mathbb N}$ we have, for some constant $C_{m,i}$
\begin{align*}
\left\| \varphi _i \Delta _j f_n\right\| _{W^{m,p}} &= \sum _{k\in ({\mathbb N}\cup \{ 0\})^3; |k|\leq m} \left\| \partial ^k (\varphi _i \Delta _j f_n )\right\| _{L^p} \\
&= \sum _{k,l \in ({\mathbb N}\cup \{ 0\})^3; |k|+|l| \leq m} \left\| (\partial ^k \varphi _i) (\partial ^l \Delta _j f_n )\right\| _{L^p} \\
&\leq C_{m,i} \sum _{l\in ({\mathbb N}\cup \{ 0\})^3; |l|\leq m} \left\| (\partial ^l {\mathcal F}^{-1} \varphi _j) * f_n\right\| _{L^p(\nu )} .
\end{align*}
This and H\"older's inequality imply
\[
\left\| \varphi _i \Delta _j f_n\right\| _{W^{m,p}} \leq C_{m,i, j} \left\| f_n\right\| _{L^p(\nu )}
\]
for some constant $C_{m,i,j}$.
Hence, for each $i,j\in {\mathbb N}\cup \{ -1,0\}$ and $m\in {\mathbb N}$, $\{ \varphi _i \Delta _j f_n ; n\in {\mathbb N}\}$ are uniformly bounded in $W^{m,p}$.
In view of the compactness of the support of $\varphi _i$, Sobolev's inequality and the Ascoli-Arzel\`a theorem yield that for each $i,j\in {\mathbb N}\cup \{ -1,0\}$, there exists a subsequence $\{ f_{n(k)}\}$ of $\{ f_n\}$ such that $\{ \varphi _i \Delta _j f_{n(k)}\}$ converges in $C_b({\mathbb R}^3)$ with the topology of uniform convergence.
By the diagonal method, we have a subsequence $\{ f_{n(k)}\}$ such that $\{ \varphi _i \Delta _j f_{n(k)}\}$ converges in $C_b({\mathbb R}^3)$ for all $i,j\in {\mathbb N}\cup \{ -1,0\}$.
Let us take $a\in [1,\infty )$ satisfying $(1/a) + (1/q) = 1/p$.
For $k,l \in {\mathbb N}$ and $N\in {\mathbb N}$ we have
\begin{align*}
&\| f_{n(k)} - f_{n(l)}\| _{B_{p,1}^{-\alpha }(\nu )} \\
&\leq \sum _{j=-1}^N 2^{-j\alpha} \| \Delta _j (f_{n(k)} -f_{n(l)})\| _{L^p(\nu )} + \sum _{j=N+1}^\infty 2^{-j\alpha} \left\| \Delta _j (f_{n(k)} -f_{n(l)}) \right\| _{L^p(\nu )}\\
&\leq \sum _{i, j=-1}^N 2^{-j\alpha} \| \varphi _i \Delta _j (f_{n(k)} -f_{n(l)})\| _{L^p(\nu )} \\
&\quad + \sum _{j=-1}^N 2^{-j\alpha} \left\| \left( 1- \sum _{i=-1}^N \varphi _i \right) \Delta _j (f_{n(k)} -f_{n(l)})\right\| _{L^p(\nu )}\\
&\quad +  2^{-\alpha (N+1)/2} \sum _{j=N+1}^\infty 2^{-j\alpha /2} \| \Delta _j (f_{n(k)} -f_{n(l)})\| _{L^p(\nu )}\\
&\leq \sum _{i, j=-1}^N 2^{-j\alpha} \| \varphi _i \Delta _j (f_{n(k)} -f_{n(l)})\| _{L^p(\nu )} \\
&\quad +  \left\| 1- \sum _{i=-1}^N \varphi _i \right\| _{L^{a}(\nu )} \sum _{j=-1}^N 2^{-j\alpha} \left\|\Delta _j (f_{n(k)} -f_{n(l)})\right\| _{L^q(\nu )}\\
&\quad +  2^{-\alpha (N+1)/2} \| f_{n(k)} -f_{n(l)}\| _{B_{p,1}^{-\alpha}(\nu )}\\
&\leq \sum _{i, j=-1}^N 2^{-j\alpha} \| \varphi _i \Delta _j (f_{n(k)} -f_{n(l)})\| _{L^p(\nu )} \\
&\quad + \left( 2^{-\alpha (N+1)/2} + \left\| {\mathbb I}_{[3/4,\infty )} \left( 2^{-N}| \cdot | \right) \right\| _{L^{a}(\nu )} \right) \| f_{n(k)} -f_{n(l)}\| _{B_{q,1}^{-\alpha}(\nu )} .
\end{align*}
Since
\[
\| f_{n(k)} -f_{n(l)}\| _{B_{q,1}^{-\alpha}(\nu )} \leq \| f_{n(k)} -f_{n(l)}\| _{L^q(\nu )} \leq 2\sup _{n\in {\mathbb N}} \| f_n\| _{L^q(\nu )} <\infty ,
\]
for any $\varepsilon \in (0,1)$ there exists a sufficiently large $N$ such that
\[
\| f_{n(k)} - f_{n(l)}\| _{B_{p,1}^{-\alpha }(\nu )} \leq \sum _{i, j=-1}^N 2^{-j\alpha} \| \varphi _i \Delta _j f_{n(k)} - \varphi _i \Delta _j f_{n(l)}\| _{L^p(\nu )} + \varepsilon .
\]
This estimate and the fact that $\{ \varphi _i \Delta _j f_{n(k)}\}$ converges in $C_b({\mathbb R}^3)$ for all $i,j\in {\mathbb N}\cup \{ -1,0\}$ yield the conclusion.
\end{proof}

\begin{lem}\label{lem:Lpestimates1}
Assume that the total mass of the measure $\nu (x)dx$ is finite.
Then:
\begin{enumerate}
\item \label{lem:Lpestimates1-1} For $\theta \in (0,9/16)$ there exists a (positive) constant $C$ depending on $\theta$ such that for $\delta \in (0,1]$ and $f\in L^4 (\nu ) \cap B_2^{15/16} (\nu )$
\[
\| f\| _{B_2^{\theta } (\nu )} \leq \delta \left( \| f\| _{L^4 (\nu )}^4 + \| f\| _{B_2^{15/16} (\nu )}^2\right) ^{7/8} + C\delta ^{-16/19} .
\]

\item \label{lem:Lpestimates1-2}For $\theta \in (0,9/16)$ there exists a (positive) constant $C$ depending on $\theta$ such that for $\delta \in (0,1]$ and $f\in L^4 (\nu ) \cap B_2^{15/16} (\nu )$
\[
\| f^2\| _{B_{4/3}^{\theta } (\nu )} \leq \delta \left( \| f\| _{L^4 (\nu )}^4 + \| f\| _{B_2^{15/16} (\nu )}^2\right) ^{7/8} + C\delta ^{-26/9} . 
\]

\item \label{lem:Lpestimates1-3} For $\theta \in (0,9/16)$ there exists a (positive) constant $C$ depending on $\theta$ such that for $\delta \in (0,1]$ and $f\in L^4 (\nu ) \cap B_2^{15/16} (\nu )$
\[
\| f^3\| _{B_1^{\theta} (\nu )} \leq \delta \left( \| f\| _{L^4 (\nu )}^4 + \| f\| _{B_2^{15/16} (\nu )}^2 \right) + C \delta ^{-10} .
\]
\end{enumerate}
\end{lem}

\begin{proof}
The proofs are similar to those in Lemma 2.2 in \cite{AK}. So, we omit them.
\end{proof}

\subsection{Approximation operators and commutator estimates}\label{subsec:approximation}

We define $\psi$, $\rho$, $\psi _N$, $\rho _M$, $P_N$, $P_{M,N}$ and $P_{M,N}^*$ as in Section \ref{sec:main}.
As remarked in Section \ref{sec:main}, $P_{M,N}^*$ is the dual operator of $P_{M,N}$ with respect to $L^2(dx)$.
In the rest of the paper, we fix $\psi$ and $\rho$ and do not mention explicitly the dependence on them, even if the constants that will appear in the estimates below are depending on them.

For simplicity of notation, we denote $B_{p,\infty }^s(\nu )$ by $B_p^s (\nu )$ for $s \in {\mathbb R}$ and $p \in [1,\infty]$.
Moreover,  when $\nu =1$, we denote $L^p(\nu)$, $B_{p,r}^s (\nu)$ and $B_p^s (\nu)$ simply by $L^p$, $B_{p,r}^s$ and $B_p^s$, respectively.

\begin{prop}[cf. Theorem 3 of Chapter IV in \cite{Stein}]\label{prop:Fmultiplier}
For $p\in (1,\infty )$, there exists a constant $C$ independent of $N$ such that
\[
\| P_N f\| _{L^p} \leq C \| f\| _{L^p}, \quad f\in L^p .
\]
\end{prop}

\begin{proof}
This proposition immediately follows from Theorem 3 of Chapter IV in \cite{Stein} and the comment just after the theorem.
\end{proof}

\begin{lem}\label{lem:PN}
Let $\alpha \in {\mathbb R}$, $\beta \in (0,\infty )$ and $p\in [1,\infty ]$.
Assume that $\alpha \neq \beta$.
Then, there exists an absolute constant $C$ such that
\[
\left\| P_N f \right\| _{B^\alpha _p(\nu )} \leq C 2^{\beta N} \| P_N f\| _{B_p^{\alpha -\beta }(\nu )}
\]
for $f\in {\mathcal S}({\mathbb R}^3)$ and $N\in {\mathbb N}$.
In particular,
\begin{align*}
\left\| P_N f \right\| _{B^\alpha _p} &\leq C 2^{\beta N} \| P_N f\| _{B_p^{\alpha -\beta }} ,\\
\left\| P_{M,N}^* f \right\| _{B^\alpha _p(\nu )} &\leq C 2^{\beta N} \| P_{M,N}^* f\| _{B_p^{\alpha -\beta }(\nu )}
\end{align*}
for $f\in {\mathcal S}({\mathbb R}^3)$ and $M,N\in {\mathbb N}$.
\end{lem}

\begin{proof}
In view of the supports of $\varphi$ and $\psi _N$,
\begin{equation}\label{eq:Nj}
\Delta _j P_N =0 ,\quad \mbox{if}\ 2^N\leq \frac{\sqrt 3}{8} 2^j.
\end{equation}
Hence, we have
\[
\left\| P_N f  \right\| _{B^\alpha _p(\nu )} \leq C 2^{\beta N} \sup _{j \in {\mathbb N} \cup \{ -1, 0\}; 2^j \leq 2^{N+3}/\sqrt 3} 2^{j(\alpha -\beta )} \left\| \Delta _j P_N f \right\| _{L^p(\nu )} ,
\]
from which the assertion follows.
\end{proof}

\begin{lem}\label{lem:comP0}
Let $\alpha \in {\mathbb R}$, $\beta \in [0,\infty )$ and $p\in (1,\infty )$.
Then, there exists a constant $C$ depending on $\beta$ and $p$ such that
\[
\left\| P_N f - f \right\| _{B^\alpha _p} \leq C 2^{-\beta N} \| f\| _{B_p^{\alpha +\beta}}
\]
for $f\in {\mathcal S}({\mathbb R}^3)$ and $N\in {\mathbb N}$.
In particular,
\[
\left\| P_N f \right\| _{B^\alpha _p} \leq C \| f\| _{B_p^{\alpha}}
\]
for $f\in {\mathcal S}({\mathbb R}^3)$ and $N\in {\mathbb N}$.
\end{lem}

\begin{proof}
In view of the supports of $\varphi$ and $\psi _N$, we have
\[
\Delta _j (I-P_N) =0 ,\quad \mbox{if}\ 2^N\geq \frac{3\sqrt 3}{4} 2^j.
\]
Hence, the commutativity of $\Delta _j$ and $I-P_N$ implies
\begin{align*}
&\left\| (I-P_N) f \right\| _{B^\alpha_p}\\
&= \sup _{j \in {\mathbb N} \cup \{ -1, 0\};  2^j \geq 4N/({3\sqrt 3})} 2^{j\alpha} \left\| (I-P_N) \Delta _j f \right\| _{L^p}\\
&\leq \left( \frac{3\sqrt 3}{4}\right) ^\beta 2^{-\beta N} \sup _{j \in {\mathbb N} \cup \{ -1, 0\};  2^j \geq 2^{N+2}/({3\sqrt 3})} 2^{j(\alpha + \beta )} \left\| (I-P_N) \Delta _j f \right\| _{L^p}.
\end{align*}
In view of this inequality and Proposition \ref{prop:Fmultiplier}, there exists a constant $C_p$ depending on $p$, such that
\[
\left\| P_N f - f \right\| _{B^\alpha _p} \leq C_p \left( \frac{3\sqrt 3}{4}\right) ^\beta 2^{-\beta N} \left\| f \right\| _{B_p^{\alpha + \beta}} .
\]
This yields the conclusion.
\end{proof}

\begin{lem}\label{lem:comP1}
Let $\alpha \in {\mathbb R}$, $\beta \in (0,\infty )$, $p\in (1,\infty )$, and assume that $\nu (x) = (1+|x|)^{-\sigma}$ $(x\in {\mathbb R}^3)$ for some $\sigma \in [0,\infty )$.
Then, for $M\in {\mathbb N}$ there exists a constant $C$ depending on $\sigma$, $\alpha$, $\beta$, $p$, $\psi$ and $\rho$ such that
\[
\left\| P_{M,N}^* f - \rho _M f \right\| _{B^\alpha _p(\nu )} \leq C M^\sigma 2^{-\beta N} \| f\| _{B_p^{\alpha +\beta}(\nu )}
\]
for $f\in {\mathcal S}({\mathbb R}^3)$ and $N\in {\mathbb N}$.
In particular,
\[
\left\| P_{M,N}^* f \right\| _{B^\alpha_p(\nu )} \leq C \left( 1+ M^\sigma 2^{-\beta N}\right) \| f\| _{B_p^{\alpha + \beta}(\nu )}
\]
for $f\in {\mathcal S}({\mathbb R}^3)$ and $N\in {\mathbb N}$.
\end{lem}

\begin{proof}
By Lemma \ref{lem:comP0} and the fact that $\nu \leq 1$, 
\[
\left\| P_N (\rho _M f) - \rho _M f \right\| _{B^\alpha _p(\nu )} \leq C_p 2^{-\beta N} \left\| \rho _M f \right\| _{B_p^{\alpha + \beta}} .
\]
Hence, it is sufficient to prove
\begin{equation}\label{eq:lemcomP03}
\left\| \rho _M f \right\| _{B_p^{\alpha + \beta}} \leq C M^\sigma \left\| f \right\| _{B_p^{\alpha + \beta}(\nu )}.
\end{equation}
Let $q$ be the H\"older conjugate of $p$.
Proposition 2.76 in \cite{BCD} and Proposition \ref{prop:Besov} imply
\begin{align*}
\left\| \rho _M f \right\| _{B_p^{\alpha + \beta}} &\leq C\sup _{h\in {\mathcal S}({\mathbb R}^3);\ \| h\| _{B_{q,1}^{-(\alpha + \beta )}} \leq 1} \langle h, \rho _M f \rangle \\
&\leq C \sup _{h\in {\mathcal S}({\mathbb R}^3); \ \| h\| _{B_{q,1}^{-(\alpha + \beta )}}\leq 1} \int _{{\mathbb R}^3} \frac{\rho _{M} h}{\nu} f \nu dx \\
&\leq C \| f \|_{B_{p}^{\alpha + \beta} (\nu )} \sup _{h\in {\mathcal S}({\mathbb R}^3);\ \| h\| _{B_{q,1}^{-(\alpha + \beta )}}\leq 1} \left\| \frac{\rho _{M} h}{\nu} \right\| _{B_{q,1}^{-(\alpha + \beta )}(\nu )},
\end{align*}
where $C$ is a constant independent of $f$ and $M$.
Note that $\rho _{M}/\nu \in C_0^\infty ({\mathbb R}^3)$ and that $\rho _{M}/\nu$ and their derivatives are dominated by $C M^\sigma$ with a constant $C$.
Since this fact and Proposition \ref{prop:Besov}  imply
\[
\left\| \frac{\rho _{M} h}{\nu} \right\| _{B_{q,1}^{-(\alpha + \beta )}(\nu )} \leq C M^\sigma \| h\| _{B_{q,1}^{-(\alpha + \beta )}},
\]
we have \eqref{eq:lemcomP03}.
\end{proof}

\begin{lem}\label{lem:comP2}
Assume that $\nu (x) = (1+|x|^2)^{-\sigma /2} $ for some $\sigma \in [0,\infty )$.
Then, for $M\in {\mathbb N}$ and $p\in [1,\infty ]$ there exists a constant $C$ depending on $\sigma$, $\psi$ and $\rho$ such that
\[
\left\| P_{M,N} f - \frac{1}{\nu} P_{M,N} (\nu f) \right\| _{L^p (\nu )} \leq C M^\sigma 2^{-N} \| f\| _{L^p(\nu )}
\]
for $f\in L^p(\nu )$ and $N\in {\mathbb N}$.
\end{lem}

\begin{proof}
It is easy to see that
\[
\nu (x)^{1/p} \rho _M (x) \nu (y)^{-1/p} \left( 1-\frac{\nu (y)}{\nu (x)}\right) \leq \sigma |x-y| \rho _M (x) \nu (x)^{-1} \nu (x-y)^{-1/p}
\]
for $x,y\in {\mathbb R}^3$.
From this fact, the definition of $P_{M,N}$ and the Plancherel theorem, it follows that
\begin{align*}
& \int _{{\mathbb R}^3} \left| P_{M,N} f (x)- \frac{1}{\nu (x)} P_{M,N} (\nu f) (x)\right| ^p \nu (x) dx\\
&= \int _{{\mathbb R}^3} \left| \int _{{\mathbb R}^3} ({\mathcal F}^{-1}\psi _N) (x-y) \rho _M (x)\left( 1-\frac{\nu (y)}{\nu (x)}\right) f(y) dy \right| ^p \nu (x) dx\\
&\leq \sigma \int _{{\mathbb R}^3} \left( \int _{{\mathbb R}^3} |x-y| \nu (x-y)^{-1/p} \left| ({\mathcal F}^{-1}\psi _N) (x-y)\right| \frac{\rho _M (x)}{{\nu (x)}} |f(y)| \nu (y)^{1/p} dy \right) ^p dx .
\end{align*}
Hence, by the fact that $\| \rho _M / \nu \| _{L^\infty} \leq C M^\sigma$ for each $M$, and Young's inequality we have
\begin{equation}\label{eq:lemcomP2-01}
\left\| P_{M,N} f - \frac{1}{\nu} P_{M,N} (\nu f) \right\| _{L^p (\nu )} \leq C M^\sigma \left( \int _{{\mathbb R}^3} |z| \nu (z)^{-1/p} \left| ({\mathcal F}^{-1}\psi _N) (z)\right| dz \right) \| f\| _{L^p(\nu )} .
\end{equation}
By the definition of inverse Fourier transform, changing the integration variables and using the integration by parts formula, we get
\begin{align*}
&\int _{{\mathbb R}^3} |z| \nu (z)^{-1/p} \left| ({\mathcal F}^{-1}\psi _N) (z)\right| dz \\
&= \frac{1}{(2\pi)^{3/2} 2^N} \int _{{\mathbb R}^3} |z| (1+2^{-N}|z|)^{\sigma /p} \left| \int _{{\mathbb R}^3} \psi (\xi ) e^{\sqrt{-1} z\cdot \xi} d\xi \right| dz .
\end{align*}
Then, for any $m\in {\mathbb N}$ the right-hand side of this equality is calculated as
\begin{align*}
&\frac{1}{(2\pi)^{3/2} 2^N} \int _{{\mathbb R}^3} \frac{|z| (1+2^{-N}|z|)^{\sigma /p}}{(1+|z|^2)^m} \left| \int _{{\mathbb R}^3} \psi (\xi ) (1+\triangle _{\xi})^m e^{\sqrt{-1} z\cdot \xi} d\xi \right| dz \\
&= \frac{1}{(2\pi)^{3/2} 2^N} \int _{{\mathbb R}^3} \frac{|z| (1+2^{-N}|z|)^{\sigma /p}}{(1+|z|^2)^m} \left| \int _{{\mathbb R}^3} \left[ (1+\triangle )^m\psi \right] (\xi )  e^{\sqrt{-1} z\cdot \xi} d\xi \right| dz .
\end{align*}
But, for $m> 2+\sigma /2$ this is dominated by $C 2^{-N}$, where $C$ is a constant depending on $\sigma$ and $\psi$.
From this inequality the proof follows.
\end{proof}

The following propositions \ref{prop:com2} and \ref{prop:com3} are about the estimate of commutators of paraproducts and the heat semigroup, respectively.

\begin{prop}\label{prop:com2}
Let $\alpha \in (0,1)$, $\beta \in {\mathbb R}$, $\gamma \in (-\infty ,0)$ and $p, p_1,p_2,p_3 \in [1,\infty ]$.
Assume that
\[
\beta + \gamma <0, \quad \alpha + \beta + \gamma >0, \quad \frac{1}{p} = \frac{1}{p_1} + \frac{1}{p_2} + \frac{1}{p_3}.
\]
Then,
\[
\| \left[ \rho (f \mbox{\textcircled{\scriptsize$<$}} g) \right] \mbox{\textcircled{\scriptsize$=$}} h - \rho f (g \mbox{\textcircled{\scriptsize$=$}} h)\| _{B_{p}^{\alpha + \beta + \gamma} (\nu )} \leq C \| \rho\| _{B_{\infty }^{\alpha + (\beta \vee 0)} (\nu )} \| f\| _{B_{p_1}^{\alpha} (\nu )} \| g\| _{B_{p_2}^{\beta} (\nu )} \| h\| _{B_{p_3}^{\gamma} (\nu )}
\]
for $\rho \in {\mathcal S}({\mathbb R}^3)$, $f\in B_{p_1}^{\alpha} (\nu )$, $g\in B_{p_2}^{\beta} (\nu )$ and $h\in B_{p_3}^{\gamma} (\nu )$ where $C$ is a constant depending on $\alpha$, $\beta$, $\gamma$, $p_1$, $p_2$ and $p_3$.
\end{prop}

\begin{proof}
Similarly to Proposition A.9 in \cite{MW3} we are able to show that
\begin{equation}\label{eq:propcom2-9}
\|  (\tilde f \mbox{\textcircled{\scriptsize$<$}} \tilde g) \mbox{\textcircled{\scriptsize$=$}} \tilde h - \tilde f (\tilde g \mbox{\textcircled{\scriptsize$=$}} \tilde h)\| _{B_{p}^{\alpha + \beta + \gamma} (\nu )} \leq C \| \tilde f\| _{B_{p_1}^{\alpha} (\nu )} \| \tilde g\| _{B_{p_2}^{\beta} (\nu )} \| \tilde h\| _{B_{p_3}^{\gamma} (\nu )}
\end{equation}
for $\tilde f\in B_{p_1}^{\alpha} (\nu )$, $\tilde g\in B_{p_2}^{\beta} (\nu )$, $\tilde h\in B_{p_3}^{\gamma} (\nu )$ and $(\alpha , \beta, \gamma )$ and $(p,p_1,p_2,p_3)$ satisfying the conditions in the statement.
By the triangle inequality we have
\begin{equation}\label{eq:propcom2-1}\begin{array}{l}
\displaystyle \| \left[ \rho (f \mbox{\textcircled{\scriptsize$<$}} g) \right] \mbox{\textcircled{\scriptsize$=$}} h - \rho f (g \mbox{\textcircled{\scriptsize$=$}} h)\| _{B_{p}^{\alpha + \beta + \gamma} (\nu )} \\
\displaystyle \leq \left\| \left[ \rho \mbox{\textcircled{\scriptsize$<$}} (f \mbox{\textcircled{\scriptsize$<$}} g) \right] \mbox{\textcircled{\scriptsize$=$}} h - \rho [ (f \mbox{\textcircled{\scriptsize$<$}} g) \mbox{\textcircled{\scriptsize$=$}} h ] \right\| _{B_{p}^{\alpha + \beta + \gamma} (\nu )} \\
\displaystyle \quad +  \left\| \rho \left[ (f \mbox{\textcircled{\scriptsize$<$}} g) \mbox{\textcircled{\scriptsize$=$}} h - f (g \mbox{\textcircled{\scriptsize$=$}} h ) \right] \right\| _{B_{p}^{\alpha + \beta + \gamma} (\nu )} + \| \left[ \rho \mbox{\textcircled{\scriptsize$\geqslant$}} (f \mbox{\textcircled{\scriptsize$<$}} g) \right] \mbox{\textcircled{\scriptsize$=$}} h \| _{B_{p}^{\alpha + \beta + \gamma} (\nu )} .
\end{array}\end{equation}
Let $q\in [1, \infty]$ such that $1/p = 1/q + 1/p_3$.
Applying (\ref{eq:propcom2-9}) and Proposition \ref{prop:paraproduct}, we have
\begin{equation}\label{eq:propcom2-2}\begin{array}{l}
\displaystyle \left\| \left[ \rho \mbox{\textcircled{\scriptsize$<$}} (f \mbox{\textcircled{\scriptsize$<$}} g) \right] \mbox{\textcircled{\scriptsize$=$}} h - \rho [ (f \mbox{\textcircled{\scriptsize$<$}} g) \mbox{\textcircled{\scriptsize$=$}} h ] \right\| _{B_{p}^{\alpha + \beta + \gamma} (\nu )} \\
\displaystyle \leq C \left\| \rho \right\| _{B_\infty ^\alpha (\nu )} \left\| f \mbox{\textcircled{\scriptsize$<$}} g \right\| _{B_{q}^{\beta} (\nu )}  \| h \| _{B_{p_3}^{\gamma} (\nu )} \\
\displaystyle \leq C \| \rho\| _{B_{\infty }^{\alpha } (\nu )} \| f\| _{L^{p_1}(\nu )} \| g\| _{B_{p_2}^{\beta} (\nu )} \| h\| _{B_{p_3}^{\gamma} (\nu )} .
\end{array}\end{equation}
Proposition \ref{prop:paraproduct} and (\ref{eq:propcom2-9}) imply
\begin{equation}\label{eq:propcom2-3}\begin{array}{l}
\displaystyle \left\| \rho \left[ (f \mbox{\textcircled{\scriptsize$<$}} g) \mbox{\textcircled{\scriptsize$=$}} h - f (g \mbox{\textcircled{\scriptsize$=$}} h ) \right] \right\| _{B_{p}^{\alpha + \beta + \gamma} (\nu )} \\
\displaystyle \leq C \| \rho\| _{B_{\infty }^{\alpha + \beta + \gamma} (\nu )} \left\| (f \mbox{\textcircled{\scriptsize$<$}} g) \mbox{\textcircled{\scriptsize$=$}} h - f (g \mbox{\textcircled{\scriptsize$=$}} h ) \right\| _{B_{p}^{\alpha + \beta + \gamma} (\nu )} \\
\displaystyle \leq  C \| \rho\| _{B_{\infty }^{\alpha} (\nu )} \| f\| _{B_{p_1}^{\alpha} (\nu )} \| g\| _{B_{p_2}^{\beta} (\nu )} \| h\| _{B_{p_3}^{\gamma} (\nu )} .
\end{array}\end{equation}
Proposition \ref{prop:paraproduct} implies
\begin{equation}\label{eq:propcom2-4}\begin{array}{l}
\displaystyle \| \left[ \rho \mbox{\textcircled{\scriptsize$\geqslant$}} (f \mbox{\textcircled{\scriptsize$<$}} g) \right] \mbox{\textcircled{\scriptsize$=$}} h \| _{B_{p}^{\alpha + \beta + \gamma} (\nu )} \\
\displaystyle \leq C \left\| \rho \mbox{\textcircled{\scriptsize$\geqslant$}} (f \mbox{\textcircled{\scriptsize$<$}} g) \right\| _{B_{q}^{\alpha + \beta} (\nu )}  \| h \| _{B_{p_3}^{\gamma} (\nu )} \\
\displaystyle \leq C \left\| \rho \right\| _{B_\infty ^{\alpha + (\beta \vee 0)} (\nu )}\left\| f \mbox{\textcircled{\scriptsize$<$}} g \right\| _{B_{q}^{\beta \wedge 0}  (\nu )}  \| h \| _{B_{p_3}^{\gamma} (\nu )} \\
\displaystyle \leq C \| \rho\| _{B_{\infty }^{\alpha + (\beta \vee 0)} (\nu )} \| f\| _{B_{p_1}^{\alpha} (\nu )} \| g\| _{B_{p_2}^{\beta} (\nu )} \| h\| _{B_{p_3}^{\gamma} (\nu )} .
\end{array}\end{equation}
Therefore, we have the assertion from (\ref{eq:propcom2-1})--(\ref{eq:propcom2-4}).
\end{proof}

\begin{prop}\label{prop:com3}
Assume that $\nu (x)= (1+|x|^2)^{-\sigma /2}$ for some $\sigma \in [0,\infty )$.
For $\alpha \in [0,1)$, $\beta \in (-\infty ,0)$, $\gamma \in {\mathbb R}$ such that $\gamma \geq \alpha + \beta$, $\varepsilon \in (0, \max\{ 1,1-\alpha \}]$, and $p, p_1, p_2 \in [1,\infty]$ such that $1/p = 1/p_1 + 1/p_2$, there exists a constant $C$ depending on $\sigma$, $\psi$, $\rho$, $\alpha$, $\beta$, $\gamma$ and $\varepsilon$ such that
\begin{align*}
& \left\| e^{t\triangle } P_N^2\left[ \rho _M \left( f \mbox{\textcircled{\scriptsize$<$}} g\right) \right] - \rho _M \left[ f \mbox{\textcircled{\scriptsize$<$}} \left( e^{t\triangle } P_N^2 g\right) \right] \right\| _{B_p^{\gamma }(\nu )} \\
& \leq C (1+ M^\sigma 2^{-N} ) \left( 1 +t^{-(\gamma - \alpha -\beta )/2} \right) \| f \|_{B_{p_1}^{\alpha +\varepsilon}(\nu )} \| g \|_{B_{p_2}^{\beta }(\nu )} 
\end{align*}
for $f,g \in C_0^\infty ({\mathbb R}^3)$ such that ${\rm supp} f, {\rm supp} g \subset \{ x\in {\mathbb R}^3; |x|\leq 2M\}$, $t\in (0,\infty )$ and $N\in {\mathbb N}$.
\end{prop}

\begin{proof}
For simplicity we denote $C \left( 1 +t^{-(\gamma - \alpha -\beta )/2} \right)$ by $C_t$, where the constant $C$ can be different from line to line.
We remark that similarly to Proposition 2.6 in \cite{AK} we are able to show that
\begin{equation}\label{eq:propcom3-9}
\left\| e^{t\triangle} P_N^2 (\tilde f \mbox{\textcircled{\scriptsize$<$}} \tilde g) - \tilde f \mbox{\textcircled{\scriptsize$<$}} (e^{t\triangle} P_N^2 \tilde g) \right\| _{B_p^\gamma (\nu)} \leq C_t \| \tilde f \| _{B_{p_1}^{\alpha} (\nu)} \| \tilde g\| _{B_{p_2}^\beta (\nu)}
\end{equation}
for $\tilde f\in B_{p_1}^{\alpha}(\nu )$ and $\tilde g \in B_{p_1}^\beta (\nu )$.
The triangle inequality and Proposition \ref{prop:paraproduct} imply
\begin{align*}
& \left\| e^{t\triangle } P_N^2\left[ \rho _M \mbox{\textcircled{\scriptsize$<$}} \left( f \mbox{\textcircled{\scriptsize$<$}} g\right) \right] - \rho _M \mbox{\textcircled{\scriptsize$<$}} \left[ f \mbox{\textcircled{\scriptsize$<$}} \left( e^{t\triangle } P_N^2 g\right) \right] \right\| _{B_p^{\gamma }(\nu )} \\
&\leq \left\| e^{t\triangle } P_N^2\left[ \rho _M \mbox{\textcircled{\scriptsize$<$}} \left( f \mbox{\textcircled{\scriptsize$<$}} g\right) \right] - \rho _M \mbox{\textcircled{\scriptsize$<$}} \left[ e^{t\triangle } P_N^2 \left( f \mbox{\textcircled{\scriptsize$<$}} g\right) \right] \right\| _{B_p^{\gamma }(\nu )} \\
&\quad + \left\| \rho _M \mbox{\textcircled{\scriptsize$<$}} \left[ e^{t\triangle } P_N^2 \left( f \mbox{\textcircled{\scriptsize$<$}} g\right) - f \mbox{\textcircled{\scriptsize$<$}} \left( e^{t\triangle } P_N^2 g\right) \right] \right\| _{B_p^{\gamma }(\nu )} \\
&\leq \left\| e^{t\triangle } P_N^2\left[ \rho _M \mbox{\textcircled{\scriptsize$<$}} \left( f \mbox{\textcircled{\scriptsize$<$}} g\right) \right] - \rho _M \mbox{\textcircled{\scriptsize$<$}} \left[ e^{t\triangle } P_N^2 \left( f \mbox{\textcircled{\scriptsize$<$}} g\right) \right] \right\| _{B_p^{\gamma }(\nu )} \\
&\quad + \| \rho _M \| _{L^\infty (\nu )} \left\|  e^{t\triangle } P_N^2 \left( f \mbox{\textcircled{\scriptsize$<$}} g\right) - f \mbox{\textcircled{\scriptsize$<$}} \left( e^{t\triangle } P_N^2 g\right) \right\| _{B_p^{\gamma }(\nu )} .
\end{align*}
Hence, applying (\ref{eq:propcom3-9}) and Proposition \ref{prop:paraproduct}, we have
\begin{align*}
& \left\| e^{t\triangle } P_N^2\left[ \rho _M \mbox{\textcircled{\scriptsize$<$}} \left( f \mbox{\textcircled{\scriptsize$<$}} g\right) \right] - \rho _M \mbox{\textcircled{\scriptsize$<$}} \left[ f \mbox{\textcircled{\scriptsize$<$}} \left( e^{t\triangle } P_N^2 g\right) \right] \right\| _{B_p^{\gamma }(\nu )} \\
&\leq C_t \| \rho _M \|_{B_{\infty}^{\alpha}(\nu )} \| f \mbox{\textcircled{\scriptsize$<$}} g \|_{B_{p}^{\beta }(\nu )} + C_t \| f \|_{B_{p_1}^{\alpha}(\nu )} \| g \|_{B_{p_2}^{\beta }(\nu )} \\
&\leq C_t \left( \| f \|_{L^{p_1}(\nu )} \| g \|_{B_{p_2}^{\beta }(\nu )} + \| f \|_{B_{p_1}^{\alpha}(\nu )} \| g \|_{B_{p_2}^{\beta }(\nu )} \right) .
\end{align*}
In view of this inequality, to finish the proof it is sufficient to show that
\begin{align}
\left\| e^{t\triangle } P_N^2\left[ \rho _M \mbox{\textcircled{\scriptsize$\geqslant$}} \left( f \mbox{\textcircled{\scriptsize$<$}} g\right) \right] \right\| _{B_p^{\gamma }(\nu )}&\leq  C_t M^\sigma \| f \|_{B_{p_1}^{\alpha +\varepsilon}(\nu )} \| g \|_{B_{p_2}^{\beta }(\nu )} \label{eq:propcom3-1}\\
\left\| \rho _M \mbox{\textcircled{\scriptsize$\geqslant$}} \left[ f \mbox{\textcircled{\scriptsize$<$}} \left( e^{t\triangle } P_N^2 g\right) \right] \right\| _{B_p^{\gamma }(\nu )} &\leq  C M^\sigma \| f \|_{B_{p_1}^{\alpha}(\nu )} \| g \|_{B_{p_2}^{\beta }(\nu )} \label{eq:propcom3-2}.
\end{align}
Proposition 2.76 in \cite{BCD}, Proposition \ref{prop:Besov} and the support of $g$ imply
\begin{align*}
\| g \|_{B_{p_2}^\beta } &\leq C\sup _{h\in {\mathcal S}({\mathbb R}^3);\ \| h\| _{B_{p_1,1}^{-\beta}} \leq 1} \langle h, g\rangle \\
&\leq C \sup _{h\in {\mathcal S}({\mathbb R}^3); \ \| h\| _{B_{p_1,1}^{-\beta}}\leq 1} \int _{{\mathbb R}^3} \frac{\rho _{2M} h}{\nu} g \nu dx \\
&\leq C \| g \|_{B_{p_2}^\beta (\nu )} \sup _{h\in {\mathcal S}({\mathbb R}^3);\ \| h\| _{B_{p_1,1}^{-\beta}}\leq 1} \left\| \frac{\rho _{2M} h}{\nu} \right\| _{B_{p_1}^{-\beta} (\nu)} .
\end{align*}
Note that $\rho _{2M}/\nu \in C_0^\infty ({\mathbb R}^3)$ and that $\rho _{M}/\nu$ and their derivatives are dominated by $C M^\sigma$ with a constant $C$.
Since this fact implies
\[
\left\| \frac{\rho _{2M} h}{\nu} \right\| _{B_{p_1,1}^{-\beta} (\nu)} \leq C M^\sigma \| h \| _{B_{p_1,1}^{-\beta}},
\]
we have
\begin{equation}\label{eq:propcom3-g}
\| g \|_{B_{p_2}^\beta } \leq  C M^\sigma \| g \|_{B_{p_2}^\beta (\nu )}.
\end{equation}
By the triangle inequality, Propositions \ref{prop:Besov} and \ref{prop:paraproduct}, Lemma \ref{lem:comP0} and the supports of $f$ and $g$, we have
\begin{align*}
&\left\| e^{t\triangle } P_N^2\left[ \rho _M \mbox{\textcircled{\scriptsize$\geqslant$}} \left( f \mbox{\textcircled{\scriptsize$<$}} g\right) \right] \right\| _{B_p^{\gamma }(\nu )} \\
&\leq \left\| e^{t\triangle } \left[ \rho _M \mbox{\textcircled{\scriptsize$\geqslant$}} \left( f \mbox{\textcircled{\scriptsize$<$}} g\right) \right] \right\| _{B_p^{\gamma }(\nu )} + \left\| (I - P_N^2) e^{t\triangle } \left[ \rho _M \mbox{\textcircled{\scriptsize$\geqslant$}} \left( f \mbox{\textcircled{\scriptsize$<$}} g\right) \right] \right\| _{B_p^{\gamma }(\nu )} \\
&\leq C_t \left\| \rho _M \mbox{\textcircled{\scriptsize$\geqslant$}} \left( f \mbox{\textcircled{\scriptsize$<$}} g\right) \right\| _{B_p^{\alpha + \beta}(\nu )} + \left\| (I - P_N^2) e^{t\triangle } \left[ \rho _M \mbox{\textcircled{\scriptsize$\geqslant$}} \left( f \mbox{\textcircled{\scriptsize$<$}} g\right) \right] \right\| _{B_p^{\gamma }} \\
&\leq C_t \left\| \rho _M \right\| _{B_\infty^{\alpha +1} (\nu )} \left\| f \mbox{\textcircled{\scriptsize$<$}} g \right\| _{B_p^{\beta }(\nu )} + C_t 2^{-N} \left\| \rho _M \mbox{\textcircled{\scriptsize$\geqslant$}} \left( f \mbox{\textcircled{\scriptsize$<$}} g\right) \right\| _{B_p^{\alpha + \beta + 1}} \\
&\leq C_t \left( \left\| f \mbox{\textcircled{\scriptsize$<$}} g \right\| _{B_p^{\beta }(\nu )} + 2^{-N} \left\| \rho _M \right\| _{B_\infty ^{\alpha +2}} \left\|  f \mbox{\textcircled{\scriptsize$<$}} g \right\| _{B_p^{\beta }} \right) \\
&\leq C_t \left( \| f \|_{L^{p_1}(\nu )} \| g \|_{B_{p_2}^{\beta }(\nu )} + 2^{-N} \| f \|_{L^{p_1}(1)} \| g \|_{B_{p_2}^\beta }\right) \\
&\leq C_t \left( \| f \|_{B_{p_1}^{\alpha +\varepsilon}(\nu )} \| g \|_{B_{p_2}^{\beta }(\nu )} +  C 2^{-N} \| f \|_{L^{p_1}(\nu )} \| g \|_{B_{p_2}^\beta }\right) .
\end{align*}
Thus, \eqref{eq:propcom3-1} follows from \eqref{eq:propcom3-g}.
From Propositions \ref{prop:Besov} and \ref{prop:paraproduct}, we have
\begin{align*}
&\left\| \rho _M \mbox{\textcircled{\scriptsize$\geqslant$}} \left[ f \mbox{\textcircled{\scriptsize$<$}} \left( e^{t\triangle } P_N^2 g\right) \right] \right\| _{B_p^{\gamma }(\nu )} \\
&\leq C\left\| \rho _M \right\| _{B_\infty ^{-\beta + |\gamma | + 2}(\nu )} \left\| f \mbox{\textcircled{\scriptsize$<$}} \left( e^{t\triangle } P_N^2 g\right) \right\| _{B_p^{\beta -1}(\nu )} \\
&\leq C\left\| f \right\| _{L^{p_1} (\nu )} \left\| e^{t\triangle } P_N^2 g \right\| _{B_{p_2} ^{\beta -1}(\nu )} \\
&\leq C \left\| f \right\| _{B_{p_1}^\varepsilon (\nu )} \left\| P_N^2 g \right\| _{B_{p_2} ^{\beta -1}(\nu )} \\
&\leq C \left\| f \right\| _{B_{p_1}^\varepsilon (\nu )} \left( \left\| g \right\| _{B_{p_2} ^{\beta -1}(\nu )} + \left\| (I-P_N^2) g \right\| _{B_{p_2} ^{\beta -1}(\nu )}\right) .
\end{align*}
Since Lemma \ref{lem:comP0} and \eqref{eq:propcom3-g} imply
\[
\left\| (I-P_N^2) g \right\| _{B_{p_2} ^{\beta -1}(\nu )} \leq \left\| (I-P_N^2) g \right\| _{B_{p_2} ^{\beta -1}} \leq C 2^{-N} \left\| g \right\| _{B_{p_2} ^{\beta}} \leq  C M^\sigma 2^{-N} \left\| g \right\| _{B_{p_2} ^{\beta}(\nu )},
\]
\eqref{eq:propcom3-2} follows.
\end{proof}

\section{Infinite-dimensional Ornstein-Uhlenbeck process}\label{sec:OU}

In this section we introduce the relevant infinite-dimensional Ornstein-Uhlenbeck process and its polynomials.
The polynomials of the Ornstein-Uhlenbeck process appear in the renormalization and the transformation of the stochastic partial differential equation associated to the $\Phi ^4_3$-model on the full Euclidean space ${\mathbb R}^3$.

Let $\mu _0$ be the centered Gaussian measure on ${\cal S}'({\mathbb R}^3)$ with the covariance operator $[2(-\triangle + m_0^2)]^{-1}$ where $m_0>0$ as in Section \ref{sec:main}, let $\dot W_t(x)$ be a Gaussian white noise with parameter $(t,x)\in  (-\infty ,\infty) \times {\mathbb R}^3$, and let $W_t$ be the cylindrical Brownian motion on $L^2(dx)$ given by
\[
\langle f, W_t \rangle = \left\{ \begin{array}{cl}
_{{\cal S}((-\infty ,\infty) \times {\mathbb R}^3)} \langle f \otimes {\mathbb I}_{[0,t]}, \dot W \rangle _{{\cal S}'((-\infty ,\infty) \times {\mathbb R}^3)}, &t\in [0,\infty )\\
_{{\cal S}((-\infty ,\infty) \times {\mathbb R}^3)} \langle f \otimes {\mathbb I}_{[t,0]}, \dot W \rangle _{{\cal S}'((-\infty ,\infty) \times {\mathbb R}^3)}, &t\in (-\infty ,0),
\end{array}\right.
\]
for $f\in {\mathcal S}({\mathbb R}^3)$.
We define the ${\mathcal S}'({\mathbb R}^3)$-valued stochastic process $Z_t$ by the stochastic convolution
\begin{equation}\label{eq:solZ}
Z_t = \int _{-\infty }^t e^{(t-s)(\triangle - m_0^2)} d W_s , \quad t\in (-\infty ,\infty ).
\end{equation}
We remark that (\ref{eq:solZ}) is an equation not only for positive $t$, but also for negative $t$.
It is well-known that $Z_t$ is the infinite-dimensional Ornstein-Uhlenbeck process and the law of $Z_t$ is $\mu _0$ for $t\in (-\infty ,\infty )$.
Hence, $Z_t$ is the stationary Ornstein-Uhlenbeck process.
Moreover, $Z_t$ is the mild solution to the stochastic partial differential equation on ${\mathbb R}^3$:
\begin{equation}\label{eq:SDEZ}
\partial _t Z_t(x) = \dot W_t(x) - (-\triangle +m_0^2)Z_t(x), \quad (t,x)\in (-\infty ,\infty) \times {\mathbb R}^3.
\end{equation}

To find the covariance operator of our approximation $P_N Z_t $ of $Z_t$ with $P_N$ as in Section \ref{sec:main}, we prepare a lemma as follows.

\begin{lem}\label{lem:covPZ}
For $f,g \in {\mathcal S}({\mathbb R}^3)$ and $s,t\in (-\infty ,\infty )$ such that $s<t$
\[
E\left[ \langle f, P_N Z_t \rangle \langle g, P_N Z_s \rangle \right] = \langle V_N (t-s) f,g \rangle
\]
where $\{ V_N (t): t\in [0,\infty )\}$ is a family of bounded linear operators on ${B_{p,r}^\alpha}$ for $\alpha \in {\mathbb R}$ and $p,r\in [1,\infty ]$ given by
\[
V_N (t) f := P_N ^2 [2(m_0^2-\triangle )]^{-1} e^{t(\triangle -m_0^2)} f.
\]
Moreover, the following estimates hold:
\begin{enumerate}
\item \label{lem:covPZ1} For $\alpha \in {\mathbb R}$ and $\beta , \gamma \in [0,\infty )$
\[
\| V_N(t) f \| _{B_{p,r}^{\alpha + \beta + \gamma}} \leq C 2^{\beta N} (1+ t^{-\gamma /2}) e^{-m_0^2 t} \| f\| _{B_{p,r}^{\alpha -2}}, \quad t\in ( 0,\infty ) ,\ f\in {\mathcal S}({\mathbb R}^3) ,
\]
where $C$ is a constant depending on $\alpha$, $\beta$ and $\gamma$.

\item \label{lem:covPZ2} For $\alpha \in {\mathbb R}$ and $\beta \in [0,\infty )$
\[
\| V_N(0) f \| _{B_{p,r}^{\alpha + \beta}} \leq C 2^{\beta N} \| f\| _{B_{p,r}^{\alpha -2}}, \quad f\in {\mathcal S}({\mathbb R}^3)
\]
where $C$ is a constant depending on $\alpha$ and $\beta$.

\item \label{lem:covPZ3} For $\alpha \in {\mathbb R}$, $\beta \in [0,\infty )$ and $\varepsilon \in (0,1)$
\[
\sum _{N=1}^\infty 2^{-\beta N} \left\| (P_{N+1} -P_N) \left[ 2(m_0^2 - \triangle )\right] ^{-1} f \right\| _{B_{p,r}^{\alpha + \beta -\varepsilon}} \leq C \| f\| _{B_{p,r}^{\alpha -2}}
\]
for $f\in {\mathcal S}({\mathbb R}^3)$, where $C$ is a constant depending on $\alpha$, $\beta$ and $\varepsilon$.

\item \label{lem:covPZ4} For $\alpha \in {\mathbb R}$, $\beta , \gamma \in [0,\infty )$ and $s,t\in [0,\infty )$ such that $s<t$
\[
\| (V_N(t) -V_N(s)) f \| _{B_{p,r}^{\alpha + \beta}} \leq C 2^{\beta N} (t-s)^{\gamma /2} \| f\| _{B_{p,r}^{\alpha +\gamma -2}}, \quad f\in {\mathcal S}({\mathbb R}^3) ,
\]
where $C$ is a constant depending on $\alpha$, $\beta$ and $\gamma$.
\end{enumerate}
\end{lem}

\begin{proof}
For fixed $s,t$, consider the time reversal $\tilde W_u := W_{s-u}$.
Then, we have
\begin{align*}
&E\left[ \langle f, P_N Z_t \rangle \langle g, P_N Z_s \rangle \right] \\
&= E\left[ \int _{-\infty }^t \left\langle e^{(t-u)(\triangle - m_0^2)} P_N f, d W_u \right\rangle \int _{-\infty }^s \left\langle e^{(s-u)(\triangle - m_0^2)} P_N g, d W_u \right\rangle \right] \\
&= E\left[ \int _{s-t}^\infty \left\langle e^{(t-s+u)(\triangle - m_0^2)} P_N f, d \tilde W_u \right\rangle \int _0^\infty \left\langle e^{u(\triangle - m_0^2)} P_N g, d \tilde W_u \right\rangle \right] .
\end{align*}
Hence, by the fact that $\tilde W_u$ is also the white noise, It\^o's isometry yields
\begin{align*}
&E\left[ \langle f, P_N Z_t \rangle \langle g, P_N Z_s \rangle \right] \\
&= \int _0^\infty \left\langle e^{(t-s+u)(\triangle - m_0^2)} P_N f,  e^{u(\triangle - m_0^2)} P_N g \right\rangle du \\
&= \left\langle P_N \left( \int _0^\infty e^{2u(\triangle - m_0^2)} du \right) e^{(t-s)(\triangle - m_0^2)} P_N f, g \right\rangle .
\end{align*}
In view of the commutativity between $P_N$ and the heat semigroup, we have the first assertion.
To show \ref{lem:covPZ1}, note that $\int _0^\infty e^{2u(\triangle - m_0^2)} du$ coincides with the resolvent operator $[2(m_0^2-\triangle )]^{-1}$, and $[2(m_0^2-\triangle )]^{-1}$ is a bounded operator from $B_{p,r}^{\alpha -2}$ to $B_{p,r}^{\alpha}$.
By using this fact, Proposition \ref{prop:paraproduct} and Lemmas \ref{lem:PN} and \ref{lem:comP0}, we calculate that
\begin{align*}
\| V_N (t) f \| _{B_{p,r}^{\alpha + \beta + \gamma}}
&\leq C \left\| P_N ^2 [2(m_0^2-\triangle )]^{-1} e^{t(\triangle -m_0^2)} f \right\| _{B_{p,r}^{\alpha + \beta + \gamma}} \\
&\leq C 2^{\beta N} (1+ t^{-\gamma /2}) e^{-m_0^2 t} \left\| P_N ^2 f \right\| _{B_{p,r}^{\alpha -2}} \\
&\leq C 2^{\beta N} (1+ t^{-\gamma /2}) e^{-m_0^2 t} \left\| f \right\| _{B_{p,r}^{\alpha -2}}.
\end{align*}
Hence, the estimate for $t\in (0,\infty )$ holds.
Similarly, we have \ref{lem:covPZ2}.
We show \ref{lem:covPZ3} in the case that $r\in [1,\infty )$.
Let $g = [2(m_0^2-\triangle )]^{-1} f$.
In view of the support of $\varphi$ and $\psi$, we have
\begin{align*}
& \sum _{N=1}^\infty 2^{-\beta N} \left\|  (P_{N+1} -P_N) g \right\| _{B_{p,r}^{\alpha +\beta -\varepsilon}} \\
&= \sum _{N=1}^\infty 2^{-\beta N} \left( \sum _{j=-1}^\infty 2^{(\alpha +\beta -\varepsilon )rj} \left\| (P_{N+1} -P_N) \Delta _j g \right\| _{L^p}^r \right) ^{1/r} \\
&\leq C \sum _{N=1}^\infty 2^{-\varepsilon N} \left( \sum _{j=-1}^\infty 2^{\alpha rj} \left\| (P_{N+1} -P_N)^2 \Delta _j g \right\| _{L^p}^r \right) ^{1/r} .
\end{align*}
Since Young's inequality and a equation of $\psi _N$ similar to \eqref{eq:scalephi} imply
\begin{align*}
\left\| (P_{N+1} -P_N) \Delta _j g \right\| _{L^p} &= \left| ( {\mathcal F}^{-1}\psi _{N+1}-{\mathcal F}^{-1}\psi _N ) * \Delta _j g \right\| _{L^p} \\
&\leq \left\| {\mathcal F}^{-1}\psi _{N+1}- {\mathcal F}^{-1}\psi _N \right\| _{L^1} \left\| \Delta _j g \right\| _{L^p} ,
\end{align*}
and
\begin{align*}
\left\| {\mathcal F}^{-1}\psi _{N+1}- {\mathcal F}^{-1}\psi _N \right\| _{L^1} &= 2^{3N} \left\| \left( {\mathcal F}^{-1}\psi _1 \right) (2^N \cdot )- \left( {\mathcal F}^{-1}\psi \right) (2^N \cdot ) \right\| _{L^1} \\
&= \left\| {\mathcal F}^{-1}\psi _1 - {\mathcal F}^{-1}\psi \right\| _{L^1}  ,
\end{align*}
respectively, we have
\[
\sum _{N=1}^\infty 2^{-\beta N} \left\|  (P_{N+1} -P_N) g \right\| _{B_{p,r}^{\alpha +\beta -\varepsilon}} \leq C \left\| g\right\| _{B_{p,r}^{\alpha}} .
\]
Hence, \ref{lem:covPZ3} follows from the fact that $\left\| g\right\| _{B_{p,r}^{\alpha}} \leq C \left\| f\right\| _{B_{p,r}^{\alpha -2}}$.
The case where $r=\infty$ is similar. So we omit it.

We also obtain \ref{lem:covPZ4} similarly, because
\[
\left\| \left( e^{t(\triangle -m_0^2)} - e^{s(\triangle -m_0^2)} \right) f\right\| _{B_{p,r}^{\alpha -2}} \leq C (t-s)^{\gamma /2} \| f\| _{B_{p,r}^{\alpha +\gamma -2}}.
\]
\end{proof}

Now we consider the renormalization of polynomials of the Ornstein-Uhlenbeck processes.
Let 
\begin{align*}
C_1^{(N)} &:= \langle V_N(0) \delta _0, \delta _0 \rangle \\
C_2^{(M,N)} (x) &:= 2\sum _{i,j=-1}^\infty {\mathbb I}_{[-1,1]}(i-j) \int _0^\infty e^{t(\triangle -m_0^2)} P_N^2 \left[ \left( \Delta _i P_N \left[ \rho _M^2 V_N(t) \left( \rho _M^2 \Delta _j P_N \delta _x \right) \right] \right) ^2 \right] (x) dt
\end{align*}
where $\delta _x$ is the Dirac delta function.
Note that $C_1^{(N)}$ and $C_2^{(M,N)} (x)$ are same as those in Section \ref{sec:main}.
Define
\begin{align*}
{\mathcal Z}^{(1,N)}_t &:= P_N Z_t ,\\
{\mathcal Z}^{(2,N)}_t &:= (P_N Z_t) ^2 - C_1^{(N)} , \\
{\mathcal Z}^{(3,N)}_t &:= (P_N Z_t) ^3 - 3 C_1^{(N)} P_N Z_t , \\
{\mathcal Z}^{(1,M,N)}_t &:= P_{M,N} Z_t ,\\
{\mathcal Z}^{(2,M,N)}_t &:= (P_{M,N} Z_t) ^2 - \rho _M ^2 C_1^{(N)} , \\
{\mathcal Z}^{(3,M,N)}_t &:= (P_{M,N} Z_t) ^3 - 3 \rho _M^2 C_1^{(N)} P_{M,N} Z_t , \\
{\mathcal Z}^{(0,2,M,N)}_t &:= \int _{-\infty}^t e^{(t-s)(\triangle -m_0^2)} P_{M,N}^* {\mathcal Z}^{(2,M,N)}_s ds , \\
{\mathcal Z}^{(0,3,M,N)}_t &:= \int _{-\infty}^t e^{(t-s)(\triangle -m_0^2)} P_{M,N}^* {\mathcal Z}^{(3,M,N)}_s ds , \\
{\mathcal Z}^{(2,2,M,N)}_t &:= {\mathcal Z}^{(2,M,N)}_t \mbox{\textcircled{\scriptsize$=$}}\int _{-\infty}^t e^{(t-s)(\triangle -m_0^2)} P_N^2 {\mathcal Z}^{(2,M,N)}_s ds -C_2^{(M,N)} , \\
{\mathcal Z}^{(2,3,M,N)}_t &:= {\mathcal Z}^{(2,M,N)}_t \mbox{\textcircled{\scriptsize$=$}}\int _{-\infty}^t e^{(t-s)(\triangle -m_0^2)} P_N^2 {\mathcal Z}^{(3,M,N)}_s ds - 3 C_2^{(M,N)}{\mathcal Z}^{(1,M,N)}_t , \\
\widehat{\mathcal Z}^{(2,3,M,N)}_t& := {\mathcal Z}^{(2,M,N)}_t \mbox{\textcircled{\scriptsize$=$}} \left( P_{M,N}{\mathcal Z}^{(0,3,M,N)}_t \right) - 3 C_2^{(M,N)} \rho _M^2 {\mathcal Z}^{(1,M,N)}_t
\end{align*}
for $t\in (-\infty , \infty)$ and $N\in {\mathbb N}$.
We remark that
\[
{\mathcal Z}^{(k,M,N)}_t = \rho _M^k {\mathcal Z}^{(k,N)}_t
\]
for $k=1,2,3$.
Moreover, in view of the stationarity of $Z$, by explicit calculation (see the proof of Proposition \ref{prop:Z}) we have
\begin{align*}
C_1^{(N)} &= E\left[ (P_N Z_t) ^2 (x) \right] \\
C_2^{(M,N)}(x) &= E\left[ \left( {\mathcal Z}^{(2,M,N)}_t \mbox{\textcircled{\scriptsize$=$}}\int _{-\infty}^t e^{(t-s)(\triangle -m_0^2)} P_N^2 {\mathcal Z}^{(2,M,N)}_s ds \right) (x) \right]
\end{align*}
for $x\in {\mathbb R}^3$ and $M,N \in {\mathbb N}$.
In particular, $C_1^{(N)}$ is independent of $x$ and $t$, and ${\mathcal Z}^{(k,N)}$ is a Wick polynomial of $P_N Z$ for $k=1,2,3$.
For later use, we see asymptotics of $C_1^{(N)}$ and $C_2^{(M,N)}(x)$ as $N$ goes to infinity.

\begin{prop}\label{prop:C}
There exist positive constants $c_1$ and $c_2$ such that
\begin{align*}
&\lim _{N\rightarrow \infty} 2^{-N} C_1^{(N)} = c_1 , \\
&-c_2 \leq C_2^{(M,N)}(x) \leq c_2 N = \frac{c_2}{\log 2} \cdot \log 2^{N}
\end{align*}
for $x\in {\mathbb R}^3$ and $M, N\in {\mathbb N}$.
\end{prop}

\begin{proof}
Since
\begin{align*}
\left[ P_N ^2 \delta _0 \right] (x) &= (2\pi )^{-3/2} \left[ {\mathcal F}^{-1}(\psi _N^2)\right] (x) \\
&= (2\pi )^{-3/2} 2^{3N} \left[ {\mathcal F}^{-1}(\psi ^2)\right] (2^N x) ,\\
\left[ [2(m_0^2-\triangle )]^{-1} \delta _0 \right] (x)&= \int _0^\infty (8\pi u)^{-3/2} \exp \left(-\frac{|x|^2}{8u} -2m_0^2 u\right) du,
\end{align*}
we are able to calculate $C_1^{(N)}$ explicitly as
\begin{align*}
C_1^{(N)} &= \frac{2^{3N}}{8 (2\pi)^3} \int _0^\infty \int _{{\mathbb R}^3} u^{-3/2} \left[ {\mathcal F}^{-1}(\psi ^2)\right] (2^N x) \exp \left(-\frac{|x|^2}{8u} -2m_0^2 u\right) dx du\\
&= \frac{2^N}{8 (2\pi)^3} \int _0^\infty \int _{{\mathbb R}^3} v^{-3/2} \left[ {\mathcal F}^{-1}(\psi ^2)\right] (y) \exp \left(-\frac{|y|^2}{8v} -2^{-2N+1} m_0^2 v \right) dy dv \\
&= \frac{2^N}{(2\pi )^{3/2}} \int _0^\infty \int _{{\mathbb R}^3} (\psi ^2)(\xi ) \exp \left(-\frac{4v|\xi |^2}{2} -2^{-2N+1} m_0^2 v \right) d\xi dv  \\
&= \frac{2^N}{(2\pi )^{3/2}} \int _{{\mathbb R}^3} \frac{(\psi ^2)(\xi )}{2|\xi |^2 +2^{-2N+1} m_0^2} d\xi.
\end{align*}
Hence, by the compactness of the support of $\psi$ we have the first assertion.

For $i\in {\mathbb N}\cup \{ -1,0\}$, $\alpha \in {\mathbb R}$ and $x\in {\mathbb R}^3$, the fact that $\Delta _i \Delta _j = 0$ for $|i-j|\geq 2$ and the Plancherel theorem imply
\begin{align*}
\| \Delta _i \delta _x\| _{B_{2,2}^{\alpha}}^2 &= \sum _{j=-1}^\infty 2^{2\alpha j} \| \Delta _i \Delta _j \delta _x\| _{L^2} ^2 \\
&\leq C 2^{2\alpha i} \| \Delta _i \delta _x\| _{L^2}^{2} \\
&= C 2^{2\alpha i} \| \psi _i {\mathcal F} \delta _x\| _{L^2}^{2} .
\end{align*}
Hence, by the explicit calculation:
\[
\| \psi _i {\mathcal F} \delta _x\| _{L^2}^{2} = \frac{1}{(2\pi )^{3/2}} \| \psi _i \| _{L^2}^{2} = \frac{2^{3i}}{(2\pi )^{3/2}} \| \psi \| _{L^2}^{2},
\]
we have
\begin{equation}\label{eq:normdirac}
\| \Delta _i \delta _x\| _{B_{2,2}^{\alpha}}^2 \leq C 2^{(2\alpha +3) i} 
\end{equation}
for $i\in {\mathbb N}\cup \{ -1,0\}$, $\alpha \in {\mathbb R}$ and $x\in {\mathbb R}^3$.
On the other hand, Proposition \ref{prop:paraproduct} and Lemma \ref{lem:covPZ} imply that
\begin{align*}
&\left\| \Delta _i P_N \left[ \rho _M^2 V_N(t) \left( \rho _M^2 \Delta _j P_N \delta _x \right) \right] \right\| _{B_{2,2}^{3/2}}\\
&\leq C \| \rho _M^2 \| _{B_{2,2}^{1/2+ \beta }} \left( \left\| [2(m_0^2-\triangle )]^{-1} e^{t(\triangle -m_0^2)} P_N^2 \left( \rho _M^2 \mbox{\textcircled{\scriptsize$<$}}\Delta _j P_N \delta _x \right) \right\| _{B_{2,2}^{3/2}} \right. \\
&\quad \hspace{7cm} \left. + \left\| V_N(t) \left( \rho _M^2 \mbox{\textcircled{\scriptsize$\geqslant$}}\Delta _j P_N \delta _x \right) \right\| _{B_{2,2}^{3/2}} \right) \\
&\leq C \left( \left\| e^{t(\triangle -m_0^2)} P_N^2 \left( \rho _M^2 \mbox{\textcircled{\scriptsize$<$}}\Delta _j P_N \delta _x \right) \right\| _{B_{2,2}^{-1/2}} + \left\| \rho _M^2 \mbox{\textcircled{\scriptsize$\geqslant$}}\Delta _j P_N \delta _x \right\| _{B_{2,2}^{-1/2}} \right).
\end{align*}
Since Proposition \ref{prop:Besov} and \eqref{eq:propcom3-9} imply that for sufficiently small $\varepsilon \in (0,1]$
\begin{align*}
&\left\| e^{t(\triangle -m_0^2)} P_N^2 \left( \rho _M^2 \mbox{\textcircled{\scriptsize$<$}}\Delta _j P_N \delta _x \right) \right\| _{B_{2,2}^{-1/2}}\\
&\leq \left\| e^{t(\triangle -m_0^2)} P_N^2 \left( \rho _M^2 \mbox{\textcircled{\scriptsize$<$}}\Delta _j P_N \delta _x \right) - \rho _M^2 \mbox{\textcircled{\scriptsize$<$}} \left( e^{t(\triangle -m_0^2)} P_N^2 \Delta _j P_N \delta _x \right) \right\| _{B_{2,\infty}^{-1/2+\varepsilon}}\\
&\quad \hspace{6cm} + \left\| \rho _M^2 \mbox{\textcircled{\scriptsize$<$}} \left( e^{t(\triangle -m_0^2)} P_N^2 (\Delta _j P_N \delta _x) \right) \right\| _{B_{2,2}^{-1/2}}\\
&\leq C (1+t^{-3\varepsilon /2}) \| \rho _M^2 \| _{B_{\infty ,\infty}^{1-\varepsilon}} \left\| \Delta _j P_N \delta _x \right\| _{B_{2,\infty }^{-3/2-\varepsilon}} + C \| \rho _M^2 \| _{L^\infty} \left\| P_N^3 (e^{t(\triangle -m_0^2)} \Delta _j  \delta _x) \right\| _{B_{2,2}^{-1/2}},
\end{align*}
by Propositions \ref{prop:Besov} and \ref{prop:paraproduct}, Lemma \ref{lem:comP0}, \cite[Lemma 2.4]{BCD} and \eqref{eq:normdirac} we have
\begin{align*}
&\left\| \Delta _i P_N \left[ \rho _M^2 V_N(t) \left( \rho _M^2 \Delta _j P_N \delta _x \right) \right] \right\| _{B_{2,2}^{3/2}}^2\\
&\leq C \left( (1+t^{-3\varepsilon /2}) \left\| \Delta _j P_N \delta _x \right\| _{B_{2,2}^{-3/2}} + \left\| P_N^3 (e^{t(\triangle -m_0^2)} \Delta _j  \delta _x) \right\| _{B_{2,2}^{-1/2}} + \left\| \rho _M^2 \mbox{\textcircled{\scriptsize$\geqslant$}}\Delta _j P_N \delta _x \right\| _{B_{2,2}^{-3/2}} \right) ^2\\
&\leq C \left( (1+t^{-3\varepsilon /2}) \left\| \Delta _j  \delta _x \right\| _{B_{2,2}^{-3/2}} + e^{-t2^{2j+1}} \left\| \Delta _j  \delta _x \right\| _{B_{2,2}^{-1/2}} + \| \rho _M^2 \| _{B_{\infty ,\infty}^{2}} \left\| \Delta _j \delta _x \right\| _{B_{2,2}^{-3/2}} \right) ^2\\
&\leq C \left( 1+ t^{-3\varepsilon} + 2^{2j} e^{-t2^{2j+1}} \right).
\end{align*}
From this inequality, Propositions \ref{prop:Besov} and \ref{prop:paraproduct}, Lemma \ref{lem:comP0} and Corollary \ref{cor:BS}, we have
\begin{align*}
&\left\| e^{t(\triangle -m_0^2)} P_N^2 \left[ \left( \Delta _i P_N \left[ \rho _M^2 V_N(t) \left( \rho _M^2 \Delta _j P_N \delta _x \right) \right] \right) ^2 \right] \right\| _{L^\infty} \\
&\leq C \left\| e^{t(\triangle -m_0^2)} P_N^2 \left[ \left( \Delta _i P_N \left[ \rho _M^2 V_N(t) \left( \rho _M^2 \Delta _j P_N \delta _x \right) \right] \right) ^2 \right] \right\| _{B_{\infty ,1}^0}\\
&\leq C e^{-m_0^2 t} \left\| \left( \Delta _i P_N \left[ \rho _M^2 V_N(t) \left( \rho _M^2 \Delta _j P_N \delta _x \right) \right] \right) ^2 \right\| _{B_{1,1}^{3/2}}\\
&\leq C e^{-m_0^2 t} \left\| \Delta _i P_N \left[ \rho _M^2 V_N(t) \left( \rho _M^2 \Delta _j P_N \delta _x \right) \right] \right\| _{B_{2,2}^{3/2}} ^2\\
&\leq C e^{-m_0^2 t} \left( 1+ t^{-3\varepsilon} + 2^{2j} e^{-t2^{2j+1}} \right) .
\end{align*}
In view of this inequality and \eqref{eq:Nj}, there exists $j_0\in {\mathbb N}$ independent of $M$ and $N$ such that
\begin{align*}
&\left\| \sum _{i,j=-1}^\infty {\mathbb I}_{[-1,1]}(i-j) \int _0^\infty e^{t(\triangle -m_0^2)} P_N^2 \left[ \left( \Delta _i P_N \left[ \rho _M^2 V_N(t) \left( \rho _M^2 \Delta _j P_N \delta _x \right) \right] \right) ^2 \right] dt \right\| _{L^\infty}\\
&\leq C \sum _{i,j=-1}^{N + j_0} {\mathbb I}_{[-1,1]}(i-j) \int _0^\infty e^{-m_0^2 t} \left( 1+ t^{-3\varepsilon} + 2^{2j} e^{-t2^{2j+1}} \right) dt\\
&\leq C N.
\end{align*}
Thus, we have the upper bound of $C_2^{M,N}(x)$.

Since for $f\in {\mathcal S}'({\mathbb R}^3)$ the Fourier transform of $(\Delta _i f)^2$ is supported in $\{ x\in {\mathbb R}^3; |x|\leq (16/3)2^{i}\}$, there exists $j_0\in {\mathbb N}$ independent of $M$ and $N$ such that 
\begin{align*}
&C_2^{(M,N)} (x)\\
&= 2\sum _{i,j=-1}^{N-j_0-1} {\mathbb I}_{[-1,1]}(i-j) \int _0^\infty e^{t(\triangle -m_0^2)} \left[ \left( \Delta _i P_N \left[ \rho _M^2 V_N(t) \left( \rho _M^2 \Delta _j P_N \delta _x \right) \right] \right) ^2 \right] (x) dt\\
&\quad + 2\sum _{i,j=N-j_0}^{N+j_0} {\mathbb I}_{[-1,1]}(i-j) \int _0^\infty e^{t(\triangle -m_0^2)} P_N^2 \left[ \left( \Delta _i P_N \left[ \rho _M^2 V_N(t) \left( \rho _M^2 \Delta _j P_N \delta _x \right) \right] \right) ^2 \right] (x) dt \\
&\geq - C \sup_{i,j \in {\mathbb N}\cup \{ -1,0\}} \int _0^\infty \left\| e^{t(\triangle -m_0^2)} P_N^2 \left[ \left( \Delta _i P_N \left[ \rho _M^2 V_N(t) \left( \rho _M^2 \Delta _j P_N \delta _x \right) \right] \right) ^2 \right] \right\| _{L^\infty} dt .
\end{align*}
Hence, by the estimate above we have the lower bound of $C_2^{M,N}(x)$.
\end{proof}

Now we discuss the functionals of $P_N Z$ defined above.
Similarly to Proposition 3.3 in \cite{AK}, the following proposition holds.
However, we are not able to use the Fourier expansion now, because we consider the $\Phi ^4$-model not on a torus as it was the case in \cite{AK}, but on ${\mathbb R}^3$.
Instead of the Fourier expansion, we apply estimates of the covariance operator that appeared in Lemma \ref{lem:covPZ} above.

\begin{prop}\label{prop:Z}
Let $\varepsilon \in (0,1]$ and $\nu \in L^1(dx)$ such that $0\leq \nu \leq 1$.
Then for all $T\in (0,\infty )$, $p\in [1,\infty )$ and $\alpha \in [0,\infty )$ the following properties hold:
\begin{enumerate}
\item \label{prop:Zk} For $k=1,2$, $\{ 2^{-\alpha N}{\mathcal Z}^{(k,N)}; N\in {\mathbb N}\}$ converges almost surely in $C([0,\infty ); B_{p,p}^{\alpha -(k/2)-\varepsilon }(\nu ))$ and satisfies
\[
\sup _{N\in {\mathbb N}} 2^{-p\alpha N} E\left[ \sup_{t\in [0,T]} \left\| {\mathcal Z}_t^{(k,N)} \right\| _{B_{p,p}^{\alpha -(k/2)-\varepsilon }(\nu )}^p\right] < \infty  .
\]
In particular, $\{ 2^{-\alpha N} {\mathcal Z}^{(k,M,N)}; N\in {\mathbb N}\}$ converges almost surely in $C([0,\infty ); B_{p,p}^{\alpha -(k/2)-\varepsilon }(\nu ))$ for each $M$ and satisfies
\[
\sup _{M,N\in {\mathbb N}} 2^{-p\alpha N} E\left[ \sup_{t\in [0,T]} \left\| {\mathcal Z}_t^{(k,M,N)} \right\| _{B_{p,p}^{\alpha -(k/2)-\varepsilon }(\nu )}^p\right] < \infty .
\]

\item \label{prop:Z0k} For $k=2,3$, $\{ 2^{-\alpha N} {\mathcal Z}^{(0,k,M, N)}; N\in {\mathbb N}\}$ converges almost surely in $C([0,\infty ); B_{p,p}^{\alpha + 2 - (k/2) -\varepsilon }(\nu ))$ for each $M$ and satisfies
\[
\sup _{M, N\in {\mathbb N}}2^{-p\alpha N} E\left[  \sup _{t\in [0,T]} \left\| {\mathcal Z}^{(0,k,M,N)}_t \right\| _{B_{p,p}^{\alpha + 2 - (k/2) -\varepsilon}(\nu )}^p\right] < \infty .
\]
Moreover, for $\gamma \in (0,1/4)$, ${\mathcal Z}^{(0,3,M,N)}$ is $\gamma$-H\"older continuous in time on $[0,T]$ for any $M, N\in {\mathbb N}$ almost surely, and
\[
\sup _{M,N\in {\mathbb N}} E\left[ \left( \sup _{s,t \in [0,T]; s\neq t} \frac{\left\| {\mathcal Z}^{(0,3,M,N)}_t - {\mathcal Z}^{(0,3,M,N)}_s \right\| _{L^p(\nu )}}{(t-s)^{\gamma}} \right) ^p\right] < \infty .
\]

\item \label{prop:Z2k} For $k=2,3$ and each $M$, $\{ 2^{-\alpha N} {\mathcal Z}^{(2,k,M,N)}; N\in {\mathbb N}\}$ converges almost surely in $C([0,\infty ); B_{p,p}^{\alpha + 1 - (k/2) -\varepsilon }(\nu ))$, and it holds that
\[
\sup _{M,N\in {\mathbb N}}  2^{-p\alpha N} E\left[  \sup _{t\in [0,T]}  \left\| {\mathcal Z}^{(2,k,M,N)}_t \right\| _{B_{p,p}^{\alpha + 1 - (k/2) -\varepsilon}(\nu )}^p\right] < \infty .
\]

\item \label{prop:Z13^k} For $k\in {\mathbb N}$, it holds that
\[
\sup _{M, N\in {\mathbb N}}2^{-p\alpha N} E\left[ \sup _{t\in [0,T]} \left\| {\mathcal Z}_t^{(1,M,N)} \left( P_{M,N} {\mathcal Z}^{(0,3,M,N)}_t \right) ^k \right\| _{B_{p,p}^{\alpha -(1/2)-\varepsilon}(\nu )}^p\right] < \infty .
\]

\item \label{prop:tildeZ23} It holds that
\[
2^{-p\alpha N} E\left[ \sup _{t\in [0,T]} \left\| \widehat{\mathcal Z}^{(2,3,M,N)}_t \right\| _{B_{p,p}^{\alpha -1/2-\varepsilon}(\nu )}^p\right] \leq C \left( 1+ M^\sigma 2^{-N} \right) ^p ,
\]
where $C$ is a constant independent of $M$ and $N$.

\end{enumerate}
\end{prop}

\begin{proof}
First, we prove \ref{prop:Zk}.
The definition of Besov spaces implies
\[
E\left[ \left\| {\mathcal Z}_t^{(k,N)} - {\mathcal Z}_s^{(k,N)} \right\| _{B_{2,2}^{\alpha -k/2-\varepsilon} (\nu )} ^2\right] = \sum _{j=-1}^\infty 2^{2\alpha -(k+2\varepsilon )j} E\left[ \left\| \Delta _j \left( {\mathcal Z}_t^{(k,N)} - {\mathcal Z}_s^{(k,N)} \right) \right\| _{L^2(\nu )}^2 \right] .
\]
From this inequality and the translation invariance of the law of ${\mathcal Z}^{(k,N)}$ and the assumption that $\nu \in L^1(dx)$, we have
\begin{align*}
&2^{-2\alpha N} E\left[ \left\| {\mathcal Z}_t^{(k,N)} - {\mathcal Z}_s^{(k,N)} \right\| _{B_{2,2}^{\alpha -k/2-\varepsilon} (\nu )} ^2\right] \\
& \leq C \sum _{j=-1}^\infty 2^{2\alpha (j-N)} 2^{-(k+2\varepsilon )j} E\left[ \left\langle \Delta _j \left( {\mathcal Z}_t^{(k,N)} - {\mathcal Z}_s^{(k,N)} \right) , \delta _0 \right\rangle ^2 \right] \\
& = C \sum _{j=-1}^\infty 2^{2\alpha (j-N)} 2^{-(k+2\varepsilon )j} E\left[ \left\langle {\mathcal Z}_t^{(k,N)} - {\mathcal Z}_s^{(k,N)} , \Delta _j \delta _0 \right\rangle ^2 \right] .
\end{align*}
Since ${\mathcal Z}^{(k,N)}_t$ is given by the Wick polynomials of $P_N Z_t$, Lemma \ref{lem:covPZ} and Plancherel theorem imply
\begin{align*}
E\left[ {\mathcal Z}^{(k,N)}_t(x) {\mathcal Z}^{(k,N)}_s(y) \right] &= k! E\left[ P_N Z_t(x) P_N Z_s(y) \right]^k \\
&= k! \langle V_N(t-s) \delta _x , \delta _y \rangle ^k\\
&= k! \langle [2(m_0^2-\triangle )]^{-1} e^{(t-s)(\triangle -m_0^2)} (P_N \delta _x) , (P_N \delta _y) \rangle ^k .
\end{align*}
Hence, from the fact that $\Delta _j \delta _0 = (2\pi )^{-3/2} {\mathcal F}^{-1}\varphi _j$, it holds that
\begin{align*}
&2^{-2\alpha N} E\left[ \left\| {\mathcal Z}_t^{(k,N)} - {\mathcal Z}_s^{(k,N)} \right\| _{B_{2,2}^{\alpha -k/2-\varepsilon} (\nu )} ^2\right] \\
&\leq C \sum _{j=-1}^\infty 2^{2\alpha (j-N)} 2^{-(k+2\varepsilon )j} \int _{{\mathbb R}^3} \int _{{\mathbb R}^3} \left[ \left\langle [2(m_0^2-\triangle )]^{-1} (P_N \delta _x) , (P_N \delta _y) \right\rangle ^k \right. \\
&\quad \hspace{3cm}\left. - \left\langle [2(m_0^2-\triangle )]^{-1} e^{(t-s)(\triangle -m_0^2)} (P_N \delta _x) , (P_N \delta _y) \right\rangle ^k\right] \Delta _j \delta _0 (x) \Delta _j \delta _0 (y) dx dy\\
&= C \sum _{j=-1}^\infty 2^{2\alpha (j-N)} 2^{-(k+2\varepsilon )j} \\
&\quad \times \int _{{\mathbb R}^3} \cdots \int _{{\mathbb R}^3} \left[ \left( I - e^{(t-s)(\triangle _{z_1}-m_0^2)}\cdots e^{(t-s)(\triangle _{z_k}-m_0^2)} \right) {\phantom{\int}} \right. \\
&\quad \left. \hspace{1cm} [2(m_0^2-\triangle _{z_1})]^{-1} \cdots [2(m_0^2-\triangle _{z_k})]^{-1}\int _{{\mathbb R}^3} (\Delta _j \delta _0) (x) (P_N \delta _x) (z_1) \cdots (P_N \delta _x)(z_k)  dx \right] \\
&\quad \times \left[ \int _{{\mathbb R}^3} (\Delta _j \delta _0) (y) (P_N \delta _y)(z_1) \cdots (P_N \delta _y)(z_k)  dy \right]dz_1 \cdots dz_k
\end{align*}
where $\triangle _{z_l}$ is the Laplacian acting on functions with respect to the parameter $z_l$.
Since
\begin{align*}
&I - e^{(t-s)(\triangle _{z_1}-m_0^2)}\cdots e^{(t-s)(\triangle _{z_k}-m_0^2)}  \\
&= \sum _{l=1}^k \left( I- e^{(t-s)(\triangle _{z_l}-m_0^2)}\right) e^{(t-s)(\triangle _{z_{l+1}}-m_0^2)} \cdots e^{(t-s)(\triangle _{z_k}-m_0^2)},
\end{align*}
by Proposition \ref{prop:Besov} we have
\begin{equation}\label{eq:PZk-90}\begin{array}{l}
\displaystyle 2^{-2\alpha N} E\left[ \left\| {\mathcal Z}_t^{(k,N)} - {\mathcal Z}_s^{(k,N)} \right\| _{B_{2,2}^{\alpha -k/2-\varepsilon} (\nu )} ^2\right] \\
\displaystyle \leq C (t-s)^{\varepsilon /2}\sum _{j=-1}^\infty 2^{2\alpha (j-N)} 2^{-(k+2\varepsilon )j} \\
\displaystyle \quad \times \left\| \int _{{\mathbb R}^3} (\Delta _j \delta _0) (x) \overbrace{(P_N \delta _x)\otimes \cdots \otimes (P_N \delta _x)}^{k} dx \right\| _{\scriptsize \overbrace{B_{2,2}^{-1+(\varepsilon /2)} \times B_{2,2}^{-1} \times \cdots \times B_{2,2}^{-1}}^k} ^2 .
\end{array}\end{equation}
An explicit calculation implies
\[
(\Delta _{j_1}P_N \delta _x) (z_1) = \frac{1}{(2\pi )^3} \int _{{\mathbb R}^3} \varphi _{j_1} (\xi _1) \psi _N (\xi _1) e^{-\sqrt{-1}x \cdot \xi_1} e^{\sqrt{-1}z_1\cdot \xi } d\xi _1.
\]
Hence, it holds that
\begin{align*}
& \int _{{\mathbb R}^3} (\Delta _j \delta _0) (x) (\Delta _{j_1} P_N \delta _x) (z_1) \cdots (\Delta _{j_k} P_N \delta _x)(z_k) dx \\
&= \frac{1}{(2\pi )^{3k}} \int _{{\mathbb R}^3} \cdots \int _{{\mathbb R}^3} \left( \prod _{l=1}^k \varphi _{j_l} (\xi _l) \psi _N (\xi _l) e^{\sqrt{-1}z_l\cdot \xi _l} \right) \\
&\quad \hspace{4cm} \times \left( \int _{{\mathbb R}^3} (\Delta _j \delta _0) (x) e^{-\sqrt{-1}x \cdot (\xi _1 + \cdots + \xi_k)} dx \right) d\xi _1 \cdots d\xi_k \\
&= \frac{1}{(2\pi )^{3k+(3/2)}} \int _{{\mathbb R}^3} \cdots \int _{{\mathbb R}^3} \left( \prod _{l=1}^k \varphi _{j_l} (\xi _l) \psi _N (\xi _l) e^{\sqrt{-1}z_l\cdot \xi _l} \right) \\
&\quad \hspace{4cm} \times \left( \int _{{\mathbb R}^3} ({\mathcal F}^{-1}\varphi _j)(x) e^{-\sqrt{-1}x \cdot (\xi _1 + \cdots + \xi_k)} dx \right) d\xi _1 \cdots d\xi_k \\
&= \frac{1}{(2\pi )^{3k}} \int _{{\mathbb R}^3} \cdots \int _{{\mathbb R}^3} \left( \prod _{l=1}^k \varphi _{j_l} (\xi _l) \psi _N (\xi _l) \right) \varphi _j (\xi _1 + \cdots + \xi_k) e^{\sqrt{-1}(z_1 \cdot \xi _1 + \cdots + z_k \cdot \xi_k)} d\xi _1 \cdots d\xi_k .
\end{align*}
This formula implies that
\begin{align*}
&\left\| \int _{{\mathbb R}^3} (\Delta _j \delta _0) (x) \overbrace{(P_N \delta _x)\otimes \cdots \otimes (P_N \delta _x)}^{k} dx \right\| _{\scriptsize \overbrace{B_{2,2}^{-1+(\varepsilon /2)} \times B_{2,2}^{-1} \times \cdots \times B_{2,2}^{-1}}^k} ^2\\
&= \frac{1}{(2\pi )^{3k}} \sum _{j_1,\cdots ,j_k =-1}^\infty 2^{-(2-\varepsilon ) j_1} 2^{-2j_2} \cdots 2^{-2j_k} \\
&\quad \hspace{3cm} \times \int _{{\mathbb R}^3} \cdots \int _{{\mathbb R}^3} \left( \prod _{l=1}^k \varphi _{j_l} (\xi _l) ^2 \psi _N (\xi _l)^2 \right) \varphi _j (\xi _1 + \cdots + \xi_k)^2 d\xi _1 \cdots d\xi_k.
\end{align*}
By this equality, \eqref{eq:PZk-90} and the fact that $\varphi _{j_l} (\xi _l) \psi _N (\xi _l) =0$ for $j_l > N+3$, we have
\begin{align*}
&2^{-2\alpha N} E\left[ \left\| {\mathcal Z}_t^{(k,N)} - {\mathcal Z}_s^{(k,N)} \right\| _{B_{2,2}^{\alpha -k/2-\varepsilon} (\nu )} ^2\right] \\
&\leq C (t-s)^{\varepsilon /2}\sum _{j=-1}^\infty 2^{2\alpha (j-N)} 2^{-(k+2\varepsilon )j} \sum _{j_1,\cdots ,j_k =-1}^{N+3} 2^{-(2-\varepsilon ) j_1} 2^{-2j_2} \cdots 2^{-2j_k} \\
&\qquad \times  \int _{{\mathbb R}^3} \cdots \int _{{\mathbb R}^3} \varphi _{j_1} (\xi _1) ^2 \cdots \varphi _{j_k} (\xi _k) ^2 \varphi _j (\xi _1 + \cdots + \xi_k)^2 d\xi _1 \cdots d\xi_k .
\end{align*}
When $j\geq N+6+k$, $j_1, \dots j_m \leq N+3$ and $|2^{-j_l} {\xi}_l|\leq 8/3$ for $l=1,\dots ,k$, then
\[
\left| 2^{-j} {\xi}_1 + \cdots  + 2^{-j}\xi_k \right| \leq 2^{-3-k} \cdot k \cdot \frac{8}{3} < \frac{3}{4}.
\]
Hence, again in view of the support of $\varphi$, it holds that
\begin{align*}
& 2^{-2\alpha N} E\left[ \left\| {\mathcal Z}_t^{(k,N)} - {\mathcal Z}_s^{(k,N)} \right\| _{B_{2,2}^{\alpha -k/2-\varepsilon} (\nu )} ^2\right] \\
& \leq C (t-s)^{\varepsilon /2}\sum _{j=-1}^{N+5+k} 2^{-(k+2\varepsilon )j}  \sum _{j_1,\cdots ,j_k =-1}^{N+3} 2^{-(2-\varepsilon ) j_1} 2^{-2j_2} \cdots 2^{-2j_k} \\
& \qquad \times  \int _{{\mathbb R}^3} \cdots \int _{{\mathbb R}^3} \varphi _{j_1} (\xi _1) ^2 \cdots \varphi _{j_k} (\xi _k) ^2 \varphi _j (\xi _1 + \cdots + \xi_k)^2 d\xi _1 \cdots d\xi_k .
\end{align*}
Hence, we have
\begin{align*}
& 2^{-2\alpha N} E\left[ \left\| {\mathcal Z}_t^{(k,N)} - {\mathcal Z}_s^{(k,N)} \right\| _{B_{2,2}^{\alpha -k/2-\varepsilon} (\nu )} ^2\right] \\
& \leq C (t-s)^{\varepsilon /2} \sum _{j_1,\cdots ,j_k =-1}^{N+3} 2^{-(2-\varepsilon ) j_1} 2^{-2j_2} \cdots 2^{-2j_k} \\
& \qquad \times \int _{{\mathbb R}^3} \cdots \int _{{\mathbb R}^3} \left( 2^{ j_1}|\xi _1|^{-1}\right) ^{2-\varepsilon } \left( 2^{j_2}|\xi _2|^{-1}\right) ^2 \cdots \left( 2^{j_k}|\xi _k|^{-1}\right) ^2 \\
& \qquad \hspace{0.5cm} \times \left( 1+ |\xi _1 + \cdots + \xi_k|^2 \right) ^{-[(k/2)+\varepsilon]}\varphi _{j_1} (\xi _1) ^2 \cdots \varphi _{j_k} (\xi _k) ^2 \chi (\xi _1 + \cdots + \xi_k)^2 d\xi _1 \cdots d\xi_k \\
&\quad + C (t-s)^{\varepsilon /2}\sum _{j=0}^{N+5+k} 2^{-(k+2\varepsilon )j}  \sum _{j_1,\cdots ,j_k =-1}^{N+3} 2^{-(2-\varepsilon ) j_1} 2^{-2j_2} \cdots 2^{-2j_k} \\
& \qquad \times \int _{{\mathbb R}^3} \cdots \int _{{\mathbb R}^3} \left( 2^{ j_1}|\xi _1|^{-1}\right) ^{2-\varepsilon } \left( 2^{j_2}|\xi _2|^{-1}\right) ^2 \cdots \left( 2^{j_k}|\xi _k|^{-1}\right) ^2\\
& \qquad \hspace{1cm} \times \left( 2^{j} |\xi _1 + \cdots + \xi_k|^{-1} \right) ^{k+2\varepsilon} \varphi _{j_1} (\xi _1) ^2 \cdots \varphi _{j_k} (\xi _k) ^2 \varphi _j (\xi _1 + \cdots + \xi_k)^2 d\xi _1 \cdots d\xi_k \\
& \leq C (t-s)^{\varepsilon /2} \int _{{\mathbb R}^3} \cdots \int _{{\mathbb R}^3} \frac{|\xi _1|^{\varepsilon }}{|\xi _1|^2 \cdots |\xi _k|^2(1+ |\xi _1 + \cdots + \xi_k|^2 )^{(k/2)+\varepsilon}} d\xi _1 \cdots d\xi_k .
\end{align*}
Thus, for $k=1,2$ we obtain
\[
2^{-2\alpha N} E\left[ \left\| {\mathcal Z}_t^{(k,N)} - {\mathcal Z}_s^{(k,N)} \right\| _{B_{2,2}^{\alpha -k/2-\varepsilon} (\nu )} ^2\right] \leq C (t-s)^{\varepsilon /2} .
\]
This inequality and the hypercontractivity of Gaussian polynomials (see e.g. Theorem 2.7.2 in \cite{NoPe}) yield
\begin{equation}\label{eq:PZ1-1}
2^{-p\alpha N}E\left[ \left\| {\mathcal Z}_t^{(k,N)} - {\mathcal Z}_s^{(k,N)} \right\| _{B_{p,p}^{\alpha -k/2-\varepsilon} (\nu )} ^p \right] \leq C (t-s)^{\varepsilon p /2} .
\end{equation}
In view of Lemma \ref{lem:covPZ}\ref{lem:covPZ3}, by a similar calculation we have
\begin{equation}\label{eq:PZ1-2}
\sup _{t\in [0,T]} \sum _{N=1}^\infty 2^{-\alpha N} E\left[ \left\| {\mathcal Z}_t^{(k,N+1)} - {\mathcal Z}_t^{(k,N)} \right\| _{B_{p,p}^{\alpha -k/2-\varepsilon} (\nu )} ^p \right] ^{1/p} < \infty .
\end{equation}
Indeed, similarly to the proof of Lemma \ref{lem:covPZ} we can prove that
\[
E\left[ \langle f, P_{N+1}Z_t \rangle \langle g, P_N Z_t \rangle \right] = \left\langle P_N P_{N+1} [2(m_0^2-\triangle )]^{-1} f, g \right\rangle ,
\]
and hence
\begin{align*}
&E\left[ \left\langle {\mathcal Z}_t^{(k,N+1)} - {\mathcal Z}_t^{(k,N)} , \Delta _j \delta _0 \right\rangle ^2 \right] \\
&=k! \int _{{\mathbb R}^3} \int _{{\mathbb R}^3} \left[ \left\langle [2(m_0^2 -\triangle )]^{-1}P_{N+1} \delta _x, P_{N+1} \delta _y \right\rangle ^k + \left\langle [2(m_0^2 -\triangle )]^{-1}P_{N} \delta _x, P_{N} \delta _y \right\rangle ^k \right. \\
&\quad \hspace{3cm} \left. -2 \left\langle [2(m_0^2 -\triangle )]^{-1}P_{N+1} \delta _x, P_{N} \delta _y \right\rangle ^k \right] \Delta _j \delta _0 (x) \Delta _j \delta _0 (y) dx dy\\
&= k! \int _{{\mathbb R}^3} \cdots \int _{{\mathbb R}^3} \left[ [2(m_0^2-\triangle _{z_1})]^{-1} \cdots [2(m_0^2-\triangle _{z_k})]^{-1}{\phantom{\int}} \right. \\
&\quad \qquad \sum _{m=1}^k \int _{{\mathbb R}^3} (\Delta _j \delta _0) (x) (P_N \delta _x) (z_1) \cdots (P_N \delta _x) (z_{m-1}) \\
&\quad \left. \hspace{1cm} {\phantom{\int}} \times [ (P_{N+1}-P_N) \delta _x] (z_{m}) (P_{N+1} \delta _x)(z_{m+1}) \cdots (P_{N+1} \delta _x)(z_k)  dx \right] \\
&\quad \quad \times \sum _{\tilde{m}=1}^k \int _{{\mathbb R}^3} (\Delta _j \delta _0) (y) (P_N \delta _y) (z_1) \cdots (P_N \delta _y) (z_{\tilde{m}-1}) \\
&\quad \left. \hspace{1cm} {\phantom{\int}} \times [ (P_{N+1}-P_N) \delta _y] (z_{\tilde{m}}) (P_{N+1} \delta _y)(z_{\tilde{m}+1}) \cdots (P_{N+1} \delta _y)(z_k)  dy \right]  dz_1 \cdots dz_k \\
&\leq C \sum _{m=1}^k \left\| \int _{{\mathbb R}^3} (\Delta _j \delta _0) (x) \overbrace{(P_N \delta _x)\otimes \cdots \otimes (P_N \delta _x)}^{m-1} \otimes [(P_{N+1} -P_N) \delta _x] \right. \\
&\quad \left. \hspace{4cm} \phantom{\int} \otimes \overbrace{(P_{N+1} \delta _x)\otimes \cdots \otimes (P_{N+1} \delta _x)}^{k-m} dx \right\| _{\scriptsize \overbrace{B_{2,2}^{-1} \times B_{2,2}^{-1} \times \cdots \times B_{2,2}^{-1}}^k} ^2 .
\end{align*}
Similarly to above, from this inequality and Lemma \ref{lem:covPZ} we have \eqref{eq:PZ1-2}.
Applying the argument in the proof of Proposition 3.3 in \cite{AK}, from \eqref{eq:PZ1-1} and \eqref{eq:PZ1-2} we obtain the almost-sure convergence of $\{ 2^{-\alpha N} {\mathcal Z}^{(k,N)} ; N\in {\mathbb N}\}$ in $C([0,\infty ); B_{p,p}^{\alpha -k/2-\varepsilon }(\nu ))$.
Thus, we get \ref{prop:Zk}.

We have \ref{prop:Z0k} in a similar way, because of the facts that
\begin{align*}
&E\left[ \left( \Delta _i {\mathcal Z}^{(0,k,M,N)}_t\right) (x) \left( \Delta_j {\mathcal Z}^{(0,k,M,N)}_s\right) (y) \right] \\
&= \int _{-\infty}^t \int _{-\infty}^s E\left[ \left\langle {\mathcal Z}^{(k,M,N)}_u, P_{M,N} e^{(t-u)(\triangle -m_0^2)} \Delta _i \delta _x \right\rangle \right. \\
&\quad \hspace{5cm} \times \left. \left\langle {\mathcal Z}^{(k,M,N)}_v, P_{M,N} e^{(s-v)(\triangle -m_0^2)} \Delta_j \delta _y \right\rangle \right] du dv \\
&= k! \int _{-\infty}^t \int _{-\infty}^s \int _{{\mathbb R}^3} \int _{{\mathbb R}^3} \langle V_N(|u-v|) \delta _{z_1}, \delta _{z_2} \rangle ^k \\
&\quad \hspace{1cm} \times \left( \rho _M^k P_{M,N} e^{(t-u)(\triangle -m_0^2)} \Delta _i \delta _x \right) (z_1) \left( \rho _M^k P_{M,N} e^{(s-v)(\triangle -m_0^2)} \Delta _j \delta _y \right) (z_2)  dz_1 dz_2 du dv\\
&= k! \int _0^\infty \int _0^\infty \int _{{\mathbb R}^3} \int _{{\mathbb R}^3} \langle V_N(|(t-s)-(\tilde{u}-\tilde{v})|) \delta _{z_1}, \delta _{z_2} \rangle ^k \\
&\quad \hspace{2cm} \times \left( \rho _M^{k+1} e^{\tilde{u}(\triangle -m_0^2)} (\Delta_i P_N \delta _x) \right) (z_1) \left( \rho _M^{k+1} e^{\tilde{v}(\triangle -m_0^2)} (\Delta_j P_N \delta _y) \right) (z_2)  dz_1 dz_2 d\tilde u d\tilde v .
\end{align*}
and that
\begin{align*}
&\int _{{\mathbb R}^3} \left( \rho _M^{k+1} e^{\tilde{u}(\triangle -m_0^2)} (\Delta_i P_N \delta _x) \right) (z_1) e^{-\sqrt{-1}z_1\cdot \xi} dz_1\\
&= \frac{1}{(2\pi )^{3/2}} \int _{{\mathbb R}^3}  (\mathcal F \rho _M^{k+1})(\xi -\eta ) \left( e^{-\tilde{u}( \eta ^2+m_0^2)} \varphi _i(\eta ) \psi _N (\eta ) e^{{-\sqrt{-1}\eta\cdot x}}\right) d\eta.
\end{align*}
For the estimate of the H\"older continuity of ${\mathcal Z}_t^{(0,3,M,N)}$, we show
\[
E\left[ \left\| {\mathcal Z}_t^{(0,3,M,N)} - {\mathcal Z}_s^{(0,3,M,N)} \right\| _{L^p (\nu )} ^p\right] \leq C (t-s)^{p(\gamma + \varepsilon )}
\]
for sufficiently small $\varepsilon >0$ and sufficiently large $p$ in a similar way as above, and apply the Kolmogorov continuity theorem.

Next we prove \ref{prop:Z2k} for $k=2$.
However, it is very complicated and difficult to write down all details explicitly.
So, we give only a sketch of the proof.
In view of the above argument of hypercontractivity, it is sufficient to show that
\begin{align}
\label{eq:propZ22a}& 2^{-2\alpha N} E\left[ \left\| {\mathcal Z}^{(2,2,M,N)}_t - {\mathcal Z}^{(2,2,M,N)}_s \right\| _{B_{2,2}^{\alpha -\varepsilon }(\nu )}^2\right] \leq C(t-s)^{\varepsilon /2} \\
\label{eq:propZ22b}&\sum _{N=1}^\infty 2^{-2\alpha N} E\left[ \left\| {\mathcal Z}^{(2,2,M,N+1)}_t - {\mathcal Z}^{(2,2,M,N)}_t \right\| _{B_{2,2}^{\alpha -\varepsilon }(\nu )}^2\right] \leq C ,
\end{align}
where $C$ is a constant independent of $M$ and $N$.
By the definition of ${\mathcal Z}^{(2,2,M,N)}$ and Proposition \ref{prop:paraproduct}
\begin{align*}
&E\left[ \left\| {\mathcal Z}^{(2,2,M,N)}_t - {\mathcal Z}^{(2,2,M,N)}_s \right\| _{B_{2,2}^{\alpha -\varepsilon }(\nu )}^2\right] \\
&\leq C E\left[ \left\| \left( {\mathcal Z}^{(2,M,N)}_t - {\mathcal Z}^{(2,M,N)}_s \right) \mbox{\textcircled{\scriptsize$=$}}\int _{-\infty}^s e^{(s-u)(\triangle -m_0^2)} P_N^2 {\mathcal Z}^{(2,M,N)}_u du \right\| _{B_{2,2}^{\alpha -\varepsilon }(\nu )}^2\right] \\
&\quad + C E\left[ \left\| {\mathcal Z}^{(2,M,N)}_t \mbox{\textcircled{\scriptsize$=$}} \int _{-\infty}^s \left[ e^{(t-s)(\triangle -m_0^2)} - I \right] e^{(s-u)(\triangle -m_0^2)} P_N^2 {\mathcal Z}^{(2,M,N)}_u du \right\| _{B_{2,2}^{\alpha -\varepsilon }(\nu )}^2\right] \\
&\quad + C E\left[ \left\| {\mathcal Z}^{(2,M,N)}_t \mbox{\textcircled{\scriptsize$=$}} \int _s^t e^{(t-u)(\triangle -m_0^2)} P_N^2 {\mathcal Z}^{(2,M,N)}_u du \right\| _{B_{2,2}^{\alpha -\varepsilon }(\nu )}^2\right] .
\end{align*}
We have now to estimate all terms on the right-hand side. However, we consider only the last term, because the others are similarly obtained.
By the definitions of Besov spaces and paraproducts, commutativity of the heat semigroup and $\Delta _j$, and the supports of $\varphi$ and $\psi$, we have
\begin{align*}
&2^{-2\alpha N} \left\| {\mathcal Z}^{(2,M,N)}_t \mbox{\textcircled{\scriptsize$=$}} \int _s^t e^{(t-u)(\triangle -m_0^2)} P_N^2 {\mathcal Z}^{(2,M,N)}_u du \right\| _{B_{2,2}^{\alpha -\varepsilon }(\nu )}^2 \\
&= 2^{-2\alpha N} \sum _{l=-1}^\infty 2^{2(\alpha -\varepsilon )l}\\
&\quad \hspace{1cm} \times \left\| \Delta _l\left[  \sum _{i, j; |i-j|\leq 1} (\Delta _{i} {\mathcal Z}^{(2,M,N)}_t) \left( \Delta _{j} \int _s^t e^{(t-u)(\triangle -m_0^2)} P_N^2 {\mathcal Z}^{(2,M,N)}_u du \right) \right] \right\| _{L^2(\nu )} ^2
\\
&\leq C \sum _{l=-1}^\infty 2^{-2 \varepsilon l}  \sum _{i_1, j_1; |i_1-j_1|\leq 1} \sum _{i_2, j_2; |i_2-j_2|\leq 1} \\
&\quad \hspace{0.5cm} \times \int _{{\mathbb R}^3} \int _{{\mathbb R}^3} \Delta _l \left[ (\Delta _{i_1} {\mathcal Z}^{(2,M,N)}_t) \left( \int _s^t e^{(t-u)(\triangle -m_0^2)} P_N^2 \Delta _{j_1} {\mathcal Z}^{(2,M,N)}_u du \right) \right] (x)\\
&\quad \hspace{1cm} \times \Delta _l \left[ (\Delta _{i_2} {\mathcal Z}^{(2,M,N)}_t) \left( \int _s^t e^{(t-v)(\triangle -m_0^2)} P_N^2 \Delta _{j_2} {\mathcal Z}^{(2,M,N)}_v dv \right) \right] (y) \nu (x) \nu (y) dx dy
\end{align*}
for some constant $C$.
Hence, we need to estimate
\begin{align*}
&E\left[ \left\langle {\mathcal Z}^{(2,M,N)}_t , \Delta _{i_1} \delta _x \right\rangle \left\langle {\mathcal Z}^{(2,M,N)}_u , e^{(t-u)(\triangle -m_0^2)} P_N^2 \Delta _{j_1} \delta _x \right\rangle \right. \\
&\hspace{2cm} \left. \times \left\langle {\mathcal Z}^{(2,M,N)}_t , \Delta _{i_2} \delta _y \right\rangle \left\langle {\mathcal Z}^{(2,M,N)}_v , e^{(t-v)(\triangle -m_0^2)} P_N^2 \Delta _{j_2} \delta _y \right\rangle \right] .
\end{align*}
This term is estimated by the product formula (see Theorem 2.7.10 in \cite{NoPe}) or by the calculation of the expectation of Gaussian polynomials by pairing (see Theorem 1.28 in \cite{Ja}).
The rest is similar to that of \ref{prop:Zk}, and by combining the estimates above we obtain \eqref{eq:propZ22a}.
We also have \eqref{eq:propZ22b} by repeating a similar argument.
We obtain the case that $k=3$ in a similar way, however the calculation is more complicated. So, we skip giving details.

Next we show \ref{prop:Z13^k}.
We fix $\varepsilon \in (0,1/8]$ and $\alpha \in [0,\infty )$.
First we show 
\begin{equation}\label{eq:propZ13a}
\sup _{M, N\in {\mathbb N}}2^{-p\alpha N} E\left[ \sup _{t\in [0,T]} \left\| {\mathcal Z}_t^{(1,M,N)} \mbox{\textcircled{\scriptsize$=$}} \left( P_{M,N} {\mathcal Z}^{(0,3,M,N)}_t \right) ^k \right\| _{B_{p,p}^{\alpha -\varepsilon}(\nu )}^p\right] < \infty
\end{equation}
for $p\in [1,\infty )$ by induction in $k$.
Similarly to the proof of \ref{prop:Z2k} we have \eqref{eq:propZ13a} in the case that $k=1$.
Assume that $k\geq 2$ and that \eqref{eq:propZ13a} holds for $k-1$.
We decompose the product as
\begin{align*}
&{\mathcal Z}_t^{(1,M,N)} \mbox{\textcircled{\scriptsize$=$}} \left( P_{M,N} {\mathcal Z}^{(0,3,M,N)}_t \right) ^k \\
&= {\mathcal Z}_t^{(1,M,N)} \mbox{\textcircled{\scriptsize$=$}} \left[ \left( P_{M,N} {\mathcal Z}^{(0,3,M,N)}_t \right) \mbox{\textcircled{\scriptsize$=$}} \left( P_{M,N} {\mathcal Z}^{(0,3,M,N)}_t \right) ^{k-1} \right] \\
&\quad + {\mathcal Z}_t^{(1,M,N)} \mbox{\textcircled{\scriptsize$=$}} \left[ \left( P_{M,N} {\mathcal Z}^{(0,3,M,N)}_t \right) \mbox{\textcircled{\scriptsize$<$}} \left( P_{M,N} {\mathcal Z}^{(0,3,M,N)}_t \right) ^{k-1} \right] \\
&\quad \hspace{2cm} - \left( P_{M,N} {\mathcal Z}^{(0,3,M,N)}_t \right) \left[ {\mathcal Z}_t^{(1,M,N)} \mbox{\textcircled{\scriptsize$=$}} \left( P_{M,N} {\mathcal Z}^{(0,3,M,N)}_t \right) ^{k-1} \right]\\
&\quad + \left( P_{M,N} {\mathcal Z}^{(0,3,M,N)}_t \right) \left[ {\mathcal Z}_t^{(1,M,N)} \mbox{\textcircled{\scriptsize$=$}} \left( P_{M,N} {\mathcal Z}^{(0,3,M,N)}_t \right) ^{k-1} \right] \\
&\quad + {\mathcal Z}_t^{(1,M,N)} \mbox{\textcircled{\scriptsize$=$}} \left[ \left( P_{M,N} {\mathcal Z}^{(0,3,M,N)}_t \right) \mbox{\textcircled{\scriptsize$>$}} \left( P_{M,N} {\mathcal Z}^{(0,3,M,N)}_t \right) ^{k-1} \right] \\
&\quad \hspace{2cm} - \left( P_{M,N} {\mathcal Z}^{(0,3,M,N)}_t \right) ^{k-1} \left[ {\mathcal Z}_t^{(1,M,N)} \mbox{\textcircled{\scriptsize$=$}} \left( P_{M,N} {\mathcal Z}^{(0,3,M,N)}_t \right) \right] \\
&\quad + \left( P_{M,N} {\mathcal Z}^{(0,3,M,N)}_t \right) ^{k-1}\left[ {\mathcal Z}_t^{(1,M,N)} \mbox{\textcircled{\scriptsize$=$}} \left( P_{M,N} {\mathcal Z}^{(0,3,M,N)}_t \right) \right].
\end{align*}
Then, by Propositions \ref{prop:paraproduct} and \ref{prop:com2} we have, for some constant $C>0$
\begin{align*}
&\left\| {\mathcal Z}_t^{(1,M,N)} \mbox{\textcircled{\scriptsize$=$}} \left( P_{M,N} {\mathcal Z}^{(0,3,M,N)}_t \right) ^k \right\| _{B_{p,p}^{\alpha -\varepsilon}(\nu )} \\
&\leq C\left\| {\mathcal Z}_t^{(1,M,N)}\right\| _{B_{2p,2p}^{\alpha -(1+\varepsilon )/2}(\nu )} \left\| \left( P_{M,N} {\mathcal Z}^{(0,3,M,N)}_t \right) \mbox{\textcircled{\scriptsize$=$}} \left( P_{M,N} {\mathcal Z}^{(0,3,M,N)}_t \right) ^{k-1} \right\| _{B_{2p,2p}^{1/2 +\varepsilon }(\nu )} \\
&\quad + C\left\| {\mathcal Z}_t^{(1,M,N)}\right\| _{B_{2p,2p}^{\alpha -1/2 -\varepsilon }(\nu )} \left\| \left( P_{M,N} {\mathcal Z}^{(0,3,M,N)}_t \right) \right\|  _{B_{4p,4p}^{1/2 -\varepsilon }(\nu )} \left\| \left( P_{M,N} {\mathcal Z}^{(0,3,M,N)}_t \right) ^{k-1} \right\| _{B_{4p,4p}^{2\varepsilon }(\nu )}\\
&\quad + C\left\| P_{M,N} {\mathcal Z}^{(0,3,M,N)}_t \right\| _{L^{2p}(\nu )} \left\| {\mathcal Z}_t^{(1,M,N)} \mbox{\textcircled{\scriptsize$=$}} \left( P_{M,N} {\mathcal Z}^{(0,3,M,N)}_t \right) ^{k-1} \right\| _{B_{2p,2p}^{\alpha -\varepsilon }(\nu )} \\
&\quad + C\left\| \left( P_{M,N} {\mathcal Z}^{(0,3,M,N)}_t \right) ^{k-1} \right\|_{L^{2p}(\nu )} \left\|  {\mathcal Z}_t^{(1,M,N)} \mbox{\textcircled{\scriptsize$=$}} \left( P_{M,N} {\mathcal Z}^{(0,3,M,N)}_t \right) \right\| _{B_{2p,2p}^{\alpha -\varepsilon }(\nu )} .
\end{align*}
In view of the induction assumption, Proposition \ref{prop:paraproduct}, \ref{prop:Zk} and \ref{prop:Z0k}, we see that \eqref{eq:propZ13a} also holds for $k$.
Since
\begin{align*}
&{\mathcal Z}_t^{(1,M,N)} \left( P_{M,N} {\mathcal Z}^{(0,3,M,N)}_t \right) ^k \\
&= {\mathcal Z}_t^{(1,M,N)} \mbox{\textcircled{\scriptsize$<$}} \left( P_{M,N} {\mathcal Z}^{(0,3,M,N)}_t \right) ^k + {\mathcal Z}_t^{(1,M,N)} \mbox{\textcircled{\scriptsize$>$}} \left( P_{M,N} {\mathcal Z}^{(0,3,M,N)}_t \right) ^k \\
&\quad +  {\mathcal Z}_t^{(1,M,N)} \mbox{\textcircled{\scriptsize$=$}} \left( P_{M,N} {\mathcal Z}^{(0,3,M,N)}_t \right) ^k,
\end{align*}
by applying Proposition \ref{prop:paraproduct}, \ref{prop:Zk}, \ref{prop:Z0k} and \eqref{eq:propZ13a}, we have \ref{prop:Z13^k} .

Finally we show \ref{prop:tildeZ23}.
Decompose $\widehat{\mathcal Z}^{(2,3,M,N)}_t$ as
\begin{align*}
\widehat{\mathcal Z}^{(2,3,M,N)}_t
&= \rho _M^2 {\mathcal Z}^{(2,3,M,N)}_t\\
&\quad + \left( P_{M,N} \int _{-\infty}^t e^{(t-s)(\triangle -m_0^2)} P_N\left( \rho _M \mbox{\textcircled{\scriptsize$\geqslant$}} {\mathcal Z}^{(3,M,N)}_s \right) ds \right) \mbox{\textcircled{\scriptsize$=$}} {\mathcal Z}_t^{(2,M,N)} \\
&\quad + \left( \rho _M \left[ \int _{-\infty}^t e^{(t-s)(\triangle -m_0^2)} P_N^2 \left( \rho _M \mbox{\textcircled{\scriptsize$<$}} {\mathcal Z}^{(3,M,N)}_s \right) ds \right. \right. \\
&\quad \hspace{2cm} \left. \left. - \left( \rho _M \mbox{\textcircled{\scriptsize$<$}} \int _{-\infty}^t e^{(t-s)(\triangle -m_0^2)} P_N^2 {\mathcal Z}^{(3,M,N)}_s ds \right) \right] \right) \mbox{\textcircled{\scriptsize$=$}} {\mathcal Z}_t^{(2,M,N)} \\
&\quad + \left[ \rho _M \left( \rho _M \mbox{\textcircled{\scriptsize$<$}} \int _{-\infty}^t e^{(t-s)(\triangle -m_0^2)} P_N^2 {\mathcal Z}^{(3,M,N)}_s ds \right) \right] \mbox{\textcircled{\scriptsize$=$}} {\mathcal Z}_t^{(2,M,N)} \\
&\quad \hspace{2cm} - \rho _M ^2 \left[ \left( \int _{-\infty}^t e^{(t-s)(\triangle -m_0^2)} P_N^2 {\mathcal Z}^{(3,M,N)}_s ds \right) \mbox{\textcircled{\scriptsize$=$}} {\mathcal Z}_t^{(2,M,N)} \right] .
\end{align*}
Hence, we get the required estimate from this, \ref{prop:Zk}, \ref{prop:Z0k}, \ref{prop:Z2k}, Propositions \ref{prop:paraproduct}, \ref{prop:com2} and \ref{prop:com3}, and the fact that $\sup _M \| \rho _M\| _{B_\infty^2(\nu )} <\infty$.
\end{proof}

\begin{rem}
The case where $\alpha =0$ in Proposition \ref{prop:Z} is also discussed in \cite{HaLa} in the framework of the regularity structures.
\end{rem}

\section{Approximation equations and their transformation}\label{sec:trans}

In this section, we will construct an invariant probability measure and a flow associated to the $\Phi ^4_3$-measure.
We use the same notations as in Sections \ref{sec:Besov} and \ref{sec:OU}.

Let $\lambda _0\in (0,\infty )$ and $\lambda \in (0,\lambda _0]$ be fixed.
Let $\nu (x) := (1+|x|^2)^{-\sigma /2}$ and assume
\begin{equation}\label{eq:assm0}
9< \sigma ^2 < 2m_0^2 .
\end{equation}
Define a function $U_{M,N}$ and a measure $\mu _{M,N}$ as in Section \ref{sec:main}.
Recall that $\{ \mu _{M,N}\}$ is an approximation sequence for the $\Phi ^4_3$-measure, which will be constructed below as an invariant probability measure of the flow associated with the stochastic quantization equation.
For simplicity, let $C_*^{(M,N)}(x) := C_1^{(N)} -3\lambda C_2^{(M,N)}(x)$.

Consider the stochastic partial differential equation on ${\mathbb R}^3$
\begin{equation}\label{eq:SDEY}\begin{array}{rl}
\displaystyle \partial _t Y_t^{M,N}(x)
&\displaystyle =  \dot{W}_t(x) - (-\triangle +m_0^2)Y_t^{M,N}(x) \\[3mm]
&\displaystyle \quad - \lambda P_{M,N}^* \left\{ (P_{M,N} Y_t^{M,N})^3 (x) -3 C_*^{(M,N)}(x) \rho _M (x)^2 P_{M,N} Y_t^{M, N}(x) \right\}
\end{array}\end{equation}
where $\dot{W}_t(x)$ is a white noise with parameter $(t,x)\in [0,\infty)\times {\mathbb R}^3$.
First, we prove that this SPDE is associated to $\mu _{M,N}$, in the sense that $\mu _{M,N}$ is the invariant measure for $Y_t^{M,N}$.

\begin{thm}\label{thm:invN}
Let $\alpha \in (1/2,\infty)$ and $\varepsilon \in (0,\infty )$.
For each $M$ and $N$, (\ref{eq:SDEY}) has a unique global mild solution as a continuous stochastic process on $B_2^{-\alpha -\varepsilon}(\nu )$ almost surely for any initial value $Y^{M,N}_0 \in B_2^{-\alpha}(\nu )$ and the solution $Y_t^{M,N}$ satisfies
\[
E\left[ \sup _{s\in [0,T]} \left\| Y_s^{M,N}\right\| _{B_2^{-\alpha -\varepsilon}(\nu )}^2 \right] < \infty .
\]
Moreover, $\mu _{M,N}$ is the invariant measure for the approximation process $Y^{M,N}$.
\end{thm}

\begin{proof}
Consider the decomposition $Y^{M,N}_t = P_{2N} Y^{M,N}_t + P_{2N}^{\perp} Y^{M,N}_t$ where $P_{2N}^{\perp} := I - P_{2N}$ .
Then, $(P_{2N} Y^{M,N}_t, P_{2N}^{\perp} Y^{M,N}_t)$ should solve the following system of coupled stochastic partial differential equations:
\begin{align}
\label{eq:PY}
\partial _t P_{2N} Y_t^{M,N}&= P_{2N} \dot{W}_t - (-\triangle +m_0^2) P_{2N} Y_t^{M,N} \\
\nonumber & \quad - \lambda P_{M,N}^* \left\{ (P_{M,N} P_{2N }Y_t^{M,N})^3 -3C_*^{(M,N)} \rho _M ^2 P_{M,N} P_{2N} Y_t^{M, N} \right\}
\\
\label{eq:(I-P)Y}
\partial _t P_{2N}^{\perp} Y_t^{M,N}&=  P_{2N}^{\perp} \dot{W}_t - (-\triangle +m_0^2)P_{2N}^{\perp} Y_t^{M,N}.
\end{align}
It is easy to see that $\{ P_{2N} \dot{W}_t; t\in [0,\infty )\}$ and $\{ P_{2N}^{\perp} \dot{W}_t; t\in [0,\infty )\}$ are independent.
Hence, it is sufficient to solve (\ref{eq:PY}) and (\ref{eq:(I-P)Y}) independently.

The solution to (\ref{eq:(I-P)Y}) is obtained explicitly as
\begin{equation}\label{eq:(I-P)Ysol}
P_{2N}^{\perp} Y_t^{M,N} = e^{t(\triangle - m_0^2)} P_{2N}^{\perp} Y_0^{M,N} + \int _0^t e^{(t-s)(\triangle -m_0^2)} P_{2N}^{\perp} \dot{W}_s ds.
\end{equation}
Hence, by explicit calculation, for $s,t\in [0,\infty )$ such that $s<t$
\begin{align*}
&E\left[ \left\| P_{2N}^{\perp} Y_t^{M,N} - P_{2N}^{\perp} Y_s^{M,N}\right\| _{B_{2,2}^{-\alpha -\varepsilon}(\nu )}^2 \right] \\
&\leq 2 \left\| \left( e^{(t-s)(\triangle - m_0^2)} -I \right) P_{2N}^{\perp} Y_0^{M,N} \right\| _{B_{2,2}^{-\alpha -\varepsilon}(\nu )}^2 + 2 E\left[ \left\| \int _s^t e^{(t-u)(\triangle -m_0^2)} P_{2N}^{\perp} \dot{W}_u du \right\| _{B_{2,2}^{-\alpha -\varepsilon }(\nu )}^2 \right] \\
&\leq 2 (t-s)^{2\varepsilon}  \left\| Y_0^{M,N} \right\| _{B_{2,2}^{-\alpha}(\nu )}^2 + 2 \sum _{j=-1}^\infty 2^{-2j(\alpha +\varepsilon)}E\left[ \left\| \Delta _j \int _s^t e^{(t-u)(\triangle -m_0^2)} P_{2N}^{\perp} \dot{W}_u du \right\| _{L^2(\nu )}^2 \right] .
\end{align*}
By another explicit calculation we have
\begin{align*}
&E\left[ \left\| \Delta _j \int _s^t e^{(t-u)(\triangle -m_0^2)} P_{2N}^{\perp} \dot{W}_u du\right\| _{L^2(\nu )}^2 \right] \\
&\leq \int _s^t \int _{{\mathbb R}^3} \int _{{\mathbb R}^3} \left( [(\Delta _j P_{2N}^{\perp} p)(t-u, \cdot )] (x-y) \right) ^2 \nu (x) dx dy du\\
&\leq C \int _s^t \int _{{\mathbb R}^3} \varphi _j (|\xi|)^2 e^{-2(t-u)(|\xi |^2+m_0^2)} d\xi du \\
&\leq C 2^{j} (t-s)
\end{align*}
where $p(t,x)$ is the heat kernel, i.e.
\[
p(t,x) = (4\pi t)^{-3/2} e^{-|x|^2/4t} , \quad (t,x)\in (0,\infty )\times {\mathbb R}^3
\]
and $C$ is a constant depending on $\nu$.
These inequalities imply that $P_{2N}^{\perp} Y_t^{M,N}$ is a $B_2^{-\alpha -\varepsilon}(\nu )$-valued continuous process.

Now we construct a solution to (\ref{eq:PY}) in $L^2(\nu )$.
Since $P_{2N}$ is a trace class operator in $L^2(dx )$ and $P_{M,N}$ is a bounded linear operator from $B_p^s (\nu )$ to $L^2(\nu )$ for each $p\in [1,\infty ]$ and $s\in {\mathbb R}$, we have the pathwise uniqueness of the local solution $P_{2N} Y_t^{M,N}$ in time to (\ref{eq:PY}) on $L^2(\nu )$ by well-known facts of the theory of stochastic partial differential equations with locally Lipschitz coefficients.

Next, we show the global existence of the solution to (\ref{eq:PY}).
Let
\[
Z_t^{Y} := e^{t(\triangle -m_0^2)} Y^{M,N}_0+ \int _0^t e^{(t-s)(\triangle -m_0^2)} \dot{W}_s ds, \quad Y_t^{M,N,(1)} := Y_t^{M,N} - Z_t^{Y}.
\]
Let $T>0$.
Define a stopping time $\tau _L := \min\{ T, \inf \{t>0; \| P_{2N} Y_t^{M,N,(1)}\| _{L^2} > L\} \}$ for $L\in [0,\infty)$.
Since
\[
\partial _t Z_t^{Y} = \dot{W}_t + (\triangle -m_0^2) Z_t^{Y}, \quad Z_0^{Y} = Y^{M,N}_0,
\]
$Y_t^{M,N,(1)}$ satisfies that $Y_0^{M,N,(1)} =0$ and
\begin{align*}
\partial _t P_{2N} Y_t^{M,N,(1)}&= - (-\triangle +m_0^2) P_{2N} Y_t^{M,N,(1)} \\
\nonumber & \quad - \lambda P_{M,N}^* \left\{ (P_{M,N} P_{2N }Y_t^{M,N})^3 -3C_*^{(M,N)} \rho _M ^2 P_{M,N} P_{2N} Y_t^{M, N} \right\} .
\end{align*}
The mild form of the solution to this equation implies
\begin{align*}
&P_{2N} Y_t^{M,N,(1)} \\
&= - \lambda \int _0^t e^{(t-s)(\triangle -m_0^2)} P_{M,N}^* \left\{ (P_{M,N} P_{2N }Y_s^{M,N})^3 -3C_*^{(M,N)} \rho _M ^2 P_{M,N} P_{2N} Y_s^{M, N} \right\} ds
\end{align*}
for $t\in [0, \tau _L )$.
Since from the definition of $P_{M,N}$
\[
\sup _{s\in [0,t]} \left\| (P_{M,N} P_{2N }Y_s^{M,N})^3 -3C_*^{(M,N)} \rho _M ^2 P_{M,N} P_{2N} Y_s^{M, N} \right\| _{B_{2}^{\beta}} <\infty
\]
for $t\in [0, \tau _L )$ and $\beta \in {\mathbb R}$, we have $P_{2N} Y_t^{M,N,(1)} \in B_{2}^{\beta}$ for $t\in [0, \tau _L )$ and $\beta \in {\mathbb R}$.
Hence, from these facts we have, for any $t \in [0,T]$,
\begin{align*}
&\sup_{s\in [0,t \wedge \tau _L]} \left\| P_{2N} Y_s^{M,N,(1)}\right\| _{L^2}^2\\
&= 2 \sup_{s\in [0,t \wedge \tau _L]} \int _0^s \int _{{\mathbb R}^3} P_{2N} Y_u^{M,N,(1)} \partial _t P_{2N} Y_u^{M,N,(1)} dx du \\
&\leq 2 \sup_{s\in [0,t \wedge \tau _L]} \int _0^s \int _{{\mathbb R}^3} (P_{2N} Y_u^{M,N,(1)}) (\triangle - m_0^2) (P_{2N} Y_u^{M,N,(1)}) dx du \\
&\quad + 2 \lambda \sup_{s\in [0,t \wedge \tau _L]} \left\{ -\int _0^s \int _{{\mathbb R}^3} (P_{2N} Y_u^{M,N,(1)}) \left( P_{M,N}^* \left[ (P_{M,N} P_{2N }Y_u^{M,N})^3\right] \right) dx du \right\} \\
&\quad + 6\lambda \left\| C_*^{(M,N)}\right\| _{L^\infty} \sup_{s\in [0,t \wedge \tau _L]} \int _0^s \int _{{\mathbb R}^3} | P_{2N} Y_u^{M,N,(1)}| \left| P_{M,N}^* P_{M,N} P_{2N }Y_u^{M,N} \right| dx du \\
&\leq 2\sup_{s\in [0,t \wedge \tau _L]} \left\{ - \int _0^s \left( m_0^2\| \nabla P_{2N} Y_u^{M,N,(1)}\| _{L^2}^2 + \| \nabla P_{2N} Y_u^{M,N,(1)}\| _{L^2}^2 \right) du \right\} \\
&\quad + \lambda \sup_{s\in [0,t \wedge \tau _L]} \left\{ - \int _0^s \| P_{M,N} P_{2N }Y_u^{M,N}\| _{L^4} ^4 du \right. \\
&\quad \hspace{4cm} \left. + \int _0^s \int _{{\mathbb R}^3} (P_{M,N} P_{2N} Z_u^Y) (P_{M,N} P_{2N }Y_u^{M,N})^3 dx du \right\} \\
&\quad + 6\lambda \left\| C_*^{(M,N)}\right\| _{L^\infty} \sup_{s\in [0,t \wedge \tau _L]} \int _0^s \int _{{\mathbb R}^3} | P_{2N} Y_u^{M,N,(1)}| \left| P_{M,N}^* P_{M,N} P_{2N }Z_u^Y \right| dx du\\
&\quad + 6\lambda \left\| C_*^{(M,N)}\right\| _{L^\infty} \sup_{s\in [0,t \wedge \tau _L]} \int _0^s \int _{{\mathbb R}^3} | P_{2N} Y_u^{M,N,(1)}| \left| P_{M,N}^* P_{M,N} P_{2N }Y_u^{M,N,(1)}\right| dx du .
\end{align*}
By H\"older's inequality we obtain
\begin{align*}
\sup_{s\in [0,t \wedge \tau _L]} \left\| P_{2N} Y_s^{M,N,(1)}\right\| _{L^2}^2
&\leq C_{M,N} \left( \sup _{s\in [0,T]} \left\| P_{M,N} Z_u^Y\right\| _{L^4}^4 + \sup _{s\in [0,T]} \left\| P_{M,N} Z_u^Y\right\| _{L^2}^2 \phantom{\int} \right. \\
&\quad \hspace{3cm} \left. + \int _0^t \sup_{u\in [0,s \wedge \tau _L]} \left\| P_{2N} Y_u^{M,N,(1)}\right\| _{L^2}^2 ds \right)
\end{align*}
where $C_{M,N}$ is a constant depending on $M$ and $N$.
Similarly to Proposition \ref{prop:Z} we have
\[
E\left[ \sup _{s\in [0,T]} \left\| P_{M,N} Z_u^Y\right\| _{L^4}^4 \right] + E\left[ \sup _{s\in [0,T]} \left\| P_{M,N} Z_u^Y\right\| _{L^2}^2 \right] < \infty
\]
for each $M,N \in {\mathbb N}$.
Hence, 
\[
E\left[ \sup_{s\in [0,t \wedge \tau _L]} \left\| P_{2N} Y_s^{M,N,(1)}\right\| _{L^2}^2 \right]
\leq C_{M,N} \left( 1+ \int _0^t E\left[ \sup_{u\in [0,s \wedge \tau _L]} \left\| P_{2N} Y_u^{M,N,(1)}\right\| _{L^2}^2 \right] ds \right) ,
\]
and by applying Gronwall's inequality we get
\[
E\left[ \sup_{s\in [0,t \wedge \tau _L]} \left\| P_{2N} Y_s^{M,N,(1)}\right\| _{L^2}^2 \right] \leq \tilde C_{M,N}.
\]
where $\tilde C_{M,N}$ is a constant depending on $M$, $N$ and $T$.
This inequality implies 
\[
E\left[ \sup_{s\in [0,T]} \left\| P_{2N} Y_s^{M,N,(1)}\right\| _{L^2(\nu )}^2 \right]  \leq \lim _{L\rightarrow \infty} E\left[ \sup_{s\in [0,T \wedge \tau _L]} \left\| P_{2N} Y_s^{M,N,(1)}\right\| _{L^2}^2 \right] \leq C_{M,N}.
\]
Similarly to Proposition \ref{prop:Z}, by noting that $Z_0^{Y} = Y^{M,N}_0 \in B_2^{-\alpha}(\nu )$, we also have
\[
E\left[ \sup _{s\in [0,T]} \left\| Z_u^Y\right\| _{B_2^{-\alpha -\varepsilon }(\nu )}^2 \right] + E\left[ \sup _{s\in [0,T]} \left\| P_{2N}^{\perp} Y_s^{M,N}\right\| _{B_2^{-\alpha -\varepsilon}(\nu )}^2 \right] < \infty .
\]
Therefore, $Y_t^{M,N} = P_{2N}^{\perp} Y_t^{M,N} + P_{2N} Y_t^{M,N,(1)} + P_{2N} Z_t^Y$ is a time global solution to (\ref{eq:SDEY}) as a continuous process in $B_2^{-\alpha -\varepsilon}(\nu )$.
Moreover, the following estimate holds:
\[
E\left[ \sup _{s\in [0,T]} \left\| Y_s^{M,N}\right\| _{B_2^{-\alpha -\varepsilon}(\nu )}^2 \right] < \infty .
\]

Since the uniqueness locally in time and the global existence of the solution are obtained for all initial value $Y^{M,N}\in B_2^{-\alpha}(\nu )$ and $\alpha \in (1/2,\infty )$, we have the uniqueness globally in time and the Markov property by a standard argument of the theory of stochastic partial differential equations.
For the invariance of $\mu _{M,N}$ under the solution of (\ref{eq:SDEY}), see \cite{AR} for the details.
\end{proof}

Similarly to the argument after the proof of Theorem 4.1 in \cite{AK}, we are able to construct a pair of random variables $(\xi _{M,N} ,\zeta)$ whose law is invariant for the system $(Y_t^{M,N}, Z_t)$.
We fix a pair of random variables $(\xi _{M,N},\zeta)$ and consider the stochastic partial differential equation
\begin{equation}\label{eq:SDEN4}\left\{ \begin{array}{rl}
\displaystyle \partial _t X_t^{M,N}(x)
&\displaystyle = \dot{W}_t(x) - (-\triangle +m_0^2) X_t^{M,N}(x) \\[3mm]
&\displaystyle \quad - \lambda P_{M,N}^* \left\{ (P_{M,N} X_t^{M,N})^3 (x) -3 C_*^{(M,N)}(x) \rho _M (x)^2 P_{M,N} X_t^{M,N}(x) \right\}\\
\displaystyle X_0^{M,N}(x)
&\displaystyle = \xi _{M,N} (x)
\end{array}\right. \end{equation}
where $\dot{W}_t$ is a Gaussian white noise independent of $(\xi _{M,N} ,\zeta)$.
Since the law of $\zeta$ is the stationary measure for the solution to (\ref{eq:SDEZ}) and independent of $\{ \dot{W}_t; t\in [0,\infty )\}$, we are able to choose $\{ \dot{W}_t; t\in (-\infty ,0)\}$ so that $\{ \dot{W}_t; t\in (-\infty ,\infty )\}$ is a Gaussian white noise and that $\zeta = \int _{-\infty}^0 e^{-s(\triangle -m_0^2)} d{W}_s$.
We remark that (\ref{eq:SDEN4}) is the same as \eqref{eq:SDEN3} in Section \ref{sec:main}.
Define $Z_t$ by (\ref{eq:SDEZ}) with the same $\dot{W}$ in (\ref{eq:SDEN4}) and with $Z_0=\zeta$, where $\zeta$ is the random variable defined above.
We remark that the pair $(X_t^{M,N}, Z_t)$ is a stationary process for $t\in [0,\infty )$ by the construction of $(\xi _{M,N}, \zeta)$.

Now we transform the SPDE (\ref{eq:SDEN4}) similarly as we did for a corresponding SPDE in \cite{AK}.
Let $X^{M,N,(1)}_t:= X_t^{M,N} - Z_t$ for $t\in [0,\infty )$.
From (\ref{eq:SDEN4}) and (\ref{eq:SDEZ}) we have
\begin{equation}\label{eq:PDEX1}\begin{array}{l}
\displaystyle \partial _t X^{M,N,(1)}_t + (- \triangle +m_0^2) X^{M,N,(1)}_t\\
\displaystyle = -\lambda P_{M,N}^* \left[ (P_{M,N} X^{M,N,(1)}_t)^3 \right] dt - 3\lambda P_{M,N}^* \left[ {\mathcal Z}_t^{(1,M,N)} (P_{M,N} X^{M,N,(1)}_t)^2 \right] \\
\displaystyle \quad - 3\lambda P_{M,N}^* \left[ {\mathcal Z}^{(2,M,N)}_t P_{M,N} X^{M,N,(1)}_t \right] dt -\lambda P_{M,N}^* {\mathcal Z}^{(3,M,N)}_t \\
\displaystyle \quad - 9 \lambda ^2 P_{M,N}^* \left[ C_2^{(M,N)} \rho _M ^2 \left( P_{M,N} X^{M,N,(1)}_t + {\mathcal Z}_t^{(1,M,N)} \right) \right] .
\end{array}\end{equation}
By letting $P_{2N}^{\perp} := I-P_{2N}$, we have
\[
\partial _t P_{2N}^{\perp} (X_t^{M,N}-Z_t) = (\triangle - m_0^2) P_{2N}^{\perp} (X_t^{M,N}-Z_t),
\]
because $P_{2N}^{\perp} P_N =0$.
Hence, by the stationarity of $(X_t^{M,N}, Z_t)$ we get
\begin{align*}
& E\left[ \| P_{2N}^{\perp} (\xi _{M,N} -\zeta ) \| _{B_2^{-1/2-\varepsilon}(\nu )}^2 \right] = E\left[ \| P_{2N}^{\perp} (X_t^{M,N} - Z_t) \| _{B_2^{-1/2-\varepsilon} (\nu )}^2 \right]\\
&\leq E\left[ \| e^{t(\triangle -m_0^2)} P_{2N}^{\perp} (\xi _{M,N} -\zeta ) \| _{B_2^{-1/2-\varepsilon}(\nu )}^2 \right] \leq e^{-2m_0^2t} E\left[ \| P_{2N}^{\perp} (\xi _{M,N} -\zeta ) \| _{B_2^{-1/2-\varepsilon}(\nu )}^2 \right]
\end{align*}
for $t\in [0,\infty )$ and $\varepsilon \in (0,\infty )$.
Hence, by letting $t\rightarrow \infty$ we have
\[
E\left[ \| P_{2N}^{\perp} (\xi _{M,N} -\zeta ) \| _{B_2^{-1/2-\varepsilon} (\nu )}^2 \right] =0.
\]
This implies that $P_{2N}^{\perp} (\xi _{M,N} -\zeta ) =0$ almost surely in $B_2^{-1/2-\varepsilon} (\nu )$.
Hence, $\Delta _j (\xi _{M,N} -\zeta ) =0$ except for finitely many $j\in {\mathbb N}\cup \{ -1,0\}$.
Thus, we have
\begin{equation}\label{eq:initialX1}
X^{M,N, (1)}_0 = \xi _{M,N} -\zeta \in \bigcap _{s\in {\mathbb R}} B_2^s(\nu ) \quad \mbox{almost surely}
\end{equation}
for each $M,N\in {\mathbb N}$.
Let 
\[
X^{M,N,(2)}_t := X_t^{M,N} - Z_t + \lambda {\mathcal Z}^{(0,3,M,N)}_t, \quad t\in [0,\infty ).
\]
Then, we have
\begin{equation}\label{eq:PDEX2}\begin{array}{l}
\displaystyle \partial _t X^{M,N,(2)}_t + (- \triangle +m_0^2) X^{M,N,(2)}_t \\
\displaystyle = -\lambda P_{M,N}^* \left[ \left( P_{M,N} X^{M,N,(2)}_t - \lambda P_{M,N} {\mathcal Z}^{(0,3,M,N)}_t \right) ^3 \right] \\
\displaystyle \quad - 3\lambda P_{M,N}^* \left[ {\mathcal Z}_t^{(1,M,N)} \left( P_{M,N} X^{M,N,(2)}_t - \lambda P_{M,N} {\mathcal Z}^{(0,3,M,N)}_t \right) ^2 \right] \\
\displaystyle \quad - 3\lambda P_{M,N}^* \left[ {\mathcal Z}^{(2,M,N)}_t \left( P_{M,N} X^{M,N,(2)}_t - \lambda P_{M,N} {\mathcal Z}^{(0,3,M,N)}_t \right) \right] \\
\displaystyle \quad - 9\lambda ^2 P_{M,N}^* \left[ C_2^{(M,N)} \rho _M ^2 \left( P_{M,N} X^{M,N,(2)}_t + {\mathcal Z}_t^{(1,M,N)} - \lambda P_{M,N} {\mathcal Z}^{(0,3,M,N)}_t \right) \right] .
\end{array}\end{equation}
Hence, the pair $(X^{M,N,(2),<}_t, X^{M,N,(2),\geqslant}_t)$ defined by 
\begin{align*}
X^{M,N,(2),<}_t &:= -3\lambda \int _0^t e^{(t-s)(\triangle -m_0^2)} P_{M,N}^* \left[ \left( P_{M,N} X^{M,N,(2)}_s - \lambda P_{M,N} {\mathcal Z}^{(0,3,M,N)}_s \right) \mbox{\textcircled{\scriptsize$<$}} {\mathcal Z}^{(2,M,N)}_s \right] ds \\
X^{M,N,(2),\geqslant}_t & := X^{M,N,(2)}_t - X^{M,N,(2),<}_t
\end{align*}
is the solution to the following partial differential equation
\begin{equation} \label{PDEpara}\left\{ \begin{array}{l}
\displaystyle (\partial _t - \triangle + m_0^2) X^{M,N,(2),<}_t \\
\displaystyle = -3\lambda P_{M,N}^* \left[ \left( P_{M,N} X^{M,N,(2),<}_t + P_{M,N} X^{M,N,(2),\geqslant}_t - \lambda P_{M,N} {\mathcal Z}^{(0,3,M,N)}_t \right) \mbox{\textcircled{\scriptsize$<$}} {\mathcal Z}^{(2,M,N)}_t \right] \\[5mm]
\displaystyle (\partial _t - \triangle + m_0^2) X^{M,N,(2),\geqslant}_t \\
\displaystyle  = - \lambda P_{M,N}^* \left[ \left( P_{M,N} X^{M,N,(2),<}_t + P_{M,N} X^{M,N,(2),\geqslant}_t - \lambda P_{M,N} {\mathcal Z}^{(0,3,M,N)}_t \right) ^3 \right] \\
\displaystyle \quad - 3\lambda P_{M,N}^* \left[ {\mathcal Z}_t^{(1,N)} \left( P_{M,N} X^{M,N,(2),<}_t + P_{M,N} X^{M,N,(2),\geqslant}_t - \lambda P_{M,N} {\mathcal Z}^{(0,3,M,N)}_t \right) ^2 \right] \\
\displaystyle \quad - 3\lambda P_{M,N}^* \left[ \left( P_{M,N} X^{M,N,(2),<}_t + P_{M,N} X^{M,N,(2),\geqslant}_t - \lambda P_{M,N} {\mathcal Z}^{(0,3,M,N)}_t \right) \mbox{\textcircled{\scriptsize$\geqslant$}} {\mathcal Z}^{(2,M,N)}_t \right] \\
\displaystyle \quad - 9\lambda ^2 P_{M,N}^* \left[ C_2^{(M,N)} \rho _M ^2 \left( P_{M,N} X^{M,N,(2),<}_t + P_{M,N} X^{M,N,(2),\geqslant}_t \right. \right. \\
\displaystyle \quad \hspace{7cm} \left. \left. + {\mathcal Z}_t^{(1,M,N)} - \lambda P_{M,N} {\mathcal Z}^{(0,3,M,N)}_t \right) \right]
\end{array}\right. \end{equation}
with initial condition $(X^{M,N,(2),<}_0, X^{M,N,(2),\geqslant}_0) = (0,X^{M,N,(2)}_0)$.
Now, we change (\ref{PDEpara}) for another equivalent equation by using the calculus of paraproducts.
By denoting
\begin{align*}
&\Psi _t^{(1,M,N)} (w) \\
&:= P_{M,N} \int _0^t e^{(t-s)(\triangle -m_0^2)} P_{M,N}^* \left[ \left( w_s - \lambda P_{M,N} {\mathcal Z}^{(0,3,M,N)}_s \right) \mbox{\textcircled{\scriptsize$<$}} {\mathcal Z}^{(2,M,N)}_s \right] ds\\
&\qquad - \rho _M^2 \left[ \left( w_t - \lambda P_{M,N} {\mathcal Z}^{(0,3,M,N)}_t \right) \mbox{\textcircled{\scriptsize$<$}} \int _0^t e^{(t-s)(\triangle -m_0^2)} P_N^2 {\mathcal Z}^{(2,M,N)}_s ds \right] , \\
&\Psi _t^{(2,M,N)} (w) \\
&:= \left( \rho _M^2 \left[ \left( w_t - \lambda P_{M,N} {\mathcal Z}^{(0,3,M,N)}_t \right) \mbox{\textcircled{\scriptsize$<$}} \int _0^t e^{(t-s)(\triangle -m_0^2)} P_N^2 {\mathcal Z}^{(2,M,N)}_s ds \right] \right) \mbox{\textcircled{\scriptsize$=$}} {\mathcal Z}^{(2,M,N)}_t \\
&\qquad - \rho _M^2 \left( w_t - \lambda P_{M,N} {\mathcal Z}^{(0,3,M,N)}_t \right) \left[ \int _0^t e^{(t-s)(\triangle -m_0^2)} P_N^2 {\mathcal Z}^{(2,M,N)}_s ds \mbox{\textcircled{\scriptsize$=$}} {\mathcal Z}^{(2,M,N)}_t \right] 
\end{align*}
for $w \in C([0,\infty ); {\mathcal S}'({\mathbb R}^3))$, we have
\begin{align*}
&(P_{M,N} X^{M,N,(2),<}_t) \mbox{\textcircled{\scriptsize$=$}} {\mathcal Z}^{(2,M,N)}_t \\
&= -3\lambda \left( P _{M,N} \int _0^t e^{(t-s)(\triangle -m_0^2)} P_{M,N}^* \left[ \left( P_{M,N} X^{M,N,(2)}_s - \lambda P_{M,N}{\mathcal Z}^{(0,3,M,N)}_s \right) \mbox{\textcircled{\scriptsize$<$}} {\mathcal Z}^{(2,M,N)}_s \right] ds \right) \\
&\quad \hspace{13cm} \mbox{\textcircled{\scriptsize$=$}} {\mathcal Z}^{(2,M,N)}_t \\
&= -3\lambda \rho _M ^2 \left( P_{M,N} X^{M,N,(2)}_t - \lambda P_{M,N} {\mathcal Z}^{(0,3,M,N)}_t \right) \left[ \int _0^t e^{(t-s)(\triangle -m_0^2)} P_N ^2 {\mathcal Z}^{(2,M,N)}_s ds \mbox{\textcircled{\scriptsize$=$}} {\mathcal Z}^{(2,M,N)}_t \right] \\
&\quad -3\lambda \Psi _t^{(1,M,N)} ( P_{M,N} X^{M,N,(2)} ) \mbox{\textcircled{\scriptsize$=$}} {\mathcal Z}^{(2,M,N)}_t -3\lambda \Psi _t^{(2,M,N)} ( P_{M,N} X^{M,N,(2)})
\end{align*}
for $t\in [0,\infty )$.
In view of this equality, (\ref{PDEpara}) is equivalent to
\begin{equation} \label{PDEpara2}\left\{ \begin{array}{l}
\displaystyle (\partial _t - \triangle + m_0^2) X^{M,N,(2),<}_t \\
\displaystyle = -3\lambda P_{M,N}^* \left[  \left( P_{M,N} X^{M,N,(2),<}_t + P_{M,N} X^{M,N,(2),\geqslant}_t - \lambda P_{M,N} {\mathcal Z}^{(0,3,M,N)}_t \right) \mbox{\textcircled{\scriptsize$<$}} {\mathcal Z}^{(2,M,N)}_t \right] \\[4mm]
\displaystyle (\partial _t - \triangle + m_0^2) X^{M,N,(2),\geqslant}_t \\[1mm]
\displaystyle  = - \lambda P_{M,N}^* \left[ \left( P_{M,N} X^{M,N,(2),<}_t + P_{M,N} X^{M,N,(2),\geqslant}_t \right) ^3 \right] \\
\displaystyle \quad + \lambda P_{M,N}^* \Phi _t^{(1,M,N)}(P_{M,N} X^{M,N,(2),<}_t + P_{M,N} X^{M,N,(2),\geqslant}_t) \\
\displaystyle \quad + \lambda P_{M,N}^* \Phi _t^{(2,M,N)}(P_{M,N} X^{M,N,(2),<}_t + P_{M,N} X^{M,N,(2),\geqslant}_t) \\
\displaystyle \quad + \lambda P_{M,N}^* \Phi _t ^{(3,M,N)} (P_{M,N} X^{M,N,(2),<}_t + P_{M,N} X^{M,N,(2),\geqslant}_t) \\
\displaystyle \quad - 3\lambda P_{M,N}^* \left[ ( P_{M,N} X^{M,N,(2),\geqslant}_t) \mbox{\textcircled{\scriptsize$=$}} {\mathcal Z}^{(2,M,N)}_t \right] \\[1mm]
\displaystyle \quad + 9\lambda ^2 P_{M,N}^* \left[ \Psi _t^{(1,M,N)} (P_{M,N} X^{M,N,(2),<} + P_{M,N} X^{M,N,(2),\geqslant} ) \mbox{\textcircled{\scriptsize$=$}} {\mathcal Z}^{(2,M,N)}_t \right] \\
\displaystyle \quad + 9\lambda ^2 P_{M,N}^* \Psi _t^{(2,M,N)} (P_{M,N} X^{M,N,(2),<} + P_{M,N} X^{M,N,(2),\geqslant}) ,
\end{array}\right. \end{equation}
where for $w \in C([0,\infty ); {\mathcal S}'({\mathbb R}^3))$
\begin{align*}
\Phi _t^{(1,M,N)}(w)
&:= - 3 \left( {\mathcal Z}_t^{(1,M,N)} - \lambda P_{M,N} {\mathcal Z}^{(0,3,M,N)}_t\right) \mbox{\textcircled{\scriptsize$\leqslant$}} w_t^2 \\
&\quad + 3 \lambda \left[ \left( 2{\mathcal Z}_t^{(1,M,N)} - \lambda P_{M,N} {\mathcal Z}^{(0,3,M,N)}_t \right) P_{M,N} {\mathcal Z}^{(0,3,M,N)}_t \right] \mbox{\textcircled{\scriptsize$\leqslant$}} w_t ,\\
\Phi _t ^{(2,M,N)} (w)
&:= - 3 \left( w_t - \lambda P_{M,N} {\mathcal Z}^{(0,3,M,N)}_t \right) \mbox{\textcircled{\scriptsize$>$}} {\mathcal Z}^{(2,M,N)}_t + 3 \lambda \widehat{\mathcal Z}^{(2,3,M,N)}_t \\
&\quad + 9\lambda \left( w_t - \lambda P_{M,N} {\mathcal Z}^{(0,3,M,N)}_t \right) \\
&\quad \hspace{1cm} \times \left( {\mathcal Z}^{(2,2,M,N)}_t - {\mathcal Z}^{(2,M,N)}_t\mbox{\textcircled{\scriptsize$=$}}\int _{-\infty }^0 e^{(t-s)(\triangle -m_0^2)} P_N^2 {\mathcal Z}^{(2,M,N)}_s ds\right) \\
&\quad - \lambda ^2 \left( 3{\mathcal Z}_t^{(1,M,N)} - P_{M,N} {\mathcal Z}^{(0,3,M,N)}_t \right) \left( P_{M,N} {\mathcal Z}^{(0,3,M,N)}_t \right) ^2, \\
\Phi _t ^{(3,M,N)} (w)
&:= -3 \left( {\mathcal Z}_t^{(1,M,N)} - \lambda P_{M,N} {\mathcal Z}^{(0,3,M,N)}_t\right) \mbox{\textcircled{\scriptsize$>$}} w_t^2 \\
&\quad + 3 \lambda \left[ \left( 2{\mathcal Z}_t^{(1,M,N)} - \lambda P_{M,N} {\mathcal Z}^{(0,3,M,N)}_t \right) P_{M,N} {\mathcal Z}^{(0,3,M,N)}_t \right] \mbox{\textcircled{\scriptsize$>$}} w_t .
\end{align*}

To show the tightness of the laws of $\{ X^{M,N}; M,N\in {\mathbb N}\}$, we consider uniform estimates similarly to those in \cite{AK}.
For $\eta \in [0,1)$, $\gamma \in (0,1/4)$ and $\varepsilon \in (0,1]$ we define ${\mathfrak X}_{\lambda , \alpha ,\eta ,\gamma }^{M,N} (t)$ and  ${\mathfrak Y}_{\varepsilon}^{M,N} (t)$ by
\begin{align*}
&{\mathfrak X}_{\lambda , \alpha , \eta ,\gamma }^{M,N} (t) \\
&:= \int _0^t \left( m_0^2 \left\| X_s^{M,N,(2)}\right\| _{L^2(\nu )}^2 + \left\| \nabla X_s^{M,N,(2),\geqslant}\right\| _{L^2(\nu )}^2 + \lambda \left\| P_{M,N} X_s^{M,N,(2)}\right\| _{L^4(\nu )}^4 \right) ds \\
&\quad + \sup _{s',t' \in [0,t]; s'<t'} \frac{(s')^\eta \left\| X^{M,N,(2)}_{t'} - X^{M,N,(2)}_{s'} \right\| _{B_{4/3}^\alpha (\nu )}}{(t'-s')^{\gamma }} ,\\
&{\mathfrak Y}_{\varepsilon}^{M,N} (t) \\
&:= \int _0^t \left\| X^{M,N,(2),<}_s \right\| _{B_{4-\varepsilon}^{1-(\varepsilon /12)}(\nu )}^3 ds + \int _0^t \left\| X^{M,N,(2),\geqslant}_s \right\| _{B_{(4/3)+\varepsilon}^{1+\varepsilon }(\nu )} ds
\end{align*}
respectively.
We remark that ${\mathfrak X}_{\lambda , \alpha ,\eta ,\gamma }^{M,N} (t)$ takes larger by $\alpha$ than that in \cite{AK}.
This improvement can be achieved in a similar way to the proofs of \cite{AK} (see \cite{Sei}).
We are going to estimate $E\left[ {\mathfrak X}_{\lambda ,\alpha , \eta ,\gamma}^{M,N} (T)\right]$ and $E\left[ {\mathfrak Y}_{\varepsilon}^{M,N} (T) ^q\right]$ for given $T \in (0,\infty )$ and $q\in (1,8/7)$.
However, since we construct the estimates in weighted Besov spaces, we have a serious problem which does not appear in \cite{AK}.
The problem is that the approximation operator $P_N$ for smoothing is not a symmetric operator on $L^2(\nu )$.
In order to overcome this problem, we shall prepare some estimates about the asymptotics of $X^{M,N}$ in $N$ for fixed $M$ in next section.

To simplify the notation, we denote by $Q$ a positive polynomial built with the following quantities 
\begin{align*}
&2^{-\beta N} \sup _{t\in [0,T]} \left\| {\mathcal Z}_t^{(1,M,N)} \right\| _{B_p^{\beta -(1/2)-\varepsilon}(\nu )}, \quad 2^{-\beta N} \sup _{t\in [0,T]} \left\| {\mathcal Z}^{(2,M,N)}_t \right\| _{B_p^{\beta -1-\varepsilon}(\nu )},\\
&2^{-\beta N} \sup _{t\in [0,T]} \left\| {\mathcal Z}^{(0,2,M,N)}_t \right\| _{B_p^{\beta +1 -\varepsilon }(\nu )},  \quad 2^{-\beta N} \sup _{t\in [0,T]} \left\| {\mathcal Z}^{(0,3,M,N)}_t \right\| _{B_p^{\beta + (1/2) -\varepsilon }(\nu )}, \\
&2^{-\beta N} \sup _{t\in [0,T]} \left\| {\mathcal Z}^{(2,2,M,N)}_t \right\| _{B_p^{\beta -\varepsilon}(\nu )}, \quad 2^{-\beta N} \sup _{t\in [0,T]} \left\| \widehat{\mathcal Z}^{(2,3,M,N)}_t \right\| _{B_p^{\beta -(1/2) -\varepsilon}(\nu )}, \quad \\
&2^{-\beta N} \sup _{t\in [0,T]} \left\| {\mathcal Z}_t^{(1,M,N)} \left( P_{M,N} {\mathcal Z}^{(0,3,N)}_t\right) \right\| _{B_p^{\beta -(1/2) -\varepsilon}(\nu )}, \\
& 2^{-\beta N} \sup _{t\in [0,T]} \left\| {\mathcal Z}_t^{(1,M,N)} \left( P_{M,N} {\mathcal Z}^{(0,3,N)}_t\right) ^2\right\| _{B_p^{\beta -(1/2)-\varepsilon}(\nu )} \\
&\mbox{and} \quad \sup _{\substack{s,t\in [0,T]\\ s<t}} \frac{\left\| {\mathcal Z}^{(0,3,M,N)}_t - {\mathcal Z}^{(0,3,M,N)}_s \right\| _{L^p(\nu )}}{(t-s)^{\gamma }},
\end{align*}
for $\beta \in [0,\infty )$, $p\in [1,\infty )$ and $\varepsilon \in (0,1]$ with coefficients depending on $\lambda _0$, $\alpha$, $\varepsilon$, $\eta$, $\gamma$ and $T$, and we also denote by $C$ a positive constant depending on $\lambda _0$, $\varepsilon$, $\eta$, $\gamma$ and $T$.
Positive constants depending on an extra parameter $\delta$ are denoted by $C_\delta$.
We remark that $Q$, $C$ and $C_\delta$ can be different from line to line and that Proposition \ref{prop:Z} implies $E[Q] \leq C$ for some $C$.
Moreover, in view of Lemma \ref{lem:comP1} constants of the form $C(1+ M ^\kappa 2^{-\delta N} )$ appear many times in the estimate in the following sections.
So, for simplicity we denote $C(1+ M ^\kappa 2^{-\delta N} )$ for some $\kappa \in {\mathbb N}$ and $\delta \in (0,1)$ by $K_{M,N}$.

\section{Bounds for the behaviour of $X^{M,N,(2)}$ in $N$}\label{sec:asymp}

In this section we consider asymptotics of $X^{M,N,(2)}$ in $N$ for fixed $M$, which appear in commutator estimates between the smoothing operator $P_N$ and the weight $\nu$.
Besides, we often consider Besov norms without weights.
As in Proposition \ref{prop:Besov} we need change of weights for the Besov embedding theorem.
However, for fixed $M$ by using the localization of interactions by $\rho _M$ we are able to obtain sufficiently good asymptotics of Besov norms without weights.

We use the same notation as in Section \ref{sec:trans}.
In this section, for simplicity we denote
\begin{align*}
F_t^{(1)}({\mathcal Z}) &= {\mathcal Z}_t^{(1,M,N)} - \lambda P_{M,N} {\mathcal Z}^{(0,3,M,N)}_t \\
F_t^{(2)}({\mathcal Z}) &= {\mathcal Z}_t^{(2,M,N)} - 2 \lambda {\mathcal Z}_t^{(1,M,N)} P_{M,N} {\mathcal Z}_t^{(0,3,M,N)} + ( \lambda P_{M,N} {\mathcal Z}^{(0,3,M,N)}_t)^2 \\
F_t^{(3)}({\mathcal Z}) &= 3 \lambda \left( \widehat{\mathcal Z}^{(2,3,M,N)}_t + {\mathcal Z}_t^{(2,M,N)} \mbox{\textcircled{\scriptsize$<$}} P_{M,N} {\mathcal Z}_t^{(0,3,M,N)} + {\mathcal Z}_t^{(2,M,N)} \mbox{\textcircled{\scriptsize$>$}} P_{M,N} {\mathcal Z}_t^{(0,3,M,N)} \right)\\
&\quad - 3 {\mathcal Z}_t^{(1,M,N)} (\lambda P_{M,N} {\mathcal Z}_t^{(0,3,M,N)})^2 + (\lambda P_{M,N} {\mathcal Z}_t^{(0,3,M,N)})^3 .
\end{align*}
We remark that $F_t^{(1)}({\mathcal Z})$, $F_t^{(2)}({\mathcal Z})$ and $F_t^{(3)}({\mathcal Z})$ depend on $M$ and $N$.

\begin{lem}\label{lem:estF}
Let $\varepsilon \in (0,1)$.
For each $p\in [1,\infty )$ it holds that
\begin{align*}
\sup _{t\in [0,T]} \| F_t^{(1)}({\mathcal Z}) \| _{B_p ^{\alpha -\varepsilon}}(\nu) &\leq Q 2^{(\alpha +1/2)N}, \quad \alpha > -\frac 12\\
\sup _{t\in [0,T]} \| F_t^{(2)}({\mathcal Z}) \| _{B_p ^{\alpha -\varepsilon}}(\nu) &\leq Q 2^{(\alpha +1)N}, \quad \alpha > -1 \\
\sup _{t\in [0,T]} \| F_t^{(3)}({\mathcal Z}) \| _{B_p ^{\alpha -\varepsilon}}(\nu) &\leq Q 2^{(\alpha +1)N}, \quad \alpha > -1 .
\end{align*}
\end{lem}

\begin{proof}
The estimates follows from Propositions \ref{prop:paraproduct} and \ref{prop:Z}.
\end{proof}

\begin{prop}\label{prop:XN2est}
For any $\varepsilon \in (0,1)$ and $M,N\in {\mathbb N}$ it holds that
\[
\int _0^T E\left[ \| X_s^{M,N,(2)}\| _{L^2}^2 + \| \nabla X_s^{M,N,(2)}\| _{L^2}^2 + \lambda \| P_{M,N} X_s^{M,N,(2)}\| _{L^4}^4 \right] ds \leq K_{M,N} 2^{\varepsilon N} .
\]
Moreover, for each $t\in [0,T]$, it holds that
\begin{align*}
&\| X_t^{M,N,(2)}\| _{L^2}^2 + \int _0^t \left( \| X_s^{M,N,(2)}\| _{L^2}^2 + \| \nabla X_s^{M,N,(2)}\| _{L^2}^2 \right) ds + \lambda \int _0^t \| P_{M,N} X_s^{M,N,(2)}\| _{L^4}^4 ds\\
&\leq C \| X_0^{M,N,(2)}\| _{L^2}^2 + K_{M,N} 2^{\varepsilon N} Q < \infty
\end{align*}
for $M,N\in {\mathbb N}$ almost surely.
\end{prop}

\begin{proof}
Let $h_n (x) := \exp (n^{-1} \sqrt{1+|x|^2})$ for $x\in {\mathbb R}^3$ and $n\in {\mathbb N}$.
From (\ref{eq:PDEX2}) we have
\begin{align*}
& \| X_t^{M,N,(2)} \| _{L^2(h_n)}^2 - \| X_0^{M,N,(2)}\| _{L^2(h_n)}^2 \\
&= 2 \int _0^t \int _{{\mathbb R}^3} X_s^{M,N,(2)} [(\triangle - m_0^2) X_s^{M,N,(2)}] h_n dx ds\\
&\quad - 2 \lambda \int _0^t \int _{{\mathbb R}^3} X_s^{M,N,(2)} \left( P_{M,N}^* \left[ (P_{M,N} X_s^{M,N,(2)} )^3 \right] \right) h_n dx ds \\
&\quad - 6\lambda \int _0^t \int _{{\mathbb R}^3} X_s^{M,N,(2)} \left( P_{M,N}^* \left[ F_s^{(1)}({\mathcal Z}) (P_{M,N} X_s^{M,N,(2)} )^2 \right] \right) h_n dx ds \\
&\quad - 6 \lambda \int _0^t \int _{{\mathbb R}^3} X_s^{M,N,(2)} \left( P_{M,N}^* \left[ F_s^{(2)}({\mathcal Z}) P_{M,N} X_s^{M,N,(2)} \right] \right) h_n dx ds \\
&\quad + 2 \lambda \int _0^t \int _{{\mathbb R}^3} X_s^{M,N,(2)} \left( P_{M,N}^* F_s^{(3)}({\mathcal Z}) \right) h_n dx ds \\
&\quad + 18 \lambda ^2 \int _0^t \int _{{\mathbb R}^3} X_s^{M,N,(2)} P_{M,N}^*  \left[ C_2^{(M,N)} \rho _M ^2 \left( P_{M,N} X^{M,N,(2)}_s - \lambda P_{M,N} {\mathcal Z}_t^{(0,3,M,N)}  \right) \right] h_n dx ds \\
&= - 2 \int _0^t \left( \| \nabla X_s^{M,N,(2)} \| _{L^2(h_n)}^2 + m_0^2 \| X_s^{M,N,(2)} \| _{L^2(h_n)}^2\right) ds \\
&\quad- 2 \int _0^t \int _{{\mathbb R}^3} X_s^{M,N,(2)} \nabla X_s^{M,N,(2)} \cdot \nabla h_n dx ds - 2 \lambda \int _0^t \int _{{\mathbb R}^3} (P_{M,N} X_s^{M,N,(2)} )^4 dx ds \\
&\quad + 2 \lambda \int _0^t \int _{{\mathbb R}^3} (1-h_n) X_s^{M,N,(2)} P_{M,N}^* \left[ (P_{M,N} X_s^{M,N,(2)} )^3 \right] dx ds \\
&\quad - 6 \lambda \int _0^t \int _{{\mathbb R}^3} F_s^{(1)}({\mathcal Z}) \nu ^{-1} (P_{M,N} X_s^{M,N,(2)} )^3 h_n \nu dx ds \\
&\quad + 6 \lambda \int _0^t \int _{{\mathbb R}^3} (1-h_n) X_s^{M,N,(2)} \left( P_{M,N}^* \left[  F_s^{(1)}({\mathcal Z}) (P_{M,N} X_s^{M,N,(2)} )^2 \right] \right) dx ds \\
&\quad - 6 \lambda \int _0^t \int _{{\mathbb R}^3} F_s^{(2)}({\mathcal Z}) \nu ^{-1} ( P_{M,N} X_s^{M,N,(2)}) ^2 h_n \nu dx ds \\
&\quad + 6 \lambda \int _0^t \int _{{\mathbb R}^3} (1-h_n) X_s^{M,N,(2)} \left( P_{M,N}^* \left[ F_s^{(2)}({\mathcal Z}) P_{M,N} X_s^{M,N,(2)} \right] \right) dx ds \\
&\quad + 2 \lambda \int _0^t \int _{{\mathbb R}^3} F_s^{(3)}({\mathcal Z}) \nu ^{-1} (P_{M,N} X_s^{M,N,(2)}) h_n \nu dx ds \\
&\quad - 2 \lambda \int _0^t \int _{{\mathbb R}^3} (1-h_n) X_s^{M,N,(2)} P_{M,N}^* F_s^{(3)}({\mathcal Z}) dx ds \\
&\quad + 18 \lambda ^2 \int _0^t \int _{{\mathbb R}^3} C_2^{(M,N)} \nu ^{-1 }\rho _M ^2 \left[ P_{M,N} (h_n X_s^{M,N,(2)}) \right]  \left[ P_{M,N} X^{N,(2)}_s - \lambda P_{M,N} {\mathcal Z}_s^{(0,3,M,N)} \right] \nu dx ds.
\end{align*}
Hence, Proposition \ref{prop:Besov} and Lemma \ref{lem:PN} imply
\begin{align*}
& \| X_t^{M,N,(2)} \| _{L^2(h_n)}^2 - \| X_0^{M,N,(2)}\| _{L^2(h_n)}^2 \\
&\leq - 2 \int _0^t \left( \| \nabla X_s^{M,N,(2)} \| _{L^2(h_n)}^2 + m_0^2 \| X_s^{M,N,(2)} \| _{L^2(h_n)}^2\right) ds \\
&\quad- 2 \int _0^t \int _{{\mathbb R}^3} X_s^{M,N,(2)} \nabla X_s^{M,N,(2)} \cdot \nabla h_n dx ds  - 2 \lambda \int _0^t \int _{{\mathbb R}^3} (P_{M,N} X_s^{M,N,(2)} )^4 dx ds \\
&\quad + 2 \lambda \int _0^t \int _{{\mathbb R}^3} (1-h_n) X_s^{M,N,(2)} P_{M,N}^* \left[ (P_{M,N} X_s^{M,N,(2)} )^3 \right] dx ds \\
&\quad + \lambda C \left( \sup _{s\in [0,T]}\left\| F_s^{(1)}({\mathcal Z}) \right\| _{B_{6/\varepsilon}^{-1/2 + \varepsilon}(\nu )} \right) \int _0^t \left\| \nu ^{-1} ( P_{M,N} X_s^{M,N,(2)} )^3 \right\| _{B_{6/(6-\varepsilon )}^{(1-\varepsilon)/2}(\nu )} ds \\
&\quad + 6 \lambda \int _0^t \int _{{\mathbb R}^3} (1-h_n) X_s^{M,N,(2)} \left( P_{M,N}^* \left[ F_s^{(1)}({\mathcal Z}) (P_{M,N} X_s^{M,N,(2)} )^2 \right] \right) dx ds \\
&\quad + \lambda C \left( \sup _{s\in [0,T]}\left\| F_s^{(2)}({\mathcal Z}) \right\| _{B_{6/\varepsilon }^{-1+\varepsilon}(\nu )} \right) \int _0^t \left\| \nu ^{-1} ( P_{M,N} X_s^{M,N,(2)}) ^2 \right\| _{B_{6/(6-\varepsilon )}^{1-\varepsilon /2}(\nu )} ds \\
&\quad + 6 \lambda \int _0^t \int _{{\mathbb R}^3} (1-h_n) X_s^{M,N,(2)} \left( P_{M,N}^* \left[ F_s^{(2)}({\mathcal Z}) P_{M,N} X_s^{M,N,(2)} \right] \right) dx ds \\
&\quad + \lambda C \left( \sup _{s\in [0,T]}\left\| F_s^{(3)}({\mathcal Z}) \right\| _{B_2^{-1}(\nu )} \right) \int _0^t \| \nu ^{-1} P_{M,N} X_s^{M,N,(2)}\| _{B_2^1(\nu )} ds \\
&\quad + 2 \lambda \int _0^t \int _{{\mathbb R}^3} (1-h_n) X_s^{M,N,(2)} \left( P_{M,N}^* F_s^{(3)}({\mathcal Z}) \right) dx ds \\
&\quad + 18 \lambda ^2 \int _0^t \left\| C_2^{(M,N)} \nu ^{-1} \rho _M^2 P_{M,N} \left( h_n X_s^{M,N,(2)} \right) \right\| _{L^2} \left\| P_{M,N} X_s^{M,N,(2)} \right\| _{L^2} ds \\
&\quad + 18 \lambda ^3 \left( \sup _{t\in [0,T]} \left\| P_{M,N} {\mathcal Z}_t^{(0,3,M,N)} \right\| _{L^2(\nu )} \right) \int _0^t \left\| C_2^{(M,N)} \nu ^{-1} \rho _M^2 P_{M,N} \left( h_n X_s^{M,N,(2)} \right) \right\| _{L^2 (\nu )} ds.
\end{align*}
Note that from Propositions \ref{prop:BS} and \ref{prop:Besov}, Lemma \ref{lem:Lpestimates1} and the fact that $\nu ^{-1}$ and the derivative of $\nu ^{-1}$ are bounded by $C M^\sigma$, it follows that
\begin{align*}
&\left\| \nu ^{-1} ( P_{M,N} X_s^{M,N,(2)} )^3 \right\| _{B_{6/(6-\varepsilon )}^{(1-\varepsilon)/2}(\nu )} \\
&\leq \left\| \nu ^{-1} ( P_{M,N} X_s^{M,N,(2)} )^3 \right\| _{B_{6/(6-\varepsilon )}^{(1-\varepsilon)/2}} \\
&\leq C \left\| \nu ^{-1} ( P_{M,N} X_s^{M,N,(2)} )^3 \right\| _{B_{1}^{1/2}} \\
&\leq C \left\| \nu ^{-1} ( P_{M,N} X_s^{M,N,(2)} )^3 \right\| _{L^1}^{1/2} \left\| \nu ^{-1} ( P_{M,N} X_s^{M,N,(2)} )^3 \right\| _{W^{1,1}}^{1/2}\\
&\leq \delta \left\| ( P_{M,N} X_s^{M,N,(2)} )^3 \right\| _{W^{1,1}(\nu )} + C M^{2\sigma} \delta ^{-1} \left\| ( P_{M,N} X_s^{M,N,(2)} )^3 \right\| _{L^1(\nu )}\\
&\leq \delta \left( \| P_{M,N} X_s^{M,N,(2)} \| _{L^4(\nu )}^4 + \| \nabla P_{M,N} X_s^{M,N,(2)} \| _{L^2(\nu )}^2 \right) + C M^{14\sigma /3} \delta ^{-7/3}
\end{align*}
for any $\delta \in (0,1]$, and similarly it follows that
\begin{align*}
&\left\| \nu ^{-1} ( P_{M,N} X_s^{M,N,(2)}) ^2 \right\| _{B_{6/(6-\varepsilon )}^{1-\varepsilon /2}(\nu )} \\
&\leq \delta \left\| \nabla P_{M,N} X_s^{M,N,(2)}\right\| _{L^2(\nu )}^2 + \delta \left\| P_{M,N} X_s^{M,N,(2)} \right\| _{L^4(\nu )}^4 + C M^{6\sigma} \delta ^{-3}
\end{align*}
for any $\delta \in (0,1]$.
Moreover, in view of Proposition \ref{prop:C}
\begin{align*}
&\left\| C_2^{(M,N)} \nu ^{-1} \rho _M^2 P_{M,N} \left( h_n X_s^{M,N,(2)} \right) \right\| _{L^2(\nu )} \\
&\leq C M^\sigma N \| X_s^{M,N,(2)}\| _{L^2(\nu )} \\
&\leq \delta \| X_s^{M,N,(2)}\| _{L^2(\nu )}^2 + C M^{2\sigma} N ^2 \delta ^{-1}.
\end{align*}
Hence, by applying Lemma \ref{lem:Lpestimates1} and Proposition \ref{prop:C}, for $\delta \in (0,1]$ we have
\begin{equation}\label{eq:propXNest01}\begin{array}{l}
\displaystyle \| X_t^{M,N,(2)} \| _{L^2(h_n)}^2 - \| X_0^{M,N,(2)}\| _{L^2(h_n)}^2\\
\displaystyle  + 2 \int _0^t \left( \| \nabla X_s^{M,N,(2)} \| _{L^2(h_n)}^2 + m_0^2 \| X_s^{M,N,(2)} \| _{L^2(h_n)}^2\right) ds +  2 \lambda \int _0^t \left\| P_{M,N} X_s^{M,N,(2)} \right\| _{L^4} ^4 ds \\
\displaystyle \leq - 2 \int _0^t \int _{{\mathbb R}^3} X_s^{M,N,(2)} \nabla X_s^{M,N,(2)} \cdot \nabla h_n dx ds \\
\displaystyle \quad + \lambda \delta \int _0^t \left( \left\| \nabla X_s^{M,N,(2)}\right\| _{L^2}^2 + \left\| X_s^{M,N,(2)} \right\| _{L^2}^2 + \left\| P_{M,N} X_s^{M,N,(2)} \right\| _{L^4}^4\right) ds  \\
\displaystyle \quad + 2 \lambda \int _0^t \int _{{\mathbb R}^3} (1-h_n) X_s^{M,N,(2)} P_{M,N}^* \left[ (P_{M,N} X_s^{M,N,(2)} )^3 \right] dx ds \\
\displaystyle \quad + 6 \lambda \int _0^t \int _{{\mathbb R}^3} (1-h_n) X_s^{M,N,(2)} \left( P_{M,N}^* \left[ F_s^{(1)}({\mathcal Z}) (P_{M,N} X_s^{M,N,(2)} )^2 \right] \right) dx ds \\
\displaystyle \quad + 6 \lambda \int _0^t \int _{{\mathbb R}^3} (1-h_n) X_s^{M,N,(2)} \left( P_{M,N}^* \left[ F_s^{(2)}({\mathcal Z}) P_{M,N} X_s^{M,N,(2)} \right] \right) dx ds \\
\displaystyle \quad + 2 \lambda \int _0^t \int _{{\mathbb R}^3} (1-h_n) X_s^{M,N,(2)} \left( P_{M,N}^* F_s^{(3)}({\mathcal Z}) \right) dx ds + C_{\delta} M^{6\sigma} 2^{6 \varepsilon N} Q.
\end{array}\end{equation}
Note that  the stationarity of $X^{N,(1)}$ implies
\begin{align*}
&E\left[ \| X_t^{M,N,(2)} \| _{L^2(h_n)}^2 \right] - E\left[ \| X_0^{M,N,(2)}\| _{L^2(h_n)}^2 \right] \\
&= 2 \lambda E\left[ \int _{{\mathbb R}^3} {\mathcal Z}_t^{(0,3,M,N)} X_t^{M,N,(2)} h_n dx \right] - \lambda ^2 E\left[ \| {\mathcal Z}_t^{(0,3,M,N)} \| _{L^2(h_n)}^2 \right] \\
&\quad - 2 \lambda E\left[ \int _{{\mathbb R}^3} {\mathcal Z}_0^{(0,3,M,N)} X_0^{M,N,(2)} h_n dx \right] + \lambda ^2 E\left[ \| {\mathcal Z}_0^{(0,3,M,N)} \| _{L^2(h_n)}^2 \right] ,
\end{align*}
and that H\"older's inequality, Lemma \ref{lem:comP1}, Proposition \ref{prop:Z}, and the stationarity of $X^{M,N,(1)}$ imply
\begin{align*}
& \sup _{s\in [0,t]} E\left[ \left| \int _{{\mathbb R}^3} {\mathcal Z}_s^{(0,3,M,N)} X_s^{M,N,(2)} h_n dx \right| \right] \\
& \leq \delta \sup _{s\in [0,t]} E\left[ \| X_s^{M,N,(2)} \| _{L^2(h_n)}^2 \right] + C_\delta K_{M,N} \sup _{s\in [0,t]} E\left[ \| {\mathcal Z}_s^{(0,3,M,N)} \| _{L^2(h_n)}^2 \right]\\
& \leq 2 \delta \sup _{s\in [0,t]} E\left[ \| X_s^{M,N,(1)} \| _{L^2(h_n)}^2 \right] + 2 \lambda ^2 \delta \sup _{s\in [0,t]} E\left[ \| {\mathcal Z}_s^{(0,3,M,N)} \| _{L^2(h_n)}^2 \right] + C_\delta K_{M,N} M^{\sigma} \\
&\leq \frac{2\delta}{T} E\left[ \int _0^T \| X_s^{M,N,(1)} \| _{L^2(h_n)}^2 ds \right] + C_\delta K_{M,N} M^{\sigma} \\
&\leq \frac{2\delta}{T} E\left[ \int _0^T \| X_s^{M,N,(2)} \| _{L^2(h_n)}^2 ds \right] + C_\delta K_{M,N} M^{\sigma}
\end{align*}
for $\delta \in (0,1]$.
Since $|\nabla h_n| \leq (2n)^{-1} h_n$ (recalling that $h_n \geq 0$), by taking expectation first and then limit $n\rightarrow \infty$ in (\ref{eq:propXNest01}), we have
\begin{align*}
&\int _0^T E\left[\| X_s^{M,N,(2)} \| _{L^2}^2 + \| \nabla X_s^{M,N,(2)} \| _{L^2}^2 \right] ds + \lambda \int _0^T E\left[ \| P_{M,N} X_s^{M,N,(2)} \| _{L^4}^4 \right] ds \\
&\leq \lambda C \delta \int _0^T E\left[ \left\| X_s^{M,N,(2)}\right\| _{L^2}^2 + \left\| \nabla X_s^{M,N,(2)}\right\| _{L^2}^2 + \left\| P_{M,N} X_s^{M,N,(2)} \right\| _{L^4}^4 \right] ds \\
&\quad + \frac{C \delta }{T} E\left[ \int _0^T \| X_s^{M,N,(2)} \| _{L^2}^2 ds \right]  + C_{\delta} K_{M,N} M^{6\sigma} 2^{6\varepsilon N}
\end{align*}
for $\delta \in (0,1]$.
Since $\varepsilon \in (0,1)$ is arbitrary, by taking $\delta$ sufficiently small, we get the first assertion of the lemma.

The first assertion implies in particular that
\[
\int _0^t \| \nabla X_s^{M,N,(2)}\| _{L^2}^2 ds , \quad \int _0^t \| P_{M,N} X_s^{M,N,(2)}\| _{L^4}^4 ds
\]
are finite almost surely for each $t\in [0,T]$.
Moreover, the first assertion and the stationarity of $X^{M,N,(1)}$, Lemma \ref{lem:comP1} and Proposition \ref{prop:Z} also imply that
\begin{align*}
&E\left[ \| X_t^{M,N,(2)}\| _{L^2}^2 \right] \\
&\leq 2 E\left[ \| X_t^{M,N,(1)}\| _{L^2}^2 \right] + 2\lambda E\left[ \| {\mathcal Z}_t^{(0,3,M,N)} \| _{L^2}^2 \right] \\
&\leq \frac{2}{T} E\left[ \int _0^T \| X_s^{M,N,(1)}\| _{L^2}^2 ds \right] + C M^\sigma \\
&\leq \frac{4}{T} E\left[ \int _0^T \| X_s^{M,N,(2)}\| _{L^2}^2 ds \right] + \frac{4}{T} E\left[ \int _0^T \| {\mathcal Z}_t^{(0,3,M,N)} \| _{L^2}^2 ds \right] + C M^\sigma \\
&< \infty .
\end{align*}
Hence, $\| X_t^{M,N,(2)}\| _{L^2}^2$ is also finite almost surely for each $t\in [0,T]$.
In view of these facts, by letting $n\rightarrow \infty$ in (\ref{eq:propXNest01}) we have
\begin{align*}
& \| X_t^{M,N,(2)} \| _{L^2}^2 - \| X_0^{M,N,(2)}\| _{L^2}^2\\
& + 2 \int _0^t \left( \| \nabla X_s^{M,N,(2)} \| _{L^2}^2 + m_0^2 \| X_s^{M,N,(2)} \| _{L^2}^2\right) ds +  2 \lambda \int _0^t \left\| P_{M,N} X_s^{M,N,(2)} \right\| _{L^4} ^4 ds \\
& \leq \lambda \delta \int _0^t \left( \left\| X_s^{M,N,(2)}\right\| _{L^2}^2 + \left\| \nabla X_s^{M,N,(2)}\right\| _{L^2}^2 + \left\| P_{M,N} X_s^{M,N,(2)} \right\| _{L^4}^4 \right) ds \\
&\quad + C_{\delta} M^{6\sigma} 2^{6\varepsilon N} Q.
\end{align*}
almost surely for $\delta \in (0,1]$ and $t\in [0,T]$.
By taking $\delta$ sufficiently small
we get the second assertion of the lemma.
\end{proof}

\begin{cor}\label{cor:XN2est1}
For $\alpha \in (1/2 ,\infty )$ and $M, N \in {\mathbb N}$
\[
2^{-\alpha N} \sup _{t\in [0,T]} E\left[\| X_t^{M,N,(2)}\| _{B_2^1}^2 \right] \leq K_{M,N}.
\]
\end{cor}

\begin{proof}
For any $\varepsilon \in (0,1)$, by the stationarity of $X^{M,N,(1)}$, Lemma \ref{lem:comP1} and Proposition \ref{prop:Z} we have
\begin{align*}
&\sup _{t\in [0,T]} E\left[\| X_t^{M,N,(2)}\| _{B_2^1}^2 \right] \\
&\leq 2 \sup _{t\in [0,T]} E\left[\| X_t^{M,N,(1)}\| _{B_2^1}^2 \right] + 2 \lambda ^2 \sup _{t\in [0,T]} E\left[\| {\mathcal Z}_t^{(0,3,M,N)} \| _{B_2^1}^2 \right] \\
&\leq \frac{2}{T} \int _0^T E\left[\| X_t^{M,N,(1)}\| _{B_2^1}^2 \right] dt + C M^{\sigma} 2^{(1+ \varepsilon ) N/2} \\
&\leq \frac{2}{T} \int _0^T E\left[\| X_t^{M,N,(2)}\| _{B_2^1}^2 \right] dt + C M^{\sigma} 2^{(1+ \varepsilon ) N/2} .
\end{align*}
Hence, by Proposition \ref{prop:XN2est} we obtain the assertion.
\end{proof}

\begin{prop}\label{prop:XN2est2}
For $\alpha \in [0,1/2]$, sufficiently small $\varepsilon \in (0,1)$ and $t\in [0,T]$
\begin{align*}
\| X^{M,N,(2)}_t\| _{B_2^{\alpha }}&\leq C \| X^{M,N,(2)}_0\| _{B_2^{\alpha}} + K_{M,N} 2^{[(3+2\varepsilon )/4]N} \| X^{M,N,(2)}_0\| _{L^2}^{3/2} \\
&\quad + K_{M,N} 2^{(1+\varepsilon )N} \| X^{M,N,(2)}_0\| _{L^2}^{1/2} + K_{M,N} 2^{[(9/8) + \varepsilon ]N} Q.
\end{align*}
\end{prop}

\begin{proof}
In view of the mild form of the solution to (\ref{eq:PDEX2}) we have
\begin{align*}
X^{M,N,(2)}_t &= e^{t(\triangle - m_0^2)} X^{M,N,(2)}_0 - \lambda \int _0^t e^{(t-s)(\triangle - m_0^2)} P_{M,N}^* \left[ (P_{M,N} X_s^{M,N,(2)} )^3 \right] ds \\
&\quad - 3\lambda \int _0^t e^{(t-s)(\triangle - m_0^2)} P_{M,N}^* \left[ F_s^{(1)}({\mathcal Z}) (P_{M,N} X_s^{M,N,(2)} )^2 \right] ds \\
&\quad - 3 \lambda \int _0^t e^{(t-s)(\triangle - m_0^2)} P_{M,N}^* \left[ F_s^{(2)}({\mathcal Z}) P_{M,N} X_s^{M,N,(2)} \right] ds \\
&\quad + \lambda \int _0^t e^{(t-s)(\triangle - m_0^2)} P_{M,N}^* F_s^{(3)}({\mathcal Z}) ds \\
&\quad + 9 \lambda ^2 \int _0^t e^{(t-s)(\triangle - m_0^2)} P_{M,N}^*\left[ C_2^{(M,N)} \rho _M ^2 \left( P_{M,N} X^{M,N,(2)}_s -\lambda P_{M,N} {\mathcal Z}_s^{(0,3,M,N)} \right) \right] ds
\end{align*}
for $t\in [0,T]$.
Hence, by Lemmas \ref{lem:PN} and \ref{lem:comP1} and Proposition \ref{prop:BS} we have
\begin{align*}
&\| X^{M,N,(2)}_t\| _{B_2^{\alpha }} \\
&\leq C \| X^{M,N,(2)}_0\| _{B_2^{\alpha}} + \lambda C \int _0^t (t-s)^{-(1-\varepsilon )/4}\left\| P_{M,N}^* \left[ (P_{M,N} X^{M,N,(2)}_s)^3 \right] \right\| _{B_2^{\alpha -(1-\varepsilon )/2}}  ds \\
&\quad + \lambda C \int _0^t (t-s)^{-(1-\varepsilon ) /2} \left\| P_{M,N}^* \left[ F_s^{(1)}({\mathcal Z}) (P_{M,N} X^{M,N,(2)}_s)^2 \right] \right\| _{B_2^{\alpha -1 +\varepsilon}} ds \\
&\quad + \lambda C \int _0^t (t-s)^{-\alpha /2} \left\| P_{M,N}^* \left[ F_s^{(2)}({\mathcal Z}) P_{M,N} X^{M,N,(2)}_s \right] \right\| _{B_2^0} ds \\
&\quad + \lambda C \int _0^t (t-s)^{-\alpha /2} \left\| P_{M,N}^* F_s^{(3)}({\mathcal Z}) \right\| _{B_2^0} ds \\
&\quad + \lambda ^2 C \int _0^t (t-s)^{-\alpha /2} \left\| P_{M,N}^* \left[ C_2^{(M,N)} \rho _M ^2 P_{M,N} X^{M,N,(2)}_s \right] \right\| _{B_2^0} ds \\
&\quad + \lambda ^2 C \int _0^t (t-s)^{-\alpha /2} \left\| P_{M,N}^* \left[ C_2^{(M,N)} \rho _M ^2 P_{M,N} {\mathcal Z}_s^{(0,3,M,N)} \right] \right\| _{B_2^0} ds \\
&\leq C \| X^{N,(2)}_0\| _{B_2^{\alpha}} + \lambda K_{M,N} 2^{[\alpha +(1+2\varepsilon )/4]N} \int _0^t (t-s)^{-(1-\varepsilon )/4} \left\| (P_{M,N} X^{M,N,(2)}_s)^3 \right\| _{B_2^{-3/4}}  ds \\
&\quad + \lambda K_{M,N} 2^{[(1/4)+\varepsilon ]N}\int _0^t (t-s)^{-(1-\varepsilon ) /2} \left\| F_s^{(1)}({\mathcal Z}) (P_{M,N} X^{M,N,(2)}_s)^2  \right\| _{B_2^{-3/4}} ds \\
&\quad + \lambda K_{M,N} 2^{\varepsilon N} \int _0^t  (t-s)^{-\alpha /2} \left\| F_s^{(2)}({\mathcal Z}) P_{M,N} X^{M,N,(2)}_s  \right\| _{B_2^{-\varepsilon}} ds \\
&\quad + \lambda K_{M,N} \int _0^t (t-s)^{-\alpha /2} \left\| \rho _M F_s^{(3)}({\mathcal Z}) \right\| _{L^2} ds \\
&\quad + \lambda ^2 K_{M,N} \int _0^t (t-s)^{-\alpha /2} \left\| C_2^{(M,N)} \rho _M ^2 P_{M,N} X^{M,N,(2)}_s \right\| _{L^2} ds \\
&\quad + \lambda ^2 K_{M,N} \int _0^t (t-s)^{-\alpha /2} \left\| C_2^{(M,N)}\rho _M ^2 P_{M,N} {\mathcal Z}_s^{(0,3,M,N)} \right\| _{L^2} ds .
\end{align*}
Hence, from Propositions \ref{prop:BS}, \ref{prop:Besov}, \ref{prop:paraproduct} and \ref{prop:C}, it follows that
\begin{align*}
&\| X^{M,N,(2)}_t\| _{B_2^{\alpha }} \\
&\leq C \| X^{M,N,(2)}_0\| _{B_2^{\alpha}} + \lambda K_{M,N} 2^{[\alpha +(1+2\varepsilon )/4]N} \int _0^t (t-s)^{-(1-\varepsilon )/4} \left\| (P_{M,N} X^{M,N,(2)}_s)^3 \right\| _{B_{4/3}^0}  ds \\
&\quad + \lambda K_{M,N} 2^{[(1/4)+\varepsilon ]N}\int _0^t (t-s)^{-(1-\varepsilon ) /2} \left\| F_s^{(1)}({\mathcal Z}) (P_{M,N} X^{M,N,(2)}_s)^2  \right\| _{B_{4/3}^0} ds \\
&\quad + \lambda K_{M,N} 2^{\varepsilon N} \int _0^t (t-s)^{-\alpha /2} \left\| F_s^{(2)}({\mathcal Z}) P_{M,N} X^{M,N,(2)}_s  \right\| _{B_{6/(3+2\varepsilon )}^0} ds \\
&\quad + \lambda K_{M,N} \int _0^t (t-s)^{-\alpha /2} \left\| \rho _M F_s^{(3)}({\mathcal Z}) \right\| _{L^2} ds \\
&\quad + \lambda ^2 K_{M,N} N \int _0^t  (t-s)^{-\alpha /2} \left\| \rho _M ^2 P_{M,N} X^{M,N,(2)}_s \right\| _{L^2} ds \\
&\quad + \lambda ^2 K_{M,N} N \sup _{s\in [0,T]} \left\| \rho _M^2 P_{M,N} {\mathcal Z}_s^{(0,3,M,N)} \right\| _{L^2}\\
&\leq C \| X^{M,N,(2)}_0\| _{B_2^{\alpha}} \\
&\quad + \lambda K_{M,N} M^\sigma 2^{[\alpha +(1+2\varepsilon )/4]N} \int _0^t (t-s)^{-(1-\varepsilon )/4} \left\| (P_{M,N} X^{M,N,(2)}_s)^3 \right\| _{L^{4/3}(\nu )}  ds \\
&\quad + \lambda K_{M,N} M^\sigma 2^{[(1/4)+\varepsilon ]N} \int _0^t (t-s)^{-(1-\varepsilon )/2} \left\| F_s^{(1)}({\mathcal Z}) (P_{M,N} X^{M,N,(2)}_s)^2  \right\| _{L^{4/3}(\nu )} ds \\
&\quad + \lambda K_{M,N} M^\sigma 2^{\varepsilon N} \int _0^t (t-s)^{-\alpha /2} \left\| F_s^{(2)}({\mathcal Z}) P_{M,N} X^{M,N,(2)}_s \right\| _{L^{6/(3+2\varepsilon )}(\nu )} ds \\
&\quad + \lambda K_{M,N} M^\sigma \int _0^t (t-s)^{-\alpha /2} \left\| F_s^{(3)}({\mathcal Z}) \right\| _{L^2(\nu )} ds \\
&\quad + \lambda ^2 K_{M,N} M^\sigma N  \int _0^t (t-s)^{-\alpha /2} \left\| \rho _M ^2 P_{M,N} X^{M,N,(2)}_s \right\| _{L^2(\nu )} ds \\
&\quad + K_{M,N} M^\sigma N \sup _{s\in [0,T]} \left\| {\mathcal Z}_s^{(0,3,M,N)} \right\| _{L^2(\nu )}\\
&\leq C \| X^{M,N,(2)}_0\| _{B_2^{\alpha}} \\
&\quad + \lambda K_{M,N} M^\sigma 2^{[\alpha +(1+2\varepsilon )/4]N} \int _0^t (t-s)^{-(1-\varepsilon )/4} \left\| P_{M,N} X^{M,N,(2)}_s \right\| _{L^4(\nu )}^3 ds \\
&\quad + \lambda K_{M,N} M^\sigma 2^{[(1/4)+\varepsilon ]N} \left( \sup _{s\in [0,T]} \left\| F_s^{(1)}({\mathcal Z}) \right\| _{L^4(\nu )}\right) \\
&\quad \hspace{4cm} \times \int _0^t (t-s)^{-(1-\varepsilon )/2} \left\| (P_{M,N} X^{M,N,(2)}_s)^2 \right\| _{L^2(\nu )} ds \\
&\quad + \lambda K_{M,N} M^\sigma 2^{\varepsilon N} \left( \sup _{s\in [0,T]} \left\| F_s^{(2)}({\mathcal Z}) \right\| _{L^{3/\varepsilon}(\nu )}\right) \int _0^t (t-s)^{-\alpha /2} \left\| P_{M,N} X^{M,N,(2)}_s  \right\| _{L^2(\nu )} ds \\
&\quad + \lambda ^2 K_{M,N} M^\sigma N \int _0^t (t-s)^{-\alpha /2} \left\| P_{M,N} X^{M,N,(2)}_s\right\| _{L^2(\nu )} ds \\
&\quad + \lambda K_{M,N} M^\sigma \sup _{s\in [0,T]} \left\| F_s^{(3)}({\mathcal Z}) \right\| _{L^2(\nu )} + K_{M,N} M^\sigma N \sup _{s\in [0,T]} \left\| {\mathcal Z}_s^{(0,3,M,N)} \right\| _{L^2(\nu )}
\end{align*}
for $t\in [0,T]$.
From this inequality, Lemma \ref{lem:Lpestimates1} and Propositions \ref{prop:BS} and \ref{prop:XN2est} we have
\begin{align*}
&\| X^{M,N,(2)}_t\| _{B_2^{\alpha }} \\
&\leq C \| X^{M,N,(2)}_0\| _{B_2^{\alpha}} + \lambda K_{M,N} M^\sigma 2^{[\alpha +(1+2\varepsilon )/4]N} \left( \int _0^t \left\| P_{M,N} X^{M,N,(2)}_s \right\| _{L^4 (\nu )} ^4 ds \right) ^{3/4} \\
&\quad + \lambda K_{M,N} M^\sigma2^{[(3/4)+\varepsilon]N} \left( \int _0^t \| P_{M,N} X^{M,N,(2)}_s\| _{L^4 (\nu )}^4 ds \right) ^{1/2} \\
&\quad + \lambda K_{M,N} M^\sigma 2^{(1+\varepsilon )N} \left( \int _0^t \| P_{M,N} X^{M,N,(2)}_s\| _{L^4 (\nu )}^4 ds \right) ^{1/4} + \lambda K_{M,N} M^\sigma 2^{(1+\varepsilon )N} Q\\
&\leq C \| X^{M,N,(2)}_0\| _{B_2^{\alpha}} + K_{M,N} M^\sigma 2^{[(3+2\varepsilon )/4]N} \| X^{M,N,(2)}_0\| _{L^2}^{3/2} \\
&\quad + K_{M,N} M^\sigma 2^{(1+\varepsilon )N} \| X^{M,N,(2)}_0\| _{L^2}^{1/2} + K_{M,N} M^\sigma 2^{[(9/8) + 2\varepsilon ]N} Q.
\end{align*}
Since $\varepsilon \in (0,1)$ is arbitrary, this is the desired inequality.
\end{proof}

\begin{cor}\label{cor:XN2est2}
For $\alpha \in [0,1/2]$ and $\beta \in (9/8, \infty )$
\[
2^{-(4\beta /3) N} E\left[ \sup _{s\in [0,T]} \| X_s^{M,N,(2)}\| _{B_2^\alpha }^{4/3} \right] \leq K_{M,N} .
\]
\end{cor}

\begin{proof}
Let $\varepsilon \in (0,1)$ sufficiently small.
Proposition \ref{prop:XN2est2} implies that
\begin{align*}
&E\left[ \sup _{s\in [0,T]} \| X_t^{M,N,(2)}\| _{B_2^\alpha }^{4/3} \right] ^{3/4}\\
&\leq C E\left[ \| X^{M,N,(2)}_0\| _{B_2^{\alpha}}^{4/3}\right] ^{3/4} + K_{M,N} 2^{[(3/4) + \varepsilon ]N} E\left[ \| X^{M,N,(2)}_0\| _{L^2}^2 \right] ^{3/4}\\
&\quad + K_{M,N} 2^{(1+\varepsilon )N} E\left[ \| X^{M,N,(2)}_0\| _{L^2}^{2/3} \right] ^{3/4} + K_{M,N} 2^{[(9/8) + \varepsilon ]N}\\
&\leq C E\left[ \| X^{M,N,(2)}_0\| _{B_2^{\alpha}}^2\right] ^{1/2} + K_{M,N} 2^{[(3/4) + \varepsilon ]N} E\left[ \| X^{M,N,(2)}_0\| _{L^2}^2 \right] ^{3/4}\\
&\quad + K_{M,N} 2^{(1+\varepsilon )N} E\left[ \| X^{M,N,(2)}_0\| _{L^2}^2 \right] ^{1/4} + K_{M,N} 2^{[(9/8) + \varepsilon ]N}.
\end{align*}
Hence,  by Corollary \ref{cor:XN2est1} we obtain the assertion.
\end{proof}

\section{Tightness of the laws of $\{ X^{M,N}\}$}\label{sec:tight}

Now we prove the tightness of the laws of $\{ X^{M,N}; M,N\in {\mathbb N}\}$ in the space of continuous stochastic processes on suitable Besov spaces.
We choose $\sigma \in (3,\infty )$ which satisfies \eqref{eq:assm0}, and prove some uniform estimates in the approximation sequence.
The argument is similar to that in Section 4 of \cite{AK}.
However,  some estimates are improved from those in \cite{AK} (see \cite{Sei}).

Similarly to the proofs of Lemmas 4.2, 4.3, 4.4, and 4.6--4.9 in \cite{AK}, we obtain the following proposition.

\begin{prop}\label{prop:estPhiPsi}
For $p\in [1,2]$, $\varepsilon \in (0,1/16]$ and $\delta \in (0,1]$, the following estimates hold:
\begin{align*}
& \int _0^t \| X_s^{M,N,(2),<}\| _{B_{4-\varepsilon}^{1-\varepsilon}(\nu )} ^3 dt \leq \delta \left( \int _0^t \left\| P_{M,N} X^{M,N,(2)}_s \right\| _{L^4(\nu )} ^4 ds \right) ^{7/8} + Q \delta ^{-6} ,\\
&\left\| \Psi _t^{(2,M,N)} (P_{M,N} X^{M,N,(2)}) \right\| _{B_{p}^{\varepsilon}(\nu )} \\
&\leq \delta \left( \| P_{M,N} X^{M,N,(2)}_t\| _{L^4(\nu )}^4 + \| X^{M,N,(2)}_t\| _{B_2^{15/16}(\nu )}^2 \right) ^{7/8}  + Q\delta ^{-16/19},\\
&\left\| \Phi _t ^{(1,M,N)} (P_{M,N} X^{M,N,(2)}) \right\| _{L^{4/3}(\nu )} \\
&\leq \delta \left( \| P_{M,N} X^{M,N,(2)}_t\| _{L^4(\nu )}^4 + \| P_{M,N} X^{M,N,(2)}_t\| _{B_2^{15/16}(\nu )}^2 \right) ^{7/8} + \delta ^{-26/9}Q,\\
&\left\| \Phi _t ^{(2,M,N)} (P_{M,N} X^{M,N,(2)}) \right\| _{B_{p}^{-1/2-\varepsilon}(\nu )} \\
&\leq \delta \left( \| P_{M,N} X^{M,N,(2)}_t\| _{L^4(\nu )}^4 + \| P_{M,N} X^{M,N,(2)}_t\| _{B_2^{15/16}(\nu )}^2 \right) ^{7/8} + \delta ^{-16/19} t^{-1/4-\varepsilon}Q, \\
&\left\| \Phi _t ^{(3,M,N)} (P_{M,N}  X^{M,N,(2)}) \right\| _{B_{p}^{-(1+\varepsilon)/2}(\nu )} \leq \delta \| P_{M,N} X^{M,N,(2)}_t\| _{L^4(\nu )}^{7/2} + \delta ^{-4/3} Q, \\
&\left\| (P_{M,N}  X^{M,N,(2),\geqslant}_t) \mbox{\textcircled{\scriptsize$=$}} {\mathcal Z}^{(2,M,N)}_t \right\| _{B_{p}^{\varepsilon /8}(\nu )} \\
&\leq \delta \left( \left\| \nabla X^{M,N,(2),\geqslant}_t \right\| _{L^2(\nu )}^2 + \left\| P_{M,N} X^{M,N,(2)}_t \right\| _{L^4(\nu )} ^4 \right) ^{7/8} \\
&\quad \hspace{3cm}+ \delta \left\| X^{M,N,(2),<}_t \right\| _{L^{p+\varepsilon}(\nu )}^{7/4} + \delta \left\| X^{M,N,(2),\geqslant}_t \right\| _{B_{p+\varepsilon}^{1+\varepsilon }(\nu )}^{5/6} + \delta ^{-82/23}Q .
\end{align*}
\end{prop}

Some of the estimates in Proposition \ref{prop:estPhiPsi} are different from those in \cite{AK} by $\varepsilon$.
The reason is that the Wick polynomials ${\mathcal Z}^{(k,M,N)}$ of the Ornstein-Uhlenbeck processes are $B_\infty ^{-s}$-valued continuous processes for suitable $s\in (0,\infty )$ in the case of the torus, but not in the case of the whole space ${\mathbb R}^3$.
Hence, in the present paper we need to replace the $B_\infty ^{-s}$-norms of ${\mathcal Z}^{(k,M,N)}$ in \cite{AK} by the $B_p^{-s}$-norms of ${\mathcal Z}^{(k,M,N)}$ with sufficiently large $p\in [1,\infty )$ and weight $\nu$.

\begin{prop}\label{prop:estPsi1}
For $\gamma \in (0,1/8)$, $\eta \in [0,1)$, $p\in [1,2]$, $\varepsilon \in (0,1/16]$ and $\theta \in (\varepsilon /\gamma ,1/2]$, it holds that
\begin{align*}
&\left\| \Psi _t^{(1,M,N)} (P_{M,N} X^{M,N,(2)}) \mbox{\textcircled{\scriptsize$=$}} {\mathcal Z}^{(2,M,N)}_t \right\| _{B_p^{\varepsilon}(\nu )} \\
&\leq Q \int _0^t (t-s)^{-21/32} \left\| P_{M,N} X^{M,N,(2)}_s \right\| _{B_{p+\varepsilon}^{15/16}(\nu )} ds \\
&\quad + Q \left( \sup _{s\in [0,t]} \frac{s^\eta \left\| P_{M,N} X^{M,N,(2)}_t - P_{M,N} X^{M,N,(2)}_s \right\| _{L^{p}(\nu )} }{(t-s)^{\gamma }} \right)^{\theta}\\
&\quad \hspace{0.5cm} \times \left(\| P_{M,N} X^{M,N,(2)}_t \| _{L^{p+5\varepsilon}(\nu )}^{1-\theta} + \int _0^t s^{-\eta \theta} (t-s)^{\theta \gamma - 1 -\varepsilon } \left\| P_{M,N} X^{M,N,(2)}_s \right\| _{L^{p+5\varepsilon}(\nu )} ^{1-\theta} ds \right) \\
&\quad + Q.
\end{align*}
\end{prop}

\begin{proof}
Similarly to the proof of Lemma 4.3 in \cite{AK}, we have
\begin{align*}
&\left\| \Psi _t^{(1,M,N)} (P_{M,N} X^{M,N,(2)}) \right\| _{B_{p+(\varepsilon /2)}^{\varepsilon}(\nu )} \\
&\leq Q \int _0^t (t-s)^{-21/32} \left\| P_{M,N} X^{M,N,(2)}_s \right\| _{B_{p+\varepsilon}^{15/16}(\nu )} ds  +Q \\
&\quad + Q \int _0^t (t-s)^{- 1 -\varepsilon } \left\| P_{M,N} X^{M,N,(2)}_t - P_{M,N} X^{M,N,(2)}_s \right\| _{L^{p+\varepsilon}(\nu )} ds .
\end{align*}
Once we show
\begin{equation}\label{eq:propestPsi1}\begin{array}{l}
\displaystyle \int _0^t (t-s)^{- 1 -\varepsilon } \left\| P_{M,N} X^{M,N,(2)}_t - P_{M,N} X^{M,N,(2)}_s \right\| _{L^{p+\varepsilon}(\nu )} ds \\
\displaystyle \leq \left( \sup _{s\in [0,t]} \frac{s^\eta \left\| P_{M,N} X^{M,N,(2)}_t - P_{M,N} X^{M,N,(2)}_s \right\| _{L^p(\nu )} }{(t-s)^{\gamma }} \right)^{\theta}\\
\displaystyle \quad \times \left(\| P_{M,N} X^{M,N,(2)}_t \| _{L^{p+5\varepsilon}(\nu )}^{1-\theta} + \int _0^t s^{-\eta \theta} (t-s)^{\theta \gamma - 1 -\varepsilon } \left\| P_{M,N} X^{M,N,(2)}_s \right\| _{L^{p+5\varepsilon}(\nu )} ^{1-\theta} ds \right) ,
\end{array}\end{equation}
by following the proof of Lemma 4.4 in \cite{AK} we obtain the assertion.
Hence, it is sufficient to show the validity of (\ref{eq:propestPsi1}).
The interpolation inequality on weighted $L^p$-spaces implies
\begin{align*}
&\left\| P_{M,N} X^{M,N,(2)}_t - P_{M,N} X^{M,N,(2)}_s \right\| _{L^{p+\varepsilon}(\nu )} \\
&\leq \left\| P_{M,N} X^{M,N,(2)}_t - P_{M,N} X^{M,N,(2)}_s \right\| _{L^p(\nu )}^{1/2} \left\| P_{M,N} X^{M,N,(2)}_t - P_{M,N} X^{M,N,(2)}_s \right\| _{L^q(\nu )}^{1/2}
\end{align*}
where $q= p(p+\varepsilon )/(p-\varepsilon ) \leq p + 5\varepsilon$.
From this inequality we have
\begin{align*}
&\int _0^t (t-s)^{- 1 -\varepsilon } \left\| P_{M,N} X^{M,N,(2)}_t - P_{M,N} X^{M,N,(2)}_s \right\| _{L^{p+\varepsilon}(\nu )} ds \\
&\leq \int _0^t (t-s)^{- 1 -\varepsilon } \left\| P_{M,N} X^{M,N,(2)}_t - P_{M,N} X^{M,N,(2)}_s \right\| _{L^p(\nu )}^{\theta} \\
&\quad \hspace{2cm} \times \left( \left\| P_{M,N} X^{M,N,(2)}_t\right\| _{L^{p + 5\varepsilon}(\nu )} + \left\| P_{M,N} X^{M,N,(2)}_s \right\| _{L^{p + 5\varepsilon}(\nu )} \right) ^{1-\theta} ds \\
&\leq C \left( \sup _{s\in [0,t]} \frac{s^\eta \left\| P_{M,N} X^{M,N,(2)}_t - P_{M,N} X^{M,N,(2)}_s \right\| _{L^p(\nu )} }{(t-s)^{\gamma }} \right)^{\theta}\\
& \quad \times \left(\| P_{M,N} X^{M,N,(2)}_t \| _{L^{p+5\varepsilon}(\nu )}^{1-\theta} + \int _0^t s^{-\eta \theta} (t-s)^{\theta \gamma - 1 -\varepsilon } \left\| P_{M,N} X^{M,N,(2)}_s \right\| _{L^{p+5\varepsilon}(\nu )} ^{1-\theta} ds \right) .
\end{align*}
Thus, we obtain (\ref{eq:propestPsi1}).
\end{proof}

\begin{rem}\label{rem:theta}
In the proof of Proposition 4.13 in \cite{AK}, we apply the proposition associated to Proposition \ref{prop:estPsi1} with $\theta =1$.
However, since for $p\in [1,2]$
\begin{align*}
&\int _0^t Q \left( \sup _{s\in [0,u]} \frac{s^\eta \left\| P_{M,N} X^{M,N,(2)}_u - P_{M,N} X^{M,N,(2)}_s \right\| _{L^{p}(\nu )} }{(u-s)^{\gamma }} \right)^{1/2}\\
&\quad \times \left(\| P_{M,N} X^{M,N,(2)}_u \| _{L^{p+5\varepsilon}(\nu )}^{1/2} + \int _0^u s^{-\eta /2} (u-s)^{(\gamma /2) - 1 -\varepsilon } \left\| P_{M,N} X^{M,N,(2)}_s \right\| _{L^{p+5\varepsilon}(\nu )} ^{1/2} ds \right) du \\
&\leq Q \left( \sup _{s',t' \in [0,t]; s'<t'} \frac{(s')^\eta \left\| P_{M,N} X^{M,N,(2)}_{t'} - P_{M,N} X^{M,N,(2)}_{s'} \right\| _{L^{p}(\nu )} }{(t'-s')^{\gamma }} \right) ^{1/2}\\
&\quad \times \left(\int _0^t \| P_{M,N} X^{M,N,(2)}_u \| _{L^4(\nu )}^{1/2} du + \int _0^t s^{-\eta /2} \left\| P_{M,N} X^{M,N,(2)}_s \right\| _{L^4(\nu )} ^{1/2} ds \right) \\
&\leq \delta \sup _{s',t' \in [0,t]; s'<t'} \frac{(s')^\eta \left\| P_{M,N} X^{M,N,(2)}_{t'} - P_{M,N} X^{M,N,(2)}_{s'} \right\| _{L^{p}(\nu )} }{(t'-s')^{\gamma }} \\
&\quad + \delta \int _0^t \| P_{M,N} X^{M,N,(2)}_u \| _{L^4(\nu )}^4 du + C_\delta Q,
\end{align*}
a similarly argument as the one used in \cite{AK} works in the proof of Proposition \ref{prop:global2-3} below.
\end{rem}

\begin{prop}\label{prop:estP*}
For $M,N\in {\mathbb N}$, $t\in [0,T]$ and $\varepsilon \in (0,1/16]$,
\begin{equation}\label{prop:estP*1}\begin{array}{l}
\displaystyle \left| \int _{{\mathbb R}^3} X_t^{M,N,(2), \geqslant} P_{M,N}^* \left[ \left( P_{M,N} X_t^{M,N,(2)} - \lambda P_{M,N} {\mathcal Z}^{(0,3,M,N)}_t \right) \mbox{\textcircled{\scriptsize$<$}} {\mathcal Z}^{(2,M,N)}_t \right] \nu dx \right| \\
\displaystyle \leq \delta \left( \| \nabla X_t^{M,N,(2),\geqslant}\| _{L^2(\nu )}^2 + \| X_t^{M,N,(2)}\| _{L^2(\nu )}^2 + \| P_{M,N} X_t^{M,N,(2)}\| _{L^4(\nu )}^4 \right) \\
\displaystyle \quad + \delta \left( \left\| X^{M,N,(2),<}_t \right\| _{B_{4-\varepsilon}^{1-\varepsilon }(\nu )}^3 + \left\| X^{M,N,(2),\geqslant}_t \right\| _{B_{(4/3)+\varepsilon}^{1+\varepsilon }(\nu )} \right) + C_\delta K_{M,N} Q
\end{array}\end{equation}
and the terms
\begin{align}
&\label{prop:estP*2} \left| \int _{{\mathbb R}^3} X_t^{M,N,(2)} P_{M,N}^* \Phi _t^{(1,M,N)} (P_{M,N} X^{M,N,(2)}) \nu dx \right| \\
&\label{prop:estP*3} \left| \int _{{\mathbb R}^3} X_t^{M,N,(2)} P_{M,N}^* \Phi _t^{(3,M,N)} (P_{M,N} X^{M,N,(2)}) \nu dx \right| \\
&\label{prop:estP*4} \left| \int _{{\mathbb R}^3} X_t^{M,N,(2)} P_{M,N}^* \left[ (P_{M,N} X_t^{M,N,(2), \geqslant}) \mbox{\textcircled{\scriptsize$=$}} {\mathcal Z}^{(2,M,N)}_t \right] \nu dx \right| \\
&\label{prop:estP*6} \left| \int _{{\mathbb R}^3} X_t^{M,N,(2)} P_{M,N}^* \Psi _t^{(2,M,N)} (P_{M,N} X^{M,N,(2)}) \nu dx \right|
\end{align}
are dominated by
\begin{align*}
&\delta \left( \left\| \nabla X_t^{M,N,(2),\geqslant} \right\| _{L^2(\nu )}^2 + \left\| X_t^{M,N,(2)} \right\| _{L^2(\nu )}^2 + \left\| P_{M,N} X_t^{M,N,(2)}\right\| _{L^4(\nu )}^4 \right) \\
&+ C M^{4 \sigma} 2^{-4N} \left\| X_t^{M,N,(2)}\right\| _{L^4(\nu )}^4 + C_\delta K_{M,N} Q .
\end{align*}
Moreover,
\begin{equation} \label{prop:estP*5}\begin{array}{l} 
\displaystyle \left| \int _{{\mathbb R}^3} X_t^{M,N,(2)} P_{M,N}^* \left[ \Psi _t^{(1,M,N)} (P_{M,N} X^{M,N,(2)}) \mbox{\textcircled{\scriptsize$=$}} {\mathcal Z}^{(2,M,N)}_t \right] \nu dx \right| \\
\displaystyle \leq \delta \sup _{s',t' \in [0,t]; s'<t'} \frac{(s')^\eta \left\| P_{M,N} X^{M,N,(2)}_{t'} - P_{M,N} X^{M,N,(2)}_{s'} \right\| _{L^{p}(\nu )} }{(t'-s')^{\gamma }} \\
\displaystyle \quad + \delta \int _0^t \left\| P_{M,N} X^{M,N,(2)}_s \right\| _{L^4(\nu )}^4 ds + C_\delta K_{M,N} Q,
\end{array}
\end{equation}
and
\begin{equation}\label{prop:estP*7}\begin{array}{l} 
\displaystyle \left| \int _{{\mathbb R}^3} X_t^{M,N,(2)} P_{M,N}^* \Phi _t^{(2,M,N)} (P_{M,N} X^{M,N,(2)}) \nu dx \right| \\
\displaystyle \leq \delta \left( \left\| P_{M,N} X^{M,N,(2)}_t \right\| _{L^4(\nu )}^4 + \left\| P_{M,N} X^{M,N,(2)}_t \right\| _{B_2^{15/16}(\nu )}^2 \right) ^{7/8} \\[2mm]
\displaystyle \quad \hspace{4cm} + C_\delta t^{-1/4-\varepsilon} K_{M,N} Q .
\end{array}
\end{equation}
\end{prop}

\begin{proof}
The duality between $B_{(4/3)+\varepsilon}^{1+\varepsilon /8}(\nu )$ and $B_{(4+3\varepsilon)/(1+3\varepsilon)}^{-1-\varepsilon /8}(\nu )$, Lemma \ref{lem:comP1} and Proposition \ref{prop:paraproduct} imply
\begin{align*}
&\left| \int _{{\mathbb R}^3} X_t^{M,N,(2), \geqslant} P_{M,N}^* \left[ \left( P_{M,N} X_t^{M,N,(2)} - \lambda P_{M,N} {\mathcal Z}^{(0,3,M,N)}_t \right) \mbox{\textcircled{\scriptsize$<$}} {\mathcal Z}^{(2,M,N)}_t \right] \nu dx \right| \\
&\leq \left\| X_t^{M,N,(2), \geqslant} \right\| _{B_{(4/3)+\varepsilon}^{1+\varepsilon /8}(\nu )} \\
&\quad \times \left\| P_{M,N}^* \left[ \left( P_{M,N} X_t^{M,N,(2)} - \lambda P_{M,N} {\mathcal Z}^{(0,3,M,N)}_t \right) \mbox{\textcircled{\scriptsize$<$}} {\mathcal Z}^{(2,M,N)}_t \right] \right\| _{B_{(4+3\varepsilon)/(1+3\varepsilon)}^{-1-\varepsilon /8}(\nu )} \\
&\leq K_{M,N} Q \left\| X_t^{M,N,(2), \geqslant} \right\| _{B_{(4/3)+\varepsilon}^{1+\varepsilon /8}(\nu )} \left\| P_{M,N} X_t^{M,N,(2)} - \lambda P_{M,N} {\mathcal Z}^{(0,3,M,N)}_t \right\| _{L^4(\nu )} \\
&\leq K_{M,N} Q \left\| X_t^{M,N,(2), \geqslant} \right\| _{B_{(4/3)+\varepsilon}^{1+\varepsilon /8}(\nu )}^{3/2} + \left\| P_{M,N} X_t^{M,N,(2)} \right\| _{L^4(\nu )}^3 + Q.
\end{align*}
Similarly to the proof of Lemma 4.10 in \cite{AK}, by Proposition \ref{prop:Besov} we have
\begin{align*}
&K_{M,N} Q \left\| X_t^{M,N,(2), \geqslant} \right\| _{B_{(4/3)+\varepsilon}^{1+\varepsilon /8}(\nu )}^{3/2} \\
&\leq C_\delta K_{M,N} Q \left\| X_t^{M,N,(2), \geqslant} \right\| _{B_{4/3}^{1+\varepsilon }(\nu )}^{1/4} \left\| X_t^{M,N,(2), \geqslant} \right\| _{B_{20(4+3\varepsilon)/(60-9\varepsilon )}^{1-\varepsilon /20}(\nu )}^{5/4}\\
&\leq \delta  \left\| X_t^{M,N,(2), \geqslant} \right\| _{B_{4/3}^{1+\varepsilon }(\nu )} + C_\delta K_{M,N} Q \left\| X_t^{M,N,(2), \geqslant} \right\| _{B_2^{1-\varepsilon /20}(\nu )}^{5/3}\\
&\leq \delta  \left\| X_t^{M,N,(2), \geqslant} \right\| _{B_{4/3}^{1+\varepsilon }(\nu )} + \delta \left\| X_t^{M,N,(2), \geqslant} \right\| _{B_2^{1-\varepsilon /20}(\nu )}^2 + C_\delta K_{M,N} Q
\end{align*}
for $\delta \in (0,1]$.
Moreover,
\[
\left\| P_{M,N} X_t^{M,N,(2)} \right\| _{L^{4}(\nu )}^3 \leq \delta \left\| P_{M,N} X_t^{M,N,(2)} \right\| _{L^4(\nu )}^4 + C_\delta
\]
for $\delta \in (0,1]$.
From these inequalities we obtain the estimate for (\ref{prop:estP*1}).

Next we prove the estimate for (\ref{prop:estP*2}).
The duality between $P_{M,N}$ and $P_{M,N}^*$ on $L^2(dx)$ implies
\begin{align*}
&\left| \int _{{\mathbb R}^3} X_t^{M,N,(2)} P_{M,N}^* \Phi _t^{(1,M,N)} (P_{M,N} X^{M,N,(2)}) \nu dx \right| \\
&\leq \left| \int _{{\mathbb R}^3} \left( P_{M,N} X_t^{M,N,(2)} \right) \Phi _t^{(1,M,N)} (P_{M,N} X^{M,N,(2)}) \nu dx \right| \\
&\quad + \left| \int _{{\mathbb R}^3} \left( P_{M,N} X_t^{M,N,(2)} - \frac{1}{\nu} P_{M,N} (\nu X_t^{M,N,(2)} ) \right) \Phi _t^{(1,M,N)} (P_{M,N} X^{M,N,(2)}) \nu dx \right|
\end{align*}
The first term on the right-hand side is estimated, similarly to the proof of Lemma 4.10 in \cite{AK}, by Proposition \ref{prop:estPhiPsi}.
To estimate the second term on the right-hand side, applying H\"older's inequality and Lemma \ref{lem:comP2}, we dominate it by 
\begin{align*}
&\left\| P_{M,N} X_t^{M,N,(2)} - \frac{1}{\nu} P_{M,N} (\nu X_t^{M,N,(2)} ) \right\| _{L^4(\nu )} \left\| \Phi _t^{(1,M,N)} (P_{M,N} X^{M,N,(2)}) \right\| _{L^{4/3}(\nu)} \\
&\leq \left( C M^\sigma 2^{-N} \left\| X_t^{M,N,(2)} \right\| _{L^2(\nu )} \right) \\
&\quad \times C \left( \left\| \left( {\mathcal Z}_t^{(1,M,N)} - \lambda P_{M,N} {\mathcal Z}^{(0,3,M,N)}_t\right) \mbox{\textcircled{\scriptsize$\leqslant$}} (P_{M,N} X^{M,N,(2)})^2\right\| _{L^{4/3}(\nu)} \right. \\
&\quad \hspace{1cm} \left. + \left\| \left[ \left( 2{\mathcal Z}_t^{(1,M,N)} - \lambda P_{M,N} {\mathcal Z}^{(0,3,M,N)}_t \right) P_{M,N} {\mathcal Z}^{(0,3,M,N)}_t \right] \mbox{\textcircled{\scriptsize$\leqslant$}} P_{M,N} X^{M,N,(2)} \right\| _{L^{4/3}(\nu)} \right) \\
&\leq C M^\sigma 2^{-N} Q \left\| X_t^{M,N,(2)} \right\| _{L^2(\nu )} \left( \left\|(P_{M,N} X^{M,N,(2)})^2\right\| _{B_{4/3 +\varepsilon}^{(1/2)+\varepsilon}(\nu)} + \left\| P_{M,N} X^{M,N,(2)} \right\| _{B_{4/3 +\varepsilon}^{(1/2)+\varepsilon}(\nu)} \right) .
\end{align*}
Hence, Lemma \ref{lem:Lpestimates1} with the replacement of $\delta$ by $\delta \left( 1+ C M^\sigma 2^{-N} Q \| X_t^{M,N,(2)} \| _{L^2(\nu )} \right) ^{-1}$ yields
\begin{align*}
&\left| \int _{{\mathbb R}^3} \left( P_{M,N} X_t^{M,N,(2)} - \frac{1}{\nu} P_{M,N} (\nu X_t^{M,N,(2)} ) \right) \Phi _t^{(1,M,N)} (P_{M,N} X^{M,N,(2)}) \nu dx \right| \\
&\leq \delta \left( \left\| P_{M,N} X^{M,N,(2)} \right\| _{L^4 (\nu )}^4 + \left\| P_{M,N} X^{M,N,(2)} \right\| _{B_2^{15/16} (\nu )}^2\right) ^{7/8} \\
&\quad + C_\delta Q \left( 1+ C M^\sigma 2^{-N} \left\| X_t^{M,N,(2)} \right\| _{L^2(\nu )} \right) ^3\\
&\leq \delta \left( \| \nabla X_t^{M,N,(2),\geqslant}\| _{L^2(\nu )}^2 + \| X_t^{M,N,(2)}\| _{L^2(\nu )}^2 + \| P_{M,N} X_t^{M,N,(2)}\| _{L^4(\nu )}^4 \right) \\
&\quad + C M^{4 \sigma} 2^{-4N} \left\| X^{M,N,(2)}_t \right\| _{L^4(\nu )}^4 + C_\delta K_{M,N} Q
\end{align*}
for $\delta \in (0,1]$.
Thus we obtain the estimate for (\ref{prop:estP*2}).
Proofs of the estimates for (\ref{prop:estP*3}) -- (\ref{prop:estP*7}) are obtained in a similar way (see the proof of \cite[Proposition 4.11]{AK} for (\ref{prop:estP*5})).
So, we omit details.
\end{proof}

\begin{prop}\label{prop:eng}
For $\gamma \in (0,1/8)$, $\varepsilon \in (0,\gamma /4)$, $\eta \in [0,1)$ and $t\in [0,T]$
\begin{align*}
&\left\| X_t^{M,N,(2)}\right\| _{L^2(\nu )} ^2 - \left\| X_0^{M,N,(2)}\right\| _{L^2(\nu )} ^2 \\
&+ \int _0^t \left( m_0^2 \left\| X_s^{M,N,(2)}\right\| _{L^2(\nu )}^2+ \left\| \nabla X_s^{M,N,(2),\geqslant}\right\| _{L^2(\nu )}^2 + \lambda \left\| P_{M,N} X_s^{M,N,(2)}\right\| _{L^4(\nu )}^4 \right) ds\\
&\leq C \left\| X_t^{M,N,(2),<} \right\| _{L^2(\nu )}^2 + C {\mathfrak Y}_{\varepsilon}^{M,N} (t) + C M^{4 \sigma} 2^{-4N} \int _0^t \left\| X_s^{M,N,(2)}\right\| _{L^4}^4 ds + K_{M,N} Q \\
&\quad + Q\left( \sup _{s',t' \in [0,t]; s'<t'} \frac{(s')^\eta \left\| P_{M,N} X^{M,N,(2)}_{t'} - P_{M,N} X^{M,N,(2)}_{s'} \right\| _{L^{4/3}(\nu )}}{(t'-s')^{\gamma }} \right) ^{4/5} .
\end{align*}
\end{prop}

\begin{proof}
The duality between $P_{M,N}$ and $P_{M,N}^*$ on $L^2(dx)$ implies
\begin{align*}
&\int _0^t \int _{{\mathbb R}^3} X_s^{M,N,(2)} P_{M,N}^* \left[ \left( P_{M,N} X_s^{M,N,(2)} \right) ^3 \right] \nu dx ds \\
&=\int _0^t \int _{{\mathbb R}^3} |P_{M,N}  X_s^{M,N,(2)}|^4 \nu dx ds \\
&\quad - \int _0^t \int _{{\mathbb R}^3} \left[ P_{M,N}  X_s^{M,N,(2)} - \frac{1}{\nu} P_{M,N} (\nu X_s^{M,N,(2)}) \right] (P_{M,N}  X_s^{N,(2)} )^3 \nu dx ds .
\end{align*}
On the other hand, by Lemma \ref{lem:comP2}
\begin{align*}
&\left| \int _0^t \int _{{\mathbb R}^3} \left[ P_{M,N}  X_s^{M,N,(2)} - \frac{1}{\nu} P_{M,N} (\nu X_s^{M,N,(2)}) \right] (P_{M,N}  X_s^{M,N,(2)} )^3 \nu dx ds \right| \\
&\leq C M^\sigma 2^{-N} \int _0^t \left\| X_s^{M,N,(2)} \right\| _{L^4(\nu )} \left\| P_{M,N}  X_s^{M,N,(2)}\right\| _{L^4(\nu )}^3 ds .
\end{align*}
Hence, by H\"older's inequality we have
\begin{equation}\label{eq:propeng01}\begin{array}{l}
\displaystyle - \int _0^t \int _{{\mathbb R}^3} X_s^{M,N,(2)} P_{M,N}^* \left[ \left( P_{M,N} X_s^{M,N,(2)} \right) ^3 \right] \nu dx ds \\
\displaystyle \leq - \frac{1}{2} \int _0^t \left\| P_{M,N}  X_s^{M,N,(2)}\right\| _{L^4(\nu )}^4 ds + C M^{4 \sigma} 2^{-4N} \int _0^t \left\| X_s^{M,N,(2)} \right\| _{L^4(\nu )}^4 ds .
\end{array}\end{equation}
Since the integration by parts formula and H\"older's inequality imply
\begin{align*}
&\int _0^t \int _{{\mathbb R}^3} X_s^{M,N,(2),\geqslant} \left[ \left( \triangle -m_0^2 \right) X_s^{M,N,(2),\geqslant} \right] \nu dx ds \\
&\leq - \int _0^t \int _{{\mathbb R}^3} | \nabla  X_s^{M,N,(2),\geqslant} |^2 \nu dx ds - m_0^2 \int _0^t \int _{{\mathbb R}^3} \left( X_s^{M,N,(2),\geqslant} \right) ^2 \nu dx ds\\
&\quad - \int _0^t \int _{{\mathbb R}^3} X_s^{M,N,(2),\geqslant} \nabla  X_s^{M,N,(2),\geqslant} \cdot (\nabla \log \nu ) \nu dx ds \\
&\leq - \left( 1 - \frac{\sigma ^2}{2m_0^2} \right) \int _0^t \int _{{\mathbb R}^3} | \nabla  X_s^{M,N,(2),\geqslant} |^2 \nu dx ds - \frac{m_0^2}2 \int _0^t \int _{{\mathbb R}^3} \left( X_s^{M,N,(2),\geqslant} \right) ^2 \nu dx ds,
\end{align*}
from (\ref{eq:assm0}) we have
\begin{equation}\label{eq:propeng02}\begin{array}{l}
\displaystyle \int _0^t \int _{{\mathbb R}^3} X_s^{M,N,(2),\geqslant} \left[ \left( \triangle -m_0^2 \right) X_s^{M,N,(2),\geqslant} \right] \nu dx ds \\
\displaystyle \leq - c \int _0^t \int _{{\mathbb R}^3} \left[ | \nabla  X_s^{M,N,(2),\geqslant} |^2 + \left( X_s^{M,N,(2),\geqslant} \right) ^2 \right] \nu dx ds
\end{array}\end{equation}
where $c$ is a positive constant independent of $M$ and $N$.

In view of Proposition 7.2, (\ref{eq:propeng01}) and (\ref{eq:propeng02}), by following the proof of Proposition 4.11 in \cite{AK} with respect to $L^2(\nu )$ instead of the $L^2$-space on the torus, we obtain the assertion.
\end{proof}

\begin{prop}\label{prop:global2-3}
For $\alpha \in [0,1/2)$, $\gamma \in (0,1/8)$, $\eta \in [0,1)$, $\varepsilon \in (0,\gamma /4)$ and $t\in [0,T]$, 
\begin{align*}
&\sup _{s',t'\in [0,t]; s'<t'} \frac{(s')^\eta \left\| X^{M,N,(2)}_{t'} - X^{M,N,(2)}_{s'} \right\| _{B_{4/3}^\alpha (\nu )}}{(t'-s')^{\gamma }} \\
&\leq K_{M,N} \sup _{r\in [0,t]} \left( r^\eta \left\| X^{M,N,(2),\geqslant}_{r} \right\| _{B_{4/3}^{\alpha + 2\gamma }(\nu )} \right) + K_{M,N} \sup _{r\in [0,t]} \left( r^\eta \left\| X^{M,N,(2),<}_r \right\| _{B_{4/3}^{\alpha + 2\gamma }(\nu )} \right) \\
&\quad + K_{M,N} \int _0^t \left( \left\| P_{M,N} X^{M,N,(2)}_s\right\| _{L^4(\nu )}^4 + \left\| P_{M,N} X^{M,N,(2)}_s\right\| _{B_2^{15/16}(\nu )}^2 \right) ^{7/8} ds \\
&\quad + K_{M,N} {\mathfrak Y}_{\varepsilon}^N (t) + K_{M,N} Q .
\end{align*}
\end{prop}

\begin{proof}
The proof is obtained in the same way as the proof of Proposition 3.1 in \cite{Sei}, which is an improved version of Proposition 4.13 in \cite{AK}.
See also Remark \ref{rem:theta}.
\end{proof}

\begin{prop}\label{prop:global2-1}
For $\gamma \in (0,1/8)$, $\eta \in [0,1)$, $\varepsilon \in (0,\gamma /4)$, $q\in (1, 8/7) $, $t\in [0,T]$ and $\delta \in (0,1]$, 
\[
{\mathfrak Y}_{\varepsilon}^{M,N} (t) \leq C \left\| X^{M,N,(2)}_0 \right\| _{B_{4/3}^{-1+2\gamma + 3\varepsilon }(\nu )} + \delta K_{M,N} {\mathfrak X}_{\lambda ,\alpha ,\eta ,\gamma}^{M,N} (t)^{7/8} + C_\delta K_{M,N} Q .
\]
\end{prop}

\begin{proof}
The proof is almost the same as that of Proposition 4.15 of \cite{AK}.
Indeed, by following the way that we obtained (4.17) in the proof of Proposition 4.15 of \cite{AK}, and by taking $\delta$ sufficiently small we obtain the estimate.
\end{proof}

\begin{prop}\label{prop:estsup}
For $\gamma \in (0,1/8)$, $\alpha \in [0,(1-4\gamma )/2)$, $\eta \in [0,1)$, $\varepsilon \in (0,\gamma /4)$, $q\in (1, 8/7) $, $t\in [0,T]$ and $\delta \in (0,1]$, we have
\begin{align*}
&\sup _{r\in [0,t]} \left( r^\eta \left\| X^{M,N,(2),<}_{r} \right\| _{B_{4}^{\alpha + 2\gamma }(\nu )}^3 \right) + \sup _{r\in [0,t]} \left( r^\eta \left\| X^{M,N,(2),\geqslant}_{r} \right\| _{B_{4/3}^{\alpha + 2\gamma }(\nu )} \right) \\
&\leq C \left\| X^{M,N,(2)}_0 \right\| _{B_{4/3}^{\alpha + 2(\gamma - \eta )}(\nu )} + \delta K_{M,N} {\mathfrak X}_{\lambda ,\alpha , \eta ,\gamma}^{M,N} (t) + \delta K_{M,N} {\mathfrak Y}_{\varepsilon}^{M,N} (t) ^q + C_\delta K_{M,N} Q.
\end{align*}
\end{prop}

\begin{proof}
The assertion is proved by following the proof of Proposition 3.2 in \cite{Sei}  without taking the expectation, which is an improved version of Proposition 4.17 of \cite{AK}.
\end{proof}

\begin{prop}\label{prop:estall}
For $\alpha \in [0,1/2)$, $\gamma \in (0,1/8)$, $\eta \in [0,1)$, $\varepsilon \in (0, \gamma /4)$ and $q\in (1, 8/7)$, we have
\begin{align*}
&\left\| X_T^{M,N,(2)}\right\| _{L^2(\nu )} ^2 + {\mathfrak X}_{\lambda , \alpha , \eta ,\gamma}^{M,N} (T) + {\mathfrak Y}_{\varepsilon}^{M,N} (T) ^q \\
& + \sup _{r\in [0,T]} \left( r^\eta \left\| X^{M,N,(2),<}_{r} \right\| _{B_{4}^{\alpha + 2\gamma }(\nu )}^3 \right) + \sup _{r\in [0,T]} \left( r^\eta \left\| X^{M,N,(2),\geqslant}_{r} \right\| _{B_{4/3}^{\alpha + 2\gamma }(\nu )} \right) \\
& \leq \left\| X_0^{M,N,(2)}\right\| _{L^2(\nu )} ^2 + C \left\| X^{M,N,(2)}_0 \right\| _{B_{4/3}^{-1+2\gamma + 3\varepsilon }(\nu )}^q +C \left\| X^{M,N,(2)}_0 \right\| _{B_{4/3}^{\alpha + 2(\gamma - \eta )}(\nu )} \\
&\quad + C M^{4 \sigma} 2^{-4N} \int _0^T \left\| X_s^{M,N,(2)} \right\| _{L^4(\nu )}^4 ds + K_{M,N} Q .
\end{align*}
\end{prop}

\begin{proof}
Propositions \ref{prop:eng}, \ref{prop:global2-3}, \ref{prop:global2-1} and \ref{prop:estsup}, and H\"older's inequality imply that for $\delta \in (0,1]$
\begin{align*}
&\| X_T^{M,N,(2)}\| _{L^2(\nu )} ^2 + {\mathfrak X}_{\lambda , \alpha , \eta ,\gamma}^{M,N} (T) + {\mathfrak Y}_{\varepsilon}^{M,N} (T) ^q \\
& + \sup _{r\in [0,T]} \left( r^\eta \left\| X^{M,N,(2),<}_{r} \right\| _{B_{4}^{\alpha + 2\gamma }(\nu )}^3 \right) + \sup _{r\in [0,T]} \left( r^\eta \left\| X^{M,N,(2),\geqslant}_{r} \right\| _{B_{4/3}^{\alpha + 2\gamma }(\nu )} \right) \\
&\leq \delta \sup _{r\in [0,T]} \left( r^\eta \left\| X^{M,N,(2),\geqslant}_{r} \right\| _{B_{4/3}^{\alpha + 2\gamma }(\nu )} \right) + \delta \sup _{r\in [0,T]} \left( r^\eta \left\| X^{M,N,(2),<}_r \right\| _{B_{4/3}^{\alpha + 2\gamma }(\nu )} \right)\\
&\quad  + \delta \| X_T^{M,N,(2),<}\| _{L^2(\nu )}^3 + \delta K_{M,N} {\mathfrak X}_{\lambda ,\alpha ,\eta ,\gamma}^{M,N} (T) + \delta K_{M,N} {\mathfrak Y}_{\varepsilon}^{M,N} (T) ^q \\
&\quad + \left\| X_0^{M,N,(2)}\right\| _{L^2(\nu )} ^2 + C \left\| X^{M,N,(2)}_0 \right\| _{B_{4/3}^{-1+2\gamma + 3\varepsilon }(\nu )}^q  + C \left\| X^{M,N,(2)}_0 \right\| _{B_{4/3}^{\alpha + 2(\gamma - \eta )}(\nu )}\\
&\quad + C M^{4 \sigma} 2^{-4N} \int _0^T \left\| X_s^{M,N,(2)} \right\| _{L^4(\nu )}^4 ds + C_\delta K_{M,N} Q.
\end{align*}
Hence, by applying the fact that 
\[
\| X_T^{M,N,(2),<}\| _{L^2(\nu )}^3 \leq T^{-\eta} \sup _{r\in [0,T]} \left( r^\eta \left\| X^{M,N,(2),<}_{r} \right\| _{B_{4}^{2\gamma }(\nu )}^3 \right) ,
\]
and taking $\delta$ sufficiently small we have the assertion.
\end{proof}

Now we prove the uniform estimate, which is applied to prove the tightness.

\begin{thm}\label{thm:tight1}
Let $\alpha \in [0,1/2)$, $\gamma \in (0,1/8)$, $\eta \in [0,1)$, $\varepsilon \in (0, \gamma /4)$ and $q\in (1, 8/7)$.
Assume that \eqref{eq:assm0} and
\[
\alpha + 2\gamma < \min \left\{ 2\eta + \frac{1}{2}, \frac{1-\varepsilon }{2}\right\}.
\]
Then, 
\begin{align*}
&E\left[ {\mathfrak X}_{\lambda , \alpha , \eta ,\gamma}^{M,N} (T)^{1/3}\right] + E\left[ {\mathfrak Y}_{\varepsilon}^{M,N} (T) ^{q/3} \right] \\
&+ E\left[ \sup _{r\in [0,T]} \left( r^\eta \left\| X^{M,N,(2),<}_{r} \right\| _{B_4^{\alpha + 2\gamma }}^3 \right) ^{1/3} \right] + E\left[ \sup _{r\in [0,T]} \left( r^\eta \left\| X^{M,N,(2),\geqslant}_{r} \right\| _{B_{4/3}^{\alpha + 2\gamma }} \right) ^{1/3} \right] \\
&\leq K_{M,N} .
\end{align*}
\end{thm}

\begin{proof}
It is sufficient to show the estimate for some large $T\in (0,\infty )$.
Proposition \ref{prop:estall}, subadditivity of $x\mapsto x^{1/3}$ for $x\geq 0$ and the nonnegativity of each term imply
\begin{align*}
& {\mathfrak X}_{\lambda , \alpha , \eta ,\gamma}^{M,N} (T)^{1/3} + {\mathfrak Y}_{\varepsilon}^{M,N} (T) ^{q/3} \\
& + \sup _{r\in [0,T]} \left( r^\eta \left\| X^{M,N,(2),<}_{r} \right\| _{B_{4}^{\alpha + 2\gamma }(\nu )}^3 \right) ^{1/3} + \sup _{r\in [0,T]} \left( r^\eta \left\| X^{M,N,(2),\geqslant}_{r} \right\| _{B_{4/3}^{\alpha + 2\gamma }(\nu )} \right) ^{1/3} \\
& \leq 4 \left\| X_0^{M,N,(2)}\right\| _{L^2(\nu )} ^{2/3} + C \left\| X^{M,N,(2)}_0 \right\| _{B_{4/3}^{-1+2\gamma + 3\varepsilon }(\nu )}^{q/3} +C \left\| X^{M,N,(2)}_0 \right\| _{B_{4/3}^{\alpha + 2(\gamma - \eta )}(\nu )}^{1/3} \\
&\quad + C M^{4 \sigma /3} \left( 2^{-4N} \int _0^T \left\| X_s^{M,N,(2)} \right\| _{L^4(\nu )}^4 ds \right) ^{1/3} + K_{M,N} Q.
\end{align*}
But the Besov embedding theorem, Corollaries \ref{cor:XN2est1} and \ref{cor:XN2est2} yield
\begin{align*}
&M^{4 \sigma /3} E\left[ \left( 2^{-4N} \int _0^T \left\| X_s^{M,N,(2)} \right\| _{L^4(\nu )}^4 ds \right) ^{1/3} \right] \\
&\leq C M^{4 \sigma /3} 2^{-(4/3)N} E\left[ \left( \int _0^T \left\| X_s^{M,N,(2)} \right\| _{B_{2}^{3/4}}^4 ds \right) ^{1/3}\right] \\
&\leq C M^{4 \sigma /3} 2^{-(4/3)N} E\left[ \left( \int _0^T \left\| X_s^{M,N,(2)} \right\| _{B_{2}^{1/2}}^2 \left\| X_s^{M,N,(2)} \right\| _{B_{2}^{1}}^2 ds \right) ^{1/3}\right] \\
&\leq C M^{4 \sigma /3} 2^{-(4/3)N} E\left[ \left( \sup _{s\in [0,T]} \left\| X_s^{M,N,(2)} \right\| _{B_{2}^{1/2}}^{4/3} \right) ^{1/2} \left( \int _0^t \left\| X_s^{M,N,(2)} \right\| _{B_{2}^{1}}^2 ds \right) ^{1/3}\right] \\
&\leq C M^{4 \sigma /3} 2^{-(4/3)N} E\left[ \sup _{s\in [0,T]} \left\| X_s^{M,N,(2)} \right\| _{B_{2}^{1/2}}^{4/3} \right] ^{1/2} \left( 1 + E \left[ \int _0^t \left\| X_s^{M,N,(2)} \right\| _{B_{2}^{1}}^2 ds \right] ^{2/3} \right) \\
&\leq K_{M,N},
\end{align*}
so that we obtain
\begin{equation}\label{eq:thmtight1-1}\begin{array}{l}
\displaystyle E\left[ {\mathfrak X}_{\lambda ,\alpha , \eta ,\gamma}^{M,N} (T)^{1/3} \right] + E\left[ {\mathfrak Y}_{\varepsilon}^{M,N} (T) ^{q/3}\right] \\
\displaystyle + E\left[ \sup _{r\in [0,T]} \left( r^\eta \left\| X^{M,N,(2),<}_{r} \right\| _{B_{4}^{\alpha + 2\gamma }(\nu )}^3 \right) ^{1/3} \right] \\
\displaystyle + E\left[ \sup _{r\in [0,T]} \left( r^\eta \left\| X^{M,N,(2),\geqslant}_{r} \right\| _{B_{4/3}^{\alpha + 2\gamma }(\nu )} \right) ^{1/3}\right] \\
\displaystyle \leq 4 E\left[ \left\| X_0^{M,N,(2)}\right\| _{L^2(\nu )} ^{2/3} \right] + C E\left[ \left\| X^{M,N,(2)}_0 \right\| _{B_{4/3}^{-1+2\gamma + 3\varepsilon }(\nu )}^{q/3} \right] \\
\displaystyle \quad + C E\left[ \left\| X^{M,N,(2)}_0 \right\| _{B_{4/3}^{\alpha + 2(\gamma - \eta )}(\nu )}^{1/3}\right] + K_{M,N}.
\end{array}\end{equation}
From the stationarity of the marginal laws of $X^{M,N,(1)}$ we have, for any $T>0$:
\begin{align*}
& E\left[ \left\| X_0^{M,N,(2)}\right\| _{L^2(\nu )} ^{2/3} \right] \\
&\leq \frac{1}{T} \int _0^T E\left[ \left\| X_t^{M,N,(1)}\right\| _{L^2(\nu )} ^{2/3} \right] dt + \lambda ^{2/3} E\left[ \left\| {\mathcal Z}_0^{(0,3,M,N)}\right\| _{L^2(\nu )} ^{2/3} \right]\\
&\leq T^{-1/3} E\left[ \left( \int _0^T \left\| X_t^{M,N,(2)}\right\| _{L^2(\nu )} ^2 dt \right) ^{1/3}\right] + 2\lambda ^{2/3} \sup _{t\in [0,T]}E\left[ \left\| {\mathcal Z}_t^{(0,3,M,N)}\right\| _{L^2(\nu )} ^{2/3} \right] \\
&\leq T^{-1/3} E\left[ {\mathfrak X}_{\lambda , \alpha , \eta ,\gamma}^{M,N} (T) ^{1/3}\right] + C.
\end{align*}
Hence, by taking $T$ sufficiently large and applying \eqref{eq:thmtight1-1} we have
\begin{equation}\label{eq:thmtight1-2}\begin{array}{l}
\displaystyle E\left[ {\mathfrak X}_{\lambda ,\alpha , \eta ,\gamma}^{M,N} (T)^{1/3} \right] + E\left[ {\mathfrak Y}_{\varepsilon}^{M,N} (T) ^{q/3}\right] \\
\displaystyle + E\left[ \sup _{r\in [0,T]} \left( r^\eta \left\| X^{M,N,(2),<}_{r} \right\| _{B_{4}^{\alpha + 2\gamma }(\nu )}^3 \right) ^{1/3} \right] \\
\displaystyle + E\left[ \sup _{r\in [0,T]} \left( r^\eta \left\| X^{M,N,(2),\geqslant}_{r} \right\| _{B_{4/3}^{\alpha + 2\gamma }(\nu )} \right) ^{1/3}\right] \\
\displaystyle \leq C E\left[ \left\| X^{M,N,(2)}_0 \right\| _{B_{4/3}^{-1+2\gamma + 3\varepsilon }(\nu )}^{q/3} \right] + C E\left[ \left\| X^{M,N,(2)}_0 \right\| _{B_{4/3}^{\alpha + 2(\gamma - \eta )}(\nu )}^{1/3}\right] + K_{M,N}.
\end{array}\end{equation}
The stationarity of the marginal laws of $X^{M,N,(1)}$ also implies that
\begin{align*}
&E\left[ \left\| X^{M,N,(2)}_0 \right\| _{B_{4/3}^{-1+2\gamma + 3\varepsilon }(\nu )} ^{q/3} \right] \\
&\leq \frac{1}{T} \int _0^T E\left[ \left\| X^{M,N,(1)}_t \right\| _{B_{4/3}^{-1+2\gamma + 3\varepsilon }(\nu )} ^{q/3} \right] dt + \lambda ^{q/3} E\left[ \left\| {\mathcal Z}^{(0,3,M,N)}_0 \right\| _{B_{4/3}^{-1+2\gamma + 3\varepsilon }(\nu )} ^{q/3} \right] \\
&\leq C E\left[ \left( \int _0^T\left\| X^{M,N,(2)}_t \right\| _{B_{4/3}^{-1+2\gamma + 3\varepsilon }(\nu )}^{q} dt \right) ^{1/3} \right] + 2\lambda ^{q/3} \sup _{t\in [0,T]} E\left[ \left\| {\mathcal Z}^{(0,3,M,N)}_t \right\| _{B_{4/3}^{-1+2\gamma + 3\varepsilon }(\nu )} ^{q/3} \right] \\
&\leq C E\left[ \left( \delta ^3 {\mathfrak X}_{\lambda ,\alpha , \eta ,\gamma}^{M,N} (T) + C_\delta \right) ^{1/3}\right] + C \\
&\leq C \delta E\left[ {\mathfrak X}_{\lambda , \alpha , \eta ,\gamma}^{M,N} (T) ^{1/3}\right] + C_\delta .
\end{align*}
Similarly it holds that
\[
E\left[ \left\| X^{M,N,(2)}_0 \right\| _{B_{4/3}^{\alpha + 2(\gamma - \eta )}(\nu )}^{1/3} \right] \leq C\delta E\left[ {\mathfrak X}_{\lambda ,\alpha , \eta ,\gamma}^{M,N} (T) ^{1/3}\right] + C_\delta .
\]
By taking a sufficiently small $\delta$ in these inequalities and combining them with (\ref{eq:thmtight1-2}) we obtain the assertion.
\end{proof}

\begin{rem}
The assumption on $\eta$ is improved from those in \cite{AK} and \cite{Sei}.
This improvement comes from applying the stationarity of $X^{M,N,(1)} = X^{M,N} - Z$ instead of that of $X^{M,N}$ ($Z$ is the Ornstein-Uhlenbeck process of Section \ref{sec:OU}).
\end{rem}

Theorem \ref{thm:tight1} yields the tightness of the laws of $\{ X^{M,N}\}$, which was the target in the present paper.
Now, by using Theorem \ref{thm:tight1} we manage to prove our key theorem (Theorem \ref{thm:tight2}).

\vspace{3mm}
\noindent
{\it Proof of Theorem \ref{thm:tight2}.}
Choose $\sigma$, $\alpha$, $\gamma$, $\eta$, $\varepsilon$ and $q$ so that they satisfy the assumptions in Theorem \ref{thm:tight1}.
Then, in view of Proposition \ref{prop:estall} and Theorem \ref{thm:tight1}, we have the bound
\begin{equation}\label{eq:thmtight2-1} \begin{array}{l}
\displaystyle E\left[ {\mathfrak X}_{\lambda , \alpha , \eta ,\gamma}^{M_N,N} (T)^{1/3}\right] + E\left[ {\mathfrak Y}_{\varepsilon}^{M_N,N} (T) ^{q/3} \right] \\
\displaystyle + E\left[ \sup _{r\in [0,T]} \left( r^\eta \left\| X^{M_N,N,(2),<}_{r} \right\| _{B_4^{\alpha + 2\gamma }(\nu )} ^3 \right) ^{1/3} \right] \\
\displaystyle + E\left[ \sup _{r\in [0,T]} \left( r^\eta \left\| X^{M_N,N,(2),\geqslant}_{r} \right\| _{B_{4/3}^{\alpha + 2\gamma }(\nu )} \right) ^{1/3} \right] \leq C
\end{array}\end{equation}
for some positive constant $C$.
Let $T\in (0,\infty )$ and $t_0 \in (0,T)$.
Then, (\ref{eq:thmtight2-1}) and Chebyshev's inequality imply that, for any $\varepsilon '>0$
\begin{align*}
&\limsup _{h\downarrow 0} \sup _{N\in {\mathbb N}} P \left( \sup _{s,t\in [t_0,T]; |s-t|<h} \left\| X^{M_N,N,(2)}_t - X^{M_N,N,(2)}_s \right\| _{B_{4/3}^\alpha (\nu )} > \varepsilon' \right) \\
&\leq \limsup _{h\downarrow 0} \left( \frac{h^\gamma }{\varepsilon ' t_0^\eta } \right) ^{1/3} \sup _{N\in {\mathbb N}} E \left[ \left( \sup _{\substack{s,t\in [t_0,T];\\ s<t, t-s<h}} \frac{ s^\eta \left\| X^{M_N, N,(2)}_t - X^{M_N, N,(2)}_s \right\| _{B_{4/3}^\alpha (\nu )}}{(t-s)^\gamma } \right) ^{1/3} \right] \\
&=0,
\end{align*}
moreover:
\begin{align*}
&\limsup _{R\rightarrow \infty} \sup _{N\in {\mathbb N}} P \left( \left\| X^{M_N,N,(2)}_{t_0} \right\| _{B_{4/3}^{\alpha + 2\gamma}(\nu )} >R \right) \\
&\leq \limsup _{R\rightarrow \infty} \frac{1}{R^{1/3} t_0^\eta} \sup _{N\in {\mathbb N}} E \left[ \sup _{r\in [0,T]} \left( r^\eta \left\| X^{M_N, N,(2)}_r \right\| _{B_{4/3}^{\alpha + 2\gamma} (\nu )} \right) ^{1/3} \right] =0.
\end{align*}
In view of Proposition \ref{prop:cptembedding}, we have then tightness of the laws of $X^{M_N,N,(2)}$ as probability measures on $C([t_0,T]; B_q^\alpha (\nu ))$ for $\alpha \in [0, 1/2)$ and $q\in [1,4/3)$ (c.f. Theorem 4.2 of Chapter I in \cite{IW}).
Hence, an argument similar to the one in the proof of Theorem 4.19 in \cite{AK} yields the tightness of the laws of $\{ X^{M_N,N}\}$ on $C([0,\infty ); {B_{q}^{-1/2-\varepsilon }(\nu )})$ for $q\in [1,4/3)$.
The embedding theorem of weighted Besov spaces (see Proposition \ref{prop:Besov}) implies the tightness of the laws of $\{ X^{M_N,N}\}$ on $C([0,\infty ); {B_{4/3}^{-1/2-\varepsilon - \delta}(\nu ^{4/(3q)})})$ where $\delta = 3[(1/q)-(3/4)]$.
Noting that 
\[
\nu ^{4/(3q)}(x) = (1+|x|^2)^{-2\sigma /(3q)}
\]
and $\varepsilon \in (0,1)$, $q\in [1,4/3)$ and $\sigma \in (3,\sqrt{2m_0^2})$ are arbitrary (see \eqref{eq:assm0}), we obtain the tightness of the laws of $\{ X^{M_N,N}\}$ on $C([0,\infty ); {B_{4/3}^{-1/2-\varepsilon }(\nu )})$ for all $\varepsilon \in (0,1/16]$ and $\nu (x) =  (1+|x|^2)^{-\sigma /2}$ with $\sigma$ satisfying \eqref{eq:assm0}.

Again by the tightness of the laws of $X^{M_N,N,(2)}$ on $C([t_0,T]; B_q^\alpha (\nu ))$ for $\alpha \in [0, 1/2)$ and $q\in [1,4/3)$ and Proposition \ref{prop:Besov} we get the tightness of the laws of $X^{M_N,N,(2)}$ on $C([t_0,T]; B_p^{-1/2-\varepsilon} (\nu ^{p/q}))$, where $p\in [1,\infty )$ is restricted by
\[
\alpha = -\frac{1}{2} -\varepsilon + 3\left( \frac{1}{q} - \frac{1}{p}\right) .
\]
Since $\varepsilon \in (0,1/16]$, $\alpha \in [0, 1/2)$ and $q\in [1,4/3)$ are arbitrary, we have the tightness of the laws of $X^{M_N,N,(2)}$ on $C([t_0,T]; B_{12/5}^{-1/2-\varepsilon} (\nu ^{9/5}))$ for all $\varepsilon \in (0,1/16]$ and $\nu (x) =  (1+|x|^2)^{-\sigma /2}$ with $\sigma$ satisfying \eqref{eq:assm0}.
Again similarly to the proof of Theorem 4.19 in \cite{AK} (see also the proof of Theorem 3.4 in \cite{Sei}), we have the tightness of the laws of $X^{M_N,N}$ on $C([0,\infty ) ; B_{12/5}^{-1/2-\varepsilon} (\nu ^{9/5}))$ with $\varepsilon$ and $\nu$ as in Theorem \ref{thm:tight2}.
Thus, we obtain the tightness of the laws of $X^{M_N,N}$ on $C([0,\infty ); {B_{4/3}^{-1/2-\varepsilon }(\nu )}\cap B_{12/5}^{-1/2-\varepsilon} (\nu ^{9/5}))$.
Moreover, the tightness holds for all $\sigma \in (3,\infty )$, because the weighted Besov spaces have the monotonicity property in $\sigma$.

The proofs of the other assertions run in the same way as in the proof of Theorem 4.19 in \cite{AK}.
So, we omit them.
\qed

\section{Properties of the $\Phi ^4_3$-measure obtained in Theorem \ref{thm:tight2}}\label{sec:propertiesPhi43}

\subsection{Rotation and reflection invariance}\label{subsec:invariance}

In this section, we prove Theorem \ref{thm:Phi43measure}\ref{thm:Phi43measure2}.
Define $\psi _N$, $\rho _M$, $P_{M,N}$, $\mu _{M,N}$ and $Z_{M,N}$ as in Section \ref{sec:main}.
Here, we remark that $Z_{M,N}$ is the normalizing constant.
The proofs are very simple, because we are able to prove $\mu _{M,N}$ is rotation invariant if $\psi$ and $\rho$ are radial functions as follows.

\noindent
{\it Proof of Theorem \ref{thm:Phi43measure}\ref{thm:Phi43measure2}.}
Assume that $\psi$ and $\rho$ are radial (i.e. rotation invariant) functions.
Then, so are $\psi _N$ and $\rho _M$.
Hence, we have $U_{M,N}(\phi ) = U_{M,N}(\theta \phi)$ for any rotation $\theta$ acting in ${\mathbb R}^3$.
This fact and the rotation and reflection invariance of the free field measure $\mu _0$ yield the rotation and reflection invariance of $\mu _{M,N}$.
Therefore, by taking the limit of the subsequence $\mu _{M_{N(k)},N(k)}$ for $k\rightarrow \infty$ with $M_{N(k)}, N(k)\rightarrow \infty$, which converges to $\mu$, we obtain Theorem \ref{thm:Phi43measure}\ref{thm:Phi43measure2}.
\qed

\subsection{Reflection positivity of the limit measure}\label{subsec:RP}

In this section we prove Theorem \ref{thm:Phi43measure}\ref{thm:Phi43measure5}.
In the proof we show that the image measure $P_N \mu _{M,N}$ of $\mu _{M,N}$ with respect to $P_N$ has the reflection positivity property by means of the Markov property of $P_N \phi$ under the free field $\mu _0$ for each $N$.

\noindent
{\it Proof of Theorem \ref{thm:Phi43measure}\ref{thm:Phi43measure5}.}
Let $\{ g_\varepsilon \} _{\varepsilon >0}$ be a sequence in $C_b^\infty ({\mathbb R}^3)$ such that the support of $g_\varepsilon$ is included in $\{ (x_1,x_2,x_3); x_1>\varepsilon \}$, $\sup _{\varepsilon >0} \| g_\varepsilon \| _{L^\infty} \leq 1$ and $\lim _{\varepsilon \downarrow 0} g_{\varepsilon}(x) ={\mathbb I}_{(0,\infty)} (x_1)$ for $x=(x_1,x_2,x_3)\in {\mathbb R}^3$, and let
\begin{align*}
U_{M,N}^{\varepsilon, +}(\phi ) &:= \int _{{\mathbb R}^3} \left[ \frac{\lambda}{4} (g_\varepsilon P_{M,N} \phi )^4 -\frac{3\lambda}{2} \left( C_1^{(N)} -3\lambda C_2^{(M,N)}\right) \rho _M^2 g_\varepsilon ^2 (g_\varepsilon P_{M,N} \phi )^2\right] dx, \\
U_{M,N}^{\varepsilon, -}(\phi ) &:= \int _{{\mathbb R}^3} \left[ \frac{\lambda}{4} (g_\varepsilon ^\dagger P_{M,N} \phi )^4 -\frac{3\lambda}{2} \left( C_1^{(N)} -3\lambda C_2^{(M,N)}\right) \rho _M^2 (g_\varepsilon ^\dagger )^2(g_\varepsilon ^\dagger P_{M,N} \phi )^2\right] dx .
\end{align*}
Denote by $\mu _0(d\phi | (P_N \phi )_{x_1=0} )$ the regular conditional probability measure of $\mu _0$ with respect to the $\sigma$-field generated by
\[
\left\{ \langle f, P_N \phi \rangle ; f\in {\mathscr H}^{-1}({\mathbb R}^3),\ {\rm supp}f \subset \{ (x_1,x_2,x_3); x_1=0\} \right\} .
\]
Here, note that ${\mathscr H}^{-1}({\mathbb R}^3)$ is the Sobolev space generated by the norm $\| f\| _{{\mathscr H}^{-1}({\mathbb R}^3)} = \| (m_0^2 -\triangle )^{-1/2} f\| _{L^2({\mathbb R}^3, dx)}$, which coincides with the Cameron-Martin space of the free field measure $\mu _0$.
Similarly to the proof of Theorem 5 in \cite{Ne73a}, we can prove that $P_N \phi$ under the free field measure $\mu _0$ is a Markov field.
Indeed, even if we replace ${\mathscr H}^{-1}({\mathbb R}^3)$ by a closed linear subspace
\[
{\mathscr H}_N^{-1}({\mathbb R}^3) := \left\{ P_N f ;\ f\in {\mathscr H}^{-1}({\mathbb R}^3) \right\} ,
\]
the proof goes in a similar way.
Moreover, $g_\varepsilon P_{M,N} \phi$ and $g_\varepsilon ^\dagger P_{M,N} \phi$ are measurable with respect to the $\sigma$-field generated by $\{ \langle f, P_{M,N} \phi \rangle ; {\rm supp} f \subset \{ (x_1,x_2,x_3); x_1>0\} \}$ and $\{ \langle f, P_{M,N} \phi \rangle ; {\rm supp} f \subset \{ (x_1,x_2,x_3); x_1<0\} \}$, respectively.
Hence, by these facts and the assumption on $F$, we have 
\begin{align*}
&\int \overline{F^\dagger (P_N \phi )} F (P_N \phi ) \exp (- U_{M,N}^{\varepsilon, +}(\phi ) - U_{M,N}^{\varepsilon, -}(\phi ))\mu _0 (d\phi ) \\
&= \int \left( \int \overline{F^\dagger (P_N \phi  )} \exp ( - U_{M,N}^{\varepsilon, -}(\phi ))\mu _0(d\phi | (P_N \phi )_{x_1=0} )\right) \\
&\quad \hspace{1cm} \times \left( \int F (P_N \phi ) \exp (- U_{M,N}^{\varepsilon, +}(\phi ) ) \mu _0(d\phi | (P_N \phi )_{x_1=0} ) \right) \mu _0 (d\phi ) .
\end{align*}
Note that $C_2^{(M,N) \dagger} = C_2^{(M,N)}$, because $\delta _{-x}(\cdot ) = \delta _x(-\cdot )$, $\rho ^\dagger = \rho$ and $\psi ^\dagger = \psi$ imply
\begin{align*}
&\left\langle (m_0^2-\triangle )^{-1} e^{t(\triangle -m_0^2)}\left( \rho _M^2\Delta _i \delta _{-x} \right) , \rho _M^2 e^{t(\triangle -m_0^2)} P_N^2 \Delta _j \delta _{-x} \right\rangle \\
&= \left\langle (m_0^2-\triangle )^{-1} e^{t(\triangle -m_0^2)}\left( \rho _M^2 \Delta _i \delta _x \right) , \rho _M^2 e^{t(\triangle -m_0^2)} P_N^2 \Delta _j \delta _x \right\rangle .
\end{align*}
We also note that $U_{M,N}^{\varepsilon, -}(\phi ^\dagger ) = U_{M,N}^{\varepsilon, +}(\phi )$ and $(P_N \phi )^\dagger = P_N (\phi ^\dagger )$ follow in a similarly way.
Since the laws of $(P_N \phi )^\dagger$ and $P_N \phi$ under $\mu _0(d\phi | (P_N \phi )_{x_1=0} )$ are same,
\begin{align*}
&\int \overline{F^\dagger (P_N \phi )} \exp ( - U_{M,N}^{\varepsilon, -}(\phi ))\mu _0(d\phi | (P_N \phi )_{x_1=0} ) \\
&= \int \overline{F^\dagger (P_N \phi ^\dagger )} \exp ( - U_{M,N}^{\varepsilon, -}(\phi ^\dagger ))\mu _0(d\phi | (P_N \phi )_{x_1=0} ) \\
&= \int \overline{F (P_N \phi )} \exp ( - U_{M,N}^{\varepsilon, +}(\phi ))\mu _0(d\phi | (P_N \phi )_{x_1=0}) .
\end{align*}
Thus, we have
\begin{align*}
&\int \overline{F^\dagger (P_N \phi )} F (P_N \phi ) \exp (- U_{M,N}^{\varepsilon, +}(\phi ) - U_{M,N}^{\varepsilon, -}(\phi ))\mu _0 (d\phi ) \\
&= \int \left| \int F (P_N \phi ) \exp (- U_{M,N}^{\varepsilon, +}(\phi ) ) \mu _0(d\phi | (P_N \phi )_{x_1=0} ) \right| ^2 \mu _0 (d\phi ) \geq 0.
\end{align*}
This inequality and the fact that
\begin{align*}
&\lim _{\varepsilon \downarrow 0} \left( U_{M,N}^{\varepsilon, +}(\phi ) + U_{M,N}^{\varepsilon, -}(\phi ) \right) \\
&= \int _{{\mathbb R}^3} \left[ \frac{\lambda}{4} ( P_{M,N} \phi )^4 -\frac{3\lambda}{2} \left( C_1^{(N)} -3\lambda C_2^{(M,N)}\right) \rho _M^2 (P_{M,N} \phi )^2\right] dx
\end{align*}
imply that
\[
\int \overline{F^\dagger (P_N \phi )} F (P_N \phi ) \mu _{M,N} (d\phi ) \geq 0.
\]
Since $F$ is bounded and Lipschitz continuous on $B_p^s$ for $s\in {\mathbb R},\ p\in [1,\infty )$, and $\mu _{M(k), N(k)}$ converges weakly to $\mu$ as $k\rightarrow \infty$,
\begin{align*}
&\lim _{k\rightarrow \infty} \int \overline{F^\dagger (\phi )} F (\phi ) \mu _{M_{N(k)},N(k)} (d\phi ) = \int \overline{F^\dagger (\phi )} F (\phi ) \mu (d\phi ), \\
& \left| F (P_{N(k)} \phi ) - F(\phi) \right| + \left| \overline{F^\dagger (P_{N(k)} \phi )} - \overline{F^\dagger (\phi )} \right| \leq C \min \{ 1, \| P_{N(k)} \phi - \phi \| _{B_p^s}\}
\end{align*}
with a constant $C$.
Therefore, we also obtain the reflection positivity of $\mu$.
\qed

\subsection{Support of the limit measure}\label{subsec:support}

In this section we prove Theorem \ref{thm:Phi43measure}\ref{thm:Phi43measure4}.
Let $\nu$, $X^{M,N}$, $X^{M,N,(1)}$, $X^{M,N,(2)}$, $Z$ and ${\mathcal Z}^{(0,3,M,N)}$ be as in Section \ref{sec:trans}, and let $X$ be the limit process of a subsequence $X^{M_{N(k)}, N(k)}$ obtained in Theorem \ref{thm:tight2}.

We choose $\alpha$, $\gamma$, $\eta$, $\varepsilon$ and $q$ as in Theorem \ref{thm:tight1}.
By applying the Skorohod representation theorem to the pair of processes
\[
(X^{M_{N(k)},N(k)}, {\mathcal Z}^{(1,M_{N(k)},N(k))} , {\mathcal Z}^{(0,3,M_{N(k)},N(k))}),
\]
we have a sequence of the pair of processes on another probability space which does not change the laws and converges almost surely.
For simplicity, we denote the almost-surely converging sequences and the limit by $\{ X^{M_{N(k)},N(k)} \}$, $\{ {\mathcal Z}^{(1,M_{N(k)},N(k))}\}$, $\{ {\mathcal Z}^{(0,3,M_{N(k)},N(k))}\}$ and $X$ again.
We remark that $X^{M_{N(k)},N(k),(2)}$ also converges almost surely on a suitable path space.

\begin{lem}\label{lem:nontrivial4}
For sufficiently small $\varepsilon \in (0,1)$, there exists a subsequence $\{ N'(k)\}$ of $\{ N(k)\}$ such that
\begin{align*}
&E\left[ \limsup _{k\rightarrow \infty} \int _0^T \left\| X_s^{M_{N'(k)},N'(k),(2)} \right\| _{B_2 ^{1-\varepsilon }(\nu )} ^2 ds \right] < \infty, \\
&E\left[ \limsup _{k\rightarrow \infty} \int _0^T \left( \left\| X_s^{M_{N'(k)},N'(k)} -Z_s \right\| _{B_2 ^{1/2 -\varepsilon }(\nu )} ^2 + \left\| X_s^{M_{N'(k)},N'(k)} -Z_s \right\| _{L^4(\nu )}^4 \right)ds \right] <\infty .
\end{align*}
\end{lem}

\begin{proof}
In this proof, for simplicity we denote the limits of $X^{M_{N(k)},N(k),(2)}$ and ${\mathcal Z}^{(0,3,M_{N(k)},N(k))}$ as $k \rightarrow \infty$ by $X^{(2)}$ and ${\mathcal Z}^{(0,3)}$, respectively.
Then,
\[
\lim _{k\rightarrow \infty} X_t^{M_{N(k)},N(k),(2)} = X_t^{(2)} -Z_t + \lambda {\mathcal Z}_t^{(0,3)}
\]
in $B_{4/3}^{1/2-\varepsilon '}(\nu )$ uniformly in $t\in [0,T]$ for any $\varepsilon ' \in (0,1]$.
Proposition \ref{prop:estall} implies
\begin{equation}\label{eq:lemnontrivial4-1}\begin{array}{l}
\displaystyle \int _0^T \left( \left\| X_s^{M_{N(k)},N(k),(2)}\right\| _{B_2^{1-\varepsilon }(\nu )}^2 + \lambda \left\| P_{M_{N(k)},N(k)} X_s^{M_{N(k)},N(k),(2)}\right\| _{L^4(\nu )}^4 \right) ds\\
\displaystyle \leq C \left\| X^{M_{N(k)},N(k),(2)}_0 \right\| _{L^2(\nu )}^2 + C \left\| X^{M_{N(k)},N(k),(2)}_0 \right\| _{B_{4/3}^{\alpha + 2(\gamma - \eta )}(\nu )} \\
\displaystyle \quad + C 2^{-(4-\varepsilon )N(k)} \int _0^T \left\| X_s^{M_{N(k)},N(k),(2)} \right\| _{L^4(\nu )}^4 ds + C Q
\end{array}\end{equation}
for some positive constant $C$.
Similarly to the proof of Theorem \ref{thm:tight1}, we have the convergence
\[
\lim _{N\rightarrow \infty} E\left[ \left( 2^{-(4-\varepsilon )N} \int _0^T \left\| X_s^{M_N,N,(2)} \right\| _{L^4(\nu )}^4 ds \right) ^{1/3} \right] =0.
\]
In particular, by taking a subsequence $\{ N'(k)\}$ of $\{ N(k)\}$ we have 
\begin{equation}\label{eq:lemnontrivial4-2}
\lim _{k\rightarrow \infty} 2^{-(4-\varepsilon )N'(k)} \int _0^T \left\| X_s^{M_{N'(k)},N'(k),(2)} \right\| _{L^4(\nu )}^4 ds =0 
\end{equation}
almost surely.
Hence \eqref{eq:lemnontrivial4-1}, \eqref{eq:lemnontrivial4-2} and a generalized version of Fatou's lemma (see \cite{Scheinberg}) imply
\begin{equation}\label{eq:lemnontrivial4-10}\begin{array}{l}
\displaystyle E \left[ \min \left\{ \int _0^T \left( \left\| X_s^{(2)}\right\| _{B_2^{1-\varepsilon }(\nu )}^2 + \lambda \left\| X_s^{(2)}\right\| _{L^4(\nu )}^4 \right) ds ,n \right\} \right] \\
\displaystyle \leq C E\left[ \min \left\{ \left\| X^{(2)}_0 \right\| _{L^2(\nu )}^2 ,n \right\} \right] + C E\left[ \min \left\{ \left\| X^{(2)}_0 \right\| _{B_{4/3}^{\alpha + 2(\gamma - \eta )}(\nu )}, n \right\} \right] + C.
\end{array}\end{equation}
for any $n\in {\mathbb N}$.
We now observe that $X_t -Z_t$ is a stationary process, because it is a limit of a stationary process (see the proof of Theorem 4.19 in \cite{AK}).
From this fact and Proposition \ref{prop:Z} we have, for any $T>0$:
\begin{align*}
&E\left[ \min \left\{ \left\| X^{(2)}_0 \right\| _{L^2(\nu )}^2 ,n \right\} \right] \\
&\leq 2 E\left[ \min \left\{ \left\| X_0 -Z_0 \right\| _{L^2(\nu )}^2 ,n \right\} \right] + E\left[ \left\| {\mathcal Z}^{(0,3)}_0 \right\| _{L^2(\nu )}^2 \right] \\
&= \frac{2}{T} \int _0^T E\left[ \min \left\{ \left\| X_s -Z_s \right\| _{L^2(\nu )}^2 ,n \right\} \right] ds + C \\
&\leq 2 E\left[ \min \left\{ \frac{1}{T} \int _0^T\left\| X^{(2)}_s \right\| _{L^2(\nu )}^2 ds , n \right\} \right] +C \\
&\leq 2 E\left[ \min \left\{ \left( \frac{1}{T} \int _0^T \left\| X^{(2)}_s \right\| _{L^4(\nu )}^4 ds \right) ^{1/2}, n \right\} \right] +C .
\end{align*}
Similarly, for any $n\in {\mathbb N}$ and some positive constant $C$:
\begin{align*}
& E\left[ \min \left\{ \left\| X^{(2)}_0 \right\| _{B_{4/3}^{\alpha + 2(\gamma - \eta )}(\nu )}, n \right\} \right] \\
& \leq E\left[ \min \left\{ \left( \frac{1}{T} \int _0^T \left\| X_s^{(2)}\right\| _{B_2^{\alpha + 2(\gamma - \eta )}(\nu )}^2 ds \right) ^{1/2} , n \right\} \right] + C \\
& \leq E\left[ \min \left\{ \left( \frac{1}{T} \int _0^T \left\| X_s^{(2)}\right\| _{B_2^{1-\varepsilon }(\nu )}^2 ds \right) ^{1/2} , n \right\} \right] + C.
\end{align*}
These inequalities and \eqref{eq:lemnontrivial4-10} imply
\begin{align*}
&E \left[ \min \left\{ \int _0^T \left( \left\| X_s^{(2)}\right\| _{B_2^{1-\varepsilon }(\nu )}^2 + \lambda \left\| X_s^{(2)}\right\| _{L^4(\nu )}^4 \right) ds ,n \right\} \right]  \\
&\leq C E \left[ \min \left\{ \left( \int _0^T \left( \left\| X_s^{(2)}\right\| _{B_2^{1-\varepsilon }(\nu )}^2 + \lambda \left\| X_s^{(2)}\right\| _{L^4(\nu )}^4 \right) ds \right) ^{1/2} ,n \right\} \right] +C.
\end{align*}
Hence, by applying Proposition \ref{prop:Aint} we have
\[
E \left[ \int _0^T \left( \left\| X_s^{(2)}\right\| _{B_2^{1-\varepsilon }(\nu )}^2 + \lambda \left\| X_s^{(2)}\right\| _{L^4(\nu )}^4 \right) ds \right] < \infty .
\]
In particular, by the stationarity of $X_t -Z_t$ and Proposition \ref{prop:Z} we get
\begin{equation}\label{eq:lemnontrivial4-11}
E\left[ \left\| X^{(2)}_0 \right\| _{L^2(\nu )}^4 \right] + E\left[ \left\| X^{(2)}_0 \right\| _{B_{4/3}^{\alpha + 2(\gamma - \eta )}(\nu )}^2 \right] < \infty .
\end{equation}
On the other hand, \eqref{eq:lemnontrivial4-1} and \eqref{eq:lemnontrivial4-2} also imply
\begin{align*}
& E\left[ \limsup _{k\rightarrow \infty} \int _0^T \left( \left\| X_s^{M_{N'(k)},N'(k),(2)}\right\| _{B_2^{1-\varepsilon}(\nu )}^2 \right. \right. \\
& \hspace{3cm} \left. \left. \phantom{\int} + \lambda \left\| P_{M_{N'(k)},N'(k)} X_s^{M_{N'(k)},N'(k),(2)}\right\| _{L^4(\nu )}^4 \right) ds \right] \\
& \leq C E\left[ \left\| X^{(2)}_0 \right\| _{L^2(\nu )}^2 \right] + C E\left[ \left\| X^{(2)}_0 \right\| _{B_{4/3}^{\alpha + 2(\gamma - \eta )}(\nu )} \right] + C.
\end{align*}
This inequality and \eqref{eq:lemnontrivial4-11} yield the first inequality. Applying Proposition \ref{prop:Z} to \eqref{eq:lemnontrivial4-10}, we obtain the second inequality.
\end{proof}

\vspace{5mm}
\noindent
{\it Proof of Theorem \ref{thm:Phi43measure}\ref{thm:Phi43measure4}.}
By the stationarity of $X^{M_N,N}$ and the fact that the marginal law $X$ is equal to $\mu$, for any $J\in {\mathbb N}$ we have
\begin{align*}
&\int \sum _{j=-1}^J 2^{-[(1/2)+\varepsilon]j} \| \Delta _j \phi \| _{L^p(\nu ^{p/2})}^2 \mu (d\phi ) \\
&= E\left[ \int _0^1 \sum _{j=-1}^J 2^{-[(1/2)+\varepsilon]j} \| \Delta _j X_t \| _{L^p(\nu ^{p/2})}^2 dt \right] \\
&= E\left[ \liminf _{k\rightarrow \infty} \int _0^1 \sum _{j=-1}^J 2^{-[(1/2)+\varepsilon]j} \left\| \Delta _j X^{M_{N(k)},N(k)}_t \right\| _{L^p(\nu ^{p/2})}^2 dt \right] \\
&\leq E\left[ \liminf _{k\rightarrow \infty} \int _0^1 \left\| \Delta _j X^{M_{N(k)},N(k)}_t \right\| _{B_p^{-1/2-\varepsilon}(\nu ^{p/2})}^2 dt \right] \\
&\leq C E\left[ \liminf _{k\rightarrow \infty} \int _0^1 \left\| \Delta _j X^{M_{N(k)},N(k),(2)}_t \right\| _{B_p^{-1/2-\varepsilon}(\nu ^{p/2})}^2 dt \right] + C .
\end{align*}
On the other hand, Proposition \ref{prop:Besov} implies
\[
\left\| X^{M_{N(k)},N(k),(2)}_t \right\| _{B_p^{-1/2-\varepsilon}(\nu ^{p/2})} \leq C \left\| X^{M_{N(k)},N(k),(2)}_t \right\| _{B_2^{1-\varepsilon}(\nu )} .
\]
Hence, by the monotone convergence theorem and Fatou's lemma we have
\begin{align*}
&\int \| \phi \| _{B_p^{-1/2-\varepsilon}(\nu ^{p/2})}^2 \mu (d\phi )\\
&\leq \liminf _{J\rightarrow \infty} \int \sum _{j=-1}^J 2^{-[(1/2)+\varepsilon]j} \| \Delta _j \phi \| _{L^p(\nu ^{p/2})}^2 \mu (d\phi )\\
&\leq C E\left[ \liminf _{k\rightarrow \infty} \int _0^1 \left\| X^{M_{N(k)},N(k),(2)}_t \right\| _{B_2^{1-\varepsilon}(\nu )}^2 dt \right] + C .
\end{align*}
Therefore, in view of Lemma \ref{lem:nontrivial4}, Theorem \ref{thm:Phi43measure}\ref{thm:Phi43measure4} holds.
\qed

\subsection{Non-Gaussianity of the limit measure}\label{subsec:nontrivial}

In this section we prove the non-Gaussianity of the $\Phi ^4_3$-measure obtained in Theorem \ref{thm:tight2} in a similar way to the proof of non-Gaussianity in \cite{GuHo2}.
Unlike the $\Phi ^4_3$-measure obtained in \cite{GuHo2}, we do not have such a nice integrability of our $\Phi ^4_3$-measure (see Proposition 4.11 in \cite{GuHo2}).
For this reason we modify the proof in \cite{GuHo2} and prove it by applying our Lemma \ref{lem:nontrivial4}.
Similarly to Section \ref{subsec:support}, we choose $\alpha$, $\gamma$, $\eta$, $\varepsilon$ and $q$ as in Theorem \ref{thm:tight1}.
Moreover, by applying the Skorohod representation theorem we assume that $X^{M_{N(k)},N(k)}$ converges to $X$ in $C([0,\infty ); {B_{4/3}^{-1/2-\varepsilon }(\nu )}\cap B_{12/5}^{-1/2-\varepsilon} (\nu ^{9/5}))$ almost surely.

\begin{lem}\label{lem:nontrivial1}
For $f\in {\mathcal S}({\mathbb R}^3)$ and $t\in (-\infty , \infty )$, we have
\begin{align*}
&\lim _{M,N \rightarrow \infty}E\left[ \left\langle f,  {\mathcal Z}_t^{(1,N)} \right\rangle ^3 \left\langle f, P_{M,N} {\mathcal Z}_t^{(0,3,M,N)} \right\rangle \right] \\
&= \int _{{\mathbb R}^3} \int _0^\infty \left( e^{s(\triangle -m_0^2)} \left[ 2(m_0^2-\triangle)\right] ^{-1} f \right) ^3 (x) \left( e^{s(\triangle -m_0^2)}f \right) (x) ds dx .
\end{align*}
\end{lem}

\begin{proof}
Since for any $t\in (-\infty , \infty )$
\begin{align*}
&E\left[ \left\langle f, {\mathcal Z}_t^{(1,N)} \right\rangle ^3 \left\langle f, P_{M,N} {\mathcal Z}_t^{(0,3,M,N)} \right\rangle \right] \\
&= \int _{-\infty}^t E\left[ \left\langle P_N Z_t , f \right\rangle ^3 \left\langle {\mathcal Z}_s^{(3,N)} , P_{M,N} e^{(t-s)(\triangle -m_0^2)} P_{M,N}^* f \right\rangle \right] ds ,
\end{align*}
by the product formula (see Theorem 2.7.10 in \cite{NoPe}) or by the calculation of the expectation of Gaussian polynomials by pairing (see Theorem 1.28 in \cite{Ja}) we have
\begin{align*}
&E\left[ \left\langle f,{\mathcal Z}_t^{(1,N)} \right\rangle ^3 \left\langle f, P_{M,N} {\mathcal Z}_t^{(0,3,M,N)} \right\rangle \right] \\
&= \int _{-\infty}^t \int _{{\mathbb R}^3} E\left[ \left\langle f, P_N Z_t \right\rangle \left\langle \delta _x, P_N Z_s \right\rangle \right] ^3  \left( P_{M,N} e^{(t-s)(\triangle -m_0^2)} P_{M,N}^* f\right) (x) dx ds .
\end{align*}
Hence, Lemma \ref{lem:covPZ} implies
\begin{align*}
&E\left[ \left\langle f, {\mathcal Z}_t^{(1,N)} \right\rangle ^3 \left\langle f,{\mathcal Z}_t^{(0,3,M,N)} \right\rangle \right] \\
&= \int _{-\infty}^t \int _{{\mathbb R}^3} \left\langle V_N(t-s) f , \delta _x \right\rangle ^3  \left( P_{M,N} e^{(t-s)(\triangle -m_0^2)} P_{M,N}^* f\right) (x) dx ds \\
&= \int _{-\infty}^t \int _{{\mathbb R}^3} \left( V_N(t-s) f \right) ^3 (x) \left( P_{M,N} e^{(t-s)(\triangle -m_0^2)} P_{M,N}^* f\right) (x) dx ds.
\end{align*}
By taking limit as $M,N\rightarrow \infty$ and exploiting that $f\in {\mathcal S}({\mathbb R}^3)$, we get the assertion.
\end{proof}

\begin{lem}\label{lem:nontrivial2}
For $j\in {\mathbb N}\cup \{ -1,0\}$ and $y\in {\mathbb R}^3$, it holds that
\[
\int _{{\mathbb R}^3} \int _0^\infty \left( e^{t(\triangle -m_0^2)} \left[ 2(m_0^2-\triangle)\right] ^{-1} \Delta _j \delta _y \right) ^3 (x) \left( e^{t(\triangle -m_0^2)}\Delta _j \delta _y \right) (x) dt dx \asymp 2^j ,
\]
where $a_j \asymp b_j$ means that there exists positive constants $c_1$ and $c_2$ satisfying $c_1 a_j \leq b_j \leq c_2 a_j$ for all $j\in {\mathbb N}\cup \{ -1,0\}$.
Moreover, the integration in the left-hand side is independent of $y$.
\end{lem}

\begin{proof}
For $x\in {\mathbb R}^3$, we have
\begin{align*}
&\int _0^\infty \left( e^{t(\triangle -m_0^2)} \left[ 2(m_0^2-\triangle)\right] ^{-1} \Delta _j \delta _y \right) ^3 (x) \left( e^{t(\triangle -m_0^2)}\Delta _j \delta _y \right) (x) dt \\
&= \int _0^\infty \int _{{\mathbb R}^3} \int _{{\mathbb R}^3} \int _{{\mathbb R}^3} \int _{{\mathbb R}^3} \left( \prod _{i=1}^3 \frac{e^{-t(m_0^2 + |\xi _i|^2)}\psi _j(\xi _i) }{2(m_0^2 + |\xi _i|^2)} ({\mathcal F} \delta _y)(\xi _i) \right) \\
&\quad \hspace{1cm} \times \left( e^{-t(m_0^2 + |\xi _4|^2)} \psi _j(\xi _4) ({\mathcal F} \delta _y)(\xi _4) \right) (2\pi )^{-6} e^{\sqrt{-1} (\xi _1+\xi _2+\xi _3 +\xi _4)\cdot x} d\xi _1 d\xi _2 d\xi _3 d\xi _4 dt \\
&= \frac{1}{8(2\pi)^{12}}\int _{{\mathbb R}^3} \int _{{\mathbb R}^3} \int _{{\mathbb R}^3} \int _{{\mathbb R}^3} \frac{\psi _j(\xi _1)\psi _j(\xi _2)\psi _j(\xi _3)\psi _j(\xi _4)}{(m_0^2 + |\xi _1|^2)(m_0^2 + |\xi _2|^2)(m_0^2 + |\xi _3|^2)} \\
&\quad \hspace{2cm} \times \frac{e^{-\sqrt{-1}(\xi _1 + \xi _2 + \xi _3 + \xi _4) \cdot y}}{4m_0^2 + |\xi _1|^2 + |\xi _2|^2 + |\xi _3|^2 + |\xi _4|^2} e^{\sqrt{-1} (\xi _1+\xi _2+\xi _3 +\xi _4)\cdot x} d\xi _1 d\xi _2 d\xi _3 d\xi _4 .
\end{align*}
Hence, by changing variables in the integrations we see that
\begin{align*}
&\int _{{\mathbb R}^3} \int _0^\infty \left( e^{t(\triangle -m_0^2)} \left[ 2(m_0^2-\triangle)\right] ^{-1} \Delta _j \delta _y \right) ^3 (x) \left( e^{t(\triangle -m_0^2)}\Delta _j \delta _y \right) (x) dt dx \\
&= \frac{2^j}{8(2\pi)^{12}} \int _{{\mathbb R}^3} \left( \int _{{\mathbb R}^3} \int _{{\mathbb R}^3} \int _{{\mathbb R}^3} \int _{{\mathbb R}^3} \frac{\psi (\xi _1)\psi (\xi _2)\psi (\xi _3)\psi (\xi _4)}{(2^{-2j} m_0^2 + |\xi _1|^2)(2^{-2j} m_0^2 + |\xi _2|^2)(2^{-2j} m_0^2 + |\xi _3|^2)} \right.\\
&\quad \hspace{1cm} \left. \times \frac{e^{-\sqrt{-1}(\xi _1 + \xi _2 + \xi _3 + \xi _4) \cdot (2^j y)}}{2^{2-2j} m_0^2 + |\xi _1|^2 + |\xi _2|^2 + |\xi _3|^2 + |\xi _4|^2} e^{\sqrt{-1} (\xi _1+\xi _2+\xi _3 +\xi _4)\cdot x} d\xi _1 d\xi _2 d\xi _3 d\xi _4 \right) dx .
\end{align*}
Since
\begin{align*}
&\int _{{\mathbb R}^3} \left( \int _{{\mathbb R}^3} \int _{{\mathbb R}^3} \int _{{\mathbb R}^3} \int _{{\mathbb R}^3} \frac{\psi (\xi _1)\psi (\xi _2)\psi (\xi _3)\psi (\xi _4)}{(2^{-2j} m_0^2 + |\xi _1|^2)(2^{-2j} m_0^2 + |\xi _2|^2)(2^{-2j} m_0^2 + |\xi _3|^2)} \right.\\
&\quad \hspace{1cm} \left. \times \frac{e^{-\sqrt{-1}(\xi _1 + \xi _2 + \xi _3 + \xi _4) \cdot (2^j y)}}{2^{2-2j} m_0^2 + |\xi _1|^2 + |\xi _2|^2 + |\xi _3|^2 + |\xi _4|^2} e^{\sqrt{-1} (\xi _1+\xi _2+\xi _3 +\xi _4)\cdot x} d\xi _1 d\xi _2 d\xi _3 d\xi _4 \right) dx \\
&= \lim _{\varepsilon \downarrow 0} \int _{{\mathbb R}^3} \left( \int _{{\mathbb R}^3} \int _{{\mathbb R}^3} \int _{{\mathbb R}^3} \int _{{\mathbb R}^3} \frac{\psi (\xi _1)\psi (\xi _2)\psi (\xi _3)\psi (\xi _4)}{(2^{-2j} m_0^2 + |\xi _1|^2)(2^{-2j} m_0^2 + |\xi _2|^2)(2^{-2j} m_0^2 + |\xi _3|^2)} \right.\\
&\quad \hspace{0cm} \left. \times \frac{e^{-\sqrt{-1}(\xi _1 + \xi _2 + \xi _3 + \xi _4) \cdot (2^j y)}}{2^{2-2j} m_0^2 + |\xi _1|^2 + |\xi _2|^2 + |\xi _3|^2 + |\xi _4|^2} e^{-\varepsilon |x|^2} e^{\sqrt{-1} (\xi _1+\xi _2+\xi _3 +\xi _4)\cdot x} d\xi _1 d\xi _2 d\xi _3 d\xi _4 \right) dx \\
&= \lim _{\varepsilon \downarrow 0} \left( \frac{\pi }{\varepsilon}\right) ^{3/2} \int _{{\mathbb R}^3} \int _{{\mathbb R}^3} \int _{{\mathbb R}^3} \int _{{\mathbb R}^3} \frac{\psi (\xi _1)\psi (\xi _2)\psi (\xi _3)\psi (\xi _4)}{(2^{-2j} m_0^2 + |\xi _1|^2)(2^{-2j} m_0^2 + |\xi _2|^2)(2^{-2j} m_0^2 + |\xi _3|^2)} \\
&\quad \hspace{2cm} \times \frac{e^{-\sqrt{-1}(\xi _1 + \xi _2 + \xi _3 + \xi _4) \cdot (2^j y)}}{2^{2-2j} m_0^2 + |\xi _1|^2 + |\xi _2|^2 + |\xi _3|^2 + |\xi _4|^2} e^{-|\xi _1+\xi _2+\xi _3 +\xi _4|^2/4\varepsilon} d\xi _1 d\xi _2 d\xi _3 d\xi _4 \\
&= (2\pi )^3 \int _{{\mathbb R}^3} \int _{{\mathbb R}^3} \int _{{\mathbb R}^3} \int _{{\mathbb R}^3} \frac{\psi (\xi _1)\psi (\xi _2)\psi (\xi _3)\psi (\xi _4)}{(2^{-2j} m_0^2 + |\xi _1|^2)(2^{-2j} m_0^2 + |\xi _2|^2)(2^{-2j} m_0^2 + |\xi _3|^2)} \\
&\quad \hspace{1.5cm} \times \frac{e^{-\sqrt{-1}(\xi _1 + \xi _2 + \xi _3 + \xi _4) \cdot (2^j y)}}{2^{2-2j} m_0^2 + |\xi _1|^2 + |\xi _2|^2 + |\xi _3|^2 + |\xi _4|^2} \delta _0 (\xi _1+\xi _2+\xi _3 +\xi _4) d\xi _1 d\xi _2 d\xi _3 d\xi _4 ,
\end{align*}
we get the assertion.
\end{proof}

Now we are ready to prove Theorem \ref{thm:Phi43measure}\ref{thm:Phi43measure3}.

\vspace{5mm}
\noindent
{\it Proof of Theorem \ref{thm:Phi43measure}\ref{thm:Phi43measure3}.}
By assuming ad absurdum that $\mu$ is a Gaussian measure on ${\mathcal S}'({\mathbb R}^3)$, we show that we are lead to a contradiction.
Let $j \in {\mathbb N}\cup \{ -1,0\}$ and $t \in (0,\infty )$.
We remark that the law of $X_t$ is $\mu$ for $t\in [0,\infty )$ and that the Gaussianity assumption implies that for any $j$, $\langle \Delta _j \delta _x, X_t \rangle$ has all moments with respect to the probability measure $P$ for $x\in {\mathbb R}^3$.
Hence, $\langle \Delta _j \delta _x, X_t -Z_t \rangle$ also has all moments with respect to the probability measure $P$ for $x\in {\mathbb R}^3$.
Since the laws of $\langle \Delta _j \delta _x, X_t \rangle$ and $\langle \Delta _j \delta _x ,Z_t \rangle$ are Gaussian, we get
\begin{equation}\label{eq:thmPhi43measure1-01}
E\left[ \langle \Delta _j \delta _x, X_t \rangle ^4 \right] = 3 E\left[ \langle \Delta _j \delta _x, X_t \rangle ^2 \right] ^2, \quad E\left[ \langle \Delta _j \delta _x, Z_t \rangle ^4 \right] = 3 E\left[ \langle \Delta _j \delta _x, Z_t \rangle ^2 \right] ^2 .
\end{equation}
Hence, we have
\begin{equation}\label{eq:nontrivial80}\begin{array}{l}
\displaystyle E\left[ \langle \Delta _j \delta _x X_t -Z_t \rangle ^4 \right] + 4 E\left[ \langle \Delta _j \delta _x, X_t -Z_t \rangle ^3 \langle \Delta _j \delta _x, Z_t \rangle \right] \\
\displaystyle  + 6 E\left[ \langle \Delta _j \delta _x, X_t -Z_t \rangle ^2 \langle \Delta _j \delta _x, Z_t \rangle ^2\right] + 4 E\left[ \langle \Delta _j \delta _x, X_t -Z_t \rangle \langle \Delta _j \delta _x, Z_t \rangle ^3 \right] \\
\displaystyle  = 3 E\left[ \langle \Delta _j \delta _x, X_t -Z_t \rangle ^2 \right] ^2 + 12 E\left[ \langle \Delta _j \delta _x, X_t -Z_t \rangle \langle \Delta _j \delta _x, Z_t \rangle \right] ^2 \\
\displaystyle  \quad + 12 E\left[ \langle \Delta _j \delta _x, X_t -Z_t \rangle ^2 \right] E\left[ \langle \Delta _j \delta _x, X_t -Z_t \rangle \langle \Delta _j \delta _x, Z_t \rangle \right] \\
\displaystyle  \quad + 6 \left( E\left[ \langle \Delta _j \delta _x, X_t -Z_t \rangle ^2 \right] + 2E\left[ \langle \Delta _j \delta _x, X_t -Z_t \rangle \langle \Delta _j \delta _x, Z_t \rangle \right] \right) \\
\displaystyle \quad \hspace{9cm} \times E\left[ \langle \Delta _j \delta _x, Z_t \rangle ^2\right]
\end{array}\end{equation}
for all $t\in [0,\infty )$.
This inequality, H\"older's inequality and \eqref{eq:thmPhi43measure1-01} imply
\begin{align*}
& \left| E\left[ \langle \Delta _j \delta _x, X_t -Z_t \rangle \langle \Delta _j \delta _x, Z_t \rangle ^3 \right] \right| \\
& \leq C E\left[ \langle \Delta _j \delta _x, X_t -Z_t \rangle ^2 \right] E\left[ \langle \Delta _j \delta _x, Z_t \rangle ^2\right] \\
& \quad + C \sum _{l=3}^4 E\left[ \langle \Delta _j \delta _x, X_t -Z_t \rangle ^4 \right] ^{l/4} E\left[ \langle \Delta _j \delta _x, Z_t \rangle ^2\right] ^{(4-l)/2} \\
& \quad + \left| E\left[ \langle \Delta _j \delta _x, X_t -Z_t \rangle \langle \Delta _j \delta _x, Z_t \rangle \right] \right| E\left[ \langle \Delta _j \delta _x, Z_t \rangle ^2\right] .
\end{align*}
On the other hand,
\begin{align*}
&\int _{{\mathbb R}^3} E\left[ \langle \Delta _j \delta _x, X_t -Z_t \rangle \langle \Delta _j \delta _x, Z_t \rangle ^3 \right] \nu ^3 (x) dx \\
&= \int _{{\mathbb R}^3} E\left[ (\Delta _j (X_t -Z_t))(x) (\Delta _j Z_t)^3 (x) \right] \nu ^3 (x) dx \\
&= E\left[ \int _{{\mathbb R}^3} [\Delta _j (X_t -Z_t)] (\Delta _j Z_t)^3 \nu ^3 dx\right] .
\end{align*}
Note that the translation invariance of the law of $Z_t$ implies the independence of $E\left[ \langle \Delta _j \delta _x, Z_t \rangle ^2\right]$ in $x\in {\mathbb R}^3$, and that other terms in \eqref{eq:nontrivial80} are calculated similarly.
Hence, we have
\begin{equation}\label{eq:thmPhi43measure1-02}\begin{array}{l}
\displaystyle \left| E\left[ \int _{{\mathbb R}^3} [\Delta _j (X_t -Z_t)] (\Delta _j Z_t)^3 \nu ^3 dx\right] \right| \\
\displaystyle \leq C E\left[ \left\| \Delta _j (X_t -Z_t) \right\| _{L^2(\nu ^3)}^2 \right] E\left[ \left\| \Delta _j Z_t \right\| _{L^2(\nu ^3)} ^2\right] \\
\displaystyle \quad + C \sum _{l=3}^4 E\left[ \left\| \Delta _j (X_t -Z_t) \right\| _{L^4(\nu ^3)} ^4 \right] ^{l/4} E\left[ \left\| \Delta _j Z_t \right\| _{L^2(\nu ^3)} ^2 \right] ^{(4-l)/2} \\
\displaystyle \quad + C \left| E\left[ \int _{{\mathbb R}^3} [\Delta _j (X_t -Z_t)] (\Delta _j Z_t) \nu ^3 dx \right] \right| E\left[ \left\| \Delta _j Z_t \right\| _{L^2(\nu ^3)} ^2\right] .
\end{array}\end{equation}
Since
\[
E\left[ \left\| \Delta _j Z_t \right\| _{L^2(\nu ^3)} ^2\right] \leq 2^{(1+2\varepsilon )j} E\left[ \left\| Z_t \right\| _{B_2^{-1/2-\varepsilon }(\nu ^3)}^2\right] ,
\]
by Proposition \ref{prop:Z} we have
\begin{equation}\label{eq:thmPhi43measure1-03}
E\left[ \left\| \Delta _j Z_t \right\| _{L^2(\nu ^3)} ^2\right] \leq C 2^{(1+2\varepsilon )j} .
\end{equation}
In view of the smoothing property of $\Delta _j$ and Proposition \ref{prop:Besov}, it holds that
\begin{align*}
&\limsup _{k\rightarrow \infty} \sup _{s\in [0,T]} \left\| \Delta _j \left( X_s^{M_{N(k)},N(k)} -Z_s \right) - \Delta _j \left( X_s -Z_s \right)\right\| _{L^4(\nu ^3)} \\
&\leq C_j \limsup _{k\rightarrow \infty} \sup _{s\in [0,T]} \left\| \Delta _j \left( X_s^{M_{N(k)},N(k)} -Z_s \right)  - \Delta _j \left( X_s -Z_s \right) \right\| _{B_{4/3}^{-1/2-\varepsilon}(\nu )} \\
&=0.
\end{align*}
From this fact, stationarity of $(X_t, Z_t)$ and Fatou's lemma we have
\begin{align*}
E\left[ \| \Delta _j (X_t -Z_t)\| _{L^4(\nu ^3)}^4 \right] &= \int _0^1 E\left[ \| \Delta _j (X_s -Z_s)\| _{L^4(\nu ^3)}^4 \right] ds \\
&= \int _0^1 E\left[ \liminf _{k\rightarrow \infty} \left\| \Delta _j \left( X_s^{M_{N(k)},N(k)} -Z_s \right) \right\| _{L^4(\nu ^3)}^4 \right] ds\\
&\leq E\left[ \liminf _{k\rightarrow \infty} \int _0^1 \left\| X_s^{M_{N(k)},N(k)} -Z_s \right\| _{L^4 (\nu ^3)}^4 ds \right]
\end{align*}
Similarly
\begin{align*}
&E\left[ \| \Delta _j (X_t -Z_t)\| _{L^2(\nu ^3)}^2 \right] \\
&\leq 2^{-[(1/2) - \varepsilon ]j} E\left[ \liminf _{k\rightarrow \infty} \int _0^1 \left\| X_s^{M_{N(k)},N(k)} -Z_s \right\|  _{B_2^{1/2-\varepsilon}(\nu ^3)}^2 ds \right] .
\end{align*}
From \eqref{eq:thmPhi43measure1-02}, \eqref{eq:thmPhi43measure1-03}, these inequalities and Lemma \ref{lem:nontrivial4} we have
\begin{equation}\label{eq:thmPhi43measure1-04}\begin{array}{l}
\displaystyle \left| E\left[ \int _{{\mathbb R}^3} [\Delta _j (X_t -Z_t)] (\Delta _j Z_t)^3 \nu ^3 dx\right] \right| \\
\displaystyle \leq C 2^{(1+2\varepsilon )j} \left| E\left[ \int _{{\mathbb R}^3} [\Delta _j (X_t -Z_t)] (\Delta _j Z_t) \nu ^3 dx \right] \right| + C 2^{[(1/2)+3\varepsilon ]j}.
\end{array}\end{equation}
Almost-sure convergence of $\{ (X^{M_{N(k)},N(k)}, {\mathcal Z}^{(1,M_{N(k)}, N(k))}, {\mathcal Z}^{(0,3,M_{N(k)}, N(k))}  )\}$ and Fatou's lemma imply
\begin{align*}
&E\left[ \int _{{\mathbb R}^3} [\Delta _j (X_t -Z_t)] (\Delta _j Z_t) \nu ^3 dx \right] \\
&= \int _0^1 E\left[ \lim _{k\rightarrow \infty} \left( \int _{{\mathbb R}^3} \left( \Delta _j X_s^{M_{N(k)},N(k), (2)} \right) \left( \Delta _j {\mathcal Z}_s^{(1,M_{N(k)}, N(k))} \right) \nu ^3 dx \right. \right. \\
&\quad \left. \left. \hspace{1cm} \phantom{\int} -\lambda \int _{{\mathbb R}^3} \left( \Delta _j P_{M_{N(k)},N(k)} {\mathcal Z}_s^{(0,3,M_{N(k)}, N(k))}\right) \left( \Delta _j {\mathcal Z}_s^{(1,M_{N(k)}, N(k))} \right) \nu ^3 dx \right) \right] ds \\
&\leq E\left[ \liminf _{k\rightarrow \infty} \int _0^1 \left\| \Delta _j X_s^{M_{N(k)},N(k), (2)} \right\| _{L^2(\nu ^3)} \left\| \Delta _j {\mathcal Z}_s^{(1,M_{N(k)}, N(k))}\right\| _{L^2(\nu ^3)} ds \right] \\
&\quad -\lambda \int _0^1 \lim _{k\rightarrow \infty} E\left[ \int _{{\mathbb R}^3} \left( \Delta _j P_{M_{N(k)},N(k)} {\mathcal Z}_s^{(0,3,M_{N(k)}, N(k))}\right) \left( \Delta _j {\mathcal Z}_s^{(1,M_{N(k)}, N(k))} \right) \nu ^3 dx \right] ds .
\end{align*}
The orthogonality of Wick polynomials yields
\[
E\left[ \left( \Delta _j P_{M_{N(k)},N(k)} {\mathcal Z}_s^{(0,3,M_{N(k)}, N(k))}\right) (x) \left( \Delta _j {\mathcal Z}_s^{(1,M_{N(k)}, N(k))} \right) (x) \right] =0
\]
for $x\in {\mathbb R}^3$ and $s\in [0,\infty )$, so that we have for some positive constant $C$:
\begin{align*}
&E\left[ \int _{{\mathbb R}^3} [\Delta _j (X_t -Z_t)] (\Delta _j Z_t) \nu ^3 dx \right] \\
&\leq C 2^{-[(1/2)-2\varepsilon ]j} E\left[ \liminf _{k\rightarrow \infty} \int _0^1 \left\| X_s^{M_{N(k)},N(k), (2)} \right\| _{B_2^{1-\varepsilon}(\nu ^3)} \left\| {\mathcal Z}_s^{(1,M_{N(k)}, N(k))}\right\| _{B_2^{-(1/2)-\varepsilon}(\nu ^3)} ds \right] .
\end{align*}
Hence, in view of Proposition \ref{prop:Z} and Lemma \ref{lem:nontrivial4}, from \eqref{eq:thmPhi43measure1-04} we have
\begin{equation}\label{eq:thmPhi43measure1-05}
\left| E\left[ \int _{{\mathbb R}^3} [\Delta _j (X_t -Z_t)] (\Delta _j Z_t)^3 \nu ^3 dx\right] \right| \leq C 2^{[(1/2)+4\varepsilon ]j}
\end{equation}
for some positive $C$.
For the left-hand side of this inequality deduce it as
\begin{align*}
&\lambda \left| E\left[ \int _0^1 \lim _{k\rightarrow \infty} \int _{{\mathbb R}^3} \left( \Delta _j P_{M_{N(k)},N(k)} {\mathcal Z}_s^{(0,3,M_{N(k)}, N(k))}\right) \left( \Delta _j {\mathcal Z}_s^{(1,M_{N(k)}, N(k))} \right) ^3 \nu ^3 dx ds \right] \right| \\
&\leq \left| E\left[ \int _{{\mathbb R}^3} [\Delta _j (X_t -Z_t)] (\Delta _j Z_t)^3 \nu ^3 dx\right] \right| \\
&\quad + \left| E\left[ \int _0^1 \lim _{k\rightarrow \infty} \int _{{\mathbb R}^3} \left( \Delta _j X_s^{M_{N(k)},N(k), (2)} \right) \left( \Delta _j {\mathcal Z}_s^{(1,M_{N(k)}, N(k))} \right) ^3 \nu ^3 dx ds \right] \right| \\
&\leq \left| E\left[ \int _{{\mathbb R}^3} [\Delta _j (X_t -Z_t)] (\Delta _j Z_t)^3 \nu ^3 dx\right] \right| \\
&\quad + E\left[ \liminf _{k\rightarrow \infty} \int _0^1 \left\| \Delta _j X_s^{M_{N(k)},N(k), (2)} \right\| _{L^2(\nu ^3)} \left\| \Delta _j {\mathcal Z}_s^{(1,M_{N(k)}, N(k))}\right\| _{L^6(\nu ^3)}^3 ds \right] .
\end{align*}
Since Proposition \ref{prop:Z} and Lemma \ref{lem:nontrivial4} imply
\begin{align*}
&E\left[ \liminf _{k\rightarrow \infty} \int _0^1 \left\| \Delta _j X_s^{M_{N(k)},N(k), (2)} \right\| _{L^2(\nu ^3)} \left\| \Delta _j {\mathcal Z}_s^{(1,M_{N(k)}, N(k))}\right\| _{L^6(\nu ^3)}^3 ds \right] \\
&\leq C 2^{[(1/2)+4\varepsilon]j} E\left[ \liminf _{k\rightarrow \infty} \int _0^1 \left\| X_s^{M_{N(k)},N(k), (2)} \right\| _{B_2^{1-\varepsilon}(\nu ^3)} \left\| {\mathcal Z}_s^{(1,M_{N(k)}, N(k))}\right\| _{B_6^{-(1/2)-\varepsilon}(\nu ^3)}^3 ds \right] \\
&\leq C 2^{[(1/2)+4\varepsilon]j} ,
\end{align*}
by applying these inequalities to \eqref{eq:thmPhi43measure1-05} we have
\begin{equation}\label{eq:thmPhi43measure1-06}\begin{array}{l}
\displaystyle \lambda \left| E\left[ \int _0^1 \lim _{k\rightarrow \infty} \int _{{\mathbb R}^3} \left( \Delta _j P_{M_{N(k)},N(k)} {\mathcal Z}_s^{(0,3,M_{N(k)}, N(k))}\right) \left( \Delta _j {\mathcal Z}_s^{(1,M_{N(k)}, N(k))} \right) ^3 \nu ^3 dx ds \right] \right| \\[4mm]
\displaystyle \leq C 2^{[(1/2)+4\varepsilon ]j}.
\end{array}\end{equation}
In view of Proposition \ref{prop:Z}, $\Delta _j {\mathcal Z}_s^{(0,3,M_{N(k)}, N(k))}$ and $\Delta _j {\mathcal Z}_s^{(1,M_{N(k)}, N(k))}$ have nice integrability with respect to both $s$ and the probability measure $P$.
Hence, by Lemmas \ref{lem:nontrivial1} and \ref{lem:nontrivial2} we have
\begin{align*}
&E\left[ \int _0^1 \lim _{k\rightarrow \infty} \int _{{\mathbb R}^3} \left( \Delta _j P_{M_{N(k)},N(k)} {\mathcal Z}_s^{(0,3,M_{N(k)}, N(k))}\right) \left( \Delta _j {\mathcal Z}_s^{(1,M_{N(k)}, N(k))} \right) ^3 \nu ^3 dx ds \right] \\
&= \lim _{k\rightarrow \infty} \int _0^1 \int _{{\mathbb R}^3} E\left[ \left\langle \Delta _j \delta _x , P_{M_{N(k)},N(k)} {\mathcal Z}_s^{(0,3,M_{N(k)}, N(k))} \right\rangle \left\langle \Delta _j \delta _x, {\mathcal Z}_s^{(1,M_{N(k)}, N(k))} \right\rangle ^3 \right] \\
&\quad \hspace{11cm}\times \nu ^3 (x) dx ds\\
&\geq c 2^j
\end{align*}
where $c$ is a positive constant.
This fact conflicts with \eqref{eq:thmPhi43measure1-06} and hence Theorem \ref{thm:Phi43measure}\ref{thm:Phi43measure3} is proven.
\qed

\appendix
\section{Appendix}

\begin{prop}\label{prop:Aint}
Let $X$ be a nonnegative random variable, $\theta \in (0,1)$ and let $c_1, c_2$ be nonnegative constants.
If for sufficiently large $n \in {\mathbb N}$ it holds that
\[
E\left[ \min \{ X,n\}\right] \leq c_1 E\left[ \min \{ X^\theta , n\}\right] + c_2,
\]
then
\[
E[X] \leq \frac{1}{1-\theta}\left( c_1^{1/(1-\theta )} + c_2\right).
\]
In particular, $X$ is integrable with respect to the probability measure.
\end{prop}

\begin{proof}
It is sufficient to prove the case where $c_1>0$.
Since
\begin{align*}
E\left[ \min \{ X^\theta , n\} \right] &= \int _0^\infty P\left( \min \{ X^\theta , n\} >\lambda \right) d\lambda \\
&\leq \int _0^n P\left( X> \lambda ^{1/\theta} \right) d\lambda ,
\end{align*}
by changing variables we have
\begin{align*}
c_1 E\left[ \min \{ X^\theta , n\} \right] &\leq \theta c_1 \int _0^{n^\theta} P\left( X> \eta \right) \eta ^{-(1-\theta )} d\eta \\
&\leq \theta \int _{c_1^{1/(1-\theta )}}^{n^\theta} P\left( X> \eta \right) d\eta + \theta c_1 \int _0^{c_1^{1/(1-\theta )}} \eta ^{-(1-\theta )} d\eta .
\end{align*}
Hence, we have
\[
c_1 E\left[ \min \{ X^\theta , n\} \right] \leq \theta E\left[ \min \{ X, n\} \right] + c_1^{1/(1-\theta )}.
\]
This inequality and the assumption yields
\[
E\left[ \min \{ X, n\} \right] \leq \frac{1}{1-\theta}\left( c_1^{1/(1-\theta )} + c_2\right) .
\]
By taking the limit as $n\rightarrow \infty$ we have the assertion.
\end{proof}


\begin{thebibliography}{100}


\bibitem{AlbStFlour}
S.~Albeverio,
\newblock Theory of {D}irichlet forms and applications,
\newblock In {\em Lectures on probability theory and statistics
  ({S}aint-{F}lour, 2000)}, vol. 1816 of {\em Lecture Notes in Math.}, pages
  1--106. Springer, Berlin, 2003.
  
\bibitem{AlBoDVGu}
S.~Albeverio, L.~Borasi, F.C.~De~Vecchi, and M.~Gubinelli,
\newblock Grassmannian stochastic analysis and the stochastic quantization of Euclidean fermions,
\newblock to appear in {\em Probab. Theory Relat. Fields}, https://doi.org/10.1007/s00440-022-01136-x (publ. online 21 May 2022), arXiv: 2004.09637.

\bibitem{AlBoRo}
S.~Albeverio, V.~Bogachev, and M.~R\"{o}ckner,
\newblock On uniqueness of invariant measures for finite- and
  infinite-dimensional diffusions,
\newblock {\em Comm. Pure Appl. Math.} 52(3): 325--362, 1999.


\bibitem{AlDVGu1}
S.~Albeverio, F.~C. De~Vecchi, and M.~Gubinelli,
\newblock Elliptic stochastic quantization,
\newblock {\em Annals of Probability} 48(4): 1693–1741, 2020.

\bibitem{AlDVGu2}
S.~Albeverio, F.~C. De~Vecchi, and M.~Gubinelli,
\newblock The elliptic stochastic quantization of some two dimensional
  {E}uclidean {QFT}s,
\newblock {\em Ann. Inst. Henri Poincar\'{e} Probab. Stat.} 57(4): 2372--2414, 2021.

\bibitem{AHFL}
S.~Albeverio, J.~E.~Fenstad, R.~H{\o}egh-Krohn, and T.~Lindstr{\o}m,
\newblock {\em Nonstandard methods in stochastic analysis and mathematical
  physics}, vol. 122 of {\em Pure and Applied Mathematics},
\newblock Academic Press, Inc., Orlando, FL, 1986.

\bibitem{AGH}
S.~Albeverio, G.~Gallavotti, and R.~H{\o}egh-Krohn,
\newblock Some results for the exponential interaction in two or more
  dimensions,
\newblock {\em Comm. Math. Phys.} 70(2): 187--192, 1979.

\bibitem{AlGieRu}
S.~Albeverio, R.~Gielerak, and F.~Russo, 
\newblock On the paths H\"older continuity in models of Euclidean quantum field theory,
\newblock {\em Stochastic Anal. Appl.} 19(5), 677--702, 2001. 

\bibitem{AHaRu2}
S.~Albeverio, Z.~Haba, and F.~Russo,
\newblock Trivial solutions for a non-linear two-space-dimensional wave
  equation perturbed by space-time white noise,
\newblock {\em Stochastics Stochastics Rep.} 56(1-2): 127--160, 1996.

\bibitem{AHaRu}
S.~Albeverio, Z.~Haba, and F.~Russo,
\newblock A two-space dimensional semilinear heat equation perturbed by
  ({G}aussian) white noise,
\newblock {\em Probab. Theory Related Fields} 121(3): 319--366, 2001.

\bibitem{AHPRS1}
S.~Albeverio, T.~Hida, J.~Potthoff, M.~R\"{o}ckner, and L.~Streit,
\newblock Dirichlet forms in terms of white noise analysis. {I}. {C}onstruction
  and {QFT} examples,
\newblock {\em Rev. Math. Phys.} 1(2-3): 291--312, 1989.

\bibitem{AHPRS2}
S.~Albeverio, T.~Hida, J.~Potthoff, M.~R\"{o}ckner, and L.~Streit,
\newblock Dirichlet forms in terms of white noise analysis. {II}. {C}losability
  and diffusion processes,
\newblock {\em Rev. Math. Phys.} 1(2-3): 313--323, 1989.

\bibitem{AHK73}
S.~Albeverio and R.~H{\o}egh-Krohn,
\newblock Uniqueness of the physical vacuum and the {W}ightman functions in the
  infinite volume limit for some non polynomial interactions,
\newblock {\em Comm. Math. Phys.} 30: 171--200, 1973.

\bibitem{AlHo}
S.~Albeverio and R.~H{\o}egh-Krohn,
\newblock The {W}ightman axioms and the mass gap for strong interactions of
  exponential type in two-dimensional space-time,
\newblock {\em J. Functional Analysis} 16: 39--82, 1974.

\bibitem{AlHK1}
S.~Albeverio and R.~H{\o}egh-Krohn,
\newblock Dirichlet forms and diffusion processes on rigged {H}ilbert spaces,
\newblock {\em Z. Wahrscheinlichkeitstheorie und Verw. Gebiete} 40(1): 1--57,
  1977.

\bibitem{AHKS}
S.~Albeverio, R.~H{\o}egh-Krohn, and L.~Streit,
\newblock Energy forms, {H}amiltonians, and distorted {B}rownian paths,
\newblock {\em J. Mathematical Phys.} 18(5): 907--917, 1977.

\bibitem{AHZ}
S.~Albeverio, R.~H{\o}egh-Krohn, and B.~Zegarli\'nski,
\newblock Uniqueness and global {M}arkov property for {E}uclidean fields: the
  case of general polynomial interactions,
\newblock {\em Comm. Math. Phys.} 123(3): 377--424, 1989.

\bibitem{AlYo}
S.~Albeverio, T.~Kagawa, Y.~Yahagi, and M.~W.~Yoshida,
\newblock Non-local Markovian symmetric forms on infinite dimensional spaces I, the closability and quasi-regularity,
\newblock {\em Comm. Math. Phys.} 388(2): 659--706, 2021.

\bibitem{AKaMR}
S.~Albeverio, H.~Kawabi, S.-R.~Mihalache, and M.~R\"ockner,
\newblock Strong uniqueness for {D}irichlet operators related to stochastic quantization under exponential/trigonometric interactions on the two-dimensional torus,
\newblock to appear in {\em  Ann. Sc. Norm. Super. Pisa Cl. Sci. (5)}, arXiv:2004.12383.

\bibitem{AKaR}
S.~Albeverio, H.~Kawabi, and M.~R\"ockner,
\newblock Strong uniqueness for both {D}irichlet operators and stochastic
  dynamics to {G}ibbs measures on a path space with exponential interactions,
\newblock {\em J. Funct. Anal.} 262(2): 602--638, 2012.

\bibitem{AKKR}
S.~Albeverio, Y.~Kondratiev, Y.~Kozitsky, and M.~R\"{o}ckner,
\newblock {\em The statistical mechanics of quantum lattice systems}, vol.~8
  of {\em EMS Tracts in Mathematics},
\newblock European Mathematical Society (EMS), Z\"{u}rich, 2009.

\bibitem{AKR}
S.~Albeverio, Y.~G.~Kondratiev, and M.~R\"ockner,
\newblock Ergodicity for the stochastic dynamics of quasi-invariant measures
  with applications to {G}ibbs states,
\newblock {\em J. Funct. Anal.} 149(2): 415--469, 1997.

\bibitem{AK}
S.~Albeverio and Sei.~Kusuoka,
\newblock The invariant measure and the flow associated to the
  ${\Phi}^4_3$-quantum field model,
\newblock {\em  Ann. Sc. Norm. Super. Pisa Cl. Sci. (5)},
  20(4): 1359--1427, 2020.

\bibitem{AlLi04}
S.~Albeverio and S.~Liang,
\newblock A remark on different lattice approximations and continuum limits for
  {$\phi^4_2$}-fields,
\newblock {\em Random Oper. Stochastic Equations} 12(4): 313--318, 2004.

\bibitem{AlLi}
S.~Albeverio and S.~Liang,
\newblock A note on the renormalized square of free quantum fields in
  space-time dimension {$d\geq4$},
\newblock {\em Bull. Sci. Math.} 131(1): 1--11, 2007.

\bibitem{ALZ}
S.~Albeverio, S.~Liang, and B.~Zegarlinski,
\newblock Remark on the integration by parts formula for the
  {$\phi^4_3$}-quantum field model,
\newblock {\em Infin. Dimens. Anal. Quantum Probab. Relat. Top.}
  9(1): 149--154, 2006.

\bibitem{AMR}
S.~Albeverio, Z.-M.~Ma, and M.~R\"ockner,
\newblock Quasi regular {D}irichlet forms and the stochastic quantization
  problem,
\newblock In {\em Festschrift {M}asatoshi {F}ukushima}, vol.~17 of {\em
  Interdiscip. Math. Sci.}, pages 27--58, World Sci. Publ., Hackensack, NJ,
  2015.

\bibitem{AlMa}
S.~Albeverio and S.~Mazzucchi,
\newblock An introduction to infinite-dimensional oscillatory and probabilistic
  integrals,
\newblock In {\em Stochastic analysis: a series of lectures}, vol.~68 of {\em
  Progr. Probab.}, pages 1--54, Birkh\"{a}user/Springer, Basel, 2015.

\bibitem{AlRo1}
S.~Albeverio and M.~R\"{o}ckner,
\newblock Classical {D}irichlet forms on topological vector spaces---the
  construction of the associated diffusion process,
\newblock {\em Probab. Theory Related Fields} 83(3): 405--434, 1989.

\bibitem{AlRo2}
S.~Albeverio and M.~R\"{o}ckner,
\newblock Classical {D}irichlet forms on topological vector
  spaces---closability and a {C}ameron-{M}artin formula,
\newblock {\em J. Funct. Anal.} 88(2): 395--436, 1990.

\bibitem{AR}
S.~Albeverio and M.~R\"ockner,
\newblock Stochastic differential equations in infinite dimensions: solutions
  via {D}irichlet forms,
\newblock {\em Probab. Theory Related Fields} 89(3): 347--386, 1991.

\bibitem{ARY}
S.~Albeverio, M.~R\"{o}ckner, and M.~W.~Yoshida,
\newblock Quantum fields,
\newblock In {\em Let us use white noise}, pages 37--65. World Sci. Publ.,
  Hackensack, NJ, 2017.

\bibitem{AlUg}
S.~Albeverio and S.~Ugolini,
\newblock A Doob h-transform of the Gross-Pitaevskii Hamiltonian,
\newblock {\em J. Stat. Phys.} 161(2), 486--508, 2015.

\bibitem{AlYo1}
S.~Albeverio and M.~W.~Yoshida,
\newblock {$H$}-{$C^1$} maps and elliptic {SPDE}s with polynomial and
  exponential perturbations of {N}elson's {E}uclidean free field,
\newblock {\em J. Funct. Anal.} 196(2): 265--322, 2002.

\bibitem{AlYo2}
S.~Albeverio and M.~W.~Yoshida,
\newblock Hida distribution construction of non-{G}aussian reflection positive
  generalized random fields,
\newblock {\em Infin. Dimens. Anal. Quantum Probab. Relat. Top.} 12(1): 21--49,
  2009.

\bibitem{BSZ}
J.~C.~Baez, I.~E.~Segal, and Z.-F.~Zhou,
\newblock {\em Introduction to algebraic and constructive quantum field
  theory},
\newblock Princeton Series in Physics. Princeton University Press, Princeton,
  NJ, 1992.

\bibitem{BCD}
H.~Bahouri, J.-Y.~Chemin, and R.~Danchin,
\newblock {\em Fourier analysis and nonlinear partial differential equations},
  vol. 343 of {\em Grundlehren der Mathematischen Wissenschaften [Fundamental
  Principles of Mathematical Sciences]},
\newblock Springer, Heidelberg, 2011.

\bibitem{BaHo}
I.~Bailleul and M.~Hoshino,
\newblock A tourist's guide to regularity structures,
\newblock arXiv: 2006.03524.

\bibitem{Bal}
T.~Ba{\l}aban,
\newblock Ultraviolet stability in field theory. {T}he {$\varphi _{3}^{4}$}
  model,
\newblock In {\em Scaling and self-similarity in physics ({B}ures-sur-{Y}vette,
  1981/1982)}, vol.~7 of {\em Progr. Phys.}, pages 297--319. Birkh\"{a}user
  Boston, Boston, MA, 1983.

\bibitem{BaDV}
N.~Barashkov and F.C.~De Vecchi,
\newblock Elliptic stochastic quantization of Sinh-Gordon QFT,
\newblock arXiv: 2108.12664 (27 Dec. 2021).

\bibitem{BaGu1}
N.~Barashkov and M.~Gubinelli,
\newblock A variational method for {$\Phi^4_3$},
\newblock {\em Duke Math. J.} 169(17): 3339--3415, 2020.

\bibitem{BaGu2}
N.~Barashkov and M.~Gubinelli,
\newblock The $\Phi^4_3$ measure via {G}irsanov's theorem,
\newblock {\em Electron. J. Probab.} 26, paper no. 81, 29pp, 2021.

\bibitem{Bat}
G.~Battle,
\newblock {\em Wavelets and renormalization}, vol.~10 of {\em Series in
  Approximations and Decompositions},
\newblock World Scientific Publishing Co., Inc., River Edge, NJ, 1999.

\bibitem{BCGNOPS78}
G.~Benfatto, M.~Cassandro, G.~Gallavotti, F.~Nicol\`o, E.~Olivieri,
  E.~Presutti, and E.~Scacciatelli,
\newblock Some probabilistic techniques in field theory,
\newblock {\em Comm. Math. Phys.} 59(2): 143--166, 1978.

\bibitem{BeGaNi}
G.~Benfatto, G.~Gallavotti, and Francesco Nicol\`o,
\newblock On the massive sine-{G}ordon equation in the first few regions of
  collapse,
\newblock {\em Comm. Math. Phys.} 83(3): 387--410, 1982.

\bibitem{BeKo1}
Y.~M. Berezansky and Y.~G. Kondratiev,
\newblock {\em Spectral methods in infinite-dimensional analysis. {V}ol. 1},
  vol. 12/1 of {\em Mathematical Physics and Applied Mathematics}.
\newblock Kluwer Academic Publishers, Dordrecht, 1995.
\newblock Translated from the 1988 Russian original by P. V. Malyshev and D. V.
  Malyshev and revised by the authors.

\bibitem{BeKo2}
Y.~M. Berezansky and Y.~G. Kondratiev,
\newblock {\em Spectral methods in infinite-dimensional analysis. {V}ol. 2},
  vol. 12/2 of {\em Mathematical Physics and Applied Mathematics}.
\newblock Kluwer Academic Publishers, Dordrecht, 1995.
\newblock Translated from the 1988 Russian original by P. V. Malyshev and D. V.
  Malyshev and revised by the authors.

\bibitem{BetzLor}
V.~Betz and J.~L\H{o}rinczi,
\newblock Uniqueness of {G}ibbs measures relative to {B}rownian motion,
\newblock {\em Ann. Inst. H. Poincar\'{e} Probab. Statist.} 39(5): 877--889,
  2003.

\bibitem{BLT}
N.~N. Bogolubov, A.~A. Logunov, and I.~T. Todorov,
\newblock {\em Introduction to axiomatic quantum field theory},
\newblock W. A. Benjamin, Inc., Reading, Mass.-London-Amsterdam, 1975.
\newblock Translated from the Russian by Stephen A. Fulling and Ludmila G.
  Popova, Edited by Stephen A. Fulling, Mathematical Physics Monograph Series,
  No. 18.

\bibitem{BChM}
V.~S. Borkar, R.~T. Chari, and S.~K. Mitter,
\newblock Stochastic quantization of field theory in finite and infinite
  volume,
\newblock {\em J. Funct. Anal.} 81(1): 184--206, 1988.

\bibitem{BoFe}
A.~Bovier and G.~Felder,
\newblock Skeleton inequalities and the asymptotic nature of perturbation
  theory for {$\Phi ^{4}$}-theories in two and three dimensions,
\newblock {\em Comm. Math. Phys.} 93(2): 259--275, 1984.

\bibitem{BDH}
D.~Brydges, J.~Dimock, and T.~R. Hurd,
\newblock The short distance behavior of {$(\phi^4)_3$},
\newblock {\em Comm. Math. Phys.} 172(1): 143--186, 1995.

\bibitem{BrFrSo}
D.~C. Brydges, J.~Fr\"ohlich, and A.~D. Sokal,
\newblock A new proof of the existence and nontriviality of the continuum
  {$\varphi ^{4}_{2}$}\ and {$\varphi ^{4}_{3}$}\ quantum field theories,
\newblock {\em Comm. Math. Phys.} 91(2): 141--186, 1983.

\bibitem{Burnap76}
C.~Burnap,
\newblock The particle structure of boson quantum field theory models,
\newblock Ph.D Thesis, Harvard University, 1976.

\bibitem{Burnap77}
C.~Burnap,
\newblock Isolated one particle states in boson quantum field theory models,
\newblock {\em Ann. Physics} 104(1): 184--196, 1977.

\bibitem{CaCh}
R.~Catellier and K.~Chouk,
\newblock Paracontrolled distributions and the 3-dimensional stochastic
  quantization equation,
\newblock {\em Ann. Probab.}, 46(5): 2621--2679, 2018.

\bibitem{ChChHaSh}
A.~Chandra, I.~Chevyrev, M.~Hairer, and H.~Shen,
\newblock Stochastic quantisation of Yang-Mills-Higgs in 3D,
\newblock arXiv: 2201.03487.

\bibitem{Const77}
F.~Constantinescu,
\newblock Nontriviality of the scattering matrix for weakly coupled $\phi ^4_3$
  models,
\newblock {\em Annals of Physics} 108(1): 37 -- 48, 1977.

\bibitem{CouRen75}
P.~Courr\`ege and P.~Renouard,
\newblock Oscillateur anharmonique mesures quasi-invariantes sur {$C({\bf R},\
  {\bf R})$} et th\'{e}orie quantique des champs en dimension {$d=1$},
\newblock In {\em Oscillateur anharmonique processus de diffusion et mesures
  quasi-invariantes}, pages 3--245. Ast\'{e}risque, No. 22--23. 1975.

\bibitem{DPDe}
G.~Da~Prato and A.~Debussche,
\newblock Strong solutions to the stochastic quantization equations,
\newblock {\em Ann. Probab.} 31(4): 1900--1916, 2003.

\bibitem{DPDe2}
G.~Da~Prato and A.~Debussche,
\newblock Gradient estimates and maximal dissipativity for the {K}olmogorov
  operator in {$\Phi^4_2$},
\newblock {\em Electron. Commun. Probab.} 25: Paper No. 9, 16, 2020.

\bibitem{DPTu}
G.~Da~Prato and L.~Tubaro,
\newblock Self-adjointness of some infinite-dimensional elliptic operators and
  application to stochastic quantization,
\newblock {\em Probab. Theory Related Fields} 118(1): 131--145, 2000.

\bibitem{DVGu}
F.~C.~De~Vecchi and M.~Gubinelli,
\newblock A note on supersymmetry and stochastic differential equations,
\newblock In {\em Geometry and Invariance in Stochastic Dynamics}, Springer Proceedings in Mathematics \& Statistics 378, pages 71--87, Springer, Cham, 2021.

\bibitem{Dimock74}
J.~Dimock,
\newblock Asymptotic perturbation expansion in the {$P(\phi)_{2}$} quantum
  field theory,
\newblock {\em Comm. Math. Phys.} 35: 347--356, 1974.

\bibitem{Dimock76}
J.~Dimock,
\newblock The non-relativistic limit of {$P(\phi )_{2}$} quantum field
  theories: two-particle phenomena,
\newblock {\em Comm. Math. Phys.} 57(1): 51--66, 1977.

\bibitem{Dimock14}
J.~Dimock,
\newblock The renormalization group according to {B}alaban {III}.
  {C}onvergence,
\newblock {\em Ann. Henri Poincar\'{e}} 15(11): 2133--2175, 2014.

\bibitem{DiEc}
J.~Dimock and J.-P. Eckmann,
\newblock Spectral properties and bound-state scattering for weakly coupled
  {$\lambda P(\varphi )_{2}$} models,
\newblock {\em Ann. Physics} 103(2): 289--314, 1977.

\bibitem{DoMi}
R.~L. Dobru\v{s}in and R.~A. Minlos,
\newblock Construction of a one-dimensional quantum field by means of a
  continuous {M}arkov field,
\newblock {\em Funkcional. Anal. i Prilo\v{z}en.} 7(4): 81--82, 1973.

\bibitem{Donald81}
M.~Donald,
\newblock The classical field limit of {$P(\varphi )_{2}$} quantum field
  theory,
\newblock {\em Comm. Math. Phys.} 79(2): 153--165, 1981.

\bibitem{Eberle}
A.~Eberle,
\newblock {\em Uniqueness and non-uniqueness of semigroups generated by
  singular diffusion operators}, vol. 1718 of {\em Lecture Notes in
  Mathematics}.
\newblock Springer-Verlag, Berlin, 1999.

\bibitem{Eck77}
J.-P.~Eckmann,
\newblock Asymptotic perturbation expansion for the {$S$}-matrix in {$P(\phi
  )_{2}$} quantum field theory models,
\newblock In {\em Quantum dynamics: models and mathematics ({P}roc. {S}ympos.,
  {C}entre for {I}nterdisciplinary {R}es., {B}ielefeld {U}niv., {B}ielefeld,
  1975)}, pages 59--68. Acta Phys. Austriaca, Suppl. XVI, 1976.

\bibitem{EcEp}
J.-P. Eckmann and H.~Epstein,
\newblock Time-ordered products and {S}chwinger functions,
\newblock {\em Comm. Math. Phys.} 64(2): 95--130, 1978/79.

\bibitem{EcEpFr}
J.-P. Eckmann, H.~Epstein, and J.~Fr\"{o}hlich,
\newblock Asymptotic perturbation expansion for the {S}-matrix and the
  definition of time ordered functions in relativistic quantum field models,
\newblock {\em Ann. Inst. H. Poincar\'{e} Sect. A (N.S.)} 25(1): 1--34,
  1976/77.

\bibitem{EckMaSe}
J.-P. Eckmann, J.~Magnen, and R.~S\'{e}n\'{e}or,
\newblock Decay properties and {B}orel summability for the {S}chwinger
  functions in {$P(\phi)_{2}$} theories,
\newblock {\em Comm. Math. Phys.} 39: 251--271, 1974/75.

\bibitem{EcOs71}
J.-P.~Eckmann and K.~Osterwalder,
\newblock On the uniqueness of the {H}amiltonian and of the representation of
  the {${\rm CCR}$} for the quartic boson interaction in three dimensions,
\newblock {\em Helv. Phys. Acta} 44: 884--909, 1971.

\bibitem{Fab}
J.~D. Fabrey,
\newblock Exponential representations of the canonical commutation relations,
\newblock {\em Comm. Math. Phys.} 19: 1--30, 1970.

\bibitem{Fed}
P.~Federbush,
\newblock Unitary renormalization of $[\phi^{4}]_{2+1}\star$,
\newblock {\em Comm. Math. Phys.} 21: 261--268, 1971.

\bibitem{Fe}
J.~Feldman,
\newblock The {$\lambda \varphi ^{4}_{3}$} field theory in a finite volume,
\newblock {\em Comm. Math. Phys.} 37: 93--120, 1974.

\bibitem{FeOs}
J.~S. Feldman and K.~Osterwalder,
\newblock The {W}ightman axioms and the mass gap for weakly coupled $(\phi
  ^4)_3$ quantum field theories,
\newblock {\em Annals of Physics} 97(1): 80--135, 1976.

\bibitem{FeRa}
J.~S.~Feldman and R.~R{\polhk{a}}czka,
\newblock The relativistic field equation of the {$\lambda \Phi ^{4}_{3}$}\
  quantum field theory,
\newblock {\em Ann. Physics} 108(1): 212--229, 1977.

\bibitem{FeFrSo}
R.~Fern\'andez, J.~Fr\"ohlich, and A.~D.~Sokal,
\newblock {\em Random walks, critical phenomena, and triviality in quantum
  field theory},
\newblock Texts and Monographs in Physics. Springer-Verlag, Berlin, 1992.

\bibitem{FrRu}
J.~Fritz and B.~R\"udiger,
\newblock Approximation of a one-dimensional stochastic {PDE} by local mean
  field type lattice systems,
\newblock In {\em Nonlinear stochastic {PDE}s ({M}inneapolis, {MN}, 1994)},
  vol.~77 of {\em IMA Vol. Math. Appl.}, pages 111--125. Springer, New York,
  1996.

\bibitem{Fr}
J.~Fr\"ohlich,
\newblock Verification of axioms for {E}uclidean and relativistic fields and
  {H}aag's theorem in a class of {$P(\varphi )_{2}$}-models,
\newblock {\em Ann. Inst. H. Poincar\'e Sect. A (N.S.)} 21: 271--317, 1974.

\bibitem{FrPa}
J.~Fr\"{o}hlich and Y.~M.~Park,
\newblock Remarks on exponential interactions and the quantum sine-{G}ordon
  equation in two space-time dimensions,
\newblock {\em Helv. Phys. Acta} 50(3): 315--329, 1977.

\bibitem{FrSei}
J.~Fr\"{o}hlich and E.~Seiler,
\newblock The massive {T}hirring-{S}chwinger model ({\rm {qed}}2): convergence
  of perturbation theory and particle structure,
\newblock {\em Helv. Phys. Acta} 49(6): 889--924, 1976.

\bibitem{FrSi77}
J.~Fr\"{o}hlich and B.~Simon,
\newblock Pure states for general {$P(\phi )_{2}$} theories: construction,
  regularity and variational equality,
\newblock {\em Ann. of Math. (2)} 105(3): 493--526, 1977.

\bibitem{FrSiSp}
J.~Fr\"ohlich, B.~Simon, and T.~Spencer,
\newblock Infrared bounds, phase transitions and continuous symmetry breaking,
\newblock {\em Comm. Math. Phys.} 50(1): 79--95, 1976.

\bibitem{FOT}
M.~Fukushima, Y.~\=Oshima, and M.~Takeda,
\newblock {\em Dirichlet forms and symmetric {M}arkov processes}, vol.~19 of
  {\em De Gruyter Studies in Mathematics},
\newblock Walter de Gruyter \& Co., Berlin, 1994.

\bibitem{Fu}
T.~Funaki,
\newblock Random motion of strings and related stochastic evolution equations.
\newblock {\em Nagoya Math. J.} 89: 129--193, 1983.

\bibitem{FHSX}
T.~Funaki, M.~Hoshino, S.~Sethuraman, and B.~Xie,
\newblock Asymptotics of {PDE} in random environment by paracontrolled calculus,
\newblock {\em Ann. Inst. H. Poincaré Probab. Statist.} 57(3): 1702--1735, 2021.

\bibitem{FuG}
M.~Furlan and M.~Gubinelli,
\newblock Weak universality for a class of 3d stochastic reaction-diffusion
  models,
\newblock {\em Probab. Theory Related Fields} 173(3-4): 1099--1164, 2019.

\bibitem{GawKup86}
K.~Gaw{\polhk{e}}dzki and A.~Kupiainen,
\newblock Asymptotic freedom beyond perturbation theory.
\newblock In {\em Ph\'{e}nom\`enes critiques, syst\`emes al\'{e}atoires,
  th\'{e}ories de jauge, {P}art {I}, {II} ({L}es {H}ouches, 1984)}, pages
  185--292. North-Holland, Amsterdam, 1986.

\bibitem{GLP}
G.~Giacomin, J.~L. Lebowitz, and E.~Presutti,
\newblock Deterministic and stochastic hydrodynamic equations arising from
  simple microscopic model systems,
\newblock In {\em Stochastic partial differential equations: six perspectives},
  vol.~64 of {\em Math. Surveys Monogr.}, pages 107--152. Amer. Math. Soc.,
  Providence, RI, 1999.

\bibitem{Gie83}
R.~Gielerak,
\newblock Verification of the global {M}arkov property in some class of
  strongly coupled exponential interactions,
\newblock {\em J. Math. Phys.} 24(2): 347--355, 1983.

\bibitem{Glaser}
V.~Glaser,
\newblock On the equivalence of the {E}uclidean and {W}ightman formulation of
  field theory,
\newblock {\em Comm. Math. Phys.} 37: 257--272, 1974.

\bibitem{Gl68}
J.~Glimm,
\newblock Boson fields with the {$:\Phi ^{4}:$} interaction in three
  dimensions.
\newblock {\em Comm. Math. Phys.} 10: 1--47, 1968.

\bibitem{GlJa1}
J.~Glimm and A.~Jaffe,
\newblock A {$\lambda \phi^{4}$} quantum field without cutoffs. I.
\newblock {\em Phys. Rev. (2)} 176: 1945–1951, 1968.

\bibitem{GlJa1.5}
J.~Glimm and A.~Jaffe,
\newblock Rigorous quantum field theory models.
\newblock {\em Bull. Amer. Math. Soc.} 76, 407–410, 1970.

\bibitem{GlJa2}
J.~Glimm and A.~Jaffe,
\newblock The {$\lambda (\Pi ^{4})_{2}$} quantum field theory without cutoffs. {II}. {T}he field operators and the approximate vacuum.
\newblock {\em Ann. of Math. (2)} 91: 362–401, 1970.

\bibitem{GlJa3}
J.~Glimm and A.~Jaffe,
\newblock The {$\lambda (\phi^{4})_{2}$} quantum field theory without cutoffs. {III}. {T}he physical vacuum.
\newblock {\em Acta Math.} 125: 203–267, 1970.

\bibitem{GlJa72}
J.~Glimm and A.~Jaffe,
\newblock Boson quantum field models,
\newblock In {\em Mathematics of contemporary physics ({P}roc. {I}nstructional
  {C}onf. ({NATO} {A}dvanced {S}tudy {I}nst.), {B}edford {C}oll., {L}ondon,
  1971)}, pages 77--143, 1972.

\bibitem{GlJa73}
J.~Glimm and A.~Jaffe,
\newblock Positivity of the {$\phi ^{4}_{3}$} {H}amiltonian,
\newblock {\em Fortschr. Physik} 21: 327--376, 1973.

\bibitem{GlJa}
J.~Glimm and A.~Jaffe,
\newblock {\em Quantum physics},
\newblock Springer-Verlag, New York-Berlin, 1981.

%

\bibitem{GlJaSp74}
J.~Glimm, A.~Jaffe, and T.~Spencer,
\newblock The {W}ightman axioms and particle structure in the {${\mathscr
  P}(\phi)_{2}$} quantum field model,
\newblock {\em Ann. of Math. (2)} 100: 585--632, 1974.

\bibitem{Gu1}
M.~Gubinelli,
\newblock Controlling rough paths,
\newblock {\em J. Funct. Anal.} 216(1): 86--140, 2004.

\bibitem{Gu2}
M.~Gubinelli,
\newblock Ramification of rough paths,
\newblock {\em J. Differential Equations} 248(4): 693--721, 2010.

\bibitem{GuHo}
M.~Gubinelli and M.~Hofmanov\'{a},
\newblock Global solutions to elliptic and parabolic {$\Phi^4$} models in
  {E}uclidean space,
\newblock {\em Comm. Math. Phys.} 368(3): 1201--1266, 2019.

\bibitem{GuHo2}
M.~Gubinelli and M.~Hofmanov\'a,
\newblock A {PDE} construction of the euclidean $\phi ^4_3$ quantum field
  theory,
\newblock {\em Comm. Math. Phys.} 384(1), 1--75, 2021.

\bibitem{GIP}
M.~Gubinelli, P.~Imkeller, and N.~Perkowski,
\newblock Paracontrolled distributions and singular {PDE}s,
\newblock {\em Forum Math. Pi} 3: e6, 75, 2015.

\bibitem{GRS72}
F.~Guerra, L.~Rosen, and B.~Simon,
\newblock Nelson's symmetry and the infinite volume behavior of the vacuum in
  {$P(\phi)_{2}$},
\newblock {\em Comm. Math. Phys.} 27: 10--22, 1972.

\bibitem{GRS}
F.~Guerra, L.~Rosen, and B.~Simon,
\newblock The {${\bf P}(\phi)_{2}$} {E}uclidean quantum field theory as
  classical statistical mechanics. {I}, {II},
\newblock {\em Ann. of Math. (2)} 101: 111--189; ibid. (2) 101 (1975),
  191--259, 1975.

\bibitem{Guerra72}
F.~Guerra,
\newblock Uniqueness of the vacuum energy density and van {H}ove phenomenon in
  the infinite-volume limit for two-dimensional self-coupled {B}ose fields,
\newblock {\em Phys. Rev. Lett.} 28: 1213--1215, 1972.

\bibitem{GueRug}
F.~Guerra and P.~Ruggiero,
\newblock New interpretation of the Euclidean-Markov field in the framework of
  physical Minkowski space-time,
\newblock {\em Phys. Rev. Lett.} 31: 1022--1025, 1973.

\bibitem{Ha}
M.~Hairer,
\newblock A theory of regularity structures,
\newblock {\em Invent. Math.} 198(2): 269--504, 2014.

\bibitem{HaIb}
M.~Hairer and M.~Iberti,
\newblock Tightness of the {I}sing--{K}ac model on the two-dimensional torus,
\newblock {\em Journal of Statistical Physics} 171(4): 632--655, 2018.

\bibitem{HaLa}
M.~Hairer and C.~Labb\'{e},
\newblock Multiplicative stochastic heat equations on the whole space,
\newblock {\em J. Eur. Math. Soc. (JEMS)} 20(4): 1005--1054, 2018.

\bibitem{HaMa}
M.~Hairer and K.~Matetski,
\newblock Discretisations of rough stochastic {PDE}s,
\newblock {\em Ann. Probab.} 46(3): 1651--1709, 2018.

\bibitem{HaMatt}
M.~Hairer and J.~Mattingly,
\newblock The strong {F}eller property for singular stochastic {PDE}s,
\newblock {\em Ann. Inst. Henri Poincar\'{e} Probab. Stat.} 54(3): 1314--1340,
  2018.

\bibitem{HaSh}
M.~Hairer and H.~Shen,
\newblock The dynamical sine-{G}ordon model,
\newblock {\em Comm. Math. Phys.} 341(3): 933--989, 2016.

\bibitem{HaX}
M.~Hairer and W.~Xu,
\newblock Large-scale behavior of three-dimensional continuous phase
  coexistence models,
\newblock {\em Comm. Pure Appl. Math.} 71(4): 688--746, 2018.

\bibitem{Heg}
G.~C. Hegerfeldt,
\newblock From {E}uclidean to relativistic fields and on the notion of
  {M}arkoff fields,
\newblock {\em Comm. Math. Phys.} 35: 155--171, 1974.

\bibitem{He}
K.~Hepp,
\newblock {\em Th\'{e}orie de la renormalisation},
\newblock Cours donn\'{e} \`a l'\'{E}cole Polytechnique, Paris. Lecture Notes
  in Physics, Vol. 2. Springer-Verlag, Berlin-New York, 1969.

\bibitem{HKPS}
T.~Hida, H.-H.~Kuo, J.~Potthoff, and L.~Streit,
\newblock {\em White noise}, vol. 253 of {\em Mathematics and its
  Applications},
\newblock Kluwer Academic Publishers Group, Dordrecht, 1993.

\bibitem{HK}
R.~H{\o}egh-Krohn,
\newblock A general class of quantum fields without cut-offs in two space-time
  dimensions,
\newblock {\em Comm. Math. Phys.} 21: 244--255, 1971.

\bibitem{HKK1}
M.~Hoshino, H.~Kawabi, and Sei.~Kusuoka,
\newblock Stochastic quantization associated with the $\exp (\phi) _2$-quantum
  field model driven by space-time white noise on the torus,
\newblock {\em Journal of Evolution Equations} 21(1): 339--375, 2021.

\bibitem{HKK2}
M.~Hoshino, H.~Kawabi, and Sei.~Kusuoka,
\newblock Stochastic quantization associated with the $\exp (\phi) _2$-quantum
  field model driven by space-time white noise on the torus in the full
  $L^1$-regime,
\newblock to appear in {\em Probability Theory and Related Fields}, https://doi.org/10.1007/s00440-022-01126-z (publ. online 13 May 2022), arXiv:2007.08171.

\bibitem{IW}
N.~Ikeda and S.~Watanabe,
\newblock {\em Stochastic differential equations and diffusion processes},
  vol.~24 of {\em North-Holland Mathematical Library},
\newblock North-Holland Publishing Co., Amsterdam; Kodansha, Ltd., Tokyo,
  second edition, 1989.

\bibitem{Iwa}
K.~Iwata,
\newblock An infinite-dimensional stochastic differential equation with state
  space {$C({\bf R})$},
\newblock {\em Probab. Theory Related Fields} 74(1): 141--159, 1987.

\bibitem{Iwa2}
K.~Iwata,
\newblock Reversible measures of a {$P(\phi)_1$}-time evolution,
\newblock In {\em Probabilistic methods in mathematical physics
  ({K}atata/{K}yoto, 1985)}, pages 195--209, Academic Press, Boston, MA, 1987.

\bibitem{Ja}
S.~Janson,
\newblock {\em Gaussian {H}ilbert spaces}, vol. 129 of {\em Cambridge Tracts
  in Mathematics},
\newblock Cambridge University Press, Cambridge, 1997.

\bibitem{JoMi}
G.~Jona-Lasinio and P.~K. Mitter,
\newblock On the stochastic quantization of field theory,
\newblock {\em Comm. Math. Phys.} 101(3): 409--436, 1985.

\bibitem{Jo1}
R.~Jost,
\newblock {\em The general theory of quantized fields}, vol. 1960 of {\em
  Mark Kac, editor. Lectures in Applied Mathematics (Proceedings of the Summer
  Seminar, Boulder, Colorado},
\newblock American Mathematical Society, Providence, R.I., 1965.

\bibitem{Kahane}
J.-P.~Kahane,
\newblock Sur le chaos multiplicatif,
\newblock {\em Ann. Sci. Math. Qu\'{e}bec} 9(2): 105--150, 1985.

\bibitem{KaRo}
H.~Kawabi and M.~R\"ockner,
\newblock Essential self-adjointness of {D}irichlet operators on a path space
  with {G}ibbs measures via an {SPDE} approach,
\newblock {\em J. Funct. Anal.} 242(2): 486--518, 2007.

\bibitem{Kl}
A.~Klein,
\newblock The semigroup characterization of {O}sterwalder-{S}chrader path
  spaces and the construction of {E}uclidean fields,
\newblock {\em J. Functional Analysis} 27(3): 277--291, 1978.

\bibitem{KLP84}
A.~Klein, L.~J.~Landau, and J.~F.~Perez,
\newblock Supersymmetry and the {P}arisi-{S}ourlas dimensional reduction: a
  rigorous proof,
\newblock {\em Comm. Math. Phys.} 94(4): 459--482, 1984.

\bibitem{KP83}
A.~Klein and J.~F.~Perez,
\newblock Supersymmetry and dimensional reduction: a nonperturbative proof,
\newblock {\em Phys. Lett. B} 125(6): 473--475, 1983.

\bibitem{Kup}
A.~Kupiainen,
\newblock Renormalization group and stochastic {PDE}s,
\newblock {\em Ann. Henri Poincar\'e} 17(3): 497--535, 2016.

\bibitem{Sei}
Sei.~Kusuoka,
\newblock An improvement of the integrability of the state space of the
  ${\Phi}^4_3$-process and the support of the ${\Phi}^4_3$-measure constructed
  by the limit of stationary processes of approximating stochastic quantization
  equations,
\newblock to appear in {\em Mathmatical Journal of Okayama University}, arXiv: 2102.07303.

\bibitem{Kusuoka}
S.~Kusuoka,
\newblock H\o egh-{K}rohn's model of quantum fields and the absolute continuity
  of measures,
\newblock In {\em Ideas and methods in quantum and statistical physics ({O}slo,
  1988)}, pages 405--424. Cambridge Univ. Press, Cambridge, 1992.

\bibitem{LeNew}
A.~Lenard and C.~M. Newman,
\newblock Infinite volume asymptotics in {$P(\phi)_{2}$} field theory,
\newblock {\em Comm. Math. Phys.} 39: 243--250, 1974.

\bibitem{MaRo}
Z.~M.~Ma and R\"ockner,
\newblock {\em Introduction to the theory of (nonsymmetric) Dirichlet forms},
\newblock Universitext, Springer-Verlag, Berlin, 1992.

\bibitem{MaSe76}
J.~Magnen and R.~S\'en\'eor,
\newblock The infinite volume limit of the {$\phi ^{4}_{3}$} model,
\newblock {\em Ann. Inst. H. Poincar\'e Sect. A (N.S.)} 24(2): 95--159, 1976.

\bibitem{MaSe77}
J.~Magnen and R.~S\'{e}n\'{e}or,
\newblock Phase space cell expansion and {B}orel summability for the
  {E}uclidean {$\phi _{3}^{4}$} theory,
\newblock {\em Comm. Math. Phys.} 56(3): 237--276, 1977.

\bibitem{Marcus}
R.~Marcus,
\newblock Stochastic diffusion on an unbounded domain,
\newblock {\em Pacific J. Math.} 84(1): 143--153, 1979.

\bibitem{MaPe}
J.~Martin and N.~Perkowski,
\newblock Paracontrolled distributions on {B}ravais lattices and weak
  universality of the 2d parabolic {A}nderson model,
\newblock {\em Ann. Inst. Henri Poincar\'{e} Probab. Stat.} 55(4): 2058--2110,
  2019.

\bibitem{MiRo}
R.~Mikulevicius and B.~L.~Rozovskii,
\newblock Martingale problems for stochastic PDE's,
\newblock In {\em Stochastic partial differential equations: six perspectives}, 243–325, Math. Surveys Monogr., 64, Amer. Math. Soc., 1999.

\bibitem{Miller15}
M.~E.~Miller,
\newblock The origins of {S}chwinger's {E}uclidean {G}reen's functions,
\newblock {\em Stud. Hist. Philos. Sci. B Stud. Hist. Philos. Modern Phys.}
  50: 5--12, 2015.

\bibitem{Mi}
S.~K.~Mitter,
\newblock Markov random fields, stochastic quantization and image analysis.
\newblock In {\em Applied and industrial mathematics ({V}enice, 1989)},
  vol.~56 of {\em Math. Appl.}, pages 101--109, Kluwer Acad. Publ.,
  Dordrecht, 1991.

\bibitem{MoWe}
A.~Moinat and H.~Weber,
\newblock Space-time localisation for the dynamic model,
\newblock {\em Communications on Pure and Applied Mathematics}
  73(12): 2519--2555, 2020.

\bibitem{MW17}
J.-C.~Mourrat and H.~Weber,
\newblock Convergence of the two-dimensional dynamic {I}sing-{K}ac model to
  {$\Phi^4_2$},
\newblock {\em Comm. Pure Appl. Math.} 70(4): 717--812, 2017.

\bibitem{MW3}
J.-C.~Mourrat and H.~Weber,
\newblock The dynamic ${\Phi}_3^4$ model comes down from infinity,
\newblock {\em Commun. Math. Phys.} 356(3): 673--753, 2017.

\bibitem{MW2}
J.-C.~Mourrat and H.~Weber,
\newblock Global well-posedness of the dynamic {$\Phi^4$} model in the plane,
\newblock {\em Ann. Probab.} 45(4): 2398--2476, 2017.

\bibitem{Ne}
E.~Nelson,
\newblock A quartic interaction in two dimensions,
\newblock In {\em Mathematical {T}heory of {E}lementary {P}articles ({P}roc.
  {C}onf., {D}edham, {M}ass., 1965)}, pages 69--73, M.I.T. Press, Cambridge,
  Mass., 1966.

\bibitem{Ne73c}
E.~Nelson,
\newblock Construction of quantum fields from {M}arkoff fields,
\newblock {\em J. Functional Analysis} 12: 97--112, 1973.

\bibitem{Ne73a}
E.~Nelson,
\newblock The free {M}arkoff field,
\newblock {\em J. Functional Analysis} 12: 211--227, 1973.

\bibitem{Ne73b}
E.~Nelson,
\newblock Probability theory and {E}uclidean field theory,
\newblock In G.~Velo and A.~Wightman, editors, {\em Constructive Quantum Field
  Theory}, vol.~25 of {\em Lecture Notes in Physics}, pages 94--124,
  Springer, Berlin, Heidelberg, 1973.

\bibitem{NeProc}
E.~Nelson,
\newblock Quantum fields and {M}arkoff fields,
\newblock In {\em Partial differential equations ({P}roc. {S}ympos. {P}ure
  {M}ath., {V}ol. {XXIII}, {U}niv. {C}alifornia, {B}erkeley, {C}alif., 1971)},
  pages 413--420, 1973.

\bibitem{Ne74}
E.~Nelson,
\newblock Remarks on {M}arkov field equations,
\newblock In {\em Functional integration and its applications ({P}roc.
  {I}nternat. {C}onf., {L}ondon, 1974)}, pages 136--143, 1975.

\bibitem{NelsonWienerIII}
E.~Nelson,
\newblock The case of the {W}iener process in quantum theory,
\newblock In P.~Masani, editor, {\em {C}ollected {W}orks of {N}orbert {Wiener},
  {V}ol. {III}}, pages 136--143, MIT Press, 1981.

\bibitem{NoPe}
I.~Nourdin and G.~Peccati,
\newblock {\em Normal approximations with {M}alliavin calculus}, volume 192 of
  {\em Cambridge Tracts in Mathematics},
\newblock Cambridge University Press, Cambridge, 2012.

\bibitem{OhRoSoWa}
T.~Oh, T.~Robert, P.~Sosoe, and Y.~Wang,
\newblock Invariant {G}ibbs dynamics for the dynamical {S}ine-{G}ordon model,
\newblock {\em Proceedings of the Royal Society of Edinburgh: Section A
  Mathematics}, page 1–17, 2020.

\bibitem{O2}
K.~Osterwalder,
\newblock Euclidean green's functions and wightman distributions,
\newblock In {\em In: Velo G., Wightman A. (eds) Constructive Quantum Field
  Theory}, vol.~25 of {\em Lecture Notes in Physics}, Springer, Berlin,
  Heidelberg, 1973.

\bibitem{O}
K.~Osterwalder,
\newblock Constructive quantum field theory: {S}calar fields,
\newblock In {\em Gauge Theories: Fundamental Interactions and rigorous
  results},
\newblock Birkh\"auser 1982.

\bibitem{OS1}
K.~Osterwalder and R.~Schrader,
\newblock Axioms for {E}uclidean {G}reen's functions,
\newblock {\em Comm. Math. Phys.} 31: 83--112, 1973.

\bibitem{OS2}
K.~Osterwalder and R.~Schrader,
\newblock Axioms for {E}uclidean {G}reen's functions. {II},
\newblock {\em Comm. Math. Phys.} 42: 281--305, 1975.

\bibitem{OsSen}
K.~Osterwalder and R.~S\'{e}n\'{e}or,
\newblock The scattering matrix is non-trivial for weakly coupled {$P(\phi
  )_{2}$} models,
\newblock {\em Helv. Phys. Acta} 49(4): 525--535, 1976.

\bibitem{PaSo}
G.~Parisi and N.~Sourlas,
\newblock Random magnetic fields, supersymmetry, and negative dimensions,
\newblock {\em Phys. Rev. Lett.} 43: 744--745, 1979.

\bibitem{PaWu}
G.~Parisi and Y.~S.~Wu,
\newblock Perturbation theory without gauge fixing,
\newblock {\em Sci. Sinica} 24(4): 483--496, 1981.

\bibitem{Pa4}
Y.~M.~Park,
\newblock Uniform bounds of the pressures of the {$\lambda \phi ^{4}_{3}$}
  field model,
\newblock {\em J. Mathematical Phys.} 17(6): 1073--1075, 1976.

\bibitem{Pa1}
Y.~M.~Park,
\newblock Convergence of lattice approximations and infinite volume limit in
  the {$(\lambda \phi ^{4}-\sigma \phi ^{2}-\tau \phi )_{3}$} field theory,
\newblock {\em J. Mathematical Phys.} 18(3): 354--366, 1977.

\bibitem{Ro1}
M.~R\"{o}ckner,
\newblock Generalized {M}arkov fields and {D}irichlet forms,
\newblock {\em Acta Appl. Math.} 3(3): 285--311, 1985.

\bibitem{Ro88}
M.~R\"{o}ckner,
\newblock Traces of harmonic functions and a new path space for the free
  quantum field,
\newblock {\em J. Funct. Anal.} 79(1): 211--249, 1988.

\bibitem{RZZ2}
M.~R\"{o}ckner, R.~Zhu, and X.~Zhu,
\newblock Ergodicity for the stochastic quantization problems on the
  2{D}-torus,
\newblock {\em Comm. Math. Phys.} 352(3): 1061--1090, 2017.

\bibitem{RZZ}
M.~R\"ockner, R.~Zhu, and X.~Zhu,
\newblock Restricted {M}arkov uniqueness for the stochastic quantization of
  {$P(\Phi)_2$} and its applications,
\newblock {\em J. Funct. Anal.} 272(10): 4263--4303, 2017.

\bibitem{Scheinberg}
S.~Scheinberg,
\newblock Fatou's lemma in normed linear spaces,
\newblock {\em Pacific J. Math.} 38: 233--238, 1971.

\bibitem{Schweber}
S.~S.~Schweber,
\newblock {\em Q{ED} and the men who made it: {D}yson, {F}eynman, {S}chwinger,
  and {T}omonaga},
\newblock Princeton Series in Physics, Princeton University Press, Princeton,
  NJ, 1994.

\bibitem{SeiSimon}
E.~Seiler and B.~Simon,
\newblock Nelson's symmetry and all that in the {Y}ukawa2 and {$(\phi
  ^{4})_{3}$} field theories,
\newblock {\em Ann. Physics} 97(2): 470--518, 1976.

\bibitem{ShZZ}
H.~Shen, R.~Zhu, and X.~Zhu,
\newblock An SPDE approach to perturbation theory of $\Phi ^4_2$: asymptoticity and short distance behavior,
\newblock arXiv: 2108.11312.

\bibitem{Shi95}
I.~Shigekawa,
\newblock An example of regular {$(r,p)$}-capacity and essential
  self-adjointness of a diffusion operator in infinite dimensions,
\newblock {\em J. Math. Kyoto Univ.} 35(4): 639--651, 1995.

\bibitem{Si}
B.~Simon,
\newblock {\em The {$P(\phi )_{2}$} {E}uclidean (quantum) field theory},
\newblock Princeton University Press, Princeton, N.J., 1974.

\bibitem{Si05}
B.~Simon,
\newblock {\em Functional integration and quantum physics},
\newblock second edition, AMS Chelsea Publishing, 2005.

\bibitem{SpZi}
T.~Spencer and F.~Zirilli,
\newblock Scattering states and bound states in {$\lambda {\cal P}(\phi
  )_{2}$},
\newblock {\em Comm. Math. Phys.} 49(1): 1--16, 1976.

\bibitem{Stein}
E.~M.~Stein,
\newblock {\em Singular integrals and differentiability properties of
  functions},
\newblock Princeton Mathematical Series, No. 30, Princeton University Press,
  Princeton, N.J., 1970.

\bibitem{Stea81}
R.~F.~Streater,
\newblock Euclidean quantum mechanics and stochastic integrals,
\newblock In {\em Stochastic integrals ({P}roc. {S}ympos., {U}niv. {D}urham,
  {D}urham, 1980)}, vol. 851 of {\em Lecture Notes in Math.}, pages 371--393,
  Springer, Berlin-New York, 1981.

\bibitem{StWi}
R.~F.~Streater and A.~S.~Wightman,
\newblock {\em P{CT}, spin and statistics, and all that},
\newblock Benjamin/Cummings Publishing Co., Inc., Reading,
  Mass.-London-Amsterdam, second edition, 1978.

\bibitem{Stro}
F.~Strocchi,
\newblock {\em An introduction to non-perturbative foundations of quantum field
  theory}, vol. 158 of {\em International Series of Monographs on Physics},
\newblock Oxford University Press, Oxford, 2013.

\bibitem{Sum1}
S.~J.~Summers,
\newblock A new proof of the asymptotic nature of perturbation theory in
  {$P(\varphi )_{2}$} models,
\newblock {\em Helv. Phys. Acta} 53(1): 1--30, 1980.

\bibitem{Sum2}
S.~J.~Summers,
\newblock A perspective on constructive quantum field theory,
\newblock In {\em Fundamentals of Physics}, Encyclopedia of Life Support Systems (Eolss Publishers), Oxford, 2012,
\newblock arXiv: 1203.3991.

\bibitem{Sy2}
K.~Symanzik,
\newblock Euclidean quantum field theory. i. equations for a scalar model,
\newblock {\em Journal of Mathematical Physics} 7(3): 510--525, 1966.

\bibitem{Sy}
K.~Symanzik,
\newblock Euclidean quantum field theory,
\newblock In {\em Varenna 1968, {P}roceedings Of The Physics School On Local
  Quantum Theory}, pages 152--226, Academic Press, New York, 1969.

\bibitem{Takeda}
M.~Takeda,
\newblock On the uniqueness of {M}arkovian selfadjoint extension of diffusion
  operators on infinite-dimensional spaces,
\newblock {\em Osaka J. Math.} 22(4): 733--742, 1985.

\bibitem{TsWe}
P.~Tsatsoulis and H.~Weber,
\newblock Exponential loss of memory for the 2-dimensional {A}llen-{C}ahn
  equation with small noise,
\newblock {\em Probab. Theory Related Fields} 177(1-2): 257--322, 2020.

\bibitem{Wi}
A.~S. Wightman,
\newblock Quantum field theory in terms of vacuum expectation values,
\newblock {\em Phys. Rev. (2)} 101: 860--866, 1956.

\bibitem{Ze84}
B.~Zegarli\'{n}ski,
\newblock Uniqueness and the global {M}arkov property for {E}uclidean fields:
  the case of general exponential interaction,
\newblock {\em Comm. Math. Phys.} 96(2): 195--221, 1984.

\bibitem{Zinoviev}
Y.~M. Zinoviev,
\newblock Equivalence of {E}uclidean and {W}ightman field theories,
\newblock {\em Comm. Math. Phys.} 174(1): 1--27, 1995.

\end{thebibliography}

\def\polhk#1{\setbox0=\hbox{#1}{\ooalign{\hidewidth
  \lower1.5ex\hbox{`}\hidewidth\crcr\unhbox0}}}

\end{document}